\documentclass{amsart}

\usepackage[utf8]{inputenc}
\usepackage{amsmath}
\usepackage{amssymb}
\usepackage{mathdots}
\usepackage{mathtools}
\usepackage{mathrsfs}
\usepackage{tikz}
\usepackage{tikz-cd}
\usetikzlibrary{decorations.markings,arrows.meta,quotes}
\usepackage[inline]{enumitem}
\usepackage{color}
\definecolor{linkcolor}{rgb}{0,0,0.6}
\usepackage{amsthm}
\usepackage{bbm}
\usepackage{multicol}

\tikzset{
  symbol/.style={
    draw=none,
    every to/.append style={
      edge node={node [sloped, allow upside down, auto=false]{$#1$}}}
  }
}

\DeclareUnicodeCharacter{2212}{\textendash}

\usepackage[english]{babel}
\usepackage{csquotes}

\usepackage[colorlinks=true,
pdfstartview=Fit,
linkcolor= linkcolor,
citecolor= linkcolor,
urlcolor= linkcolor,
hyperindex=true,
hyperfigures=false]
{hyperref}

\usepackage{orcidlink}

\title[EKOR and BT stratifications]{EKOR and BT stratifications for basic unramified $\mathrm{GU}(1,n-1)$ Rapoport-Zink spaces}
\author{Joseph Muller\,\orcidlink{0000-0002-1546-0910}}\thanks{NCTS, National Taiwan University, muller@ncts.ntu.edu.tw}
\date{}

\RenewCommandCopy{\underbrace}{\LaTeXunderbrace}

\numberwithin{equation}{subsection} 

\begin{document}


\newtheorem{introtheorem}{Theorem}
\renewcommand{\theintrotheorem}{\Alph{introtheorem}}
\newtheorem{introproposition}[introtheorem]{Proposition}
\newtheorem{theo}{Theorem}[section]
\newtheorem{prop}[theo]{Proposition}
\newtheorem{lem}[theo]{Lemma}
\newtheorem{corol}[theo]{Corollary}
\newtheorem{conj}[theo]{Conjecture}
\newtheorem*{theo*}{Theorem}
\newtheorem*{lem*}{Lemma}
\newtheorem*{prop*}{Proposition}
\newtheorem*{corol*}{Corollary}

\theoremstyle{remark}
\newtheorem{rk}[theo]{Remark}
\newtheorem{rks}[theo]{Remarks}
\newtheorem{ex}[theo]{Example}

\theoremstyle{definition}
\newtheorem{defi}[theo]{Definition}
\newtheorem*{notation}{Notation}
\newtheorem*{notations}{Notations}


\begin{abstract} 
In this paper, we establish the relation between the Ekedahl-Kottwitz-Oort-Rapoport stratification and the Bruhat-Tits stratification on basic unramified $\mathrm{GU}(1,n-1)$ Rapoport-Zink spaces with arbitrary parahoric level. More precisely, we prove that every basic EKOR stratum is a disjoint union of copies of an explicitly defined fine Deligne-Lusztig variety. Using this comparison and the closure relations for fine Deligne-Lusztig varieties, we give an explicit combinatorial description of the closure of every basic EKOR stratum. As further applications, we determine which KR strata are entirely contained in the basic locus, and we prove the smoothness of the irreducible components of the closure of certain EKOR strata.
\end{abstract}

\maketitle

\tableofcontents

\vspace{1.5cm}

\section{Introduction} 

The study of the reduction modulo $p$ of Shimura varieties has long been guided by the existence of natural stratifications, interacting non-trivially with each other. In the case of good reduction, the prototypical example is the Siegel modular variety. Its special fiber is stratified by several invariants of the associated $p$-divisible group, such as its $p$-rank, its Newton polygon, the isomorphism class of its $p$-torsion group scheme, or even the isomorphism class of the full $p$-divisible group. These give rise respectively to the $p$-rank stratification, the Newton stratification, the Ekedahl-Oort (EO) stratification, and the decomposition into central leaves. The Ekedahl-Oort stratification was first studied systematically in the Siegel case by Oort \cite{oortStratificationModuliSpace2001}. The group-theoretic framework underlying this construction was later clarified through the language of $F$-zips and $G$-zips in the work of Moonen and Wedhorn, and Pink, Wedhorn and Ziegler \cite{moonenDiscreteInvariantsVarieties2004, pinkAlgebraicZipData2011, pinkFzipsAdditionalStructure2015}. This point of view allowed the Ekedahl-Oort stratification to be extended from the Siegel case to Shimura varieties of PEL type, Hodge type, and abelian type, see for instance \cite{viehmannEkedahlOortNewton2013, shenEKORStrataShimura2021, shenStratificationsGoodReductions2022}. The Newton stratification, on the other hand, records the isogeny class of the associated $p$-divisible group, or equivalently of the corresponding $F$-isocrystal with additional structure. Its group-theoretic foundations go back to Kottwitz, and to Rapoport and Richartz \cite{kottwitzIsocrystalsAdditionalStructure1997, rapoportClassificationSpecializationIsocrystals1996}. In the Siegel case, the geometry of Newton strata was studied extensively by Oort \cite{oortNewtonPolygonStrata2001}. \\
At parahoric level, the geometry of the special fiber is significantly governed by the corresponding local model. This leads to the Kottwitz-Rapoport (KR) stratification, whose strata are indexed by admissible elements in the Iwahori-Weyl group. The relevant combinatorics for $\mathrm{GL}_n$ and $\mathrm{GSp}_{2n}$ was introduced by Kottwitz and Rapoport in \cite{kottwitzMinusculeAlcovesGL2000} and used in \cite{gortzFlatnessModelsCertain2001} in regard to the geometry of the local model. Building on work of He and Lusztig on decomposition into $G$-stable pieces, the Ekedahl-Kottwitz-Oort-Rapoport (EKOR) stratification was introduced by He and Rapoport in \cite{heStratificationsReductionShimura2017} through an axiomatic and group-theoretic framework, interpolating between the EO stratification at hyperspecial level, and the KR stratification at Iwahori level in the reduction of Shimura varieties. At general parahoric level, the EKOR stratification is always finer than the KR stratification. A geometric construction of the EKOR stratification for Kisin-Pappas integral models of abelian type, using $G$-zips and mixed-characteristic local shtukas, was given by Shen, Yu and Zhang in \cite{shenEKORStrataShimura2021}, see also \cite{hesseEKORStrataShimura2020, hoffEKORstratificationSiegelModular2025} for related constructions in the Siegel case. \\
By \cite{heStratificationsReductionShimura2017}, it is well known that any Newton stratum contains at least one EKOR stratum entirely. In more favorable situations, namely when the local Shimura datum is fully Hodge-Newton decomposable, then Görtz, He and Nie proved in \cite{gortzFullyHodgeNewton2019} that the EKOR stratification is actually finer than the Newton stratification. Such can not be said of the KR strata, which usually intersect multiple Newton strata non-trivially. We note that \cite{gortzFullyHodgeNewton2019} also provides a complete classification of all cases which are fully Hodge-Newton decomposable. Incidentally, another characterization of fully Hodge-Newton decomposable local Shimura datum is the fact that the basic affine Deligne-Lusztig variety is naturally stratified by fine Deligne-Lusztig varieties, up to perfection. This is referred to as the (weak) Bruhat-Tits (BT) stratification in \cite{gortzBasicLociCoxeter2024}, but in this paper we suggest calling it the \textit{fine Bruhat-Tits stratification} instead, and we reserve the name ``BT stratification'' for the slightly coarser decomposition we recall in Section \ref{Section5}. Geometrically, the basic locus of a Shimura variety -- that is, the unique minimal Newton stratum -- is uniformized by some associated Rapoport-Zink space, over which the BT stratification has been constructed in various cases, for instance by Vollaard and Wedhorn, by Cho, and by the author \cite{vollaardSupersingularLocusShimura2010, vollaardSupersingularLocusShimura2011, choBasicLocusUnitary2018, mullerBruhatTitsStratificationBasic2026} for the unramified $\mathrm{GU}(1,n-1)$ Rapoport-Zink space at respectively hyperspecial, maximal and arbitrary parahoric levels; by Rapoport, Terstiege and Wilson, by Wu and by He, Luo and Shi \cite{rapoportSupersingularLocusShimura2014, wuSupersingularLocusUnitary2016, heBasicLocusRamified2026} for the ramified $\mathrm{GU}(1,n-1)$ Rapoport-Zink space at respectively self-dual level, exotic good reduction and maximal lattice stabilizer levels. See also \cite{zachosBasicLocusUnitary2025, zachosBasicLocusRegular2025} for other integral models in the ramified case. Other examples include $\mathrm{GSpin}$ cases \cite{howardRapoportZinkSpaces2017, okiRapoportZinkSpacesSpinor2020, heBasicLocusGSpin2025}, quaternionic cases \cite{wangBruhatTitsStratification2020, okiSupersingularLociShimura2022, terakadoSupersingularLocusShimura2023}, and low rank cases \cite{howardSupersingularLocusGU222014, wangBruhatTitsStratificationGU2021, okiSupersingularLocusShimura2023}, among many others.\\
In the fully Hodge-Newton decomposable case, the EKOR strata in the basic affine Deligne-Lusztig variety are a disjoint union of copies of some fine Deligne-Lusztig variety. It is natural to expect the same result to hold geometrically, namely over the associated basic Rapoport-Zink space and with respect to the EKOR strata constructed in \cite{shenEKORStrataShimura2021}. The first goal of this paper is to prove precisely this fact in the case of unramified $\mathrm{GU}(1,n-1)$ Rapoport-Zink spaces at arbitrary parahoric levels, and in doing so to explicitly describe the relation between the BT stratification built by the author in \cite{mullerBruhatTitsStratificationBasic2026} and the EKOR stratification. We then use this comparison, together with the closure relations for fine Deligne-Lusztig varieties, to determine explicitly the closure of every EKOR stratum. Let us state our results more precisely.

\subsection*{The Rapoport-Zink space}

We fix an integer $n \geq 2$, a prime $p > 2$, a $p$-adic field $F$ with residue field $\kappa_F = \mathbb F_q$ and uniformizer $\pi$, and a quadratic unramified extension $E/F$. Let $\mathbb V$ be an $n$-dimensional non-degenerate $E/F$-hermitian space, and let $G = \mathrm{GU}(\mathbb V)$ denote the associated group of unitary similitudes, seen as a connected reductive group over $\mathrm{Spec}(F)$. The discriminant $\varepsilon$ of $\mathbb V$ is allowed to take any value in $\{\pm 1\}$, so that $G$ may not be quasi-split: this happens when $n$ is even and $\varepsilon = -1$. Let $\breve F$ denote the completion of the maximal unramified extension of $F$, with residue field $\kappa_{\breve F} = \overline{\mathbb F_q}$, and let $\sigma \in \mathrm{Gal}(\breve F/F)$ be the automorphism induced by the Frobenius $x \mapsto x^q$ on the residue fields. Fix a quadratic unramified extension $K/F$ embedded in $\breve F$. The two $F$-isomorphisms $E \simeq K$ induce a decomposition
\begin{equation*}
    \mathbb V_K \coloneq \mathbb V \otimes_F K = \mathbb V_{K,0} \oplus \mathbb V_{K,1}.
\end{equation*}
Providing $\mathbb V_K$ with Frobenius and Verschiebung operators induced by some basic element $b \in G(K) \subseteq G(\breve F)$, we equip the $K$-vector space $\mathbb V_{K,0}$ with a non-degenerate hermitian form, and consider the group $J = J_b \simeq \mathrm{GU}(\mathbb V_{K,0})$. The discriminant of $\mathbb V_{K,0}$ is equal to $(-1)^{n-1}\varepsilon$. We define two integers
\begin{align*}
    \underline{\varepsilon} \coloneq \frac{1-\varepsilon}{2} = \begin{cases}
        0 & \text{if } \varepsilon = 1, \\
        1 & \text{if } \varepsilon = -1,
    \end{cases} & & \nu_{\varepsilon} \coloneq \left\lfloor \frac{n-\underline{\varepsilon}}{2} \right\rfloor.
\end{align*}
We note that $\nu_{\varepsilon} = \nu_{-\varepsilon} = \frac{n-1}{2}$ when $n$ is odd, but $|\nu_{\varepsilon}-\nu_{-\varepsilon}|=1$ when $n$ is even. The maximal parahoric subgroups of $G(F)$ are given by the stabilizers of some $\mathcal O_E$-lattice $L \subset \mathbb V$ satisfying $\pi L^{\dagger} \subseteq L \subseteq L^{\dagger}$, where $L^{\dagger}$ denotes the dual lattice. The index $h = [L:\pi L^{\dagger}]$ can be any integer $0 \leq h \leq n$ such that $h \equiv \underline{\varepsilon} \mod 2$. Thus, $h = 2r + \underline{\varepsilon}$ with $r \in \{0, \ldots , \nu_{\varepsilon}\}$. For any such $r$, we write $h_r^{\varepsilon} \coloneq 2r + \underline{\varepsilon}$. A general parahoric subgroup is an intersection of maximal parahoric subgroups as above. We fix a self-dual decreasing lattice chain $L_{\bullet} = (L_r)_{r \in \mathbb Z}$ in $\mathbb V$, and consider, for any non-empty subset $I \subseteq \{0, \ldots , \nu_{\varepsilon}\}$, the associated standard parahoric subgroup $K_I$. For instance, if $I = I_{\mathrm{Iw}} \coloneq \{0, \ldots , \nu_{\varepsilon}\}$, then $K_{I_{\mathrm{Iw}}}$ is an Iwahori subgroup. For every $r \in \mathbb Z$, we write $\overline r$ for its image in $\mathbb Z/n\mathbb Z$. We define
\begin{align*}
    \overline I & \coloneq \{\overline r \mid r \in I\} \cup \{\overline{-r-\underline{\varepsilon}} \mid r\in I\} \subseteq \mathbb Z/n\mathbb Z,\\
    I_{\mathbb Z} & \coloneq \{r \in \mathbb Z \mid \overline r \in \overline I\} \subseteq \mathbb Z.
\end{align*}
For every $i \in \overline I$, we define
\begin{equation*}
    \ell_i \coloneq \min\left\{\ell \in \{1,\ldots ,n\} \mid i+\ell \in \overline I\right\}.
\end{equation*}
For example, $\overline{I_{\mathrm{Iw}}}=\mathbb Z/n\mathbb Z$ and then $\ell_i=1$ for every $i$. If $I=\{r\}$ for some $0\leq r\leq\nu_{\varepsilon}$, then $\overline I=\{\overline r,\overline{-r-\underline{\varepsilon}}\}$ and
\begin{equation*}
    \ell_{\overline r} = \begin{cases}
        n-h_r^{\varepsilon} & \text{if } h_r^{\varepsilon} \not = n,\\
        n & \text{if } h_r^{\varepsilon} = n,
    \end{cases} \qquad \ell_{\overline{-r-\underline{\varepsilon}}} = \begin{cases}
        h_r^{\varepsilon} & \text{if } h_r^{\varepsilon} \not = 0,\\
        n & \text{if } h_r^{\varepsilon} = 0.
    \end{cases}
\end{equation*}
Finally, we fix a conjugacy class of cocharacters $\mu:\mathbb G_{m,K}\to G_K$ with dominant coweight $(1,0^{n-1})$. The (relative) Rapoport-Zink space $\mathcal N^{I,\varepsilon}$ of parahoric level $K_I$ is a regular formal scheme over $\mathrm{Spf}(\mathcal O_{\breve F})$, locally formally of finite type. We write $\mathcal N_{\mathrm{red}}^{I,\varepsilon}$ for its reduced special fiber over $\mathrm{Spec}(\overline{\mathbb F_q})$. Via covariant Dieudonn\'e theory, its geometric points can be described as follows.
\begin{introproposition}[Proposition \ref{Prop:PointsArbitraryParahoricBis}]\label{IntropropA}
    There is a bijection
    \begin{equation*}\hspace{-1cm}
        \mathcal N^{I,\varepsilon}(\overline{\mathbb F_q}) \simeq \left\{ \begin{array}{c}
        \text{Decreasing chains }(Z_r)_{r\in I_{\mathbb Z}}\\
        \text{of }\mathcal O_{\breve F}\text{-lattices in }\mathbb V_{\breve F,0}
        \end{array}
        \;\middle|\; \exists j\in\mathbb Z,\ \forall r\in I_{\mathbb Z},
        \begin{array}{l}
        1)\ \ \pi^{j+1}Z_{-r-\underline{\varepsilon}}^{\vee}\overset{1}{\subseteq}Z_r\subseteq\pi^jZ_{-r-\underline{\varepsilon}}^{\vee},\\
        2)\ \ Z_{r+\ell_{\overline r}}\overset{\ell_{\overline r}}{\subseteq}Z_r,\\
        3)\ \ Z_{r+n}=\pi Z_r.
        \end{array}\right\}.
    \end{equation*}
\end{introproposition}
Here, $\mathbb V_{\breve F,0}=\mathbb V_{K,0}\otimes_K\breve F$ is equipped with the non-degenerate $\breve F$-valued $\sigma$-sesquilinear pairing obtained by scalar extension, and given a lattice $Z\subset\mathbb V_{\breve F,0}$ we write $Z^{\vee}$ for the dual lattice with respect to this pairing. Moreover, for $x\geq0$ the notation $Z'\overset{x}{\subseteq}Z$ means $Z'\subseteq Z$ and $x=[Z:Z']$. \\
For $j\in\mathbb Z$, there is an open and closed formal subscheme $\mathcal N_j^{I,\varepsilon}\hookrightarrow\mathcal N^{I,\varepsilon}$ whose $\overline{\mathbb F_q}$-points correspond to those $z=(Z_r)_{r\in I_{\mathbb Z}}$ satisfying condition 1) of Proposition \ref{IntropropA} with respect to $j$. We have $\mathcal N_j^{I,\varepsilon}\cap\mathcal N_{j'}^{I,\varepsilon}=\emptyset$ if $j\not=j'$, and $\mathcal N_j^{I,\varepsilon}\not=\emptyset$ if and only if $nj$ is even.

\subsection*{KR and EKOR stratifications}

The local model associated with $\mathcal N^{I,\varepsilon}$ is isomorphic to the classical parahoric local model studied in \cite{gortzFlatnessModelsCertain2001} for the cocharacter $\mu'\coloneq\mu^{\sigma}$ with dominant coweight $(1^{n-1},0)$ (the ``Drinfeld case''). Here, the occurrence of $\mu'$ rather than $\mu$ appears to be the natural choice when using \textit{covariant} Dieudonn\'e theory. The local model carries an action of the parahoric group scheme associated with $K_I$, and the orbits on its special fiber are indexed by the set $\mathrm{Adm}(\{\mu'\})_I$, which is the image in a double quotient space of the $\{\mu'\}$-admissible set $\mathrm{Adm}(\{\mu'\})$ in the Iwahori-Weyl group. At Iwahori level, we simply have $\mathrm{Adm}(\{\mu'\})_{I_{\mathrm{Iw}}}=\mathrm{Adm}(\{\mu'\})$. The Bruhat order induces a partial order on $\mathrm{Adm}(\{\mu'\})$ and $\mathrm{Adm}(\{\mu'\})_I$. Since any element of $\mathrm{Adm}(\{\mu'\})$ turns out to be entirely determined by its support, we obtain the following elementary description.
\begin{introproposition}[Proposition \ref{Prop:IsomAdmMuPoset}]
    The poset $\mathrm{Adm}(\{\mu'\})_I$ is isomorphic to the poset $\mathcal P_I\coloneq\{J\mid J\subsetneq \overline I\}$ consisting of the proper subsets of $\overline I$, ordered by inclusion.
\end{introproposition}
For every $J\in\mathcal P_I$, we obtain a potentially empty KR stratum $\mathrm{KR}_J$, which is a locally closed subscheme of $\mathcal N_{\mathrm{red}}^{I,\varepsilon}$. We can describe its points as follows.
\begin{introproposition}[Proposition \ref{Prop:PointsOfKRStrata}]\label{IntropropC}
    For $J\in\mathcal P_I$ and $j$ such that $nj$ is even, there is a bijection
    \begin{equation*}
        (\mathrm{KR}_J\cap\mathcal N_{j,\mathrm{red}}^{I,\varepsilon})(\overline{\mathbb F_q})\simeq\left\{(Z_r)_{r\in I_{\mathbb Z}}\in\mathcal N_j^{I,\varepsilon}(\overline{\mathbb F_q})\ \middle|\ \forall r\in I_{\mathbb Z},\quad \overline r\in\overline I\setminus J\iff\tau(Z_{r+\ell_{\overline r}})\subseteq\pi^{j+1}Z_{-r-\underline{\varepsilon}}^{\vee}\right\}.
    \end{equation*}
\end{introproposition}
Here, $\tau\coloneq\mathrm{id}\otimes\sigma^2$ is the operator on $\mathbb V_{\breve F,0}$ acting through $\sigma^2$ on the scalars. Consider the involution $\theta_I:\overline I\xrightarrow{\sim}\overline I$ given by
\begin{equation*}
    \theta_I(i)\coloneq-i-\underline{\varepsilon}-\ell_i.
\end{equation*}
It satisfies $\ell_i=\ell_{\theta_I(i)}$ for every $i\in\overline I$, and has zero, one or two fixed points. Its fixed points and its action on $\overline I$ are computed explicitly after Definition \ref{Defi:DefiInvolutionThetaI}. We define
\begin{equation*}
    \mathcal P_{0,I}\coloneq\{J\in\mathcal P_I\mid J\cup\theta_I(J)\subsetneq\overline I\}.
\end{equation*}
It follows from Proposition 5.6 of \cite{gortzFullyHodgeNewton2019} that $\mathrm{KR}_J\not=\emptyset$ if and only if $J\in\mathcal P_{0,I}$. We recover this result in Corollary \ref{Corol:KRStrataNonEmpty} directly from Proposition \ref{IntropropC}, using the ``crucial Lemma'' of \cite{mullerBruhatTitsStratificationBasic2026} (Lemma \ref{Lem:CrucialLemma}) for one direction. Over every non-empty KR stratum $\mathrm{KR}_J$, following the construction of \cite{shenEKORStrataShimura2021}, we construct in Section \ref{Section4.3} a morphism to an associated stack of $G$-zips. Pulling back its orbit stratification gives the EKOR stratification of $\mathrm{KR}_J$. Varying $J$, we obtain the EKOR stratification of $\mathcal N_{\mathrm{red}}^{I,\varepsilon}$, whose indexing set is denoted ${}^I\mathrm{Adm}(\{\mu'\})_0$. Proposition \ref{Prop:DescriptionIndexSetOfEKORStrata} gives a bijection
\begin{equation*}
    {}^I\mathrm{Adm}(\{\mu'\})_0\xrightarrow{\sim}
    \left\{\mathbf c=(c_i)_{i\in\overline I}\in\prod_{i\in\overline I}\{0,\ldots,\ell_i\}\ \middle|\ \exists i\in\overline I,\quad c_i+c_{\theta_I(i)}\leq\ell_i-1\right\}.
\end{equation*}
For a tuple $\mathbf c$ in the right-hand side, we denote the corresponding element by $v_{\mathbf c}\in{}^I\mathrm{Adm}(\{\mu'\})_0$, and define
\begin{equation*}
    J_{\mathbf c}\coloneq\{i\in\overline I\mid c_i=\ell_i\}\in\mathcal P_{0,I}.
\end{equation*}
The corresponding EKOR stratum $\mathrm{EKOR}_{v_{\mathbf c}}$ is non-empty and is contained in $\mathrm{KR}_{J_{\mathbf c}}$.

\subsection*{Comparison with the BT stratification}

In \cite{mullerBruhatTitsStratificationBasic2026}, we defined another stratification of $\mathcal N_{\mathrm{red}}^{I,\varepsilon}$, which we called the \textit{Bruhat-Tits stratification}. Its strata are indexed by pairs $(\mathcal T,\mathbf\Lambda)$, called \textit{Bruhat-Tits indices}, see Definition \ref{Defi:BTIndex}. The corresponding stratum is denoted $\mathcal N_{\mathcal T,\mathbf\Lambda}^{0}$. Essentially, a BT index consists of
\begin{enumerate}
    \item a non-empty $\theta_I$-stable subset $\mathcal T\subseteq\overline I$,
    \item a facet $\mathbf\Lambda$ in the Bruhat-Tits building $\mathrm{BT}(J,F)$, subject to numerical constraints, encoded by $\mathcal T$, on its dimension and on the types of its vertices.
\end{enumerate}
This is the natural generalization of the Bruhat-Tits stratifications defined in \cite{vollaardSupersingularLocusShimura2011} and \cite{choBasicLocusUnitary2018} when $\#I=1$, which justifies the terminology. It is, however, generally \textit{coarser} than the (weak) Bruhat-Tits stratification of fully Hodge-Newton decomposable affine Deligne-Lusztig varieties defined in \cite{gortzBasicLociCoxeter2024}. Indeed, each BT stratum $\mathcal N_{\mathcal T,\mathbf\Lambda}^{0}$ is further stratified by a family of fine Deligne-Lusztig varieties. The resulting refinement of $\mathcal N_{\mathrm{red}}^{I,\varepsilon}$ is what we call the \textit{fine Bruhat-Tits stratification}. As observed in \cite{mullerBruhatTitsStratificationBasic2026}, the BT and fine BT stratifications agree if and only if the triple $(G,\{\mu\},K_I)$ is of \textit{Coxeter type} in the sense of \cite{gortzBasicLociCoxeter2024}. In our setting, this occurs precisely when $(I = \{0\} \text{ and } \varepsilon = 1)$, or $(I = \{\nu_{\varepsilon}\} \text{ and } 2\nu_{\varepsilon} + \underline{\varepsilon} = n)$, or when $(I = \{0, \nu_{\varepsilon}\} \text{, } n \text{ is even and } \varepsilon = 1)$. The first two cases correspond to $K_I$ being hyperspecial.\\
The group $J(F)$ acts on $\mathcal N^{I,\varepsilon}$, and the KR and EKOR strata are invariant under this action. Via its action on its own building, $J(F)$ also acts compatibly on the set of BT indices. If $\mathcal T$ is fixed, the $J(F)$-orbits of BT indices of the form $(\mathcal T,\mathbf\Lambda)$ are classified by
\begin{equation*}
    \mathrm{Num}(\mathcal T)\coloneq\left\{\mathbf m=(m_i)_{i\in\mathcal T}\in\mathbb Z_{\geq0}^{\mathcal T}\ \middle|\ \forall i\in\mathcal T,\quad m_i+m_{\theta_I(i)}\leq\ell_i-1\right\}.
\end{equation*}
For every $\mathbf m\in\mathrm{Num}(\mathcal T)$, we fix a representative $\mathbf\Lambda(\mathbf m)$ of the corresponding $J(F)$-orbit. We also define
\begin{equation*}
    \mathcal B_{\mathcal T}\coloneq\left\{\mathbf b=(b_i)_{i\in\overline I\setminus\mathcal T}\ \middle|\ \forall i\in\overline I\setminus\mathcal T,\quad 0\leq b_i\leq\ell_i\text{ and }b_i+b_{\theta_I(i)}\geq\ell_i\right\}.
\end{equation*}
Given $\mathbf m\in\mathrm{Num}(\mathcal T)$ and $\mathbf b\in\mathcal B_{\mathcal T}$, define the tuple $\mathbf c(\mathbf b)=(c_i)_{i\in\overline I}$ by
\begin{equation*}
    c_i \coloneq \begin{cases}
        m_i & \text{if } i \in \mathcal T,\\
        b_i & \text{if } i \in \overline I \setminus \mathcal T.
    \end{cases}
\end{equation*}
One of the main results of \cite{mullerBruhatTitsStratificationBasic2026}, recalled in Theorem \ref{Theo:IsomorphismBTStrataAndFineDLVarieties}, is the construction of an isomorphism
\begin{equation*}
    \mathcal N_{\mathcal T,\mathbf\Lambda(\mathbf m)}^{0}\xrightarrow{\sim}\bigsqcup_{\mathbf b\in\mathcal B_{\mathcal T}}X_{\mathcal K_{\mathcal T,\mathbf m}}\{w_{\mathbf c(\mathbf b)}^{-1}\},
\end{equation*}
where the varieties on the right-hand side are explicitly defined fine Deligne-Lusztig varieties in a partial flag variety of type $\mathcal K_{\mathcal T,\mathbf m}$. We make the corresponding fine BT strata $J(F)$-invariant by setting
\begin{equation*}
    X_{\mathcal T,\mathbf m,\mathbf b}\coloneq\bigsqcup_{g\in J(F)/\mathrm{Stab}_{J(F)}(\mathbf\Lambda(\mathbf m))}g\cdot X_{\mathcal K_{\mathcal T,\mathbf m}}\{w_{\mathbf c(\mathbf b)}^{-1}\}.
\end{equation*}
Proposition \ref{Prop:BijectionIndexSetsBTEKOR} gives a bijection
\begin{equation*} \hspace{-1cm}
    \left\{(\mathcal T,\mathbf m,\mathbf b)\ \middle|\ \begin{array}{l}
        \mathcal T\subseteq\overline I\text{ is non-empty and }\theta_I\text{-stable},\\
        \mathbf m\in\mathrm{Num}(\mathcal T),\\
        \mathbf b\in\mathcal B_{\mathcal T}
    \end{array}\right\}
    \xrightarrow{\sim}{}^I\mathrm{Adm}(\{\mu'\})_0,
    \qquad
    (\mathcal T,\mathbf m,\mathbf b)\longmapsto v_{\mathbf c(\mathbf b)}.
\end{equation*}
Our main comparison theorem states that this bijection identifies the two stratifications.
\begin{introtheorem}[Theorem \ref{Theo:ComparisonBTEKOR}]\label{IntrotheoremG}
    Let $\mathcal T\subseteq\overline I$ be a non-empty $\theta_I$-stable subset, let $\mathbf m\in\mathrm{Num}(\mathcal T)$, let $\mathbf b\in\mathcal B_{\mathcal T}$, and write $\mathbf c=\mathbf c(\mathbf b)$. Then
    \begin{equation*}
        X_{\mathcal T,\mathbf m,\mathbf b}=\mathrm{EKOR}_{v_{\mathbf c}}.
    \end{equation*}
\end{introtheorem}
In particular, $X_{\mathcal T,\mathbf m,\mathbf b}$ is contained in the KR stratum indexed by
\begin{equation*}
    J_{\mathbf c}=\{i\in\overline I\mid c_i=\ell_i\}=\{i\in\overline I\setminus\mathcal T\mid b_i=\ell_i\}.
\end{equation*}

\subsection*{Closure relations and applications}

The principal application of Theorem \ref{IntrotheoremG} is an explicit computation of the closure relations of the EKOR stratification. From what we can tell, the closure order on the quotient stacks used to define the KR and EKOR strata does not by itself give the closure relations on the Rapoport-Zink space: the morphisms from the basic Rapoport-Zink space to the corresponding local-model and $G$-zip quotient stacks are indeed not smooth. We instead use Theorem \ref{IntrotheoremG} to reduce the problem to the closure relations for fine Deligne-Lusztig varieties, and then carry out the resulting combinatorial computation. \\
We define a partial order $\mathrel{\dot{\leq}}$ on ${}^I\mathrm{Adm}(\{\mu'\})_0$ as follows. For $v_{\mathbf c},v_{\mathbf c'}\in{}^I\mathrm{Adm}(\{\mu'\})_0$, we write $v_{\mathbf c'}\mathrel{\dot{\leq}}v_{\mathbf c}$ if, for every $i\in\overline I$,
\begin{equation*}
    (c_i'\leq c_i\text{ and }c_{\theta_I(i)}'\leq c_{\theta_I(i)})
    \quad\text{or}\quad
    (c_i'+c_{\theta_I(i)}'\leq\ell_i-1\text{ and }c_i+c_{\theta_I(i)}\geq\ell_i).
\end{equation*}
The two conditions are invariant under replacing $i$ by $\theta_I(i)$, so it is enough to check them once on every $\theta_I$-orbit. On a singleton orbit this is the usual order on $\{0,\ldots,\ell_i\}$, whereas on an orbit of size two it is the coordinatewise order on the finite grid $\{0, \ldots, \ell_i\}^2$, together with the extra condition that any point below the diagonal $x+y = \ell_i$ is smaller than any point lying on or above the diagonal. Our closure theorem is the following.
\begin{introtheorem}[Theorem \ref{Theo:ClosureFineBTAndEKORStrata}]\label{IntrotheoremClosure}
    Let $v_{\mathbf c}\in{}^I\mathrm{Adm}(\{\mu'\})_0$. Then
    \begin{equation*}
        \overline{\mathrm{EKOR}_{v_{\mathbf c}}}
        =\bigsqcup_{\substack{v_{\mathbf c'}\in{}^I\mathrm{Adm}(\{\mu'\})_0\\v_{\mathbf c'}\mathrel{\dot{\leq}}v_{\mathbf c}}}
        \mathrm{EKOR}_{v_{\mathbf c'}}.
    \end{equation*}
    Equivalently, if $\mathbf c$ corresponds to $(\mathcal T,\mathbf m,\mathbf b)$, then
    \begin{equation*}
        \overline{X_{\mathcal T,\mathbf m,\mathbf b}}
        =\bigsqcup_{(\mathcal T',\mathbf m',\mathbf b')}X_{\mathcal T',\mathbf m',\mathbf b'},
    \end{equation*}
    where $(\mathcal T',\mathbf m',\mathbf b')$ runs over the triples such that
    \begin{equation*}
        \mathcal T\subseteq\mathcal T', \qquad m_i'\leq m_i \quad \text{for every } i \in \mathcal T, \qquad b_i'\leq b_i \quad \text{for every } i \in \overline I \setminus \mathcal T'.
    \end{equation*}
\end{introtheorem}
We also obtain the closure of every KR stratum. For $v_{\mathbf c}\in{}^I\mathrm{Adm}(\{\mu'\})_0$, set
\begin{equation*}
    \mathcal T_{\mathbf c}\coloneq\{i\in\overline I\mid c_i+c_{\theta_I(i)}\leq\ell_i-1\}.
\end{equation*}
Then Corollary \ref{Corol:ClosureKRStrata} gives, for every $J\in\mathcal P_{0,I}$,
\begin{equation*}
    \overline{\mathrm{KR}_J} = \bigsqcup_{\substack{v_{\mathbf c}\in{}^I\mathrm{Adm}(\{\mu'\})_0\\
    \mathcal T_{\mathbf c}\cap(\overline I\setminus(J\cup\theta_I(J)))\not=\emptyset}\\
    J_{\mathbf c}\subseteq J} \mathrm{EKOR}_{v_{\mathbf c}}.
\end{equation*}

We finish by recording some further consequences of the comparison theorem. First, an EKOR stratum is non-empty if and only if its index belongs to ${}^I\mathrm{Adm}(\{\mu'\})_0$. In that case, it is smooth and has pure dimension equal to the length $\ell(v_{\mathbf c})$ of its index, as stated in Corollary \ref{Corol:NonEmptyEKORStrata}. This also determines which EKOR strata of the corresponding Shimura variety are contained in the basic locus. In particular, it gives the following criterion for an entire KR stratum to be basic, that is entirely contained in the basic locus.
\begin{introproposition}[Proposition \ref{Prop:EntirelyBasicKRStrata}]
    Let $J\in\mathcal P_{0,I}$. The KR stratum indexed by $J$ is basic if and only if there exists $i \in \overline I \setminus (J\cup\theta_I(J))$ such that $\ell_i = 1$.
\end{introproposition}
Finally, the smoothness of the closures of the BT strata proved in \cite{mullerBruhatTitsStratificationBasic2026} yields the following statement about certain EKOR strata.
\begin{introproposition}[Proposition \ref{Prop:SmoothnessIrreducibleComponentsClosureEKOR}]
    Let $v_{\mathbf c}\in{}^I\mathrm{Adm}(\{\mu'\})_0$. Suppose that, for every $i\in\overline I$,
    \begin{equation*}
        c_i\leq\ell_i-1\implies c_i+c_{\theta_I(i)}\leq\ell_i-1.
    \end{equation*}
    Then every irreducible component of the Zariski closure $\overline{\mathrm{EKOR}_{v_{\mathbf c}}}$ is smooth.
\end{introproposition}
It might be reasonable to conjecture that the irreducible components of the closure of every EKOR stratum are smooth. As explained after Proposition \ref{Prop:SmoothnessIrreducibleComponentsClosureEKOR}, however, the method used in \cite{mullerBruhatTitsStratificationBasic2026} does not extend directly to all the fine Deligne-Lusztig varieties occurring here.

\section*{Notations}\label{SectionNotations}

Throughout the paper, we fix the following notations. Let $n \geq 2$. Let $p>2$ be an odd prime. Let $E/F$ be an unramified quadratic extension of $p$-adic fields, and let $\mathcal O_F$ and $\mathcal O_E$ denote the respective rings of integers. Let $\pi$ be a common uniformizer of $F$ and $E$, and let $\kappa_F \coloneq  \mathcal O_F/(\pi)$ and $\kappa_E \coloneq  \mathcal O_E/(\pi)$ denote the residue fields of $F$ and of $E$ respectively. Let $q \coloneq  \#\kappa_F$ so that we have $q^2 = \#\kappa_E$. Let $x\mapsto x^*$ denote the non-trivial element of $\mathrm{Gal}(E/F)$. If $\kappa$ is a perfect field containing $\kappa_F$, let $W_{\mathcal O_F}(\kappa)$ denote the ring of relative Witt vectors over $\kappa$. Recall that if $F^{u} \simeq \mathrm{Frac}(W(\kappa_F)) \subseteq F$ denotes the maximal unramified subextension of $F/\mathbb Q_p$, then we have 
\begin{equation*}
    W_{\mathcal O_F}(\kappa) \simeq \mathcal O_F \otimes_{\mathcal O_{F^u}} W(\kappa),
\end{equation*}
where $W(\kappa)$ is the absolute ring of Witt vectors over $\kappa$. We write $F_{\kappa} \coloneq  \mathrm{Frac}(W_{\mathcal O_F}(\kappa))$, so that $F_{\kappa}$ is a complete discretely valued extension of $F$ with uniformizer $\pi$ and residue field $\kappa$. In particular, after fixing an identification $\kappa_F \simeq \mathbb F_q$, we write $\breve F \coloneq  F_{\overline{\mathbb F_{q}}}$, which is the completion of the maximal unramified extension of $F$, and $K \coloneq  F_{\mathbb F_{q^2}}$. We fix once and for all an embedding $\mathbb F_{q^2} \hookrightarrow \overline{\mathbb F_q}$, which induces an embedding $K \subset \breve F$. \\
If $\kappa$ is as above, let $\sigma \in \mathrm{Aut}(F_{\kappa}/F)$ denote the automorphism lifting the $q$-power Frobenius automorphism $x \mapsto x^q$ on $\kappa$. Since $E/F$ is unramified, we can fix an $F$-embedding $\varphi_0: E \hookrightarrow \breve F$. The other $F$-embedding is given by $\varphi_1: x \mapsto \varphi_0(x^*)$. Then necessarily, we have $\varphi_j(E) = K$ for $j=0,1$. We fix an element $\delta \in \mathcal O_E^{\times}$ such that $\delta^* = -\delta$.\\
If $\varepsilon \in \{\pm 1\}$, we define two integers
\begin{align*}
    \underline{\varepsilon} \coloneq  \frac{1-\varepsilon}{2} = \begin{cases}
        0 & \text{if } \varepsilon = 1, \\
        1 & \text{if } \varepsilon = -1,
    \end{cases} & & \nu_{\varepsilon} \coloneq  \left\lfloor \frac{n-\underline{\varepsilon}}{2} \right\rfloor.
\end{align*}
We note that $\nu_{\varepsilon} = \nu_{-\varepsilon} = \frac{n-1}{2}$ when $n$ is odd, but $|\nu_{\varepsilon}-\nu_{-\varepsilon}|=1$ when $n$ is even.\\
All hermitian products will be linear in the first variable, and semi-linear in the second. When working with tuples, we often write $x^a$ instead of ``$x$ repeated $a$ times''. For instance, we will write 
\begin{equation*}
    (1^a, 0^{b}) = (\underbrace{1,\ldots , 1}_{a \text{ times}}, \underbrace{0,\ldots , 0}_{b \text{ times}}).
\end{equation*}
The context should always make any confusion unlikely. Finally, if $k \geq 0$ we write $A_{k}$ for the $k\times k$ matrix with $1$'s on the antidiagonal and $0$ everywhere else. If $r \in \mathbb Z$, we write $\overline{r} \in \mathbb Z/n\mathbb Z$ for its image modulo $n$.\\
If $(W,S)$ is a (not necessarily finite) Coxeter system with length function $\ell$, and if $I,J \subseteq S$, we write $W_I$ for the subgroup generated by $I$, $W^{I}$ (resp. ${}^JW$) for the set of elements which are reduced-$I$ (resp. $J$-reduced), and ${}^JW^I \coloneq  {}^JW \cap W^I$. Recall that an element $w \in W$ is reduced-$I$ (resp. $J$-reduced) if and only if $\ell(wv) \geq \ell(w)$ for all $v \in W_I$ (resp. $\ell(vw) \geq \ell(w)$ for all $v \in W_J$), or equivalently if $w$ is the unique element of shortest length in the right coset $wW_I$ (resp. left coset $W_Jw$). Then $w \in {}^JW^I$ if and only if $w$ is the unique element of shortest length in the double coset $W_JwW_I$.

\section{The Rapoport-Zink space}

\subsection{The PEL datum}\label{Section2.1}

First, we fix the setup for the algebraic PEL datum underlying the construction of the (relative) basic unramified unitary Rapoport-Zink space of signature $(1,n-1)$ and at parahoric level. Given $\varepsilon \in \{\pm 1\}$, let $(\mathbb V, H) \coloneq  (\mathbb V^{\varepsilon}, H)$ be a non-degenerate $n$-dimensional $E/F$-hermitian space of discriminant $\mathrm{disc}(\mathbb V, H) = \varepsilon$. Recall that the discriminant is by definition the image of $(-1)^{\frac{n(n-1)}{2}}\det(H)$ in the group $F^{\times}/\mathrm{Norm}_{E/F}(E^{\times}) \simeq \{\pm 1\}$. The hermitian space $(\mathbb V, H)$ is said to be \textit{split} if $\varepsilon = 1$ and \textit{non-split} if $\varepsilon = -1$. The Witt index of $(\mathbb V, H)$ is equal to $\nu_{\varepsilon}$. We fix a Witt decomposition
\begin{equation*}
    \mathbb V = \mathbb V^{\mathrm{an}} \oplus \bigoplus_{k=1}^{\nu_{\varepsilon}} \mathbb H_k,
\end{equation*}
where the $\mathbb H_k$'s are isomorphic to hyperbolic planes, and $\mathbb V^{\mathrm{an}}$ is a maximal anisotropic subspace. We also fix a basis $(f_{-k},f_{k})$ of $\mathbb H_k$ in which $H_{|\mathbb H_k \times \mathbb H_k}$ is represented by the matrix 
\begin{equation*}
    \begin{pmatrix}
        0 & 1 \\
        1 & 0
    \end{pmatrix}.
\end{equation*}
We note that 
\begin{equation*}
    \dim(\mathbb V^{\mathrm{an}}) = n - 2\nu_{\varepsilon} = \begin{cases}
        0 & \text{if } n \text{ is even and } \varepsilon = 1,\\
        1 & \text{if } n \text{ is odd},\\
        2 & \text{if } n \text{ is even and } \varepsilon = -1.
    \end{cases} 
\end{equation*}
We fix a basis $\mathbf e^{\mathrm{an}}$ of $\mathbb V^{\mathrm{an}}$ as follows. If $\varepsilon = -1$, there is a vector $f_{-} \in \mathbb V^{\mathrm{an}}$ such that $H(f_-,f_-) = \pi$; and if additionally $n$ is even, there is a vector $f_+$ orthogonal to $f_-$ such that $H(f_+,f_+)=1$. If $\varepsilon = 1$ and $n$ is odd, there is also a vector $f_{+} \in \mathbb V^{\mathrm{an}}$ such that $H(f_+,f_+) = 1$. The basis is given by 
\begin{equation*}
    \mathbf e^{\mathrm{an}} = \begin{cases}
        \emptyset & \text{if } n \text{ is even and } \varepsilon = 1,\\
        (f_+) & \text{if } n \text{ is odd and } \varepsilon = 1, \\
        (f_-) & \text{if } n \text{ is odd and } \varepsilon = -1, \\
        (f_+,f_-) & \text{if } n \text{ is even and } \varepsilon = -1.
    \end{cases}
\end{equation*}
The hermitian space $\mathbb V$ is equipped with the concatenated basis 
\begin{equation*}
    \mathbf e = \begin{cases}
        (f_{-\nu_{\varepsilon}}, \ldots , f_{-1}, \phantom{f_+,\,} f_1, \ldots , f_{\nu_{\varepsilon}}) & \text{if } n \text{ is even and } \varepsilon = 1,\\
        (f_{-\nu_{\varepsilon}}, \ldots , f_{-1}, f_+, f_1, \ldots , f_{\nu_{\varepsilon}}) & \text{if } n \text{ is odd and } \varepsilon = 1, \\
        (f_{-\nu_{\varepsilon}}, \ldots , f_{-1},  \phantom{f_+,\,}f_1, \ldots , f_{\nu_{\varepsilon}}, f_-) & \text{if } n \text{ is odd and } \varepsilon = -1, \\
        (f_{-\nu_{\varepsilon}}, \ldots , f_{-1}, f_+, f_1, \ldots , f_{\nu_{\varepsilon}},f_-) & \text{if } n \text{ is even and } \varepsilon = -1.
    \end{cases}
\end{equation*}
We will also write $\mathbf e = (\mathbf e_1, \ldots , \mathbf e_n)$ in the order specified above. We let $\Omega_H$ denote the Gram matrix of $\mathbf e$ in $H$. Explicitly, we have 
\begin{equation*} 
   \Omega_H = \begin{cases}
            A_n & \text{if } \varepsilon = 1, \\
            {}\\
            \begin{pmatrix}
                 A_{n-1} & \\
                 & \pi 
            \end{pmatrix}
            & \text{if } \varepsilon = -1.
    \end{cases}
\end{equation*}
Let $G \coloneq  \mathrm{GU}(\mathbb V, H)$ be the group of unitary similitudes of $(\mathbb V,H)$, seen as a connected reductive group over $F$. Explicitly, for an $F$-algebra $R$, we have 
\begin{equation*}
    G(R) = \left\{g \in \mathrm{GL}_{E \otimes_F R}(\mathbb V_R) \;\middle|\; \exists c(g) \in R^{\times}, \forall v,w \in \mathbb V_R, H(gv,gw) = c(g)H(v,w) \right\},
\end{equation*}
where $\mathbb V_R \coloneq  \mathbb V \otimes_F R$, $c(g)$ is identified with $1 \otimes c(g) \in E \otimes_F R$, and $H$ is extended to $\mathbb V_R$ via 
\begin{equation*}
    \forall v,w \in \mathbb V, \forall a,b \in R, \quad  H(v \otimes a, w \otimes b) \coloneq  H(v,w) \otimes ab \in E \otimes_F R.
\end{equation*}
If $R$ is a $K$-algebra, then $E \otimes_{F} R \simeq R \times R$ as $R$-algebras via the map induced by $\lambda \otimes a \mapsto (\varphi_0(\lambda)a, \varphi_1(\lambda)a)$. Under this isomorphism, the involution $\lambda \otimes a \mapsto (\lambda \otimes a)^* \coloneq  \lambda^* \otimes a$ corresponds to swapping the two coordinates of $R \times R$. If $(1,0)$ and $(0,1)$ denote the two natural idempotents of $K\times K \subseteq R \times R$, their inverse images in $E \otimes_{F} K$ are given by 
\begin{equation*}
    \eta_j \coloneq  \frac{1}{2\varphi_j(\delta)}\left(1 \otimes \varphi_j(\delta) + \delta\otimes 1\right).
\end{equation*}
Let us write 
\begin{equation*}
    \forall j=0,1, \qquad \mathbb V_{R,j} \coloneq  \eta_j\mathbb V_R,
\end{equation*}
so that we have $\mathbb V_R = \mathbb V_{R,0} \oplus \mathbb V_{R,1}$. We equip $\mathbb V_{K,j}$ with the $K$-basis $\mathbf e^{(j)} \coloneq  \eta_j(\mathbf e\otimes 1)$. By scalar extension to $R$, this induces an identification $\mathbb V_{R,j} \simeq R^n$. We can then rewrite $G(R)$ as follows.
\begin{equation*}
    G(R) \simeq \left\{g = (g_0,g_1) \in \mathrm{GL}(R^n)\times \mathrm{GL}(R^n) \, \middle | \, \exists c \in R^{\times}, g_1 = c\Omega_H^{-1}g_0^{-T}\Omega_H \right\}.
\end{equation*}
Sending $g$ to $(g_0,c)$ gives an isomorphism 
\begin{equation}\label{Eq:BaseChangeG}
    G_K \coloneq  G \otimes_F K \simeq \mathrm{GL}(\mathbb V_{K,0}) \times_K \mathbb G_{m,K}.
\end{equation}
We also fix the conjugacy class of cocharacters $\{\mu\}$ corresponding to the signature $(1,n-1)$. It is represented by
\begin{equation*}
\mu:\mathbb G_{m,K}\longrightarrow G_K,\qquad z\longmapsto (\mathrm{Diag}(z,1,\ldots,1),z),
\end{equation*}
where the diagonal matrix is written in the basis $\mathbf e^{(0)}$. The weight-one subspace $\mathbb V_K^{1}$ of $\mathbb V_K$ for this cocharacter has dimension $1$ in $\mathbb V_{K,0}$ and dimension $n-1$ in $\mathbb V_{K,1}$. Hence, for every $\lambda \in E$, the characteristic polynomial of $\lambda\otimes 1$ on $\mathbb V_K^{1}$ is
\begin{equation*}
    \det(T - \lambda\otimes 1 \mid \mathbb V_K^{1}) =  (T-\varphi_0(\lambda))(T-\varphi_1(\lambda))^{n-1}.
\end{equation*}
The \textit{reflex field} of this PEL datum is the field of definition of $\{\mu\}$, that is $K$ if $n \not = 2$, and $F$ if $n=2$.\\
In the basis $\mathbf e^{(0)}$, we define the element $b_0 \in \mathrm{GL}_K(\mathbb V_{K,0})$ with matrix 
\begin{equation}\label{Eq:DefinitionElementb}
    b_0 \coloneq  \begin{pmatrix}
        & \pi \\
        I_{n-1} &
    \end{pmatrix},
\end{equation}
and define $b \coloneq  (b_0, \pi) \in G(K)$. Now, let $\kappa$ be a perfect field containing $\mathbb F_{q^2}$. We define two semilinear operators on $\mathbb V_{F_{\kappa}}$
\begin{align*}
    \mathcal F \coloneq  b(\mathrm{id}\otimes\sigma), & & \mathcal V \coloneq  \pi\mathcal F^{-1}.
\end{align*}
We have $\mathcal F^2 = \pi(\mathrm{id}\otimes \sigma^{2})$ and $\mathcal V^2 = \pi(\mathrm{id}\otimes \sigma^{-2})$. The operators $\mathcal F$ and $\mathcal V$ are $1$-homogeneous with respect to the $\mathbb Z/2\mathbb Z$-grading $\mathbb V_{F_{\kappa}} = \mathbb V_{F_{\kappa},0} \oplus \mathbb V_{F_{\kappa},1}$. We equip $\mathbb V_{F_{\kappa},0}$ with a non-degenerate $F_{\kappa}$-valued form which is linear in the first variable and $\sigma$-linear in the second:
\begin{equation}\label{Eq:DefinitionForm(,)OnV0}
    \forall v,w \in \mathbb V_{F_{\kappa},0}, \qquad (v,w) \coloneq  H(v,\mathcal Fw) \in \eta_0(E \otimes_F F_{\kappa}) \simeq F_{\kappa}.
\end{equation}
We write $\tau\coloneq  \pi^{-1}\mathcal F^2 = \mathrm{id}\otimes \sigma^2$ on $\mathbb V_{F_{\kappa}}$. Then we have 
\begin{align*}
    \forall v,w \in \mathbb V_{F_{\kappa},0}, & & (v,w) & = (w,\tau^{-1}(v))^{\sigma},\\
    & & (\tau(v),\tau(w)) & = (v,w)^{\sigma^2}.
\end{align*}
Taking $\tau$-fixed vectors, we identify $\mathbb V_{K,0}$ with
\begin{equation*}
    \mathbb V_{K,0} = \mathbb V_{F_{\kappa},0}^{\tau} = \{v \in \mathbb V_{F_{\kappa},0}\,|\, \tau(v) = v\}.
\end{equation*} 
Then $(\mathbb V_{K,0},(\cdot,\cdot))$ is a non-degenerate $K/F$-hermitian space of dimension $n$, equipped with the distinguished basis $\mathbf e^{(0)}$ defined above.

\begin{prop}
    We have $\mathrm{disc}(\mathbb V_{K,0},(\cdot,\cdot)) = (-1)^{n-1}\mathrm{disc}(\mathbb V,H)$.
\end{prop}

\begin{rk}
    It makes sense to compare the two discriminants since $E$ and $K$ are isomorphic. Thus $\mathrm{Norm}_{E/F}(E^{\times}) = \mathrm{Norm}_{K/F}(K^{\times})$.
\end{rk}

\begin{proof}
The Gram matrix of $(\cdot,\cdot)$ in the basis $\mathbf e^{(0)}$ is given by 
\begin{equation*}
    \Omega_{\mathbb V_{K,0}} = \begin{cases}
        \begin{pmatrix}
            1 &  \\
            & \pi A_{n-1}
        \end{pmatrix} & \text{if } \varepsilon = 1,\\
        {}\\
        \pi A_{n} & \text{if } \varepsilon = -1.
    \end{cases}
\end{equation*}
Comparing with $\Omega_H$, the result follows.
\end{proof}
It follows from the Proposition that the Witt index of $(\mathbb V_{K,0},(\cdot,\cdot))$ is equal to the integer $\nu_{-\varepsilon}$. We define a connected reductive group $J \coloneq  J_b$ as follows. For any $F$-algebra $R$, we put 
\begin{equation*}
    J(R) = \{g \in G(R \otimes_F \breve F) \,|\, gb\sigma(g)^{-1} = b \}.
\end{equation*}

\begin{prop}\label{Prop:GroupJIsomorphism}
    We have $J \simeq \mathrm{GU}(\mathbb V_{K,0},(\cdot,\cdot))$.
\end{prop}

\begin{proof}
    Any $g \in J(R)$ decomposes as $g = (g_0,g_1)$ where $g_i \in \mathrm{GL}_{R \otimes_F \breve F}(\mathbb V_{R \otimes_F \breve F,i})$. We note that 
    \begin{equation*}
    \mathbb V_{R \otimes_F \breve F,0} = \mathbb V_{K,0} \otimes_{K} (R \otimes_{F} \breve F) = (\mathbb V_{K,0} \otimes_F R) \otimes_{K\otimes_F R} (R \otimes_F \breve{F}).
    \end{equation*} 
    It is easy to check that $g_0$ preserves $\mathbb V_{K,0} \otimes_F R \subset \mathbb V_{R \otimes_F \breve F,0}$, and that $g \mapsto (g_0)_{|\mathbb V_{K,0} \otimes_F R}$ induces the desired isomorphism.
\end{proof}

For later reference, we define a non-degenerate $F$-valued alternating pairing on $\mathbb V$ by the formula 
\begin{equation*}
    \forall v,w \in \mathbb V, \qquad \langle v,w \rangle_H \coloneq  \mathrm{Trace}_{E/F}\left(\frac{H(v,w)}{2\delta}\right).
\end{equation*}
We have $\langle \lambda v, w \rangle_H = \langle v, \lambda^*w \rangle_H$ for all $v,w \in \mathbb V$ and $\lambda \in E$. Moreover, $H$ can be recovered from $\langle\cdot,\cdot\rangle_H$ via the formula 
\begin{equation*}
    \forall v,w \in \mathbb V, \qquad H(v,w) = \langle \delta v, w \rangle_H + \delta \langle v,w \rangle_H.
\end{equation*}
If $L \subset \mathbb V$ is an $\mathcal O_E$-lattice, then the dual lattice with respect to $\langle \cdot,\cdot \rangle_H$ is also equal to $L^{\dagger}$. If $R$ is an $F$-algebra, we extend $\langle\cdot,\cdot\rangle_H$ to $\mathbb V_R$ via 
\begin{equation*}
    \forall v,w \in \mathbb V, \forall a,b \in R, \quad  \langle v \otimes a, w \otimes b\rangle_H \coloneq  ab\langle v,w \rangle_H \in R.
\end{equation*}
If $R = F_{\kappa}$ for some perfect field $\kappa$ containing $\mathbb F_{q^2}$, we also have $\langle \mathcal F v, w \rangle_H = \langle v, \mathcal V w \rangle_H^{\sigma}$ for all $v,w \in \mathbb V_{F_{\kappa}}$.\\

If $L \subset \mathbb V$ is an $\mathcal O_E$-lattice, we write $L^{\dagger}$ for its dual lattice. Explicitly,
\begin{equation*}
    L^{\dagger} = \{v \in \mathbb V \,|\, H(v,L) \subseteq \mathcal O_E\}.
\end{equation*}
Besides, if $L, L' \subset \mathbb V$ are two $\mathcal O_E$-lattices and $t \geq 0$ is an integer, we write 
\begin{equation*}
    L \overset{t}{\subseteq} L' \iff L \subseteq L' \text{ and } \mathrm{length}_{\mathcal O_E}(L'/L) = t.
\end{equation*}
In this case, we also write $t = [L':L]$. This notation can be generalized. Namely, for any two lattices $L$ and $L'$, we define 
\begin{equation}\label{Eq:DefinitionIndexLattices}
    [L':L] \coloneq  \mathrm{length}_{\mathcal O_E}(L'/L\cap L') - \mathrm{length}_{\mathcal O_E}(L/L\cap L') \in \mathbb Z.
\end{equation}
Then we have $[L':L] = [L':L''] + [L'':L]$ for every third lattice $L''$. \\
We extend the basis $\mathbf e$ to a $\pi$-periodic family of vectors. Namely, for every $1 \leq r_0 \leq n$ and $k \in \mathbb Z$, define 
\begin{equation*}
    \mathbf e_{r_0+kn} \coloneq  \pi^k \mathbf e_{r_0}.
\end{equation*}
Using the Gram matrix $\Omega_H$, one may check that for all $r,r'\in \mathbb Z$, 
\begin{align}\label{Eq:ValueOfHOnPeriodicBasis}
\begin{split}
    H(\mathbf e_r, \mathbf e_{r'}) & = 0 \text{ if } \overline{r+r'-1+\underline{\varepsilon}} \not = 0, \\
    H(\mathbf e_r, \mathbf e_{n+1-\underline{\varepsilon}-r}) & = 1.
\end{split}
\end{align}
Then, for $r \in \mathbb Z$ we define an $\mathcal O_E$-lattice in $\mathbb V$
\begin{equation*}
    L_r \coloneq  \bigoplus_{k=r+1}^{r+n} \mathcal O_{E} \cdot \mathbf e_k.
\end{equation*}
The collection $L_{\bullet} = (L_r)_{r\in \mathbb Z}$ forms a complete self-dual (decreasing) chain of $\mathcal O_E$-lattices in $\mathbb V$. Indeed, we clearly have $L_{r+1} \overset{1}{\subseteq} L_r$ and $L_{r+n} = \pi L_r$ for all $r\in \mathbb Z$, and as follows from \eqref{Eq:ValueOfHOnPeriodicBasis} we have
\begin{equation*}
    L_r^{\dagger} = L_{-r-\underline{\varepsilon}}.
\end{equation*}
We note that 
\begin{equation*}
    \forall 0 \leq r \leq \nu_{\varepsilon}, \qquad \pi L_r^{\dagger} \overset{n-h_r^{\varepsilon}}{\subseteq} L_r \overset{h_r^{\varepsilon}}{\subseteq} L_r^{\dagger},
\end{equation*}
where we wrote
\begin{equation}\label{Eq:DefiIntegershrvarepsilon}
    \forall 0 \leq r \leq \nu_{\varepsilon}, \qquad h_r^{\varepsilon} \coloneq  2r+\underline{\varepsilon}, \quad 0 \leq h_r^{\varepsilon} \leq n.
\end{equation}
Given a non-empty subset $I \subseteq \{0, \ldots , \nu_{\varepsilon}\}$, we define
\begin{align}\label{Eq:DefImathbbZ}
    \overline{I} &  \coloneq  \{\overline r \mid r \in I\} \cup \{\overline{-r-\underline{\varepsilon}} \mid r\in I\} & & \subseteq \mathbb Z/n\mathbb Z, \\
    I_{\mathbb Z} & \coloneq  \{r \in \mathbb Z \mid \overline r \in \overline I\} & & \subseteq \mathbb Z. \nonumber
\end{align}
We consider the following self-dual subchain of $L_{\bullet}$
\begin{equation}\label{Eq:DefL^I_bullet}
    L^{I}_{\bullet} \coloneq  (L_r)_{r \in I_{\mathbb Z}}.
\end{equation} 
We also define a parahoric subgroup scheme of $G$ over $\mathrm{Spec}(\mathcal O_F)$
\begin{equation}\label{Eq:DefiParahoricGroupScheme}
    \mathcal G_I \coloneq  \mathrm{Aut}_{\mathcal O_E}^{\mathrm{sim}}(L^{I}_{\bullet}).
\end{equation}
Explicitly, for any $\mathcal O_F$-algebra $R$, $\mathcal G_I(R)$ is the set of collections $(g_r)_{r \in I_{\mathbb Z}}$ of $\mathcal O_E\otimes_{\mathcal O_F} R$-linear automorphisms 
\begin{equation*}
    g_r: L_r\otimes_{\mathcal O_F} R \xrightarrow{\sim} L_r \otimes_{\mathcal O_F} R,
\end{equation*}
such that 
\begin{itemize}
    \item the $g_r$'s are compatible with the transition morphisms of the chain $L^{I}_{\bullet} \otimes_{\mathcal O_F} R$,
    \item there exists a unique constant $c \in R^{\times}$ such that, for all $r \in I_{\mathbb Z}$,
    \begin{equation*}
        \forall v \in L_r \otimes_{\mathcal O_F}R, \forall w \in L_{-r-\underline{\varepsilon}} \otimes_{\mathcal O_F}R, \qquad \langle g_r(v), g_{-r-\underline{\varepsilon}}(w)\rangle_H = c\langle v,w \rangle_H.
    \end{equation*}
\end{itemize}
We have $\mathcal G_I \otimes_{\mathcal O_F} F = G$. We also define 
\begin{equation*}
    K_I \coloneq  \mathcal G_I(\mathcal O_F) = \bigcap_{r \in I} \mathrm{Stab}_{G(F)}(L_r) = \{ g \in G(F) \mid \forall r \in I, g(L_r) = L_r\}.
\end{equation*}
Then $K_I$ is the \textit{standard parahoric subgroup of $G(F)$ of level $I$}. Every parahoric subgroup of $G(F)$ is conjugate to $K_I$ for some $I$.\\

If $\kappa$ is a perfect field containing $\mathbb F_{q^2}$, for any $W_{\mathcal O_F}(\kappa)$-lattice $\Lambda \subset \mathbb V_{F_{\kappa},0}$, we write $\Lambda^{\vee}$ for its dual lattice with respect to the form $(\cdot,\cdot)$ defined in \eqref{Eq:DefinitionForm(,)OnV0}, i.e.
\begin{equation}\label{Eq:DefiDualLatticeVee}
    \Lambda^{\vee} \coloneq  \{v \in \mathbb V_{F_{\kappa},0} \mid (v, \Lambda) \subseteq W_{\mathcal O_F}(\kappa) \}.
\end{equation}
One may check that for every $W_{\mathcal O_F}(\kappa)$-lattice $\Lambda \subset \mathbb V_{F_{\kappa},0}$, we have 
\begin{equation}\label{Eq:FormulaVeeTau}
    (\Lambda^{\vee})^{\vee} = \tau(\Lambda), \qquad \tau(\Lambda^{\vee}) = \tau(\Lambda)^{\vee}.
\end{equation}
For $r \in \mathbb Z$, consider the base change $L_{r,\mathcal O_{K}} \coloneq  L_r \otimes_{\mathcal O_F} \mathcal O_{K}$. One may check that for every $r \in \mathbb Z$,
\begin{equation*}
    \forall \lambda \in \mathcal O_E, \qquad \det(T-\lambda\otimes 1 \mid L_{r,\mathcal O_{K}}/\mathcal VL_{r,\mathcal O_{K}}) = (T-\varphi_0(\lambda))(T-\varphi_1(\lambda))^{n-1},
\end{equation*}
which corresponds to the $(1,n-1)$-signature condition. Intersecting with $\mathbb V_{K,0}$, we obtain a complete self-dual (decreasing) chain of $\mathcal O_K$-lattices in $\mathbb V_{K,0}$ 
\begin{equation}\label{Eq:DefiLatticeLambda}
    \Lambda_{\bullet} = (\Lambda_r)_{r \in \mathbb Z}, \qquad \Lambda_r \coloneq  L_{r,\mathcal O_{K}} \cap \mathbb V_{K,0} = \eta_0 L_{r,\mathcal O_{K}}.
\end{equation}
We expand the basis $\mathbf e^{(0)}$ of $\mathbb V_{K,0}$ to a $\pi$-periodic family of vectors by setting, for all $1 \leq r_0 \leq n$ and $k \in \mathbb Z$,
\begin{equation*}
    \mathbf e^{(0)}_{r_0+kn} \coloneq  \pi^k \mathbf e_{r_0}^{(0)} = \eta_0(\mathbf e_{r_0+kn} \otimes 1).
\end{equation*}
It follows that
\begin{equation*}
    \forall r \in \mathbb Z, \quad \Lambda_r = \bigoplus_{k=r+1}^{r+n} \mathcal O_K \cdot \mathbf e_{k}^{(0)}.
\end{equation*}
Besides, using the Gram matrix $\Omega_{\mathbb V_{K,0}}$ one may check that for all $r,r'\in \mathbb Z$, 
\begin{align*}
\begin{split}
    (\mathbf e_r^{(0)}, \mathbf e_{r'}^{(0)}) & = 0 \text{ if } \overline{r+r'-2+\underline{\varepsilon}} \not = 0, \\
    (\mathbf e_r^{(0)}, \mathbf e_{2-\underline{\varepsilon}-r}^{(0)}) & = 1.
\end{split}
\end{align*}
It follows that
\begin{equation*}
    \forall r \in \mathbb Z, \quad \pi\Lambda_r^\vee=\Lambda_{(1-\underline{\varepsilon})-r},
\end{equation*}
and
\begin{equation*}
    \forall 0 \leq r \leq \nu_{-\varepsilon}, \qquad \pi\Lambda_{-r}^{\vee} \overset{h_r^{-\varepsilon}}{\subseteq} \Lambda_{-r} \overset{n-h_{r}^{-\varepsilon}}{\subseteq} \Lambda_{-r}^{\vee}.
\end{equation*}
\begin{defi}\label{Defi:Elementg_0}
    Assume that $n$ is even. We denote by $g_0 \in G(F)$ the element whose matrix in the basis $\mathbf e$ is 
    \begin{equation*}
        g_0 \coloneq \begin{pmatrix}
             & \pi I_{n/2} \\
            I_{n/2} & 
        \end{pmatrix}.
    \end{equation*}
\end{defi}
It is straightforward to check that $g_0^T\Omega_H g_0 = \pi \Omega_H$, from which it follows that $c(g_0) = \pi$.
\begin{prop}\label{Prop:ImageOfalpha}
Let $\alpha_G:G(F)\to \mathbb Z$ denote the composition of the multiplier map $c:G(F) \to F^{\times}$ with the $\pi$-adic valuation $v_{\pi}$. We have 
\begin{equation*}
    \mathrm{Im}(\alpha_G) = \begin{cases}
        2 \mathbb Z & \text{if } n \text{ is odd}, \\
        \mathbb Z & \text{if } n \text{ is even}.
    \end{cases}
\end{equation*}
\end{prop}
\begin{proof}
    Since $c(\pi^k \mathrm{id}) = \pi^{2k}$, we always have $2\mathbb Z \subset \mathrm{Im}(\alpha_G)$. The existence of $g_0$ already proves that $\mathrm{Im}(\alpha_G) = \mathbb Z$ when $n$ is even. Assume now that $n$ is odd. Since we have 
    \begin{equation*}
        \mathrm{Norm}_{E/F}(\det(g)) = c(g)^n,
    \end{equation*}
    we deduce that $2v_{\pi}(\det(g)) = n\alpha_G(g)$, which forces $\alpha_G(g)$ to be even.
\end{proof}
We define the dual complete chain $\widetilde{L}_{\bullet}$ by the formula
\begin{equation}\label{Eq:DefDualChain}
     \widetilde{L}_{\bullet} \coloneq  (L_{\nu_{\varepsilon}-r}^{\dagger})_{r \in \mathbb Z}.
\end{equation}
Consider the hermitian product $H' \coloneq  \pi H$ and let $\cdot^{\dagger'}$ denote duality with respect to $H'$. Clearly, we have $L^{\dagger'} = \pi^{-1} L^{\dagger}$ for every $\mathcal O_E$-lattice $L \subset \mathbb V$, and we have 
\begin{equation*}
    \mathrm{disc}(\mathbb V, H') = (-1)^n\mathrm{disc}(\mathbb V,H) = (-1)^{n}\varepsilon \eqcolon \varepsilon'.
\end{equation*}
We point out that $\nu_{\varepsilon} = \nu_{\varepsilon'}$. The definition of $\widetilde{L}_{\bullet}$ is designed so that we have 
\begin{equation*}
    \forall 0 \leq r \leq \nu_{\varepsilon'}, \qquad \pi\widetilde{L}_{r}^{\dagger'} \overset{n-h_{r}^{\varepsilon'}}{\subseteq} \widetilde{L}_{r} \overset{h_{r}^{\varepsilon'}}{\subseteq} \widetilde{L}_{r}^{\dagger'},
\end{equation*}
thus it is more relevant to think of $\widetilde{L}_{\bullet}$ as a complete chain of lattices in $(\mathbb V, H')$ rather than $(\mathbb V, H)$. When $n$ is even, the hermitian spaces defined by $H$ and by $H'$ are isomorphic. In this case, $g_0$ defines an isometry
\begin{equation*}
    g_0 : (\mathbb V, H') \xrightarrow{\sim} (\mathbb V, H)
\end{equation*}
which identifies $\widetilde{L}_{\bullet}$ with $L_{\bullet}$. Indeed, we have 
\begin{equation*}
    \forall r \in \mathbb Z, \qquad g_0(L_r) = L_{r + \frac{n}{2}},
\end{equation*}
so that $g_0(\widetilde{L}_r) = g_0(L_{\nu_{\varepsilon}-r}^{\dagger}) = g_0(L_{r - \nu_{\varepsilon} - \underline{\varepsilon}}) = g_0(L_{r-\frac{n}{2}}) = L_{r}$. More generally, for any $n$, let $I \subseteq \{0, \ldots , \nu_{\varepsilon}\}$ be a non-empty subset. We define 
\begin{equation}\label{Eq:DefiWidetildeI}
    \widetilde I \coloneq  \{\nu_{\varepsilon}-r \,|\, r \in I\} \subseteq \{0,\ldots ,\nu_{\varepsilon}\}, 
\end{equation}
and we define a subchain $\widetilde{L}_{\bullet}^{I}$ of $\widetilde{L}_{\bullet}$ via the formula 
\begin{equation}\label{Eq:DefiWidetildeLI}
    \widetilde{L}_{\bullet}^{I} \coloneq  (\widetilde{L}_{\widetilde r})_{\widetilde r \in \widetilde I_{\mathbb Z}} = (L^{\dagger}_r)_{r \in I_{\mathbb Z}},
\end{equation}
where $I_{\mathbb Z}$ is defined in \eqref{Eq:DefImathbbZ}, and similarly for $\widetilde I_{\mathbb Z}$ defined with respect to $\varepsilon'$. We note that $\widetilde I_{\mathbb Z} = \nu_{\varepsilon} - I_{\mathbb Z}$. When $n$ is even, $g_0$ identifies $\widetilde{L}_{\bullet}^{I}$ with $L_{\bullet}^{\widetilde I}$.\\

For every non-empty subset $I\subseteq \{0,\ldots,\nu_{\varepsilon}\}$, the following PEL datum
\begin{equation*}
    \mathcal D_I \coloneq  (E/F,\mathbb V,H,b,\{\mu\} ; L_{\bullet}^I),
\end{equation*}
is called the \textit{relative basic unramified unitary PEL datum of signature $(1,n-1)$ and of parahoric level $I$}.

\subsection{Unitary Rapoport-Zink spaces of parahoric level}

We recall the definition of the relative Rapoport-Zink spaces attached to the PEL datum introduced in Section \ref{Section2.1}.
\begin{defi} Let $S$ be a scheme over which $p$ is locally nilpotent.
\begin{enumerate}
    \item A \textit{formal $\mathcal O_F$-module} over $S$ is a formal $p$-divisible group $X$ over $S$ together with an $\mathcal O_F$-action, i.e. a ring morphism $i:\mathcal O_F \to \mathrm{End}(X)$.
    \item Assume that $S$ is an $\mathcal O_F$-scheme. The formal $\mathcal O_F$-module $X$ is said to be \textit{strict} if the $\mathcal O_F$-action on $\mathrm{Lie}(X)$ induced by $i$ coincides with the natural action given by the structure morphism $\mathcal O_F \to \mathcal O_S$.
    \item Assume that $S$ is an $\mathcal O_F$-scheme. We say that a strict formal $\mathcal O_F$-module $X$ is \textit{supersingular} if all its slopes are $\frac{1}{2}$. 
\end{enumerate}
\end{defi}
Let $\textbf{Nilp}$ (resp. $\textbf{Nilp}_{\mathcal O_{\breve F}}$) denote the category of $\mathcal O_K$-schemes (resp. $\mathcal O_{\breve F}$-schemes) $S$ over which $\pi$ is locally nilpotent. The following definition is a special case of a notion studied in \cite{mihatschRELATIVEUNITARYRZSPACES2022} and \cite{choBasicLocusUnitary2018}.
\begin{defi}
    Let $0 \leq h \leq n$ and let $S \in \textbf{Nilp}$. A \textit{supersingular hermitian $\mathcal O_E$-$\mathcal O_F$-$h$-module of signature $(1,n-1)$} over $S$ is a triple $(X,i_X,\lambda_X)$ where 
    \begin{enumerate}
        \item $X$ is a supersingular strict formal $\mathcal O_F$-module over $S$,
        \item $i_X: \mathcal O_E \to \mathrm{End}(X)$ is an $\mathcal O_E$-action which extends the $\mathcal O_F$-action on $X$, and which satisfies the signature $(1,n-1)$ condition
        \begin{equation*} 
            \forall \lambda \in \mathcal O_E, \qquad \mathrm{charpol}(i_{X}(\lambda)\,|\, \mathrm{Lie}(X)) = (T-\varphi_0(\lambda))(T-\varphi_1(\lambda))^{n-1} \in  \mathcal O_S[T],
        \end{equation*}   
        \item $\lambda_X: X \to X^{\vee}$ is an $\mathcal O_E$-linear polarization such that $\mathrm{Ker}(\lambda_X) \subseteq X[\pi]$ and $\mathrm{Ker}(\lambda_X)$ has order $q^{2h}$. Here, $X^{\vee}$ is the Serre dual of $X$ equipped with the $\mathcal O_E$-action given by $i_{X^{\vee}}(\lambda) \coloneq  i_{X}(\lambda^*)^{\vee}$. 
    \end{enumerate}
\end{defi}
For $0 \leq h \leq n$, let us write $h = h_r^{\varepsilon} = 2r + \underline{\varepsilon}$ for some $\varepsilon \in \{\pm 1\}$ and $r \in \{0, \ldots , \nu_{\varepsilon}\}$. By (covariant) relative Dieudonné theory, the datum $(\mathbb V^{\varepsilon}_{K}, L_{r,\mathcal O_K}, \mathcal F, \mathcal V, \langle \cdot,\cdot \rangle_H)$ defines a supersingular hermitian $\mathcal O_E$-$\mathcal O_F$-$h$-module of signature $(1,n-1)$ over $\kappa_K = \mathbb F_{q^2}$, denoted $\mathbb X^{[h]} = (\mathbb X,i_{\mathbb X},\lambda_{\mathbb X}^{[h]})$. By construction, $(\mathbb X,i_{\mathbb X})$ depends only on $\varepsilon \in \{\pm 1\}$ while only the polarization $\lambda_{\mathbb X}^{[h]}$ depends on the value of $h$ of fixed parity $\underline{\varepsilon}$.

For $S \in \textbf{Nilp}$, let $\mathcal N_{E/F,\mathbb X^{[h]}}(S)$ denote the set of tuples $(X,i_X,\lambda_X,\rho_X)$ up to isomorphism, where 
\begin{itemize}
\item $(X,i_X,\lambda_X)$ is a (supersingular) hermitian $\mathcal O_E$-$\mathcal O_F$-$h$-module of signature $(1,n-1)$ over $S$,
\item $\rho_X:X\times_S \overline{S} \to \mathbb X \times_{\mathrm{Spec}(\mathbb F_{q^2})} \overline S$ is an $\mathcal O_E$-linear quasi-isogeny making the following diagram commute locally on $\overline S$ up to a scalar in $F^{\times}$ 
\begin{equation*}
\begin{tikzcd}
X_{\overline S} \arrow[r,"(\lambda_{X})_{\overline S}"] \arrow[d,swap,"\rho_X"] & X_{\overline S}^{\vee} \\
\mathbb X_{\overline S} \arrow[r,"(\lambda_{\mathbb X}^{[h]})_{\overline S}"] & \mathbb X_{\overline S}^{\vee} \arrow[u,swap,"\rho_X^{\vee}"]
\end{tikzcd}
\end{equation*}
\end{itemize}
In the last item, $\overline S \coloneq  S \times_{\mathrm{Spec}(\mathcal O_K)} \mathrm{Spec}(\mathbb F_{q^2})$ is the special fiber of $S$. An isomorphism 
\begin{equation*}
    (X,i_X,\lambda_X,\rho_X) \xrightarrow{\sim} (X',i_{X'},\lambda_{X'},\rho_{X'})
\end{equation*} 
is an $\mathcal O_E$-linear isomorphism $\beta:X\xrightarrow{\sim} X'$ such that $\rho_{X'} \circ \beta_{\overline S} = \rho_X$ and $\beta^{\vee}\circ\lambda_{X'}\circ\beta$ differs from $\lambda_X$ locally on $S$ by a scalar in $\mathcal O_F^{\times}$. 

Recall that an embedding $K \hookrightarrow \breve F$ is fixed in our Notations \ref{SectionNotations}. The following statement is essentially proved in \cite{mihatschRELATIVEUNITARYRZSPACES2022}, see also \cite{choBasicLocusUnitary2018}.
\begin{theo}
    The functor $\mathcal N_{E/F,\mathbb X^{[h]}}\otimes \mathcal O_{\breve F}$ on $\textbf{Nilp}_{\mathcal O_{\breve F}}$ is represented by a formal scheme over $\mathrm{Spf}(\mathcal O_{\breve F})$ which is locally formally of finite type and regular.
\end{theo}
From now on we write $\mathcal N_{\mathbb X^{[h]}}$ for the formal scheme over $\mathrm{Spf}(\mathcal O_{\breve F})$ given by the Theorem. 
\begin{rk}\label{Rk:IndependenceFramingObjectMaxPara}
    By \cite{choBasicLocusUnitary2018} Remark 3.32, up to quasi-isogeny, there is a unique supersingular hermitian $\mathcal O_E$-$\mathcal O_F$-$h$-module of signature $(1,n-1)$ over $\mathrm{Spec}(\overline{\mathbb F_{q}})$. In particular, the formal scheme $\mathcal N_{\mathbb X^{[h]}}$ does not depend on the precise choice of $\mathbb X^{[h]}$. This allows us to write $\mathcal N^{h} \coloneq  \mathcal N_{\mathbb X^{[h]}}$ when we do not need to make a reference to $\mathbb X^{[h]}$. The space $\mathcal N^{h}$ is called the \textit{relative basic unramified unitary PEL Rapoport-Zink space of signature $(1,n-1)$ of maximal parahoric level $h$}. We also write $\mathcal N^{h}_{\mathrm{red}}$ for its reduced special fiber over $\mathrm{Spec}(\overline{\mathbb F_q})$.
\end{rk}
For $S \in \textbf{Nilp}_{\mathcal O_{\breve F}}$ and $(X,i_X,\lambda_X,\rho_X) \in \mathcal N^{h}(S)$, locally on $S$ there is a constant $c(\rho_X) \in F^{\times}$ such that 
\begin{equation*}
    \rho_X^{\vee}\circ (\lambda_{\mathbb X}^{[h]})_{\overline S} \circ \rho_X = c(\rho_X) \circ (\lambda_X)_{\overline{S}}.
\end{equation*}
It follows that the $F$-height $\mathrm{ht}_F(\rho_X)$ of $\rho_X$ is divisible by $n$, and locally on $S$ we have 
\begin{equation*}
    v_{\pi}(c(\rho_X)) = \frac{\mathrm{ht}_F(\rho_X)}{n} \in \mathbb Z.
\end{equation*}
This defines a continuous map 
\begin{equation*}
    \Delta: \mathcal N^h \longrightarrow \mathbb Z, \qquad (X,i_X,\lambda_X,\rho_X) \longmapsto v_{\pi}(c(\rho_X)),
\end{equation*}
where $\mathbb Z$ is given the discrete topology. We have a decomposition into a disjoint union of open and closed formal subschemes
\begin{equation*}
    \mathcal N^h = \bigsqcup_{j\in \mathbb Z} \mathcal N_j^{h},
\end{equation*}
where $\mathcal N_{j}^{h} \coloneq  \Delta^{-1}\{j\}$. 
\begin{rk}
    In \cite{choBasicLocusUnitary2018}, the main object of study is the component $\mathcal N_{0}^{h}$ consisting of quasi-isogenies of height $0$. This is actually enough to understand the whole of $\mathcal N^h$ according to \eqref{Eq:ActionOfJ} and Proposition \ref{Prop:ConnectedComponentNonEmptyMaxLevel} below. 
\end{rk}
By Dieudonné theory, the group $J(F) \simeq \mathrm{GU}(\mathbb V_{K,0},(\cdot,\cdot))$ of Proposition \ref{Prop:GroupJIsomorphism} is identified with the group of quasi-isogenies $\mathbb X^{[h]} \to \mathbb X^{[h]}$, i.e. $\mathcal O_E$-linear quasi-isogenies $g:\mathbb X \to \mathbb X$ such that $g^{\vee}\circ \lambda_{\mathbb X}^{[h]} \circ g$ differs from $\lambda_{\mathbb X}^{[h]}$ by a scalar in $F^{\times}$. We have a natural action of $J(F)$ on $\mathcal N^{h}$ given by 
\begin{equation}\label{Eq:ActionOfJ}
    \forall g \in J(F), \qquad g\cdot (X,i_X,\lambda_X,\rho_X) \coloneq  (X,i_X,\lambda_X,(g)_{\overline{S}}\circ \rho_X).
\end{equation}
This action induces, for $g \in J(F)$, an isomorphism 
\begin{equation}\label{Eq:ActionOfJOnConnectedComponents}
    g: \mathcal N_{j}^{h} \xrightarrow{\sim} \mathcal N_{j+\alpha_J(g)}^{h},
\end{equation}
where $\alpha_J: J(F) \to \mathbb Z$ is defined as in Proposition \ref{Prop:ImageOfalpha} relative to $J(F) \simeq \mathrm{GU}(\mathbb V_{K,0},(\cdot,\cdot))$. 
\begin{prop}\label{Prop:PointsMaximalParahoric}
    Let $k$ be a perfect field containing $\overline{\mathbb F_{q}}$. For $j \in \mathbb Z$, there is a natural bijection
    \begin{equation*}
        \mathcal N_{j}^{h}(k) \simeq \left\{ \begin{array}{c}
        W_{\mathcal O_F}(k)\text{-lattices in }\mathbb V_{F_k,0}\\
        A \subseteq B 
        \end{array}
        \;\middle|\; 
        \begin{array}{c}
        \pi^{j+1} A^{\vee} \overset{1}{\subseteq} B \subseteq \pi^{j}A^{\vee},\\
        \pi^{j+1} B^{\vee} \overset{1}{\subseteq} A \subseteq \pi^{j}B^{\vee},\\
        \pi B \subseteq A \overset{h}{\subseteq} B 
        \end{array}\right\},
    \end{equation*}
    where $\cdot^{\vee}$ denotes the dual lattice with respect to the pairing $(\cdot,\cdot)$ as defined in \eqref{Eq:DefiDualLatticeVee}.
\end{prop}

\begin{proof}
    The proof is essentially the same as \cite{choBasicLocusUnitary2018} Propositions 2.2 and 2.4. Let $(X,i_X,\lambda_X,\rho_X) \in \mathcal N_{j}^{h}(k)$. The image via $\rho_X$ of the relative Dieudonné module of $X$ is an $\mathcal O_E \otimes_{\mathcal O_F} W_{\mathcal O_F}(k)$-lattice $M \subset \mathbb V_{F_k}$ which is stable under $\mathcal F$ and $\mathcal V$. Moreover, if we decompose $M = M_0 \oplus M_1$ where $M_i \coloneq  \eta_i M$ for $i=0,1$ then the compatibility of $\rho_X$ with respect to the polarizations is equivalent to 
    \begin{align*}
        \pi^{j+1}M_1^{\dagger} \overset{n-h}{\subseteq} M_0 \overset{h}{\subseteq} \pi^{j}M_1^{\dagger}, & & \pi^{j+1}M_0^{\dagger} \overset{n-h}{\subseteq} M_1 \overset{h}{\subseteq} \pi^{j}M_0^{\dagger}.
    \end{align*}
    Here, given a $W_{\mathcal O_F}(k)$-lattice $L$ in $\mathbb V_{F_k,e}$, we wrote 
    \begin{equation*}
        L^{\dagger} \coloneq  \{v \in \mathbb V_{F_k,e+1} \mid \langle v, L \rangle_H \subseteq W_{\mathcal O_F}(k)\},
    \end{equation*}
    where $e$ is understood in $\mathbb Z/2\mathbb Z$. Eventually, the signature $(1,n-1)$ condition is equivalent to 
    \begin{align*}
        \dim_{k}(M_0/\mathcal VM_1) = 1, & & \dim_k(M_1/\mathcal VM_0) = n-1.
    \end{align*}
    Define $A \coloneq  M_0$ and $B \coloneq  \pi^{j}M_1^{\dagger}$. From the definition of the form $(\cdot,\cdot)$ in \eqref{Eq:DefinitionForm(,)OnV0}, we have 
    \begin{equation*}
        \pi^{j+1} B^{\vee} = \mathcal F M_1.
    \end{equation*}
    By the signature condition, we have $\mathcal VM_0 \overset{n-1}{\subseteq} M_1 \overset{1}{\subseteq} \pi^{-1}\mathcal VM_0 = \mathcal F^{-1}M_0$. Applying $\mathcal F$, we obtain $\pi^{j+1}B^{\vee} \overset{1}{\subseteq} A \overset{n-1}{\subseteq} \pi^{j}B^{\vee}$.\\
    On the other hand, the signature condition also gives $\mathcal VM_1 \overset{1}{\subseteq} M_0 \overset{n-1}{\subseteq} \mathcal F^{-1}M_1$. Taking duals and using \eqref{Eq:FormulaVeeTau}, this is equivalent to $\pi^{j+1}A^{\vee} \overset{1}{\subseteq} B \overset{n-1}{\subseteq} \pi^{j}A^{\vee}$ as desired.\\
    Conversely, the lattice $M$ can easily be recovered from the pair $(A \subseteq B)$, giving rise to the inverse map.
\end{proof}
\begin{rk}
    Upon taking duals, the conditions 
    \begin{align*}
        \pi^{j+1} A^{\vee} \overset{1}{\subseteq} B \subseteq \pi^{j}A^{\vee}, & & \pi^{j+1} B^{\vee} \overset{1}{\subseteq} A \subseteq \pi^{j}B^{\vee},
    \end{align*}
    are equivalent to 
    \begin{align*}
        \pi^{j+1} B^{\vee} \overset{1}{\subseteq} \tau(A) \subseteq \pi^{j}B^{\vee}, & & \pi^{j+1} A^{\vee} \overset{1}{\subseteq} \tau(B) \subseteq \pi^{j}A^{\vee},
    \end{align*}
\end{rk}
\begin{prop}\label{Prop:ConnectedComponentNonEmptyMaxLevel}
    We have 
    \begin{equation*}
        \mathcal N_{j}^{h} \not = \emptyset \iff nj \text{ is even}.
    \end{equation*}
\end{prop}

\begin{proof}
    Clearly, $\mathbb X^{[h]} \otimes_{\mathbb F_{q^2}} \overline{\mathbb F_{q}}$ together with the identity defines an $\overline{\mathbb F_{q}}$-rational point of $\mathcal N_{0}^{h}$. By \eqref{Eq:ActionOfJOnConnectedComponents}, it follows that $\mathcal N_{j}^{h} \not = \emptyset$ for all $j \in \mathrm{Im}(\alpha_J)$. By Proposition \ref{Prop:ImageOfalpha}, $\mathrm{Im}(\alpha_J)$ is precisely the set of integers $j \in \mathbb Z$ such that $nj$ is even.\\
    Conversely, assume that $\mathcal N_{j}^{h} \not = \emptyset$ for some $j \in \mathbb Z$. Since it is locally formally of finite type, it has a point $(A \subseteq B) \in \mathcal N_{j}^{h}(\overline{\mathbb F_{q}})$. On the one hand, we have 
    \begin{align*}
        [A^{\vee}:A] & = [A^{\vee}:\pi^jA^{\vee}] + [\pi^jA^{\vee}:B] + [B:A]\\
        & = nj + n-1 + h,
    \end{align*}
    where $[\cdot:\cdot]$ is defined for $(\mathbb V_{\breve F,0}, (\cdot,\cdot))$ similarly to \eqref{Eq:DefinitionIndexLattices}. On the other hand, there is a distinguished lattice $\Lambda_{0, \mathcal O_{\breve F}} \coloneq  \Lambda_{0}\otimes_{\mathcal O_K} \mathcal O_{\breve F}$ in $\mathbb V_{\breve F,0}$. We have 
    \begin{align*}
        [A^{\vee}:A] & = [A^{\vee}:\Lambda_{0, \mathcal O_{\breve F}}^{\vee}] + [\Lambda_{0, \mathcal O_{\breve F}}^{\vee}:\Lambda_{0, \mathcal O_{\breve F}}] + [\Lambda_{0, \mathcal O_{\breve F}}:A]\\
        & = 2[\Lambda_{0, \mathcal O_{\breve F}}:A] + n - (1-\underline{\varepsilon}).
    \end{align*}
    It follows that 
    \begin{align*}
        & nj + n-1 + h \equiv n - (1-\underline{\varepsilon}) \mod 2 \\
        \iff & nj \equiv h + \underline{\varepsilon} \equiv 0 \mod 2,
    \end{align*}
    since $\varepsilon$ has been determined precisely so that $h \equiv \underline{\varepsilon} \mod 2$.
\end{proof}
More generally, we now introduce the Rapoport-Zink spaces of deeper parahoric levels. Let $\varepsilon \in \{\pm 1\}$. The lattice chain in $\mathbb V^{\varepsilon}$
\begin{equation*}
    L_{\nu_{\varepsilon}} \overset{1}{\subseteq} L_{\nu_{\varepsilon}-1} \overset{1}{\subseteq} \ldots \overset{1}{\subseteq} L_0,
\end{equation*}
induces a family of $\mathcal O_E$-linear isogenies $\alpha_{r+1,r}: \mathbb X^{[h_{r+1}^{\varepsilon}]} \to \mathbb X^{[h_{r}^{\varepsilon}]}$ for all $0 \leq r \leq \nu_{\varepsilon} - 1$. Given two integers $0 \leq r \leq r' \leq \nu_{\varepsilon}$, we define $\alpha_{r',r} \coloneq  \alpha_{r+1,r} \circ \alpha_{r+2,r+1} \circ \cdots \circ \alpha_{r',r'-1}$. We have $\mathrm{Ker}(\alpha_{r',r}) \subseteq \mathbb X[\pi]$ and $\mathrm{Ker}(\alpha_{r',r})$ has order $q^{h_{r'}^{\varepsilon}-h_r^{\varepsilon}} = q^{2(r'-r)}$. Moreover, up to trivialization, the polarizations can be chosen so that $\lambda_{\mathbb X}^{[h_{r'}^{\varepsilon}]} = \alpha_{r',r}^{\vee} \circ \lambda_{\mathbb X}^{[h_r^{\varepsilon}]} \circ \alpha_{r',r}$ for all $0 \leq r \leq r' \leq \nu_{\varepsilon}$.

Let $I \subseteq \{0,\ldots , \nu_{\varepsilon}\}$ be a non-empty subset. 
\begin{defi}\label{Defi:ConsecutiveIntegers}
    We say that two integers $r, r' \in I$ such that $r < r'$ are \textit{consecutive in $I$} if $I \cap \{r+1, \ldots , r'-1\} = \emptyset$.
\end{defi}
We define a functor $\mathcal N_{E/F,L^I_{\bullet}}$ as follows. For $S \in \textbf{Nilp}$, let $\mathcal N_{E/F,L^I_{\bullet}}(S)$ denote the set of tuples $(X_r,i_{X_r},\lambda_{X_r},\rho_{X_r})_{r \in I}$ up to isomorphism, where 
\begin{itemize}
\item for all $r \in I$, $(X_r,i_{X_r},\lambda_{X_r},\rho_{X_r})\in \mathcal N_{E/F,\mathbb X^{[h_{r}^{\varepsilon}]}}(S)$,
\item for every consecutive $r<r' \in I$, there exists an isogeny $\widetilde{\alpha}_{r',r}:X_{r'}\to X_r$ such that $\mathrm{Ker}(\widetilde{\alpha}_{r',r}) \subseteq X_{r'}[\pi]$, $\mathrm{Ker}(\widetilde{\alpha}_{r',r})$ has order $q^{2(r'-r)}$, and the following diagram commutes 
\begin{center}
\begin{tikzcd}[column sep = large]
(X_{r'})_{\overline S} \arrow[r,"(\widetilde{\alpha}_{r',r})_{\overline S}"] \arrow[d,swap,"\rho_{X_{r'}}"] & (X_r)_{\overline S} \\
\mathbb X_{\overline S}^{[h_{r'}^{\varepsilon}]} \arrow[r,"(\alpha_{r',r})_{\overline S}"] & \mathbb X_{\overline S}^{[h_r^{\varepsilon}]} \arrow[u,swap,"\rho_{X_r}^{-1}"]
\end{tikzcd}
\end{center}
\end{itemize}
When it exists, the isogeny $\widetilde{\alpha}_{r',r}$ is unique. An isomorphism 
\begin{equation*}
(X_r,i_{X_r},\lambda_{X_r},\rho_{X_r})_{r \in I} \xrightarrow{\sim} (X_r',i_{X_r'},\lambda_{X_r'},\rho_{X_r'})_{r \in I}
\end{equation*}
is a collection of isomorphisms $\beta_r:(X_r,i_{X_r},\lambda_{X_r},\rho_{X_r}) \xrightarrow{\sim} (X_r',i_{X_r'},\lambda_{X_r'},\rho_{X_r'})$ in the sense of the maximal parahoric case, such that 
\begin{equation*}
\beta_{r}^{-1} \circ \widetilde{\alpha'}_{r',r}\circ\beta_{r'} = \widetilde{\alpha}_{r',r}
\end{equation*}
for all consecutive $r<r' \in I$. Adapting \cite{rapoportPeriodSpacesPdivisible1996} to the relative setting as in \cite{mihatschRELATIVEUNITARYRZSPACES2022} and \cite{choBasicLocusUnitary2018}, one may prove the following Proposition.
\begin{prop}
The functor $\mathcal N_{E/F,L^I_{\bullet}}\otimes \mathcal O_{\breve F}$ is represented by a formal scheme over $\mathrm{Spf}(\mathcal O_{\breve F})$ which is locally formally of finite type and regular.
\end{prop}
From now on we write $\mathcal N_{L^I_{\bullet}}$ for the formal scheme over $\mathrm{Spf}(\mathcal O_{\breve F})$ given by the Proposition. 
\begin{rk}\label{Rk:IndependenceFramingObjectsArbitraryPara}
    Just as mentioned in Remark \ref{Rk:IndependenceFramingObjectMaxPara}, the precise choice of framing objects $\mathbb X^{[h_r^{\varepsilon}]}$ does not matter, so that we will write $\mathcal N^{I,\varepsilon} \coloneq  \mathcal N_{L^I_{\bullet}}$ when we do not need to make a reference to the framing objects. The space $\mathcal N^{I,\varepsilon}$ is called the \textit{relative basic unramified unitary PEL Rapoport-Zink space of signature $(1,n-1)$ of parahoric level $K_I$}. We also write $\mathcal N^{I,\varepsilon}_{\mathrm{red}}$ for its reduced special fiber over $\mathrm{Spec}(\overline{\mathbb F_q})$. 
\end{rk}
Let $(X_r,i_{X_r},\lambda_{X_r},\rho_{X_r})_{r \in I} \in \mathcal N^{I,\varepsilon}(S)$. For consecutive $r<r' \in I$, the relation
\begin{equation*}
    (\alpha_{r',r})_{\overline{S}} \circ \rho_{X_{r'}} = \rho_{X_r}\circ (\widetilde{\alpha}_{r',r})_{\overline{S}}
\end{equation*}
implies that $\mathrm{ht}_F(\rho_{X_{r'}}) = \mathrm{ht}_F(\rho_{X_r})$ locally on $\overline{S}$. In particular, if $S$ is connected then there exists a unique $j \in \mathbb Z$ such that $nj$ is even, and $(X_r,i_{X_r},\lambda_{X_r},\rho_{X_r}) \in \mathcal N_{j}^{h_{r}^{\varepsilon}}(S)$ for all $r \in I$. This induces a decomposition 
\begin{equation*}
    \mathcal N^{I, \varepsilon} = \bigsqcup_{j \in \mathbb Z} \mathcal N_{j}^{I, \varepsilon},
\end{equation*}
and we have $\mathcal N_{j}^{I, \varepsilon} \not = \emptyset$ if and only if $nj$ is even. Moreover, we have an action of $J(F)$ on $\mathcal N^{I, \varepsilon}$ by letting $g \in J(F)$ act on each individual component $(X_r,i_{X_r},\lambda_{X_r},\rho_{X_r})$. This induces isomorphisms 
\begin{equation*}
    g:\mathcal N_{j}^{I,\varepsilon} \xrightarrow{\sim} \mathcal N_{j+\alpha_J(g)}^{I,\varepsilon}.
\end{equation*}
\begin{rk}
    In \cite{mullerBruhatTitsStratificationBasic2026}, we studied the component $\mathcal N_{0}^{I, \varepsilon}$ consisting of quasi-isogenies of height $0$.
\end{rk}
\begin{prop}\label{Prop:PointsArbitraryParahoric}
        Let $k$ be a perfect field containing $\overline{\mathbb F_{q}}$. For $j \in \mathbb Z$ such that $nj$ is even, there is a natural bijection
    \begin{equation*}
        \mathcal N_{j}^{I,\varepsilon}(k) \simeq \left\{ \begin{array}{c}
        W_{\mathcal O_F}(k)\text{-lattices in }\mathbb V_{F_k,0}\\
        (A_r \subseteq B_r)_{r \in I} 
        \end{array}
        \;\middle|\; \forall r \in I,
        \begin{array}{c}
        \pi^{j+1} A^{\vee}_r \overset{1}{\subseteq} B_r \subseteq \pi^{j}A^{\vee}_r,\\
        \pi^{j+1} B^{\vee}_r \overset{1}{\subseteq} A_r \subseteq \pi^{j}B^{\vee}_r,\\
        \pi B_r \subseteq A_r \overset{h_{r}^{\varepsilon}}{\subseteq} B_r, 
        \end{array} \text{ and } \begin{array}{c}
        A_{r'} \subset A_r, \\
        B_r \subset B_{r'}. 
        \end{array} \right\}
    \end{equation*}
    where $r<r'$ run over all consecutive integers in $I$.
\end{prop}
The proof follows easily by construction and from Proposition \ref{Prop:PointsMaximalParahoric}.

\begin{rk}
    If $I = \{r\}$ is a singleton, then $\mathcal N^{I,\varepsilon} = \mathcal N^{h_r^{\varepsilon}}$. We recover the space of maximal parahoric level introduced in \ref{Rk:IndependenceFramingObjectMaxPara}. Besides, if $I = \{0,\ldots , \nu_{\varepsilon}\}$ then the Rapoport-Zink space $\mathcal N^{I,\varepsilon}$ is said to be of Iwahori level. 
\end{rk}

For any $0 \leq h \leq n$ of parity $\underline{\varepsilon}$, we can consider the dual framing object $(\mathbb X^{[h]})^{\vee} \coloneq  (\mathbb X^{\vee}, i_{\mathbb X^{\vee}},\lambda_{\mathbb X^{\vee}}^{[n-h]})$ defined as follows: 
\begin{itemize}
    \item $\mathbb X^{\vee}$ is the Serre dual of $\mathbb X$,
    \item $i_{\mathbb X^{\vee}}(\lambda) \coloneq  i_{\mathbb X}(\lambda^{*})^{\vee}$ for all $\lambda \in \mathcal O_E$,
    \item $\lambda_{\mathbb X^{\vee}}^{[n-h]}$ is the polarization on $\mathbb X^{\vee}$ determined by the condition 
    \begin{equation*}
        \lambda_{\mathbb X^{\vee}}^{[n-h]}\circ \lambda_{\mathbb X}^{[h]} = i_{\mathbb X}(\pi).
    \end{equation*}
\end{itemize}
Then $(\mathbb X^{[h]})^{\vee}$ is a supersingular hermitian $\mathcal O_E$-$\mathcal O_F$-$(n-h)$-module of signature $(1,n-1)$ over $\mathbb F_{q^2}$. If $h = 2r+\underline{\varepsilon}$, then by Dieudonné theory the PEL datum associated to $(\mathbb X^{[h]})^{\vee}$ is $(\mathbb V_K^{\varepsilon}, \widetilde{L}_{\nu_{\varepsilon}-r,\mathcal O_K},\mathcal F,\mathcal V,\langle\cdot,\cdot\rangle_{H'})$ where $\langle\cdot,\cdot\rangle_{H'} \coloneq  \pi\langle\cdot,\cdot\rangle_H$, and $\widetilde{L}_{\bullet}$ is the dual lattice chain introduced in \eqref{Eq:DefDualChain}. 
\begin{prop}
    The mapping 
    \begin{equation*}
        (X,i_X,\lambda_X,\rho_X) \longmapsto (X^{\vee},i_{X^{\vee}},\lambda_{X^{\vee}},(\rho_X^{\vee})^{-1})
    \end{equation*}
    defines an isomorphism $\mathcal N_{\mathbb X^{[h]}} \xrightarrow{\sim} \mathcal N_{(\mathbb X^{[h]})^{\vee}}$.
\end{prop}
\begin{rk}
    In view of Remark \ref{Rk:IndependenceFramingObjectMaxPara}, this is an isomorphism 
    \begin{equation*}
        \mathcal N^{h} \xrightarrow{\sim} \mathcal N^{n-h}.
    \end{equation*}
    This phenomenon was already noticed in \cite{choBasicLocusUnitary2018} Remark 5.2. When $n$ is odd, $h$ and $n-h$ have opposite parity, so that any choice of $\varepsilon \in \{\pm 1\}$ produces the same Rapoport-Zink spaces up to isomorphism. When $n$ is even, the identification $\mathcal N_{(\mathbb X^{[h]})^{\vee}} \xrightarrow{\sim} \mathcal N_{\mathbb X^{[n-h]}}$ can be made explicit by the action of the isometry $g_0:(\mathbb V,H') \xrightarrow{\sim} (\mathbb V,H)$ defined in Definition \ref{Defi:Elementg_0}. Indeed, if $h = 2r + \underline{\varepsilon}$, then $n-h = 2(\nu_{\varepsilon}-r)+\underline{\varepsilon}$ and we have $g_0(\widetilde{L}_{\nu_{\varepsilon}-r}) = L_{\nu_{\varepsilon}-r}$. Thus, $g_0$ induces an isomorphism $(\mathbb X^{[h]})^{\vee} \xrightarrow{\sim} \mathbb X^{[n-h]}$.
\end{rk}
More generally, for all $0 \leq r \leq \nu_{\varepsilon}-1$, we have $\mathcal O_E$-linear isogenies $\alpha_{r+1,r}^{\vee} : (\mathbb X^{[h_r^{\varepsilon}]})^{\vee} \to (\mathbb X^{[h_{r+1}^{\varepsilon}]})^{\vee}$. Given $I \subseteq \{0,\ldots , \nu_{\varepsilon}\}$ non-empty, these framing objects and isogenies define the Rapoport-Zink space $\mathcal N_{\widetilde{L}^I_{\bullet}}$ where $\widetilde{L}_{\bullet}^I$ is the lattice chain defined in \eqref{Eq:DefiWidetildeLI}.
\begin{prop}\label{Prop:IsomorphismDualityArbitraryParahoric}
    The mapping 
    \begin{equation*}
        (X_r,i_{X_r},\lambda_{X_r},\rho_{X_r})_{r \in I} \longmapsto (X^{\vee}_{\nu_{\varepsilon}-\widetilde r},i_{X^{\vee}_{\nu_{\varepsilon}-\widetilde r}},\lambda_{X^{\vee}_{\nu_{\varepsilon}-\widetilde r}},(\rho_{X_{\nu_{\varepsilon}-\widetilde r}}^{\vee})^{-1})_{\widetilde r \in \widetilde I}
    \end{equation*}
    defines an isomorphism $\Phi_{I}:\mathcal N_{L^I_{\bullet}} \xrightarrow{\sim} \mathcal N_{\widetilde{L}^I_{\bullet}}$.
\end{prop}
\begin{rk}\label{Rk:IsomorphismDualityArbitraryParahoric}
    In view of Remark \ref{Rk:IndependenceFramingObjectsArbitraryPara}, this is an isomorphism 
    \begin{equation*}
        \Phi_I:\mathcal N^{I,\varepsilon} \xrightarrow{\sim} \mathcal N^{\widetilde I, \varepsilon'},
    \end{equation*}
    where $\varepsilon' = (-1)^{n}\varepsilon$ and $\widetilde I \subseteq \{0,\ldots , \nu_{\varepsilon'}\}$ is defined in \eqref{Eq:DefiWidetildeI}. Thus again, when $n$ is odd, any choice of $\varepsilon \in \{\pm 1\}$ produces the same Rapoport-Zink spaces up to isomorphism. When $n$ is even, the identification $\mathcal N_{\widetilde{L}_{\bullet}^{I}} \xrightarrow{\sim} \mathcal N_{L_{\bullet}^{\widetilde I}}$ is induced by $g_0$, since we have $g_0(\widetilde{L}_{\bullet}^{I}) = L^{\widetilde I}_{\bullet}$. We point out that when $n$ is even and $I = \widetilde I$, the mapping of Proposition \ref{Prop:IsomorphismDualityArbitraryParahoric} gives a non-trivial automorphism of $\mathcal N^{I,\varepsilon}$.
\end{rk}
Let us reformulate Proposition \ref{Prop:PointsArbitraryParahoric} in a more uniform way that is more suited for later computations. To this end, we introduce some combinatorial notations which will be omnipresent in the remaining of the exposition.
\begin{defi}\label{Defi:DefiIntegerli}
    For all $i \in \overline I \subseteq \mathbb Z/n\mathbb Z$, we define 
    \begin{equation*}
        \ell_i \coloneq  \min \{1 \leq \ell \leq n \mid i+\ell \in \overline I \}.
    \end{equation*} 
\end{defi}
If $r < r' \in I$ are consecutive, then 
\begin{equation*}
    \ell_{\overline r} = \ell_{\overline{-r'-\underline{\varepsilon}}} = r'-r.
\end{equation*}
Besides, if we write $r_{\mathrm{min}} \coloneq \min(I)$ and $r_{\mathrm{max}} \coloneq \max(I)$, then 
\begin{equation*}
    \begin{aligned}
        \text{if } h_{r_{\mathrm{max}}}^{\varepsilon} \not = n, & \text{ then } \ell_{\overline{r_{\mathrm{max}}}} = n-h_{r_{\mathrm{max}}}^{\varepsilon}, \\
        \text{if } h_{r_{\mathrm{min}}}^{\varepsilon} \not = 0, & \text{ then } \ell_{\overline{-r_{\mathrm{min}}-\underline{\varepsilon}}} = h_{r_{\mathrm{min}}}^{\varepsilon},
    \end{aligned}
\end{equation*}
We point out that if $h_{r_{\mathrm{min}}}^{\varepsilon} = 0$ (resp. $h_{r_{\mathrm{max}}}^{\varepsilon} = n$) then $\overline{-r_{\mathrm{min}}-\underline{\varepsilon}} = \overline{r_{\mathrm{min}}}$ (resp. $\overline{r_{\mathrm{max}}} = \overline{-r_{\mathrm{max}}-\underline{\varepsilon}}$), so that we have just computed $\ell_i$ for all $i \in \overline I$.

\begin{prop}\label{Prop:PointsArbitraryParahoricBis}
    Let $k$ be a perfect field containing $\overline{\mathbb F_{q}}$. For $j \in \mathbb Z$ such that $nj$ is even, there is a natural bijection
    \begin{equation*}
        \mathcal N_{j}^{I, \varepsilon}(k) \simeq \left\{ \begin{array}{c}
        \text{Decreasing chains } (Z_r)_{r \in I_{\mathbb Z}}\\
        \text{of }W_{\mathcal O_F}(k)\text{-lattices in }\mathbb V_{F_k,0}
        \end{array}
        \;\middle|\; \forall r \in I_{\mathbb Z},
        \begin{array}{c}
        \pi^{j+1}Z_{-r-\underline{\varepsilon}}^{\vee} \overset{1}{\subseteq} Z_r \subseteq \pi^{j}Z^{\vee}_{-r-\underline{\varepsilon}},\\
        Z_{r + \ell_{\overline r}} \overset{\ell_{\overline r}}{\subseteq} Z_r,\\
        Z_{r+n} = \pi Z_r.
        \end{array}\right\}.
    \end{equation*}
\end{prop}
\begin{proof}
    Starting with a point $(A_r \subseteq B_r)_{r \in I} \in \mathcal N_{j}^{I,\varepsilon}(k)$ as in Proposition \ref{Prop:PointsArbitraryParahoric}, we define $Z_{r} \coloneq  A_{r}$ and $Z_{-r-\underline{\varepsilon}} \coloneq  B_r$ for all $r \in I$. We extend it to a decreasing chain $(Z_{r})_{r \in I_{\mathbb Z}}$ by $\pi$-periodicity. According to \cite{mullerBruhatTitsStratificationBasic2026} Lemma 2.23, for all consecutive $r < r' \in I$ we have 
    \begin{equation*}
        A_{r'} \overset{r'-r}{\subset} A_r, \qquad B_r \overset{r'-r}{\subset} B_{r'}.
    \end{equation*}
    Furthermore, we have 
    \begin{equation*}
        A_{r_{\mathrm{min}}} \overset{h_{r_{\mathrm{min}}}^{\varepsilon}}{\subseteq} B_{r_{\mathrm{min}}}, \qquad \pi B_{r_{\mathrm{max}}} \overset{n-h_{r_{\mathrm{max}}}^{\varepsilon}}{\subseteq} A_{r_{\mathrm{max}}}.
    \end{equation*} 
    This proves that $(Z_r)_{r\in I_{\mathbb Z}}$ belongs to the set on the RHS of the Proposition, and it is clear that this defines the desired bijection.
\end{proof}

With this description, we record the following Lemma.

\begin{lem}\label{Lem:LatticeChainUnderDuality}
    Let $k$ be a perfect field containing $\overline{\mathbb F_q}$, let $j \in \mathbb Z$ be such that $nj$ is even, and let $z = (Z_r)_{r\in I_{\mathbb Z}} \in \mathcal N_j^{I,\varepsilon}(k)$. Then $\Phi_I(z) \in \mathcal N_{-j}^{\widetilde I,\varepsilon'}(k)$. Moreover if we write $\Phi_I(z) = (\widetilde Z_{\widetilde r})_{\widetilde r \in \widetilde I_{\mathbb Z}}$, then 
    \begin{equation*}
        \forall r\in I_{\mathbb Z}, \qquad \widetilde Z_{\nu_{\varepsilon}-r} = \pi^{-j}Z_{-r-\underline{\varepsilon}}.
    \end{equation*}
\end{lem}

\begin{proof}
    For $r\in I$, let $M_r = M_{r,0} \oplus M_{r,1} \subseteq \mathbb V_{F_k}$ be the image under $\rho_{X_r}$ of the relative Dieudonné module of $X_r$. By construction, we have $Z_r = M_{r,0}$ and $Z_{-r-\underline{\varepsilon}} = \pi^j M_{r,1}^{\dagger}$. The relative Dieudonné module associated with $(X_r^\vee,(\rho_{X_r}^\vee)^{-1})$ is $M_{r,1}^{\dagger} \oplus M_{r,0}^{\dagger}$, and we have $\mathrm{ht}_F((\rho_{X_r}^\vee)^{-1}) = - \mathrm{ht}_F(\rho_{X_r}) = -nj$. Thus, we obtain
    \begin{equation*}
        \forall r \in I, \qquad \widetilde Z_{\nu_{\varepsilon}-r} = M_{r,1}^{\dagger} = \pi^{-j}Z_{-r-\underline{\varepsilon}}.
    \end{equation*}
    Recall that $H' = \pi H$, so that $L^{\dagger'} = \pi^{-1}L^\dagger$ for every lattice $L$. We therefore obtain
    \begin{equation*}
        \forall r \in I, \qquad \widetilde Z_{-(\nu_{\varepsilon}-r)-\underline{\varepsilon'}} = \pi^{-j} (M_{r,0}^{\dagger})^{\dagger'} = \pi^{-j-1} M_{r,0} = \pi^{-j-1} Z_r.
    \end{equation*}
    Using the identity $2\nu_{\varepsilon} + \underline{\varepsilon} + \underline{\varepsilon'} = n$, this gives the desired formula. We conclude for all $r \in I_{\mathbb Z}$ by $\pi$-periodicity.
\end{proof}

\subsection{The local model}

In this section, we introduce the local model diagram related to the relative Rapoport-Zink space $\mathcal N^{I,\varepsilon}$. Here $\varepsilon \in \{\pm 1\}$ and $I \subseteq \{0,\ldots , \nu_{\varepsilon}\}$ non-empty are fixed. Recall the lattice chain $L_{\bullet}^I$ defined in \eqref{Eq:DefL^I_bullet}. For $S$ an $\mathcal O_K$-scheme, we consider 
\begin{equation*}
    L_{\bullet}^I \otimes_{\mathcal O_F} \mathcal O_S = (L_{r} \otimes_{\mathcal O_F} \mathcal O_S)_{r \in I_{\mathbb Z}}.
\end{equation*}
Each $L_{r} \otimes_{\mathcal O_F} \mathcal O_S$ is a locally free $\mathcal O_E \otimes_{\mathcal O_F} \mathcal O_S$-module of rank $n$. 

\begin{defi}
    Let $M_{\mathrm{loc}}^{I,\varepsilon}$ denote the functor sending an $\mathcal O_K$-scheme to the set of tuples $(\mathfrak t_r)_{r \in I_{\mathbb Z}}$ consisting of subsheaves $\mathfrak t_r \subseteq L_{r} \otimes_{\mathcal O_F} \mathcal O_S$ of $\mathcal O_E \otimes_{\mathcal O_F} \mathcal O_S$-modules such that 
    \begin{enumerate}
        \item for $r,r' \in I_{\mathbb Z}$ with $r < r'$, the transition map $L_{r'} \otimes_{\mathcal O_F} \mathcal O_S \to L_{r} \otimes_{\mathcal O_F} \mathcal O_S$ maps $\mathfrak t_{r'}$ into $\mathfrak t_r$.
        \item as a sheaf of $\mathcal O_S$-modules, $\mathfrak t_{r}$ is locally a free direct summand of $L_{r} \otimes_{\mathcal O_F} \mathcal O_S$ of rank $n$,
        \item under the pairing 
        \begin{equation*}
            \langle \cdot,\cdot \rangle_H \otimes_{\mathcal O_F} \mathcal O_S : (L_r \otimes_{\mathcal O_F} \mathcal O_S) \times (L_{-r-\underline{\varepsilon}} \otimes_{\mathcal O_F} \mathcal O_S) \longrightarrow \mathcal O_S,
        \end{equation*}
        the subsheaves $\mathfrak t_r$ and $\mathfrak t_{-r-\underline{\varepsilon}}$ map to $0$ for all $r \in I_{\mathbb Z}$,
        \item for every $k \in \mathbb Z$ and $r \in I_{\mathbb Z}$, the map $\pi^k\otimes 1:L_{r} \otimes_{\mathcal O_F} \mathcal O_S \to L_{r+kn}\otimes_{\mathcal O_F} \mathcal O_S$ maps $\mathfrak t_{r}$ onto $\mathfrak t_{r+kn}$,
        \item the determinant condition is satisfied, i.e. for all $r \in I_{\mathbb Z}$,
        \begin{equation*}\hspace{-0.5cm}
            \forall \lambda \in \mathcal O_E, \qquad \det(T - \lambda\otimes 1 \mid \mathfrak t_r) = (T-\varphi_0(\lambda))^{n-1}(T-\varphi_1(\lambda)) \in \mathcal O_S[T].
        \end{equation*}
    \end{enumerate}
\end{defi}
\begin{rk}
    The signature $(n-1,1)$ condition imposed on $\mathfrak t_r$ in 5. is equivalent to the signature $(1,n-1)$ condition imposed on the quotients $(L_r \otimes_{\mathcal O_F}\mathcal O_S)/\mathfrak t_r$. 
\end{rk}
\begin{prop}\label{Prop:IsomorphismWithDrinfeldCase}
    The functor $M_{\mathrm{loc}}^{I,\varepsilon}$ is represented by a projective scheme over $\mathrm{Spec}(\mathcal O_K)$ which has strict semi-stable reduction. In particular, it is regular.
\end{prop}
\begin{proof}
    Let $S$ be a scheme over $\mathcal O_K$. We have a decomposition
    \begin{equation*}
        L_{r} \otimes_{\mathcal O_F} \mathcal O_S = (\Lambda_r \otimes_{\mathcal O_K} \mathcal O_S) \oplus (\eta_1L_{r,\mathcal O_K} \otimes_{\mathcal O_K} \mathcal O_S),
    \end{equation*}
    where each summand is a locally free sheaf of $\mathcal O_S$-modules of rank $n$, and $\Lambda_r = \eta_0L_{r,\mathcal O_K}$ was defined in \eqref{Eq:DefiLatticeLambda}. Likewise, we have a decomposition 
    \begin{equation*}
        \mathfrak t_r = \eta_0\mathfrak t_r \oplus \eta_1\mathfrak t_r,
    \end{equation*}
    where $\eta_j\mathfrak t_r \subseteq \eta_jL_{r,\mathcal O_K} \otimes_{\mathcal O_K} \mathcal O_S$ for $j = 0,1$, and the signature condition implies that $\eta_0\mathfrak t_r$ has rank $n-1$ whereas $\eta_1\mathfrak t_r$ has rank $1$. Moreover, condition 3. implies that $\eta_1\mathfrak t_r$ is the orthogonal complement of $\eta_0\mathfrak t_{-r-\underline{\varepsilon}}$ under the pairing $\langle\cdot,\cdot\rangle_H\otimes_{\mathcal O_F}\mathcal O_S$. Thus, the mapping 
    \begin{equation*}
        (\mathfrak t_r)_{r \in I_{\mathbb Z}} \longmapsto (\eta_0\mathfrak t_r)_{r \in I_{\mathbb Z}}
    \end{equation*}
    defines an isomorphism from $M_{\mathrm{loc}}^{I,\varepsilon}$ to the local model of \cite{gortzFlatnessModelsCertain2001} Section 4, denoted ``$M_{I}^{\mathrm{loc}}$'', where in the notations of loc. cit. we have $O = \mathcal O_K$, $\mathbb V_{K,0}$ is identified with $K^n$ via the basis obtained by reversing the order of the vectors in $\mathbf e^{(0)}$, the lattices $\Lambda_r \subset K^n$ for $0 \leq r \leq n-1$ of loc. cit. correspond to $\Lambda_{-r} \subset \mathbb V_{K,0}$, the set $I$ of loc. cit. consists of all the integers $0 \leq r \leq n-1$ such that $-r \in I_{\mathbb Z}$, and the dominant minuscule cocharacter $\mu$ of loc. cit. is $(1^{n-1},0)$ (Drinfeld case). The proof then follows from \cite{gortzFlatnessModelsCertain2001}, see also \cite{PappasRapoportSmithlingLocalModelsI} Theorem 2.3.
\end{proof}
We define as well an auxiliary functor $\widetilde{\mathcal N}^{I,\varepsilon}$ as follows. 
\begin{defi}
    For $S \in \textbf{Nilp}_{\mathcal O_{\breve F}}$, we define $\widetilde{\mathcal N}^{I,\varepsilon}(S)$ to be the set of isomorphism classes of tuples $(X_r,i_{X_r},\lambda_{X_r},\rho_{X_r} ; \gamma_r)_{r \in I}$ where 
    \begin{enumerate}
        \item $(X_r,i_{X_r},\lambda_{X_r},\rho_{X_r})_{r \in I} \in \mathcal N^{I,\varepsilon}(S)$, 
        \item $\gamma_r:\mathbb D(X_r)_S \xrightarrow{\sim} L_{r} \otimes_{\mathcal O_F} \mathcal O_S$ is, for every $r \in I$, an isomorphism of $\mathcal O_E \otimes_{\mathcal O_F} \mathcal O_S$-modules, where $\mathbb D(X_r)_S$ is the value on $S$ of the covariant relative Dieudonné crystal of $X_r$.  
    \end{enumerate}
    Moreover we assume that the following conditions are satisfied:
    \begin{itemize}
        \item for every consecutive $r < r' \in I$, the following diagram is commutative
        \begin{center}
            \begin{tikzcd} 
                \mathbb D(X_{r'})_S \arrow[r,"\mathbb D(\widetilde\alpha_{r',r})_S"] \arrow[d,swap,"\gamma_{r'}"] & \mathbb D(X_r)_S \arrow[d,"\gamma_r"] \\ L_{r'}\otimes_{\mathcal O_F}\mathcal O_S \arrow[r] & L_{r}\otimes_{\mathcal O_F}\mathcal O_S 
            \end{tikzcd}
        \end{center}
        \item $\gamma_r$ identifies the (non-perfect) pairing induced by $\lambda_{X_r}$ on $\mathbb D(X_r)_S$ with $\langle\cdot,\cdot\rangle_H \otimes_{\mathcal O_F} \mathcal O_S$ restricted to $(L_r\otimes_{\mathcal O_F}\mathcal O_S) \times (L_r\otimes_{\mathcal O_F}\mathcal O_S)$, étale locally on $S$ up to a unit $c \in \mathcal O_S^{\times}$ independent of $r$. 
    \end{itemize}
\end{defi}
\sloppy An isomorphism between two tuples $(X_r,i_{X_r},\lambda_{X_r},\rho_{X_r} ; \gamma_r)_{r \in I}$ and $(X_r',i_{X_r'},\lambda_{X_r'},\rho_{X_r'} ; \gamma'_r)_{r \in I}$ is simply an isomorphism 
\begin{equation*}
    \beta_{\bullet}:(X_r,i_{X_r},\lambda_{X_r},\rho_{X_r})_{r \in I} \xrightarrow{\sim} (X_r',i_{X_r'},\lambda_{X_r'},\rho_{X_r'})_{r \in I}
\end{equation*}
\fussy such that $\gamma_r' \circ \mathbb D(\beta_r)_S = \gamma_r$ for all $r \in I$. There is a natural map 
\begin{equation*}
    \widetilde{\varphi}:\widetilde{\mathcal N}^{I,\varepsilon} \longrightarrow \mathcal N^{I,\varepsilon},
\end{equation*}
induced by forgetting the extra structure $(\gamma_r)_{r \in I}$. Then $\widetilde{\varphi}$ is a $\mathcal G_I$-torsor for the étale topology, where $\mathcal G_I$ is the parahoric group scheme defined in \eqref{Eq:DefiParahoricGroupScheme}. It acts on $\widetilde{\mathcal N}^{I,\varepsilon}$ by left composition with the $\gamma_r$'s. It follows that $\widetilde{\mathcal N}^{I,\varepsilon}$ is representable by a formal scheme over $\mathrm{Spf}(\mathcal O_{\breve F})$. Besides, there is another map 
\begin{equation*}
    \widetilde{\psi}: \widetilde{\mathcal N}^{I,\varepsilon} \longrightarrow  \widehat{M}_{\mathrm{loc}}^{I,\varepsilon},
\end{equation*}
where $\widehat{M}_{\mathrm{loc}}^{I,\varepsilon}$ is the $\pi$-adic completion of $M_{\mathrm{loc}}^{I,\varepsilon}$. Given $S \in \textbf{Nilp}_{\mathcal O_{\breve F}}$ and a point $(X_r,i_{X_r},\lambda_{X_r},\rho_{X_r} ; \gamma_r)_{r \in I} \in \widetilde{\mathcal N}^{I,\varepsilon}(S)$, consider the Hodge filtration 
\begin{equation*}
    0 \longrightarrow \mathrm{Fil}^1 \mathbb D(X_r)_S \longrightarrow \mathbb D(X_r)_S \longrightarrow \mathrm{Lie}(X_r) \longrightarrow 0,
\end{equation*}
and define $\mathfrak t_{r} \coloneq  \gamma_r(\mathrm{Fil}^1 \mathbb D(X_r)_S) \subset L_{r} \otimes_{\mathcal O_F}\mathcal O_S$ for all $r \in I$. The collection $(\mathfrak t_{r})_{r \in I}$ uniquely extends to a point $(\mathfrak t_r)_{r \in I_{\mathbb Z}} \in M_{\mathrm{loc}}^{I,\varepsilon}(S)$ by self-duality and periodicity. This process gives the desired map $\widetilde{\psi}$. We note that $\mathcal G_{I}$ also acts on $M_{\mathrm{loc}}^{I,\varepsilon}$ via the mapping 
\begin{equation*}
    (g_r)_{r\in I_{\mathbb Z}} \cdot (\mathfrak t_r)_{r \in I_{\mathbb Z}} \coloneq  (g_r(\mathfrak t_r))_{r \in I_{\mathbb Z}},
\end{equation*}
and $\widetilde{\psi}$ is formally smooth and $\mathcal G_{I}$-equivariant. All three spaces sit together in the (relative) local model diagram represented in Figure \ref{Figure1}. The importance of this diagram is that étale locally around any point of $\mathcal N^{I,\varepsilon}$, there exists a formally étale morphism to $\widehat{M}_{\mathrm{loc}}^{I,\varepsilon}$, in the sense of \cite{rapoportPeriodSpacesPdivisible1996} Proposition 3.33. On the special fiber, we get a morphism of algebraic stacks
\begin{equation*}
    \lambda^{I,\varepsilon}: \mathcal N_{\mathrm{red}}^{I,\varepsilon} \longrightarrow \left[M_{\mathrm{loc},\overline{\mathbb F_q}}^{I,\varepsilon}/\mathcal G_{I,\overline{\mathbb F_q}}\right],
\end{equation*}
where $\mathcal G_{I,\overline{\mathbb F_q}} \coloneq  \mathcal G_I \times_{\mathrm{Spec}(\mathcal O_F)} \mathrm{Spec}(\overline{\mathbb F_q})$ and similarly for $M_{\mathrm{loc},\overline{\mathbb F_q}}^{I,\varepsilon}$.
\begin{figure}
    \centering
    \begin{tikzcd}
        & \widetilde{\mathcal N}^{I,\varepsilon} \arrow[dl,swap,"\widetilde{\varphi}"] \arrow[dr,"\widetilde{\psi}"]  & \\
        \mathcal N^{I,\varepsilon} & & \widehat{M}_{\mathrm{loc}}^{I,\varepsilon}
    \end{tikzcd}
    \caption{The relative local model diagram.}
    \label{Figure1}
\end{figure}

\section{The Kottwitz-Rapoport stratification}
\subsection{Iwahori-Weyl group}\label{Section3.1}

We now introduce the Iwahori-Weyl group attached to the group $G$. Recall from \eqref{Eq:BaseChangeG} that we have an isomorphism
\begin{equation*}
G_{\breve F} \simeq \mathrm{GL}(\mathbb V_{\breve F,0})\times_{\breve F}\mathbb G_{m,\breve F}.
\end{equation*}
Through the basis $\mathbf e^{(0)}$ of $\mathbb V_{\breve F,0}$, an element $g \in G(\breve F)$ will be written as a pair $(g_0,c)$, where $g_0\in \mathrm{GL}_n(\breve F)$ and $c\in \breve F^{\times}$. Let $T\subset G$ be the maximal $F$-torus consisting of the unitary similitudes whose matrix in the basis $\mathbf e$ is diagonal, and let $N\coloneq N_G(T)$ be its normalizer. Then $T$ is $K$-split, so that $T_{\breve F} \subset G_{\breve F}$ is the usual diagonal maximal torus, and $N_{\breve F}$ agrees with the normalizer of $T_{\breve F}$ in $G_{\breve F}$. Let us write 
\begin{equation}\label{Eq:NotationIIw}
    I_{\mathrm{Iw}} \coloneq  \{0,\ldots , \nu_{\varepsilon}\},
\end{equation}
and $\breve{K}_{I_{\mathrm{Iw}}} \coloneq  \mathcal G_{I_{\mathrm{Iw}}}(\mathcal O_{\breve F})$ in the notations of \eqref{Eq:DefiParahoricGroupScheme}. The \textit{Iwahori-Weyl group} of $G$ is
\begin{equation*}
\widetilde W\coloneq N(\breve F)/(T(\breve F) \cap \breve{K}_{I_{\mathrm{Iw}}}).
\end{equation*}
Define also the finite Weyl group $W \coloneq  N(\breve F)/T(\breve F) \simeq \mathfrak S_n$. The mapping 
\begin{equation*}
    (\mathrm{Diag}(a_1,\ldots , a_n),c) \in T(\breve F) \mapsto (v_{\pi}(a_1),\ldots , v_{\pi}(a_n); v_{\pi}(c)) \in \mathbb Z^n \times \mathbb Z,
\end{equation*}
defines an isomorphism $X_*(T_{\breve F}) \simeq T(\breve F)/(T(\breve F) \cap \breve{K}_{I_{\mathrm{Iw}}}) \xrightarrow{\sim} \mathbb Z^{n} \times \mathbb Z$. Then $W$ acts on $\mathbb Z^{n} \times \mathbb Z$ via
\begin{equation*}
    w \cdot (x_1,\ldots , x_n;m) \coloneq  (x_{w^{-1}(1)}, \ldots , x_{w^{-1}(n)} ; m) \in \mathbb Z^n \times \mathbb Z,
\end{equation*}
and induces an isomorphism 
\begin{equation*}
    \widetilde{W} \simeq (\mathbb Z^n\times \mathbb Z) \rtimes W.
\end{equation*}
Given $w \in W$ and $x\in \mathbb Z^n\times \mathbb Z$, the corresponding element in $\widetilde W$ is denoted $t^xw \coloneq  (x,w)$. Multiplication is given by 
\begin{equation*}
    (t^xw)(t^{x'}w') = t^{x+w\cdot x'}ww'.
\end{equation*}
The affine Weyl group $W_a \subset \widetilde{W}$ is the Iwahori-Weyl group of the derived subgroup of $G$. It is then naturally identified with the subgroup
\begin{equation*}
W_{a} = \left\{t^{(x_1,\ldots,x_n;0)}w \in \widetilde W \;\middle|\;
\sum_{i=1}^n x_i=0\right\}.
\end{equation*}
Let $\Omega$ denote the stabilizer in $\widetilde{W}$ of the alcove determined by the lattice chain $L_{\bullet,\mathcal O_{\breve F}}$, in the apartment of $T_{\breve F}$ in the Bruhat-Tits building of $G_{\breve F}$. We have a decomposition 
\begin{equation*}
    \widetilde{W} \simeq W_a \rtimes \Omega .
\end{equation*}
If we define 
\begin{align*}
    \tau_0 \coloneq  t^{(1,1,\ldots 1,0;1)}(n\ n-1\  \ldots \ 1), & & \zeta \coloneq  t^{(1,\ldots , 1;1)},
\end{align*}
then we have 
\begin{equation*}
    \Omega \simeq \langle \tau_0, \zeta \rangle \simeq \mathbb Z \oplus \mathbb Z.
\end{equation*}
The element $\dot{\zeta} \coloneq  (\pi I_n,\pi) \in G(K)$ lies in $N(K)$ and its image in $\widetilde W$ is $\zeta$. On the other hand, the element $b$ defined in \eqref{Eq:DefinitionElementb} lies in $N(K)$, and the image of $\sigma(b)$ in $\widetilde W$ is equal to $\tau_0 \in \Omega$. Thus, we will sometimes write $\dot{\tau}_0 \coloneq  \sigma(b)$ when we need to stress this fact. We note that $\dot{\tau}_0\mathbf e^{(0)}_r = \mathbf e_{r-1}^{(0)}$ for all $r \in \mathbb Z$. The affine Weyl group is a Coxeter group generated by the set $\mathbb S_{a} \coloneq  \{s_0,\ldots , s_{n-1}\}$ of simple affine reflections
\begin{align*}
\forall 1 \leq j \leq n-1, \qquad s_j=(j\ j+1), & & s_0=t^{(1,0,\ldots,0,-1;0)}(1\ n).
\end{align*}
For $0 \leq j \leq n-1$, the natural representative $\dot{s}_j$ of $s_j$ in $N(K) \subset N(\breve F)$ is determined by the condition that it has multiplier equal to $1$, and that it exchanges the vectors $\mathbf e_{j+kn}^{(0)}$ and $\mathbf e_{j+1+kn}^{(0)}$ for all $k \in \mathbb Z$, while fixing the vectors $\mathbf e_{r}^{(0)}$ whenever $\overline r \not = \overline{j}, \overline{j+1}$. If $v \in W_a$ has a reduced expression $v = s_{j_1}\ldots s_{j_k}$, we define $\dot{v} \coloneq  \dot{s}_{j_1}\ldots \dot{s}_{j_k} \in N(K)$. It does not depend on the choice of the reduced expression. Finally, if $w \in \widetilde{W}$, we can decompose $w = v\tau_0^a\zeta^b$ and define $\dot{w} \coloneq  \dot{v}\dot{\tau}_0^a\dot{\zeta}^b \in N(K)$.\\
We will always identify the indexing set of the simple reflections $s_0,\ldots , s_{n-1}$ with $\mathbb Z/n\mathbb Z$. They satisfy the following relations: for all $i,j \in \mathbb Z/n\mathbb Z$, 
\begin{align*}
    & \text{if } j \not = i \pm 1, & s_is_j & = s_js_i,\\
    & \text{if } j = i \pm 1 \text{ and } n\geq 3, & s_is_js_i & = s_js_is_j,\\
    & & s_i^2 & = \mathrm{id}.
\end{align*} 
Let $\ell$ denote the length function with respect to $\mathbb S_a$. It can be extended to $\widetilde{W}$ by setting $\ell(w\omega) \coloneq  \ell(w)$ for all $w \in W_a$ and $\omega \in \Omega$. The Bruhat order on $W_a$ can also be extended to $\widetilde W$ by setting
\begin{equation*}
    w\omega \leq w'\omega' \iff \omega = \omega' \text{ and } w \leq w'.
\end{equation*}
The Frobenius $\sigma \in \mathrm{Gal}(\breve F/F)$ acts on $G(\breve F)$ as follows 
\begin{equation*}
    \forall g = (g_0,c) \in G(\breve F), \qquad \sigma(g) = (\sigma(c)\Omega_H^{-1}(g_0^{(\sigma)})^{-T}\Omega_H, \sigma(c)), 
\end{equation*}
where $g_0^{(\sigma)}$ is obtained by applying $\sigma$ to all the entries of the matrix $g_0$. This induces an action on $\widetilde W$ which stabilizes $W_a$, $\mathbb S_a$ and $\Omega$. By direct computation, one may check that we have 
\begin{align}\label{Eq:FrobeniusActionOnSimpleReflections}
\begin{split}
    \forall j \in \mathbb Z/n\mathbb Z, \qquad  \sigma(s_j) & = s_{-j-\underline{\varepsilon}},\\
    \sigma(\tau_0) & = \tau_0^{n-1}\zeta^{2-n} = t^{(1,0, \ldots,0;1)}(1\ 2\ \ldots \ n), \\
    \sigma(\zeta)& = \tau_0^n\zeta^{1-n} = t^{(0,\ldots , 0; 1)}.
\end{split}
\end{align}
Recall the cocharacter $\mu$ defined in Section \ref{Section2.1}. We will consider $\mu' \coloneq  \sigma \circ \mu$, which is given by $z \mapsto (\mathrm{Diag}(z,\ldots , z ,1),z)$. It defines the element
\begin{equation*}
    t^{\mu'} = t^{(1, \ldots ,1, 0;1)} \in \widetilde{W}.
\end{equation*}
For $1 \leq j \leq n$, we define $t^{\mu'_j} \coloneq  t^{(1^{j-1},0,1^{n-j};1)}$, so that the orbit of $t^{\mu'}$ under the action of the finite Weyl group $W$ is $\{t^{\mu'_j} \mid 1 \leq j \leq n\}$. 
\begin{defi}
    The $\{\mu'\}$-admissible set is 
    \begin{equation*}
        \mathrm{Adm}(\{ \mu' \}) = \left\{v \in \widetilde W \mid \exists 1 \leq j \leq n, \quad v \leq t^{\mu'_j}\right\}.
    \end{equation*}
\end{defi}
Now let $I\subseteq \{0,\ldots,\nu_{\varepsilon}\}$ be non-empty, and $\breve K_I\coloneq \mathcal G_I(\mathcal O_{\breve F})\subset G(\breve F)$. We define
\begin{equation}\label{Eq:DefiSIa}
    \mathbb S_a^I \coloneq  \{s_i \mid i \not\in \overline I\} \subseteq \mathbb S_a.
\end{equation}
The mapping $I \mapsto \mathbb S_a^I$ induces a bijection between non-empty subsets of $\{0,\ldots,\nu_{\varepsilon}\}$ and proper $\sigma$-stable subsets of $\mathbb S_a$. The subgroup of $W_{a}$ attached to the facet fixed by $\breve K_I$ is
\begin{equation}\label{Eq:DefiSubgroupWI}
W_{\mathbb S_a^I}\coloneq \left\langle s\mid s \in \mathbb S_a^{I}\right\rangle \subseteq W_a.
\end{equation}
By the Iwahori-Bruhat decomposition of $G(\breve F)$, we have 
\begin{equation*}
\breve K_I\backslash G(\breve F)/\breve K_I \simeq W_{\mathbb S_a^I}\backslash \widetilde W/W_{\mathbb S_a^I}.
\end{equation*}
We define 
\begin{align*}
\begin{split}
\mathrm{Adm}(\{\mu'\})^I & \coloneq  W_{\mathbb S_a^I}\mathrm{Adm}(\{\mu'\})W_{\mathbb S_a^I} \subseteq \widetilde W,\\
\mathrm{Adm}(\{\mu'\})_I & \coloneq  W_{\mathbb S_a^I}\backslash \mathrm{Adm}(\{\mu'\})^I / W_{\mathbb S_a^I} \subseteq W_{\mathbb S_a^I} \backslash \widetilde W / W_{\mathbb S_a^I}.
\end{split}
\end{align*}
For instance, we have $\mathrm{Adm}(\{\mu'\})_{I_{\mathrm{Iw}}} = \mathrm{Adm}(\{\mu'\})$. Given $v \in \widetilde{W}$, we write $[v]_I$ for its image $W_{\mathbb S_a^I}\backslash \widetilde W/W_{\mathbb S_a^I}$. This double coset space is equipped with a partial order, induced from the Bruhat order on $\tilde W$, as follows.
\begin{equation*}
    [v]_{I} \leq [v']_{I} \iff v_{I,\mathrm{min}} \leq v'_{I,\mathrm{min}},
\end{equation*}
where $v_{I,\mathrm{min}}$ (resp. $v'_{I,\mathrm{min}}$) denotes the unique element of minimal length in $[v]_I$ (resp. in $[v']_I$). Besides, the length of a double coset $[v]_{I}$ is by definition the length of $v_{I,\mathrm{min}}$. By restriction, the subset $\mathrm{Adm}(\{\mu'\})_{I}$ is a finite poset, which we are going to describe explicitly. 

Given $x,y \in \mathbb Z/n\mathbb Z$, we define 
\begin{equation*}
    s_{[x,y]} \coloneq  s_xs_{x+1} \ldots s_{y-1}s_y \in W_a.
\end{equation*}
For instance, $s_{[a,a]} = s_a$ and $s_{[n-1,2]} = s_{n-1}s_0s_1s_2$. The given expression of $s_{[x,y]}$ is always reduced. The following Lemma can be checked by direct computation.
\begin{lem}\label{Lemma:ReducedDecompositiontmuj}
    For every $1 \leq j \leq n$, we have  
    \begin{equation*}
        t^{\mu_j'} = s_{[j+1,j-1]}\tau_0.
    \end{equation*}
\end{lem}
Given $w \in W_a$, let $\mathrm{Supp}(w) \subseteq \mathbb S_a$ denote its support, that is the set of simple reflections occurring in some (equivalently, in any) reduced decomposition of $w$. 
\begin{defi}
    We define $\mathcal P_I \coloneq  \{J \subsetneq \overline I\}$ to be the poset of proper subsets of $\overline I$, ordered by inclusion.
\end{defi}
Recall the definition of $\ell_i$ for $i \in \overline I$ in Definition \ref{Defi:DefiIntegerli}.
\begin{defi}\label{Defi:DefiBlockBi}
    For $i \in \overline I$, we define
    \begin{equation*}
            B_i \coloneq  \{i, i+1, \ \ldots \ , i+\ell_i-1\} \subseteq \mathbb Z/n\mathbb Z, \qquad w_{B_i} \coloneq  s_{[i,i+\ell_i-1]} \in W_a.
    \end{equation*}
    The subsets $B_i$'s form a partition of $\mathbb Z/n\mathbb Z$. Then, for $J \in \mathcal P_I$, define
    \begin{equation*}
        B_J \coloneq  \bigsqcup_{j \in J} B_j \subsetneq \mathbb Z/n\mathbb Z.
    \end{equation*}
\end{defi}
Since $J \in \mathcal P_{I}$, we can pick some $i_0 \in \overline I \setminus J$. Let $j_1, \ldots , j_{k}$ denote the elements of $J$ enumerated successively as we run along $\mathbb Z/n\mathbb Z$ from $i_0+1, i_0+2, \ldots , i_0-1$. We define 
\begin{equation*}
    w_{J} \coloneq  w_{B_{j_1}}w_{B_{j_2}}\ldots w_{B_{j_k}} \in W_a, \qquad v_J \coloneq  w_J\tau_0 \in \widetilde W.
\end{equation*}
\begin{lem}
    The element $w_J$ does not depend on the choice of $i_0 \in \overline I \setminus J$.
\end{lem}
\begin{proof}
    If $k = \# J = \#\overline{I} - 1$, there is nothing to prove. Else, let $i_1$ denote the first element of $\overline I \setminus J$ which occurs when running along $i_0+1, i_0+2, \ldots , i_0-1$. Let $0 \leq k' \leq k$ denote the number of elements of $J$ in $\{i_0+1, \ldots , i_1-1\}$. These elements are precisely $j_1, \ldots , j_{k'}$. We define
    \begin{equation*}
       w' \coloneq  w_{B_{j_1}}w_{B_{j_2}}\ldots w_{B_{j_{k'}}}, \qquad w'' \coloneq   w_{B_{j_{k'+1}}}w_{B_{j_{k'+2}}}\ldots w_{B_{j_k}},
    \end{equation*}
    so that we have $w_J = w'w''$. The enumeration of the elements of $J$ obtained by starting with $i_1$ instead of $i_0$, is $j'_1,\ldots , j'_k$ where 
    \begin{equation*}
        \forall 1 \leq t \leq k-k', \quad j'_t = j_{k'+t}, \qquad \forall k-k'+1 \leq t \leq k, \quad j'_t = j_{t-(k-k')}.
    \end{equation*}
    It follows that the element $w_J'$ produced with $i_1$ is $w_J' = w''w'$. Now, we have 
    \begin{equation*}
        \mathrm{Supp}(w') \subseteq \{i_0+1, \ldots , i_1-1\}, \qquad \mathrm{Supp}(w'') \subseteq \{i_1+1, \ldots , i_0-1\}.
    \end{equation*}
    Thus $w'$ and $w''$ commute, so that $w_J = w_J'$.
\end{proof}
\begin{rk}\label{Rk:wJIsShortestElementDoubleCoset}
    We point out that $w_J$ is the shortest element of the double coset $W_{\mathbb S_a^I} w_J W_{\mathbb S_a^I -1}$, where $\mathbb S_a^I -1 \coloneq  \{s_{j-1} \mid s_j \in \mathbb S_a^I\}$. Equivalently, since $\tau_0 s_j\tau_0^{-1} = s_{j-1}$ for all $j \in \mathbb Z/n\mathbb Z$, we have $\tau_0 W_{\mathbb S_a^I} \tau_0^{-1} = W_{\mathbb S_a^I -1}$, and so $v_J$ is the shortest element of the double coset $W_{\mathbb S_a^I} v_J W_{\mathbb S_a^I}$.
\end{rk}
\begin{prop}\label{Prop:IsomAdmMuPoset}
    Let $I \subseteq \{0,\ldots , \nu_{\varepsilon}\}$ be non-empty. The mapping 
    \begin{equation*}
        J \mapsto [v_J]_I
    \end{equation*}
    defines a bijection $\mathcal P_I \xrightarrow{\sim} \mathrm{Adm}(\{\mu'\})_{I}$ of posets.
\end{prop}
\begin{proof}
    Let $J \in \mathcal P_I$. The support $\mathrm{Supp}(w_J) = \{s_b \mid b \in B_J\}$ is a proper subset of $\mathbb S_a$, so that we can pick some simple reflection $s_j$ which does not occur in $w_J$. By construction, we have $w_J \leq s_{[j+1,j-1]}$. Then by Lemma \ref{Lemma:ReducedDecompositiontmuj}, we have $v_J \leq t^{\mu'_j}$. It follows that $v_J \in \mathrm{Adm}(\{\mu'\})$, and thus $[v_J]_I \in \mathrm{Adm}(\{\mu'\})_{I}$ so that the map is well-defined. We construct its inverse. Let $v \in \mathrm{Adm}(\{\mu'\})$ and let $v'$ be the shortest element in the double coset $W_{\mathbb S_a^I} v W_{\mathbb S_a^I}$. By Lemma \ref{Lemma:ReducedDecompositiontmuj}, we can write $v = w \tau_0$ where $w \in W_a$ satisfies $w \leq s_{[j+1,j-1]}$ for some $j$. We have 
    \begin{equation*}
        W_{\mathbb S_a^I} w\tau_0 W_{\mathbb S_a^I} = \underbrace{W_{\mathbb S_a^I} w (\tau_0 W_{\mathbb S_a^I} \tau_0^{-1})}_{\subseteq W_a}\tau_0,
    \end{equation*}
    thus $v' = w'\tau_0$ where $w'$ is the shortest element of the double coset $W_{\mathbb S_a^I} w (\tau_0 W_{\mathbb S_a^I}\tau_0^{-1}) = W_{\mathbb S_a^I}wW_{\mathbb S_a^I-1}$. Since $w \leq s_{[j+1,j-1]}$ for some $j$, one easily sees that $w' = w_J$ for the unique proper subset $J \subsetneq \overline I$ such that $\mathrm{Supp}(w') = B_J$. It follows that $[v_J]_I = [v]_I$, and this defines the inverse map. It is then easy to check that $J \mapsto [v_J]_I$ is an isomorphism of posets.
\end{proof}
\begin{rk}\label{Rk:DescriptionAdmMu}
    In particular, Proposition \ref{Prop:IsomAdmMuPoset} with $I = I_{\mathrm{Iw}}$ says that $J \mapsto v_J$ is an isomorphism of posets $\mathcal P_{I_{\mathrm{Iw}}} \simeq \mathrm{Adm}(\{\mu'\})$, and $\mathcal P_{I_{\mathrm{Iw}}}$ is nothing but the poset of proper subsets of $\mathbb Z/n\mathbb Z$. 
\end{rk}
\begin{corol}
    Let $I \subseteq \{0,\ldots , \nu_{\varepsilon}\}$ be non-empty. We have 
    \begin{equation*}
        \#\mathrm{Adm}(\{\mu'\})_I = 2^{\#\overline I}-1.
    \end{equation*}
\end{corol}
As pointed out in the proof of Proposition \ref{Prop:IsomAdmMuPoset}, conjugation by $\tau_0$ preserves $\mathbb S_a$, and we write $\mathrm{Ad}(\tau_0)$ for this action. Combining with \eqref{Eq:FrobeniusActionOnSimpleReflections}, we have 
\begin{equation*}
    \forall j \in \mathbb Z/n\mathbb Z, \qquad \mathrm{Ad}(\tau_0)\circ \sigma(s_j) = s_{-j-\underline{\varepsilon}-1}.
\end{equation*}
The action of $\mathrm{Ad}(\tau_0)\circ \sigma$ on $\mathbb S_a$ is involutive, and for every $w \in W_a$ we define the $\sigma$-support of $w$ as follows:
\begin{equation*}
    \mathrm{Supp}_{\sigma}(w) \coloneq  \mathrm{Supp}(w) \cup \mathrm{Ad}(\tau_0)\circ\sigma(\mathrm{Supp}(w)) \subseteq \mathbb S_a.
\end{equation*}
We also define
\begin{align*}
\begin{split}
\mathrm{Adm}(\{\mu'\})_0 & \coloneq  \{w\tau_0 \in \mathrm{Adm}(\{\mu'\}) \mid \mathrm{Supp}_{\sigma}(w) \not = \mathbb S_a\} \subseteq \mathrm{Adm}(\{\mu'\}), \\
\mathrm{Adm}(\{\mu'\})^I_0 & \coloneq  W_{\mathbb S_a^I}\mathrm{Adm}(\{\mu'\})_0W_{\mathbb S_a^I} \subseteq \mathrm{Adm}(\{\mu'\})^I,\\
\mathrm{Adm}(\{\mu'\})_{0,I} & \coloneq  W_{\mathbb S_a^I}\backslash \mathrm{Adm}(\{\mu'\})^I_0 / W_{\mathbb S_a^I} \subseteq \mathrm{Adm}(\{\mu'\})_{I}.
\end{split}
\end{align*}
\begin{defi}\label{Defi:DefiInvolutionThetaI}
    We define an involution $\theta_I: \overline I \xrightarrow{\sim} \overline I$ by the formula $\theta_I(i) \coloneq  -i-\underline{\varepsilon}- \ell_i$.
\end{defi}
This is well defined, since $\overline I$ is stable under the transformation $i \mapsto -i-\underline{\varepsilon}$, and because $i + \ell_i \in \overline I$ by definition. The involution $\theta_I$ satisfies $\ell_{i} = \ell_{\theta_I(i)}$ for all $i \in \overline I$. For all consecutive $r < r' \in I$ we have
\begin{equation}\label{Eq:ValueThetaIOrbitSizeTwo}
    \theta_I(\overline{r}) = \overline{-r'-\underline{\varepsilon}} \not = \overline r, \qquad \theta_I(\overline{-r'-\underline{\varepsilon}}) = \overline r.
\end{equation}
Besides, for $r_{\mathrm{min}} \coloneq \min(I)$ and $r_{\mathrm{max}} \coloneq \max(I)$ we have 
\begin{equation}\label{Eq:ValueThetaIOrbitSizeOne}
    \begin{aligned}
        \text{if } h_{r_{\mathrm{max}}}^{\varepsilon} \not = n, & \text{ then } \theta_I(\overline{r_{\mathrm{max}}}) = \overline{r_{\mathrm{max}}},\\
        \text{if } h_{r_{\mathrm{min}}}^{\varepsilon} \not = 0, & \text{ then } \theta_I(\overline{-r_{\mathrm{min}}-\underline{\varepsilon}}) = \overline{-r_{\mathrm{min}}-\underline{\varepsilon}}.
\end{aligned}
\end{equation}
Thus, $\theta_I$ has $0$, $1$ or $2$ fixed points depending on whether $h_{r_{\mathrm{min}}}^{\varepsilon} \not = 0$ and $h_{r_{\mathrm{max}}}^{\varepsilon} \not = n$.

\begin{lem}\label{Lem:CorrespondenceAdtausigmaAndThetaI}
    For all $J \in \mathcal P_I$, we have $\mathrm{Ad}(\tau_0)\circ\sigma(\mathrm{Supp}(w_J)) = \mathrm{Supp}(w_{\theta_I(J)})$.
\end{lem}

\begin{proof}
    This amounts to proving that $B_{\theta_I(J)} = \{-b-\underline{\varepsilon}-1 \mid b \in B_J\}$, and this reduces to the case in which $J = \{i\}$ is a singleton. By definition, we have $\ell_{\theta_I(i)} = \ell_i$, so that 
    \begin{align*}
        B_{\theta_I(i)} & = \{-i-\underline{\varepsilon}-\ell_i, -i-\underline{\varepsilon}-\ell_i+1, \ldots , -i-\underline{\varepsilon}-1 \}\\
        & = \{-b -\underline{\varepsilon}-1 \mid b \in B_i\}.
    \end{align*}
\end{proof}

We define 
\begin{equation*}
    \mathcal P_{0,I} \coloneq  \{J \in \mathcal P_I \mid J \cup \theta_I(J) \not = \overline I\}.
\end{equation*}
By Lemma \ref{Lem:CorrespondenceAdtausigmaAndThetaI}, for $J \in \mathcal P_I$ we have 
\begin{equation*}
    \mathrm{Supp}_{\sigma}(w_J) = \{s_b \mid b \in B_J \cup B_{\theta_I(J)}\} = \{s_b \mid b \in B_{J\cup\theta_I(J)}\},
\end{equation*}
and thus $v_J \in \mathrm{Adm}(\{\mu'\})_0$ if and only if $J \in \mathcal P_{0,I}$. The next Proposition follows easily. 

\begin{prop}\label{Prop:IsomAdmMuPosetBasic}
    The bijection of Proposition \ref{Prop:IsomAdmMuPoset} identifies $\mathcal P_{0,I}$ with $\mathrm{Adm}(\{\mu'\})_{0,I}$. 
\end{prop}

\subsection{The KR strata}\label{Section3.2}

In general, Kottwitz-Rapoport (KR) strata are defined by embedding the local model into an affine flag variety, and decomposing along the Schubert cells, see for instance Example 3.15 of \cite{PappasRapoportSmithlingLocalModelsI}. But since we are essentially working with the Drinfeld case, we can also define the stratification directly as follows.\footnote{In fact, this amounts to the description of \cite{PappasRapoportSmithlingLocalModelsI} Section 7.2 in terms of quiver Grassmannians.} To alleviate the notation, we fix $\varepsilon \in \{\pm 1\}$ and $I \subseteq \{0,\ldots , \nu_{\varepsilon}\}$ non-empty throughout the section, and we write $\overline{M}_{\mathrm{loc}} \coloneq  M_{\mathrm{loc}}^{I,\varepsilon} \otimes_{\mathcal O_K} \mathbb F_{q^2}$, $\overline{L}_r \coloneq  L_r \otimes_{\mathcal O_F} \mathbb F_{q^2}$, $\overline{\Lambda}_r \coloneq  \eta_0 \overline L_r = \Lambda_{r} \otimes_{\mathcal O_K} \mathbb F_{q^2}$ and $\overline{\mathcal G_I} \coloneq  \mathcal G_{I} \otimes_{\mathcal O_F} \mathbb F_{q^2}$.\\
Let $\mathbf t^{\mathrm{univ}}_{\bullet} = (\mathbf t^{\mathrm{univ}}_r)_{r \in I_{\mathbb Z}}$ denote the universal object over the special fiber $\overline{M}_{\mathrm{loc}}$, and let us define 
\begin{equation*}
    \mathcal Q_r \coloneq  \frac{\overline \Lambda_r \otimes \mathcal O_{\overline{M}_{\mathrm{loc}}}}{\eta_0\mathbf t^{\mathrm{univ}}_r},
\end{equation*}
which is a line bundle over $\overline{M}_{\mathrm{loc}}$. For $i \in \overline I$, recall the integers $1 \leq \ell_{i} \leq n$ from Definition \ref{Defi:DefiIntegerli}. For all $r \in I_{\mathbb Z}$, the two lattices $L_{r+\ell_{\overline r}}$ and $L_{r}$ are consecutive in the chain $L_{\bullet}^{I}$, and the inclusion $L_{r+\ell_{\overline r}} \subset L_{r}$ induces a map
\begin{equation}\label{Eq:DefiMapDeltarx}
    \delta_r: \mathcal Q_{r+\ell_{\overline r}} \longrightarrow \mathcal Q_{r}.
\end{equation}
For $i \in \overline I$, let $C_i \subseteq \overline{M}_{\mathrm{loc}}$ denote the reduced closed subscheme defined by the vanishing of $\delta_{r}$ for all (equivalently, for some) $r \in I_{\mathbb Z}$ such that $\overline r = i$, and let $U_i$ denote the open complement. The $C_i$'s are the irreducible components of $\overline{M}_{\mathrm{loc}}$, and by Proposition \ref{Prop:IsomorphismWithDrinfeldCase} they are smooth and intersect transversally. For $J \in \mathcal P_I$, we define a locally closed subscheme of $\overline{M}_{\mathrm{loc}}$ by
\begin{equation*}
    \mathcal O_J \coloneq  \left(\bigcap_{i \in J} U_i\right) \cap \left(\bigcap_{i \in \overline I \setminus J} C_i\right). 
\end{equation*}
Then $\mathcal O_J$ is a smooth locally closed subscheme of $\overline{M}_{\mathrm{loc}}$. Since the transition map $\overline{L}_{r+n} \to \overline{L}_{r}$ is zero, the map $\mathcal Q_{r+n} \to \mathcal Q_{r}$ obtained by composition of successive $\delta_r$'s is zero as well. It follows that the collection $(\mathcal O_J)_{J \in \mathcal P_I}$ defines a stratification of $\overline{M}_{\mathrm{loc}}$. Moreover, we have a natural closure relation
\begin{equation*}
    \overline{\mathcal O_J} = \bigsqcup_{J' \subseteq J} \mathcal O_{J'}.
\end{equation*}
By means of the bijection of posets $\mathcal P_I \xrightarrow{\sim} \mathrm{Adm}(\{\mu'\})_I$ of Proposition \ref{Prop:IsomAdmMuPoset}, we also write $\mathcal O_{[v_J]_I} \coloneq  \mathcal O_J$. \\

For each $J \in \mathcal P_I$, we define a distinguished $\mathbb F_{q^2}$-rational point $x_J \in \mathcal O_J$ as follows. Recall that we have a fixed representative $\dot{v}_J \in N(K)$ as defined in Section \ref{Section3.1}. For $r \in \mathbb Z$, we define 
\begin{equation*}
    d_J(r) \coloneq  \min\{d \geq 1 \mid \overline{r+d-1} \not \in B_J\}.
\end{equation*}
The integer $d_J(r)$ is well-defined since $B_J$ is a proper subset of $\mathbb Z/n\mathbb Z$.
\begin{lem}\label{Lem:ActionOfvJOnLattices}
    For every $r \in \mathbb Z$, we have 
    \begin{equation*}
        \pi\dot{v}_J^{-1} \Lambda_r = \left(\bigoplus_{\substack{k = r+1\\ k \not = r+d_J(r)}}^{r+n} \mathcal O_K \cdot \mathbf e_k^{(0)} \right) \oplus \pi\mathcal O_K \cdot \mathbf e^{(0)}_{r+d_J(r)}
    \end{equation*}
\end{lem}

\begin{proof}
Let $1\leq d_1 < \cdots < d_t\leq n$ be the integers such that, for all $1\leq d\leq n$, we have
\begin{equation*}
\overline{r+d-1}\notin B_J \iff d \in \{d_1,\ldots,d_t\}.
\end{equation*}
By definition, $d_1=d_J(r)$ and $d_{k+1} = d_k + d_J(r+d_k)$ for all $0 <k < t$. For $1 \leq d \leq n$, we have 
\begin{equation*}
\dot{w}_J^{-1} \mathbf e_{r+d}^{(0)} = \begin{cases}
    \mathbf e_{r+d-1}^{(0)} & \text{if } \overline{r+d-1} \in B_J,\\
    \mathbf e_{r+d_{k+1}-1}^{(0)} & \text{if } d = d_k \text{ for some } 0 <k < t,\\
    \pi \mathbf e_{r+d_1-1}^{(0)} & \text{if } d = d_t.
\end{cases}
\end{equation*}
On the other hand, $\pi\dot\tau_0^{-1}\mathbf e_r^{(0)} = \mathbf e_{r+1}^{(0)}$ for all $r \in \mathbb Z$. Thus $\pi\dot{v}_J^{-1}$ fixes all the vectors $\mathbf e_{r+d}^{(0)}$ such that $1 \leq d \leq n$ and $d \not \in \{d_1, \ldots ,d_t\}$; and it sends the vectors $\mathbf e_{r+d_1}^{(0)},\ldots,\mathbf e_{r+d_t}^{(0)}$ to  $\mathbf e_{r+d_2}^{(0)},\ldots,\mathbf e_{r+d_t}^{(0)}, \pi\mathbf e_{r+d_1}^{(0)}$. This concludes the proof.
\end{proof}
It follows from Lemma \ref{Lem:ActionOfvJOnLattices} that 
\begin{equation*}
    \forall r \in \mathbb Z, \qquad \pi\Lambda_r \overset{n-1}{\subseteq} \pi\dot{v}^{-1}_J\Lambda_r \overset{1}{\subseteq} \Lambda_r.
\end{equation*}
There is a unique point $x_J = (\mathfrak t_r)_{r \in I_{\mathbb Z}} \in \overline{M}_{\mathrm{loc}}(\mathbb F_{q^2})$ such that $\eta_0\mathfrak t_r = \pi\dot{v}_J^{-1}\Lambda_r/\pi\Lambda_r$ for all $r \in I_{\mathbb Z}$. Consider the $1$-dimensional quotient 
\begin{equation*}
    \forall r \in I_{\mathbb Z}, \qquad Q_r \coloneq  \Lambda_r / \pi\dot{v}_J^{-1}\Lambda_r,
\end{equation*}
and the maps $\delta_{r}: Q_{r+\ell_{\overline r}} \to Q_{r}$. By definition, we have $\delta_r = 0$ if and only if $\Lambda_{r+\ell_{\overline r}} \subseteq \pi\dot{v}_J^{-1}\Lambda_{r}$.
\begin{lem}\label{Lem:xJIsAPointOfOJ}
    For $r \in I_{\mathbb Z}$, we have $\Lambda_{r+\ell_{\overline r}} \subseteq \pi\dot{v}_J^{-1}\Lambda_{r} \iff \overline r \in \overline I \setminus J$. Consequently, we have $x_J \in \mathcal O_J$.
\end{lem}
\begin{proof}
    Assume first that $\overline r \in \overline I \setminus J$. We have $B_{\overline r} \cap B_J = \emptyset$, therefore $d_{J}(r) = 1$. Thus by Lemma \ref{Lem:ActionOfvJOnLattices} we have 
    \begin{equation*}
        \Lambda_{r+\ell_{\overline r}} \subseteq \Lambda_{r+1} =  \pi\dot{v}_J^{-1}\Lambda_{r}.
    \end{equation*}
    Assume now that $\overline r \in J$. Since $B_{\overline r} \subseteq B_J$, we have $\ell_{\overline r} < d_J(r) \leq n$. By Lemma \ref{Lem:ActionOfvJOnLattices}, the lattice $\pi\dot{v}_J^{-1}\Lambda_{r}$ contains the vector $\pi \mathbf e_{r+ d_J(r)}^{(0)}$ but does not contain $\mathbf e_{r+ d_J(r)}^{(0)}$. On the other hand, we have $\mathbf e_{r+ d_J(r)}^{(0)} \in \Lambda_{r+\ell_{\overline r}}$. It follows that $\Lambda_{r+\ell_{\overline r}} \not \subset \pi\dot{v}_J^{-1}\Lambda_{r}$.
\end{proof}
Next, we prove that the locally closed subvarieties $\mathcal O_J$ are precisely the orbits of $\overline{\mathcal G_{I}}$ acting on $\overline{M}_{\mathrm{loc}}$. Let 
\begin{equation*}
    \mathcal G_{I,J} \coloneq  \mathrm{Stab}_{\overline{\mathcal G_{I}}}(x_J)
\end{equation*}
denote the stabilizer of $x_J$ in $\overline{\mathcal G_{I}}$. It is a closed subgroup scheme of $\overline{\mathcal G_I}$ and the immersion $\overline{\mathcal G_I}/\mathcal G_{I,J} \to \overline{M}_{\mathrm{loc}}, [g] \mapsto g\cdot x_J$ identifies the quotient with a smooth locally closed subscheme of $\overline{M}_{\mathrm{loc}}$.
\begin{prop}\label{Prop:TheOJAreOrbits}
    The immersion $\overline{\mathcal G_I}/\mathcal G_{I,J} \to \overline{M}_{\mathrm{loc}}$ induces an isomorphism 
    \begin{equation*}
        \overline{\mathcal G_I}/\mathcal G_{I,J} \xrightarrow{\sim} \mathcal O_J.
    \end{equation*}
\end{prop}
\begin{proof}
    Since $x_J \in \mathcal O_J$ and $\mathcal O_J$ is clearly stable under the action of $\overline{\mathcal G_I}$, the map $\overline{\mathcal G_I}/\mathcal G_{I,J} \to \overline{M}_{\mathrm{loc}}$ factors through $\mathcal O_J$. Since the image of $\overline{\mathcal G_I}/\mathcal G_{I,J}$ and $\mathcal O_J$ are both reduced, it remains to prove that 
    \begin{equation*}
        \overline{\mathcal G_I}(\overline{\mathbb F_q}) \cdot x_J = \mathcal O_J(\overline{\mathbb F_q}).
    \end{equation*}
    Of course, the inclusion $\subseteq$ is clear. For the converse, we need the following Lemma.
    \begin{lem}\label{Lem:ProjectiveCover}
        Let $k$ be a field, $A$ a finite dimensional $k$-algebra and $P$ a finitely generated projective $A$-module. If $f,f':P \twoheadrightarrow Q$ are two epimorphisms onto a common $A$-module $Q$, then there exists $\gamma \in \mathrm{Aut}_{A}(P)$ such that $\gamma(\mathrm{Ker}(f)) = \mathrm{Ker}(f')$.
    \end{lem}
    \begin{proof}[Proof of Lemma \ref{Lem:ProjectiveCover}]
        Since $A$ is semiperfect, we can consider a projective cover $\varphi: P_Q \twoheadrightarrow Q$. Since both $P$ and $P_Q$ are projective, there exist maps 
        \begin{equation*}
            u:P_Q \longrightarrow P, \qquad v: P \longrightarrow P_Q,
        \end{equation*}
        such that $\varphi = f\circ u$ and $f = \varphi \circ v$. It follows that $\varphi \circ vu = \varphi$. Since $\varphi$ is surjective, we deduce that 
        \begin{equation*}
            P_Q = vu(P_Q) + \mathrm{Ker}(\varphi).
        \end{equation*}
        But since $\mathrm{Ker}(\varphi)$ is superfluous, it follows that $P_Q = vu(P_Q)$ thus $vu$ is surjective, and hence an automorphism of $P_Q$. Define $u_0 \coloneq  u \circ (vu)^{-1}$, so that $v\circ u_0 = \mathrm{id}_{P_Q}$. It follows that $u_0$ is injective and
        \begin{equation*}
            P = u_0(P_Q) \oplus \mathrm{Ker}(v).
        \end{equation*}
        Since $\mathrm{Ker}(v)$ is a direct summand of $P$, it is projective, hence there exists a map $w : \mathrm{Ker}(v) \to P_Q$ such that $f_{\mid \mathrm{Ker}(v)} = \varphi \circ w$. We define an automorphism $T \in \mathrm{Aut}_A(P)$ via
        \begin{equation*}
            T: P = u_0(P_Q) \oplus \mathrm{Ker}(v) \xrightarrow{\sim} u_0(P_Q) \oplus \mathrm{Ker}(v), \qquad  T(u_0(x),y) \coloneq  (u_0(x-w(y)),y).
        \end{equation*}
        We compute 
        \begin{equation*}
            f \circ T(u_0(x),y) = f(u_0(x - w(y)), y) = \underbrace{f\circ u_0}_{\varphi}(x-w(y)) + f_{\mid \mathrm{Ker}(v)}(y) = \varphi(x). 
        \end{equation*}
        Applying the same construction to $f'$, there exists maps $v': P \to P_Q$ and $u_0' : P_Q \to P$ such that $v' \circ u_0' = \mathrm{id}_{P_Q}$, and an automorphism $T' \in \mathrm{Aut}_A(P)$ such that $f'\circ T'$ is given by $(u_0'(x),y') \mapsto \varphi(x)$ in the decomposition $P = u_0'(P_Q) \oplus \mathrm{Ker}(v')$. By the Krull-Schmidt theorem, there is an isomorphism $\theta: \mathrm{Ker}(v) \xrightarrow{\sim} \mathrm{Ker}(v')$, and we get an automorphism 
        \begin{equation*}
            \Theta = (\beta,\theta) : P = u_0(P_Q) \oplus \mathrm{Ker}(v) \xrightarrow{\sim} u_0'(P_Q) \oplus \mathrm{Ker}(v') = P,
        \end{equation*}
        where $\beta : u_0(P_Q) \xrightarrow{\sim} u_0'(P_Q)$ is defined by $\beta (u_0(x)) \coloneq  u_0'(x)$ for all $x \in P_Q$. Finally, define the automorphism $\gamma \coloneq  T' \circ \Theta \circ T^{-1}$. We have 
        \begin{equation*}
            T^{-1}(\mathrm{Ker}(f)) = u_0(\mathrm{Ker}(\varphi))\oplus\mathrm{Ker}(v), \qquad (T')^{-1}(\mathrm{Ker}(f')) = u_0'(\mathrm{Ker}(\varphi))\oplus\mathrm{Ker}(v'),
        \end{equation*}
        and
        \begin{equation*} 
            \Theta(u_0(\mathrm{Ker}(\varphi))\oplus\mathrm{Ker}(v)) = u_0'(\mathrm{Ker}(\varphi))\oplus\mathrm{Ker}(v').
        \end{equation*}
        Therefore $\gamma(\mathrm{Ker}(f)) = \mathrm{Ker}(f')$ as desired.
    \end{proof}
    We come back to the proof of Proposition \ref{Prop:TheOJAreOrbits}. Let $k\coloneq \overline{\mathbb F_q}$ and let $Q_I$ be the cyclic quiver with vertex set $\overline I$ and with an arrow $a_i: i+\ell_i \to i$ for all $i \in \overline I$. We define 
    \begin{equation*}
    A_I \coloneq  kQ_I / \left\langle \text{all paths of length }\# \overline I \right\rangle,
    \end{equation*}
    where $kQ_I$ is the quiver algebra of $Q_I$ over $k$. We point out that the paths of length $\# \overline I$ in $Q_I$ simply consist in starting from any vertex and coming back to it after a full turn. For $r \in I_{\mathbb Z}$, define
    \begin{equation*}
        (V_I)_{\overline r} \coloneq  \left(\Lambda_r/\pi\Lambda_r\right)\otimes_{\mathbb F_{q^2}}k,
    \end{equation*}
    and let the maps $(V_I)_{\overline r+\ell_{\overline r}} \to (V_I)_{\overline r}$ be induced by the inclusions $\Lambda_{r+\ell_{\overline r}}\subseteq\Lambda_{r}$. By $\pi$-periodicity of $\Lambda_{\bullet}$, this construction does not depend on the choice of the representative $r$ of $\overline r$. Then $V_I \coloneq  \bigoplus_{i \in \overline I} (V_I)_i$ is a finitely generated $A_I$-module.\\
    
    We claim that $V_I$ is a projective $A_I$-module. For $i\in\overline I$, let $P_i$ denote the $A_I$-module defined by $(P_i)_j \coloneq  k$ for all $j \in \overline I$, and the map $(P_i)_{j+\ell_j} \to (P_i)_{j}$ is $\mathrm{id}$ if $j \not = i$, and $0$ if $j=i$. Then $P_i$ is an indecomposable projective finitely generated $A_I$-module. \\
    For each $r \in I_{\mathbb Z}$, the space $(V_I)_{\overline r}$ has a basis given by the images of the vectors $\mathbf e_{r+1}^{(0)}, \ldots , \mathbf e_{r+n}^{(0)}$. For $b \in \mathbb Z/n\mathbb Z$, we get a submodule $V(b) \subseteq V_I$ by defining $V(b)_{\overline r} \coloneq  k \cdot [\mathbf e_{r+d}^{(0)}] \subseteq (V_I)_{\overline r}$ where $1 \leq d \leq n$ is the unique integer such that $\overline{r+d} = b$. The map $(V_I)_{\overline r+\ell_{\overline r}} \to (V_I)_{\overline r}$ sends $V(b)_{\overline r+\ell_{\overline r}}$ onto $V(b)_{\overline r}$ if $b \not \in B_{\overline r}$, and to zero if $b \in B_{\overline r}$. It follows that $V(b) \simeq P_{i}$, where $i \in \overline I$ is the unique element such that $b \in B_i$. Therefore
    \begin{equation*}
        V_I = \bigoplus_{b \in \mathbb Z/n\mathbb Z} V(b) \simeq \bigoplus_{i \in \overline I} P_i^{\oplus \ell_i},
    \end{equation*}
    proving that $V_I$ is projective.\\
    
    Let $x=(\mathfrak t_r)_{r\in I_{\mathbb Z}} \in \mathcal O_J(k)$. For $r \in I_{\mathbb Z}$, set 
    \begin{equation*}
        H_{x,\overline r}\coloneq \eta_0\mathfrak t_r \subseteq \left(\Lambda_r/\pi\Lambda_r\right)\otimes_{\mathbb F_{q^2}}k = (V_I)_{\overline r}.
    \end{equation*}
    Then $H_x=(H_{x,i})_{i\in\overline I}$ is a submodule of $V_I$. We write
    \begin{equation*}
        q_x: V_I\twoheadrightarrow Q_x\coloneq V_I/H_x,
    \end{equation*}
    for the quotient map. The space $(Q_x)_i$ is one-dimensional for every $i$, and for all $r \in I_{\mathbb Z}$, the map $(Q_x)_{\overline r+\ell_{\overline r}} \to (Q_x)_{\overline r}$ is nothing but the map $\delta_r$ defined in \eqref{Eq:DefiMapDeltarx}. We know that $\delta_r$ is an isomorphism if $\overline r \in J$, and is zero if $\overline r \in \overline I \setminus J$. Since $J$ is a proper subset of $\overline I$, all the non-zero $\delta_r$ can be scaled to the identity. In other words, we have
    \begin{equation*}
        Q_x \simeq \bigoplus_{i \in \overline I\setminus J} S_i,
    \end{equation*}
    where for every $i\in \overline I \setminus J$, $S_i$ is defined as follows. Define recursively $i_0\coloneq i$ and $i_{a+1}\coloneq i_a+\ell_{i_a}$. Let $m_i\geq 1$ be the smallest integer such that $i_{m_i}\in \overline I \setminus J$. We define $S_i$ to be the $A_I$-module given by
    \begin{equation*}
        (S_i)_j= \begin{cases}
            k & \text{if } j\in \{i_1,\ldots,i_{m_i}\},\\
            0 & \text{otherwise},
        \end{cases}
    \end{equation*}
    whose maps between two non-zero spaces are the identity. Consequently, the isomorphism class of $Q_x$ only depends on $J$. Thus, we have $Q_x \simeq Q_{x_J}$. Identifying both modules via this isomorphism, Lemma \ref{Lem:ProjectiveCover} gives an automorphism $\gamma_0 \in \mathrm{Aut}_{A_I}(V_I)$ such that $\gamma_0(H_x) = H_{x_J}$. Then $\gamma_0$ has a unique lift $\gamma \in \overline{\mathcal G_I}(k)$ with multiplier equal to $1$, and we have $\gamma(x) = x_J$. This concludes the proof.
\end{proof}

For every $[v]_I \in \mathrm{Adm}(\{\mu'\})_I$, we obtain a locally closed immersion of algebraic stacks 
\begin{equation*}
    [\mathcal O_{[v]_I}/\mathcal G_{I,\overline{\mathbb F_q}}] \hookrightarrow [\overline{M}_{\mathrm{loc}}/\mathcal G_{I,\overline{\mathbb F_q}}].
\end{equation*}
\begin{defi}
    The Kottwitz-Rapoport stratum associated to $[v]_I \in \mathrm{Adm}(\{\mu'\})_I$ is the locally closed subscheme of $\mathcal N^{I,\varepsilon}_{\mathrm{red}}$ defined by 
    \begin{equation*}
        \mathrm{KR}_{[v]_I} = \mathrm{KR}_{[v]_I}^{I,\varepsilon} \coloneq  \mathcal N^{I,\varepsilon}_{\mathrm{red}} \times_{[\overline{M}_{\mathrm{loc}}/\mathcal G_{I,\overline{\mathbb F_q}}]} [\mathcal O_{[v]_I}/\mathcal G_{I,\overline{\mathbb F_q}}].
    \end{equation*}
    We also write $\mathrm{KR}_J$ instead of $\mathrm{KR}_{[v]_I}$ when $J \in \mathcal P_{I}$ and $[v]_I = [v_J]_I$.
\end{defi}
Recall the description of the $k$-points of $\mathcal N^{I,\varepsilon}_{\mathrm{red}}$ stated in Proposition \ref{Prop:PointsArbitraryParahoricBis}.
\begin{prop}\label{Prop:PointsOfKRStrata}
    Let $I \subseteq \{0,\ldots , \nu_{\varepsilon}\}$ be non-empty, let $J \in \mathcal P_{I}$, let $k$ be a perfect field containing $\overline{\mathbb F_q}$, and let $j \in \mathbb Z$ such that $nj$ is even. We have 
    \begin{equation*}
        (\mathrm{KR}_J \cap \mathcal N^{I,\varepsilon}_j)(k) \simeq \left\{ (Z_r)_{r\in I_{\mathbb Z}} \in \mathcal N^{I,\varepsilon}_j(k) \; \middle| \; 
        \forall r \in I_{\mathbb Z},\quad \overline r \in \overline I \setminus J \iff \tau(Z_{r+\ell_{\overline r}}) \subseteq \pi^{j+1} Z_{-r-\underline{\varepsilon}}^{\vee} \right\}.
    \end{equation*}
\end{prop}

\begin{proof}
Let $x = \left(X_r,i_{X_r},\lambda_{X_r},\rho_{X_r}\right)_{r\in I}\in\mathcal N_{j}^{I,\varepsilon}(k)$, and for every $r \in I$, let
\begin{equation*}
M_r = M_{r,0}\oplus M_{r,1} \subseteq \mathbb V_{F_{k}},
\end{equation*}
be the image under $\rho_{X_r}$ of the relative Dieudonné module of $X_r$. Here, $M_{r,a} \coloneq  \eta_aM_r$ for $a=0,1$. By the proofs of Proposition \ref{Prop:PointsMaximalParahoric} and Proposition \ref{Prop:PointsArbitraryParahoricBis}, we have
\begin{align*}
Z_{r}=M_{r,0}, && Z_{-r-\underline{\varepsilon}}=\pi^{j}\left(M_{r,1}\right)^{\dagger}.
\end{align*}
Besides, for all $r \in I$ we have
\begin{equation*}
    \mathbb D(X_r)_k \simeq M_r/\pi M_r,
\end{equation*}
and the Hodge filtration becomes 
\begin{center}
    \begin{tikzcd}
        0 \arrow[r] & \mathrm{Fil}^1 \mathbb D(X_r)_k \arrow[r] \arrow[d,"\rotatebox{90}{\(\sim\)}"] & \mathbb D(X_r)_k \arrow[r] \arrow[d,"\rotatebox{90}{\(\sim\)}"] & \mathrm{Lie}(X_r) \arrow[r] \arrow[d,"\rotatebox{90}{\(\sim\)}"] & 0 \\
        0 \arrow[r] & \mathcal VM_r/\pi M_r \arrow[r] & M_r/\pi M_r \arrow[r] & M_r/\mathcal VM_r \arrow[r] & 0
    \end{tikzcd}
\end{center}
For every finite extension $k'/k$, the ring of relative Witt vectors $W_{\mathcal O_F}(k')$ is faithfully flat over $W_{\mathcal O_F}(k)$, so that all the lattice inclusions in the statement of the Proposition can be checked after extension to $k'$, as well as the membership of $x$ to a given KR stratum. Thus, replacing $k$ with a suitable finite extension, we can pick a compatible collection of trivializations
\begin{equation*}
\forall r\in I, \qquad \gamma_r: M_r/\pi M_r \xrightarrow{\sim} L_{r} \otimes_{\mathcal O_{F}} k.
\end{equation*}
The point $y = (\mathfrak t_r)_{r \in I_{\mathbb Z}} \in M_{\mathrm{loc}}^{I,\varepsilon}(k)$ corresponding to $x$, together with its trivialization, via the local model diagram is determined by 
\begin{equation*}
\forall r \in I, \qquad \mathfrak t_{r} = \gamma_{r}\left(\mathrm{Fil}^1 \mathbb D(X_r)_k\right).
\end{equation*}
For $r \in I_{\mathbb Z}$, define $Q_r \coloneq  (\Lambda_r \otimes_{\mathcal O_K} k) / \eta_0 \mathfrak t_r$. Thus $Q_r$ is a $1$-dimensional $k$-vector space, and via the trivialization we have
\begin{equation*}
    \forall r \in I_{\mathbb Z}, \qquad Q_{r} \simeq Z_{r}/\pi^{j+1}\tau^{-1}(Z_{-r-\underline{\varepsilon}}^{\vee}).
\end{equation*}
Under these identifications, the morphisms $\delta_r$ are the natural maps
\begin{equation*}
    \forall r \in I_{\mathbb Z}, \qquad \delta_{r} : Z_{r+\ell_{\overline r}}/\pi^{j+1}\tau^{-1}(Z_{-r-\ell_{\overline r}-\underline{\varepsilon}}^{\vee}) \longrightarrow Z_{r}/\pi^{j+1}\tau^{-1}(Z_{-r-\underline{\varepsilon}}^{\vee}), 
\end{equation*}
induced by the inclusions $Z_{r+\ell_{\overline r}} \subsetneq Z_r$ and $Z_{-r-\underline{\varepsilon}} \subsetneq Z_{-r-\ell_{\overline r}-\underline{\varepsilon}}$. Now, let $J \in \mathcal P_{I}$. By construction, we have $x \in (\mathrm{KR}_{J} \cap \mathcal N^{I,\varepsilon}_j)(k)$ if and only if we have the equivalence 
\begin{equation*}
    \forall r \in I_{\mathbb Z}, \qquad \overline r \in \overline I \setminus J \iff \delta_{r} = 0.
\end{equation*}
With the description given above, the proposition follows.
\end{proof}
\begin{rk}\label{Rk:KRNonEmptyIFFInAdmMu0}
    By \cite{gortzFullyHodgeNewton2019} Proposition 5.6, it is known that $\mathrm{KR}_J \not = \emptyset$ if and only if $J \in \mathcal P_{0,I}$, or equivalently $[v_J]_I \in \mathrm{Adm}(\{\mu'\})_{0,I}$. We will give an independent direct proof of this fact in Corollary \ref{Corol:KRStrataNonEmpty}. For now, we will just admit it.
\end{rk}
Recall the isomorphism $\Phi_I: \mathcal N^{I,\varepsilon} \xrightarrow{\sim} \mathcal N^{\widetilde I, \varepsilon'}$ given in Proposition \ref{Prop:IsomorphismDualityArbitraryParahoric} and Remark \ref{Rk:IsomorphismDualityArbitraryParahoric}. On the special fiber of the target space, we can consider the KR strata $\mathrm{KR}_{\widetilde J}$ indexed by $\widetilde J \in \mathcal P_{0,\widetilde I}$, where $\widetilde I$ was defined in \eqref{Eq:DefiWidetildeI}. Temporarily, let us write $\ell_i^I$ instead of $\ell_i$, and consider also $\ell_i^{\widetilde I}$ for all $i \in \overline{\widetilde I}$. We have a bijection 
\begin{equation}\label{Eq:DefinitionMapVarphiI}
    \varphi_I: \overline I \xrightarrow{\sim} \overline{\widetilde I}, \qquad i \mapsto \nu_{\varepsilon}-i-\ell_i^I = \theta_I(i) + \nu_{\varepsilon} + \underline{\varepsilon} \in \overline{\widetilde I} \subseteq \mathbb Z/n\mathbb Z.
\end{equation} 
One may check that $\ell^{\widetilde I}_{\varphi_I(i)} = \ell_i^I$ for all $i \in \overline I$, and that $\varphi_I \circ \theta_I = \theta_{\widetilde I} \circ \varphi_I$.  

\begin{prop}\label{Prop:DualityKRStrata}
    The isomorphism $\Phi_I: \mathcal N^{I,\varepsilon} \xrightarrow{\sim} \mathcal N^{\widetilde I, \varepsilon'}$ induces, for all $J \in \mathcal P_{0,I}$, an isomorphism between KR strata
    \begin{equation*}
        \Phi_I: \mathrm{KR}_{J}^{I,\varepsilon} \xrightarrow{\sim} \mathrm{KR}_{\varphi_I(\theta_I(J))}^{\widetilde I, \varepsilon'}.
    \end{equation*}
\end{prop}

\begin{proof}
    Let $k$ be a perfect field containing $\overline{\mathbb F_q}$, let $j \in \mathbb Z$ be such that $nj$ is even, and let $z = (Z_r)_{r\in I_{\mathbb Z}} \in \mathcal N_j^{I,\varepsilon}(k)$. Write $\widetilde z = (\widetilde Z_{\widetilde r})_{\widetilde r\in\widetilde I_{\mathbb Z}} \coloneq \Phi_I(z)$. By Lemma \ref{Lem:LatticeChainUnderDuality}, we have
    \begin{equation*}
        \widetilde z \in \mathcal N_{-j}^{\widetilde I,\varepsilon'}(k), \qquad \widetilde Z_{\nu_{\varepsilon}-r} = \pi^{-j}Z_{-r-\underline{\varepsilon}}
    \end{equation*}
    for every $r\in I_{\mathbb Z}$. Fix such an $r$ and define $\widehat r \coloneq r + \nu_{\varepsilon} + \underline{\varepsilon} \in \widetilde I_{\mathbb Z}$. Its reduction modulo $n$ is $\overline{\widehat r} = \varphi_I(\theta_I(\overline r))$ so that $\ell_{\overline{\widehat r}}^{\widetilde I} = \ell_{\overline r}^{I}$. Applying Lemma \ref{Lem:LatticeChainUnderDuality} to $-r-\underline{\varepsilon}-\ell_{\overline r}^{I} \in I_{\mathbb Z}$ gives
    \begin{equation*}
        \widetilde Z_{\widehat r+\ell_{\overline{\widehat r}}^{\widetilde I}} = \pi^{-j}Z_{r+\ell_{\overline r}^{I}}.
    \end{equation*}
    On the other hand, since $2\nu_{\varepsilon} + \underline{\varepsilon} + \underline{\varepsilon'} = n$, periodicity gives
    \begin{equation*}
        \widetilde Z_{-\widehat r-\underline{\varepsilon'}} = \pi^{-j-1}Z_{-r-\underline{\varepsilon}}.
    \end{equation*}
    Let $\vee'$ denote duality for the hermitian form $\pi(\cdot,\cdot)$. We have $Z^{\vee'} = \pi^{-1}Z^\vee$ for every lattice $Z$. Therefore
    \begin{align*}
        \tau\left(\widetilde Z_{\widehat r+ \ell_{\overline{\widehat r}}^{\widetilde I}}\right) \subseteq \pi^{-j+1} \widetilde Z_{-\widehat r-\underline{\varepsilon'}}^{\vee'} & \iff \pi^{-j} \tau\left(Z_{r+\ell_{\overline r}^{I}}\right) \subseteq \pi^{-j+1} \left(\pi^{-j-1}Z_{-r-\underline{\varepsilon}}\right)^{\vee'} \\
        & \iff \pi^{-j} \tau\left(Z_{r+\ell_{\overline r}^{I}}\right) \subseteq \pi Z_{-r-\underline{\varepsilon}}^{\vee} \\
        & \iff \tau\left(Z_{r+\ell_{\overline r}^{I}}\right) \subseteq \pi^{j+1}Z_{-r-\underline{\varepsilon}}^{\vee}.
    \end{align*}
    Besides, by definition we have
    \begin{equation*}
        \overline{\widehat r} \in \overline{\widetilde I} \setminus\varphi_I(\theta_I(J)) \iff \overline r\in\overline I\setminus J.
    \end{equation*}
    It follows that $z \in \mathrm{KR}_{J}^{I,\varepsilon} \iff \widetilde z\in \mathrm{KR}_{\varphi_I(\theta_I(J))}^{\widetilde I,\varepsilon'}$, which concludes the proof.
\end{proof}

If $n$ is even then $\nu_{\varepsilon} + \underline{\varepsilon} = \frac{n}{2}$. If additionally we have $I = \widetilde I$, then for all $J\in \mathcal P_{0,I}$ we obtain an isomorphism
\begin{equation*}
    \mathrm{KR}_J \xrightarrow{\sim} \mathrm{KR}_{J + \frac{n}{2}},
\end{equation*}
where $J+\frac{n}{2} \coloneq  \{i+\frac{n}{2} \mid i \in J\}$.

\section{The Ekedahl-Kottwitz-Oort-Rapoport stratification}
\subsection{\texorpdfstring{$G$-Zip}{G-Zip} datum associated to KR strata}\label{Section4.2}

Following \cite{shenEKORStrataShimura2021}, there are essentially two different ways of introducing the EKOR stratification, either locally as a stratification of individual KR strata via the theory of $G$-zips, either globally via the moduli space of (truncated) local Shtukas. For our purpose the first approach is enough. Regarding $G$-zips, the references are \cite{pinkAlgebraicZipData2011} and \cite{pinkFzipsAdditionalStructure2015}, see also \cite{shenEKORStrataShimura2021} Section 1.1 for a concise summary.
\begin{defi}[\cite{pinkAlgebraicZipData2011} Definition 1.1]\label{Defi:AlgebraicZipDatum}
An \textit{algebraic zip datum over a field $k$} is a tuple
\begin{equation*}
\mathcal Z = \left(\mathbf G,\mathbf P,\mathbf Q,\varphi\right),
\end{equation*}
where $\mathbf G$ is a connected reductive group over a field $k$, $\mathbf P$ and $\mathbf Q$ are two parabolic subgroups of $\mathbf G_{k'}$ for some finite field extension $k'/k$, and 
\begin{equation*}
\varphi:\mathbf P/R_u(\mathbf P) \longrightarrow \mathbf Q/R_u(\mathbf Q)
\end{equation*}
is an isogeny of algebraic groups, where $R_u(\cdot)$ denotes the unipotent radical.
\end{defi}
If $p\in\mathbf P$ and $q\in\mathbf Q$, we denote by $\overline p$ and $\overline q$ their respective images in $\mathbf P/R_u(\mathbf P)$ and in $\mathbf Q/R_u(\mathbf Q)$. The \textit{zip group attached to $\mathcal Z$} is
\begin{equation*}
\mathbf E_{\mathcal Z} \coloneq  \{ (p,q) \in \mathbf P \times \mathbf Q \mid \varphi(\overline p) = \overline q\}.
\end{equation*}
The group $\mathbf E_{\mathcal Z}$ acts on $\mathbf G_{k'}$ on the left by
\begin{equation*}
(p,q)\cdot g \coloneq  pgq^{-1}.
\end{equation*}
Let $\mathcal Z_{\overline{k}}$ denote the base change of the zip datum $\mathcal Z$ to $\overline{k}$. Following \cite{pinkAlgebraicZipData2011} Definition 3.6, we can fix a \textit{frame} of $\mathcal Z_{\overline k}$, that is a triple $(\mathbf T,\mathbf B,g)$ consisting of a Borel subgroup and a maximal torus $\mathbf T \subset \mathbf B \subset \mathbf G_{\overline k}$, together with an element $g \in \mathbf G(\overline k)$ such that 
\begin{align*}
    \mathbf B & \subseteq \mathbf Q_{\overline k}, & \varphi(\overline{{}^g\mathbf B}) & = \overline{\mathbf B},\\
    {}^g\mathbf B & \subseteq \mathbf P_{\overline k}, & \varphi(\overline{{}^g\mathbf T}) & = \overline{\mathbf T},
\end{align*}
where $\overline{\mathbf B}$ and $\overline{\mathbf T}$ (resp. $\overline{{}^g\mathbf B}$ and $\overline{{}^g\mathbf T}$) denote their images in $\mathbf Q_{\overline k}/R_u(\mathbf Q_{\overline k})$ (resp. in $\mathbf P_{\overline k}/R_u(\mathbf P_{\overline k})$). The frame determines Levi components 
\begin{equation*}
    {}^g\mathbf T \subseteq \mathbf L_{\mathbf P} \subseteq \mathbf P_{\overline k}, \qquad \mathbf T \subseteq \mathbf L_{\mathbf Q} \subseteq \mathbf Q_{\overline k},
\end{equation*} 
so that we can identify $\mathbf L_{\mathbf P} \simeq \mathbf P_{\overline k}/R_u(\mathbf P_{\overline k})$ and $\mathbf L_{\mathbf Q} \simeq \mathbf Q_{\overline k}/R_u(\mathbf Q_{\overline k})$. The pair $(\mathbf T,\mathbf B)$ determines a finite Coxeter system $(\mathbf W,\mathbf S)$. Let $\ell$ denote the associated length function. Let $\mathbf J_{\mathbf P},\mathbf J_{\mathbf Q} \subseteq \mathbf S$ denote the types of $\mathbf P$ and $\mathbf Q$. The Coxeter systems of $\mathbf L_{\mathbf P}$ and $\mathbf L_{\mathbf Q}$ can be identified with $(\mathbf W_{\mathbf J_{\mathbf P}}, \mathbf J_{\mathbf P})$ and $(\mathbf W_{\mathbf J_{\mathbf Q}}, \mathbf J_{\mathbf Q})$ respectively, and the composition $\varphi \circ \mathrm{Ad}(g)$ induces an isomorphism
\begin{equation*}
    \psi: (\mathbf W_{\mathbf J_{\mathbf P}}, \mathbf J_{\mathbf P}) \xrightarrow{\sim} (\mathbf W_{\mathbf J_{\mathbf Q}}, \mathbf J_{\mathbf Q}).
\end{equation*}
Let ${}^{\mathbf J_{\mathbf P}}\mathbf W$ denote the set of elements $w$ which are (left) $\mathbf J_{\mathbf P}$-reduced, and fix a representative $\dot{w} \in \mathrm{N}_{\mathbf G_{\overline k}}(\mathbf T)$. We define 
\begin{equation}\label{Eq:DefiOrbitZipGroup}
    \mathbf G^w \coloneq  \mathbf E_{\mathcal Z_{\overline k}}\cdot g\mathbf B\dot{w}\mathbf B.
\end{equation}
Finally, we define a partial order $\leq_{\mathbf J_{\mathbf P},\psi}$ on ${}^{\mathbf J_{\mathbf P}}\mathbf W$ as follows:
\begin{equation*}
    \forall w,w' \in {}^{\mathbf J_{\mathbf P}}\mathbf W, \qquad w' \leq_{\mathbf J_{\mathbf P},\psi} w \iff \exists x \in \mathbf W_{\mathbf J_{\mathbf P}}, \quad xw'\psi(x)^{-1} \leq w,
\end{equation*}
where $\leq$ is the usual Bruhat order on $\mathbf W$.
\begin{theo}[\cite{pinkAlgebraicZipData2011} Theorem 1.3, Theorem 1.4, Proposition 7.3]
    For $w \in {}^{\mathbf J_{\mathbf P}}\mathbf W$, $\mathbf G^w$ is a locally closed smooth subvariety of $\mathbf G_{\overline k}$ of dimension $\dim \mathbf P + \ell(w)$. We have a stratification 
    \begin{equation*}
        \mathbf G_{\overline k} = \bigsqcup_{w \in {}^{\mathbf J_{\mathbf P}}\mathbf W} \mathbf G^w, \quad \quad \overline{\mathbf G^w} = \bigsqcup_{\substack{w' \in {}^{\mathbf J_{\mathbf P}}\mathbf W \\
         w' \leq_{\mathbf J_{\mathbf P},\psi} w}} \mathbf G^{w'}.
    \end{equation*}
    Each $\mathbf G^w$ is a single orbit under the action of $\mathbf E_{\mathcal Z_{\overline k}}$ if the differential of $\varphi$ at $1$ vanishes. 
\end{theo}
Let us go back to the original setting, that is before taking algebraic closure. The quotient stack associated with an algebraic zip datum $\mathcal Z$ over $k$ is denoted by
\begin{equation*}
[\mathbf E_{\mathcal Z}\backslash\mathbf G_{k'}].
\end{equation*}
Let $S$ be a $k'$-scheme. 
\begin{defi}[\cite{pinkFzipsAdditionalStructure2015} Definition 3.1]
    A \textit{$\mathbf G$-zip of type $\mathbf P$ over $S$} is a tuple $(\mathcal I, \mathcal I_{\mathbf P}, \mathcal I_{\mathbf Q}, \iota)$ where 
    \begin{itemize}
        \item $\mathcal I$ is a right $\mathbf G_{k'}$-torsor over $S$,
        \item $\mathcal I_{\mathbf P} \subset \mathcal I$ is a right $\mathbf P$-torsor over $S$,
        \item $\mathcal I_{\mathbf Q} \subset \mathcal I$ is a right $\mathbf Q$-torsor over $S$,
        \item $\iota: \mathcal I_{\mathbf P} / R_u(\mathbf P) \xrightarrow{\sim} \mathcal I_{\mathbf Q}/ R_u(\mathbf Q)$ is a $\varphi$-equivariant isomorphism, that is 
        \begin{equation*}
            \forall p \in \mathbf P, \forall x \in \mathcal I_{\mathbf P} / R_u(\mathbf P), \qquad \iota(x\overline p) = \iota(x)\varphi(\overline p).
        \end{equation*}
    \end{itemize}
\end{defi}
One can form the stack $\mathbf G\text{-}\mathrm{Zip}^{\mathbf P}$ classifying $\mathbf G$-zip of type $\mathbf P$, and according to \cite{pinkFzipsAdditionalStructure2015} Theorem 1.5 it is isomorphic to the quotient stack $[\mathbf E_{\mathcal Z}\backslash\mathbf G_{k'}]$. Both stacks are smooth and $0$-dimensional over $k'$. As a result, giving a morphism from $S$ to $[\mathbf E_{\mathcal Z}\backslash\mathbf G_{k'}]$ amounts to giving a $\mathbf G$-zip of type $\mathbf P$ over $S$. See \cite{shenEKORStrataShimura2021} Equation 1.1.6 for more details.\\

To each (non-empty) KR stratum $\mathrm{KR}_J$, $J \in \mathcal P_{0,I}$, we are going to define a zip datum $\mathcal Z_J$. First, given any non-empty subset $I \subseteq \{0,\ldots , \nu_{\varepsilon}\}$, we define 
\begin{equation*}
\mathbf G_I \coloneq  \mathcal G_{I,\mathbb F_q}/ R_u(\mathcal G_{I,\mathbb F_q}).
\end{equation*}
The quotient $\mathbf G_I$ is called the \textit{maximal reductive quotient} of $\mathcal G_{I,\mathbb F_q}$ and it is a connected reductive group over $\mathbb F_q$ which we can describe as follows. Put $r_{\mathrm{min}} \coloneq \min(I)$ and $r_{\mathrm{max}} \coloneq \max(I)$. We define
\begin{equation*}
    V_I^{\mathrm{min}}\coloneq L_{r_{\mathrm{min}}}^\dagger/ L_{r_{\mathrm{min}}}, \qquad V_I^{\mathrm{max}} \coloneq L_{r_{\mathrm{max}}} / \pi L_{r_{\mathrm{max}}}^\dagger, \qquad V_I^{r}\coloneq L_r/L_{r'}.
\end{equation*}
where $r<r'$ run over all consecutive integers in $I$. These are $\mathbb F_{q^2}$-vector spaces of respective dimensions
\begin{equation*}
    \dim V_I^{\mathrm{min}} = h_{r_{\mathrm{min}}}^{\varepsilon} \geq 0, \qquad  \dim V_I^{\mathrm{max}} = n-h_{r_{\mathrm{max}}}^{\varepsilon} \geq 0, \qquad  \dim V_I^{r} = \ell_{\overline r} \geq 1.
\end{equation*}
The spaces $V_I^{\mathrm{min}}$ and $V_I^{\mathrm{max}}$ carry non-degenerate $\mathbb F_{q^2}/\mathbb F_q$-hermitian forms $\overline H_I^{\mathrm{min}}$ and $\overline H_I^{\mathrm{max}}$ induced respectively by $\pi H_{|L_{r_{\mathrm{min}}}^\dagger\times L_{r_{\mathrm{min}}}^\dagger}$ and $H_{|L_{r_{\mathrm{max}}}\times L_{r_{\mathrm{max}}}}$ modulo $\pi$. There is a natural isomorphism
\begin{equation*}
    \mathbf G_I \simeq \left( \mathrm{GU}(V_I^{\mathrm{min}}, \overline H_I^{\mathrm{min}}) \times_{\mathbb G_m} \mathrm{GU}(V_I^{\mathrm{max}},\overline H_I^{\mathrm{max}}) \right) \times \prod_{\substack{r<r'\\\text{consecutive in }I}} \mathrm{Res}_{\mathbb F_{q^2}/\mathbb F_q} \mathrm{GL}(V_I^{r}).
\end{equation*}
where the maps $\mathrm{GU}(V_I^{\mathrm{min}},\overline{H}^{\mathrm{min}}_I) \to \mathbb G_m$ and $\mathrm{GU}(V_I^{\mathrm{max}}, \overline H^{\mathrm{max}}_I) \to \mathbb G_m$ are the multiplier characters, and with the convention that if $h_{r_{\mathrm{min}}}^{\varepsilon} = 0$ or $h_{r_{\mathrm{max}}}^{\varepsilon} = n$ then the corresponding $\mathrm{GU}$ factor is omitted.\\
The complete \textit{decreasing} lattice chain $L_{\bullet,\mathcal O_{\breve F}}$ and its stabilizer subgroup $\mathcal G_{I_{\mathrm{Iw}}}$ (notation of \eqref{Eq:NotationIIw}) determine a pair $(\mathbf B_I,\mathbf T_I)$ consisting of an $\mathbb F_q$-rational Borel subgroup $\mathbf B_I\subset \mathbf G_{I,\overline{\mathbb F_q}}$ and an $\mathbb F_q$-rational maximal torus $\mathbf T_I \subset \mathbf B_I$. In the standard basis induced by $\mathbf e$, the Borel $\mathbf B_I$ consists of \textit{lower} triangular matrices in $\mathbf G_I$, and $\mathbf T_I$ consists of the diagonal matrices. The associated Coxeter system of $\mathbf G_I$ is naturally identified with $(W_{\mathbb S_a^I},\mathbb S_a^{I})$ as defined in \eqref{Eq:DefiSIa} and \eqref{Eq:DefiSubgroupWI}. Note that $W_{\mathbb S_a^I}$ is finite since $\mathbb S_a^{I} \subsetneq \mathbb S_a$ is proper, and the affine Weyl group $W_a$ is of type $A$. For every subset $\mathbf J\subseteq\mathbb S_a^{I}$, we define
\begin{equation*}
W_{\mathbf J} \coloneq  \langle s \mid s \in \mathbf J\rangle \subseteq W_{\mathbb S_a^I},
\end{equation*}
and we denote by $\mathbf P_{\mathbf J} \subseteq \mathbf G_{I,\overline{\mathbb F_q}}$ the standard parabolic subgroup of type $\mathbf J$, and by $\mathbf L_{\mathbf J} \subseteq \mathbf P_{\mathbf J}$ the standard Levi subgroup containing $\mathbf T_I$. We point out that since $\sigma^2 = \mathrm{id}$ on $W_{\mathbb S_a^I}$, the parabolic subgroup $\mathbf P_{\mathbf J}$ and the Levi component $\mathbf L_{\mathbf J}$ are defined over $\mathbb F_{q^2}$.\\

From now on, we consider a fixed non-empty KR stratum $\mathrm{KR}_J$ for some $J \in \mathcal P_{0,I}$. Recall the element $v_J \coloneq  w_J\tau_0 \in \mathrm{Adm}(\{\mu'\})$. By Remark \ref{Rk:wJIsShortestElementDoubleCoset}, $v_J$ is the unique element of minimal length in the double coset $W_{\mathbb S_a^I}v_JW_{\mathbb S_a^I} \subset \widetilde{W}$. We define
\begin{equation}\label{Eq:DefiMathbfJ}
\mathbf J \coloneq  \mathbb S_a^{I} \cap \mathrm{Ad}(v_J^{-1})(\mathbb S_a^{I}) \subseteq \mathbb S_a^{I}.
\end{equation}
Recall the point $x_J \in \mathcal O_J(\mathbb F_{q^2})$ which we defined in Lemma \ref{Lem:xJIsAPointOfOJ}, as well as its stabilizer $\mathcal G_{I,J} \subseteq \mathcal G_{I,\mathbb F_{q^2}}$. By construction, the image of $\mathcal G_{I,J}$ in the maximal reductive quotient $\mathbf G_{I,\mathbb F_{q^2}}$ is exactly $\mathbf P_{\mathbf J}$. We also consider the twisted Frobenius
\begin{equation*}
\sigma_J \coloneq  \sigma\circ \mathrm{Ad}(\dot v_J): G(\breve F) \longrightarrow G(\breve F).
\end{equation*}
We use the same notation for the induced action on $\widetilde W$. By construction, we have
\begin{equation*}
\sigma_J(\mathbf J) = \mathbb S_a^{I} \cap \mathrm{Ad}(\sigma(v_J))(\mathbb S_a^{I}) \subseteq \mathbb S_a^I,
\end{equation*}
and $\sigma_J$ induces an isogeny $\sigma_J: \mathbf L_{\mathbf J} \to \mathbf L_{\sigma_J(\mathbf J)}$.

\begin{defi}
The algebraic zip datum attached to the Kottwitz-Rapoport stratum $\mathrm{KR}_J$ is
\begin{equation*}
\mathcal Z_J \coloneq  (\mathbf G_I, \mathbf P_{\mathbf J}, \mathbf P_{\sigma_J(\mathbf J)}, \sigma_J: \mathbf L_{\mathbf J} \longrightarrow \mathbf L_{\sigma_{J}(\mathbf J)}).
\end{equation*}
\end{defi}
Since $\mathbf P_{\mathbf J}$ and $\mathbf P_{\sigma_J(\mathbf J)}$ are defined over $\mathbb F_{q^2}$, $\mathcal Z_J$ is an algebraic zip datum over $k = \mathbb F_q$ with $k' = \mathbb F_{q^2}$ in the notations of Definition \ref{Defi:AlgebraicZipDatum}. To simplify the notations, we will write $\mathbf E_J$ for the zip group of $\mathcal Z_{J,\overline{\mathbb F_q}}$, and $\mathbf G_I\text{-}\mathrm{Zip}^J \simeq [\mathbf E_J \backslash \mathbf G_{I,\overline{\mathbb F_q}}]$ for the associated stack of $\mathbf G_I$-zips of type $\mathbf P_{\mathbf J}$. By \cite{shenEKORStrataShimura2021} Remark 1.1.7, the differential of $\sigma_J$ at $1$ vanishes, so that we obtain a stratification 
\begin{equation*}
    \mathbf G_{I,\overline{\mathbb F_q}} = \bigsqcup_{w \in {}^{\mathbf J}W_{\mathbb S_a^I}} \mathbf G_I^{w},
\end{equation*}
where each $\mathbf G_I^{w}$ is an orbit under the action of $\mathbf E_J$. In the remainder of this section, we will describe the poset ${}^{\mathbf J}W_{\mathbb S_a^I}$ explicitly.\\

Recall the notations $B_i$ for $i \in \overline I$ defined in \eqref{Defi:DefiBlockBi}, as well as the involution $\theta_I$ defined in Definition \ref{Defi:DefiInvolutionThetaI}. For $i \in \overline I$ we define $C_i \coloneq  \{s_b \mid b \in B_i \setminus \{i\}\}$, so that $C_i \not = \emptyset$ if and only if $\ell_i > 1$. If we define 
\begin{equation*}
    \forall i \in \overline I, \forall 1\leq a \leq \ell_i-1, \qquad t_{i,a} \coloneq  s_{i+\overline{a}},
\end{equation*} 
then we have $C_i = \{t_{i,1}, \ldots ,t_{i,\ell_i-1}\}$. Besides,
\begin{equation*}
    \mathbb S_a^I = \bigsqcup_{i \in \overline I} C_i, \qquad W_{\mathbb S_a^I} \simeq \prod_{i \in \overline I} W_{C_i}.
\end{equation*}
We can identify $W_{C_i} \simeq \mathfrak S_{\ell_i}$, and if $\ell_i > 1$ then $t_{i,a}$ is identified with the simple reflection $(a\ a+1)$ in $W_{C_i}$. 
\begin{lem}\label{Lem:DescriptionMathbfJ}
    We have 
    \begin{equation*}
        \mathbf J = \left(\bigsqcup_{i \in J} C_i \right) \sqcup \left(\bigsqcup_{i \in \overline I \setminus J} C_i^{\circ}\right), 
    \end{equation*}
    where $C_i^{\circ} \coloneq  \{t_{i,a} \mid 2 \leq a \leq \ell_i-1\}$.
\end{lem}
\begin{proof}
    For $s \in \mathbb S_a$ we have
    \begin{equation*}
        s \in \mathbf J \iff s = t_{i,a} \text{ for some } i\in \overline I,\ 1 \leq a \leq \ell_i-1, \text{ and } \mathrm{Ad}(v_J)(t_{i,a}) \in \mathbb S_a^I.
    \end{equation*}
    Recall that $v_J = w_J\tau_0$, and that $\mathrm{Ad}(\tau_0)(s_i) = s_{i-1}$ for all $i \in \mathbb Z/n\mathbb Z$. Thus, $\mathrm{Ad}(v_J)(t_{i,a}) = \mathrm{Ad}(w_J)(s_{i+\overline a -1})$. A direct computation from the definition of $w_J$ gives
    \begin{equation*}
        \mathrm{Ad}(v_J)(t_{i,a}) = \begin{cases}
           t_{i,a} & \text{if } i\in J,\\
           t_{i,a-1} & \text{if } i\not \in J \text{ and } 2 \leq a \leq \ell_i-1,\\
           \not \in \mathbb S_a^I & \text{if } i\not \in J \text{ and } a=1.
        \end{cases}
    \end{equation*}
    The result follows.
\end{proof}
For $i \in \overline I \setminus J$ and for $0 \leq a \leq \ell_i-1$, we define 
\begin{equation*}
    u_{i,a} \coloneq  t_{i,1} t_{i,2} \cdots t_{i,a},
\end{equation*}
with the convention $u_{i,0} = \mathrm{id}$.
\begin{corol}\label{Corol:DescriptionMathbfJreducedElements}
    We have
    \begin{equation*}
        {}^{\mathbf J}W_{\mathbb S_a^I} = \left\{\prod_{i\in\overline I\setminus J}u_{i,a_i} \:\middle|\; 0\leq a_i\leq\ell_i-1 \text{ for every } i\in\overline I\setminus J \right\}.
    \end{equation*}
    The expression on the right-hand side is unique, and the order of the factors $u_{i,a_i}$ is irrelevant as they all commute.
\end{corol}
\begin{proof}
    This follows directly from Lemma \ref{Lem:DescriptionMathbfJ}, as we have ${}^{\mathbf J\cap C_i}W_{C_i} = \{\mathrm{id}\}$ if $i \in J$ and 
    \begin{equation*}
        {}^{\mathbf J \cap C_i}W_{C_i} = {}^{C_i^{\circ}}W_{C_i} = \{u_{i,a} \mid 0 \leq a \leq \ell_i-1\},
    \end{equation*}
    if $i \in \overline I \setminus J$.
\end{proof}

\subsection{The EKOR strata}\label{Section4.3}

We now construct the EKOR stratification on the reduced special fiber
$\mathcal N_{\mathrm{red}}^{I,\varepsilon}$. We follow the local construction of Sections 3.3 and 3.4 of \cite{shenEKORStrataShimura2021}. We fix some $J\in\mathcal P_{0,I}$ and consider the associated non-empty Kottwitz-Rapoport stratum $\mathrm{KR}_J$. Recall the local model diagram of Figure \ref{Figure1}. We define
\begin{equation*}
\widetilde{\mathrm{KR}}_J \coloneq  \widetilde{\psi}^{-1}(\mathcal O_J) =  \widetilde{\mathcal N}_{\mathrm{red}}^{I,\varepsilon} \times_{M_{\mathrm{loc},\overline{\mathbb F_q}}^{I,\varepsilon}} \mathcal{O}_{J,\overline{\mathbb F_q}}.
\end{equation*}
By construction, $\widetilde{\mathrm{KR}}_J$ is a $\mathcal G_{I,\overline{\mathbb F_q}}$-torsor over $\mathrm{KR}_J$ via the morphism $\widetilde{\varphi}$. From now on, we regard it as a right torsor by inverting the original left action. Reducing the action to the maximal reductive quotient, we obtain a right $\mathbf G_{I,\overline{\mathbb F_q}}$-torsor over $\mathrm{KR}_J$:
\begin{equation*}
\mathcal I_J \coloneq  \widetilde{\mathrm{KR}}_J \times^{\mathcal G_{I,\overline{\mathbb F_q}}} \mathbf G_{I,\overline{\mathbb F_q}}.
\end{equation*}
Recall the distinguished point $x_J\in \mathcal O_J(\mathbb F_{q^2})$ defined in Lemma \ref{Lem:xJIsAPointOfOJ}. We set $\widetilde{\mathrm{KR}}_J^{+} \coloneq  \widetilde{\psi}^{-1}(x_J) \subset \widetilde{\mathrm{KR}}_J$, which is a right torsor over $\mathrm{KR}_J$ with respect to the action of the stabilizer $\mathcal G_{I,J,\overline{\mathbb F_q}}$ of $x_J$. Therefore $\widetilde{\mathrm{KR}}_J^{+}$ induces a $\mathbf P_{\mathbf J,\overline{\mathbb F_q}}$-torsor over $\mathrm{KR}_J$:
\begin{equation*}
\mathcal I_J^{+} \coloneq  \widetilde{\mathrm{KR}}_J^{+} \times^{\mathcal G_{I,J,\overline{\mathbb F_q}}} \mathbf P_{\mathbf J,\overline{\mathbb F_q}}.
\end{equation*}
In order to define the second parabolic reduction $\mathcal I_J^{-}$, we first define another morphism 
\begin{equation*}
    \widetilde{\psi}^c : \widetilde{\mathcal N}_{\mathrm{red}}^{I,\varepsilon} \longrightarrow M_{\mathrm{loc},\overline{\mathbb F_q}}^{I,\varepsilon}.
\end{equation*} 
Let $S$ be an $\overline{\mathbb F_q}$-scheme, and let $\widetilde{x} = (X_r,i_{X_r},\lambda_{X_r},\rho_{X_r};\gamma_r)_{r \in I}$ be a point of $\widetilde{\mathcal N}_{\mathrm{red}}^{I,\varepsilon}(S)$. Aside from the usual Hodge filtration, the Frobenius twist of the evaluation at $S$ of the Dieudonné crystal $\mathcal D_r \coloneq  \mathbb D(X_r)_S$ is also equipped with the conjugate filtration defined as follows. Since $S$ is an $\overline{\mathbb F_q}$-scheme, the $\mathcal O_S$-module $\mathcal D_r$ is equipped with Frobenius and Verschiebung morphisms
\begin{equation*}
F_r:\mathcal D_r^{(q)} \longrightarrow \mathcal D_r, \qquad V_r:\mathcal D_r\longrightarrow \mathcal D_r^{(q)},
\end{equation*}
which satisfy $V_rF_r =0$ and $F_rV_r = 0$. The Frobenius twist of the usual Hodge filtration is given by $(\mathrm{Fil}^1 \mathcal D_r)^{(q)} = \mathrm{Im}(V_r)$. We define the conjugate filtration by 
\begin{equation*}
\mathrm{Fil}^{1,c} \mathcal D_r \coloneq  \mathrm{Im}(F_r) \subset \mathcal D_r. 
\end{equation*}
Applying the trivialization, for all $r \in I$ we define 
\begin{equation*}
\mathfrak t_{r}^{c}\coloneq \gamma_r(\mathrm{Fil}^{1,c}\mathcal D_r) \subset L_{r}\otimes_{\mathcal O_F}\mathcal O_S.
\end{equation*}
The collection $(\mathfrak t_{r}^{c})_{r \in I}$ extends uniquely by periodicity and self-duality to a point $(\mathfrak t_r^{c})_{r\in I_{\mathbb Z}}$ of $M_{\mathrm{loc},\overline{\mathbb F_q}}^{I,\varepsilon}(S)$. In particular, the Cartier isomorphisms 
\begin{equation}\label{Eq:CartierIsomorphism}
    \overline{F_r}: (\mathcal D_r/\mathrm{Fil}^1\mathcal D_r)^{(q)} \xrightarrow{\sim} \mathrm{Fil}^{1,c}\mathcal D_r, \qquad \overline{V_r}: \mathcal D_r/\mathrm{Fil}^{1,c}\mathcal D_r \xrightarrow{\sim} (\mathrm{Fil}^1\mathcal D_r)^{(q)}.
\end{equation}
guarantee that the $\mathfrak t_{r}^c$ satisfy the correct signature condition. Mapping $\widetilde x$ to $(\mathfrak t_r^{c})_{r\in I_{\mathbb Z}}$ defines the desired map $\widetilde{\psi}^c$. 
\begin{lem}
We have $\widetilde{\psi}^{c}(\widetilde{\mathrm{KR}}_J)=\mathcal O_{\theta_I(J),\overline{\mathbb F_q}}$.
\end{lem}

\begin{proof}
Let $k$ be an algebraically closed field containing $\overline{\mathbb F_q}$, and let $\widetilde x = (X_r,i_{X_r},\lambda_{X_r},\rho_{X_r};\gamma_r)_{r \in I}$ be a point of $\widetilde{\mathrm{KR}}_J(k)$. We write $\widetilde{\psi}(\widetilde x) = (\mathfrak t_r)_{r \in I_{\mathbb Z}}$ and $\widetilde{\psi}^{c}(\widetilde x) = (\mathfrak t_r^c)_{r \in I_{\mathbb Z}}$. For $r \in I_{\mathbb Z}$, we consider
\begin{equation*}
    \mathcal Q_r \coloneq  \frac{\Lambda_{r,k}}{\eta_0\mathfrak t_r}, \qquad \mathcal Q_r^c \coloneq  \frac{\Lambda_{r,k}}{\eta_0\mathfrak t_r^c}.
\end{equation*}
We write $\delta_r: \mathcal Q_{r+\ell_{\overline r}} \to \mathcal Q_{r}$ and $\delta_r^c : \mathcal Q_{r+\ell_{\overline r}}^c \to \mathcal Q_{r}^c$ for the transition maps. The Cartier isomorphisms \eqref{Eq:CartierIsomorphism} and the self-duality of $(\mathfrak t_r)_{r\in I_{\mathbb Z}}$ induce
\begin{equation*}
    \mathcal Q_{r+\ell_{\overline r}}^c \xrightarrow{\sim} \eta_0(\mathfrak t_{r+\ell_{\overline r}})^{(q)} \simeq (\eta_1\mathfrak t_{r+\ell_{\overline r}})^{(q)} \simeq  \mathcal Q_{-r-\ell_{\overline r}-\underline{\varepsilon}}^{\vee,(q)}.
\end{equation*}
These identifications form a diagram, commutative up to multiplication by a unit scalar.
\begin{center}
\begin{tikzcd}
    \mathcal Q_{r+\ell_{\overline r}}^c \arrow[rr,"\delta_r^c"] \arrow[d,swap,"\rotatebox{-90}{\(\sim\)}"] & {} & \mathcal Q_{r}^c \arrow[d,"\rotatebox{90}{\(\sim\)}"] \\
    \mathcal Q_{-r-\ell_{\overline r}-\underline{\varepsilon}}^{\vee,(q)} \arrow[rr,"\delta_{-r-\underline{\varepsilon}-\ell_{\overline r}}^{\vee,(q)}"] & {} & \mathcal Q_{-r-\underline{\varepsilon}}^{\vee,(q)}
\end{tikzcd}
\end{center}
We note that the bottom map is well-defined since $\ell_{\theta_I(\overline r)} = \ell_{\overline r}$. In particular, we have
\begin{equation*}
    \delta_r^c \not = 0 \iff \delta_{-r-\underline{\varepsilon}-\ell_{\overline r}} \not = 0 \iff \theta_I(\overline r) \in J \iff \overline r \in \theta_I(J).
\end{equation*}
This proves that $\widetilde{\psi}^{c}(\widetilde{\mathrm{KR}}_J) \subseteq \mathcal O_{\theta_I(J),\overline{\mathbb F_q}}$. Since the morphism $\widetilde{\psi}^{c}$ is clearly $\mathcal G_{I,\overline{\mathbb F_q}}$-equivariant, the equality follows from Proposition \ref{Prop:TheOJAreOrbits}.
\end{proof}
We define $\widetilde{\mathrm{KR}}_J^{-} \coloneq  (\widetilde{\psi}^{c})^{-1}(x_{\theta_I(J)}) \subset \widetilde{\mathrm{KR}}_J$, which is a right torsor over
$\mathrm{KR}_J$ with respect to the action of the stabilizer $\mathcal G_{I,\theta_I(J),\overline{\mathbb F_q}}$ of $x_{\theta_I(J)}$. Since we have 
\begin{equation*}
    \mathbb S_a^I \cap \mathrm{Ad}(v_{\theta_I(J)}^{-1})(\mathbb S_a^I) = \sigma_J(\mathbf J),
\end{equation*}
we deduce that $\widetilde{\mathrm{KR}}_J^{-}$ induces a $\mathbf P_{\sigma_J(\mathbf J),\overline{\mathbb F_q}}$-torsor over $\mathrm{KR}_J$:
\begin{equation*}
    \mathcal I_J^{-} \coloneq  \widetilde{\mathrm{KR}}_J^{-} \times^{\mathcal G_{I,\theta_I(J),\overline{\mathbb F_q}}} \mathbf P_{\sigma_J(\mathbf J),\overline{\mathbb F_q}}.
\end{equation*}
It remains to define the isomorphism between $\mathcal I_J^{+}/R_u(\mathbf P_{\mathbf J,\overline{\mathbb F_q}})$ and $\mathcal I_J^{-}/R_u(\mathbf P_{\sigma_J(\mathbf J),\overline{\mathbb F_q}})$. These two quotients classify the trivializations of the graded objects associated to the Hodge and to the conjugate filtrations respectively.
The Cartier isomorphisms yield an isomorphism $(\mathcal I_J^{+}/R_u(\mathbf P_{\mathbf J,\overline{\mathbb F_q}}))^{(q)} \xrightarrow{\sim} \mathcal I_J^{-}/R_u(\mathbf P_{\sigma_J(\mathbf J),\overline{\mathbb F_q}})$ of $\mathbf L_{\mathbf J}^{(q)} \simeq \mathbf L_{\sigma_J(\mathbf J)}$-torsors. Composing with $\sigma$, we obtain a $\sigma_J$-equivariant isomorphism 
\begin{equation*}
    \iota_J: \mathcal I_J^{+}/R_u(\mathbf P_{\mathbf J,\overline{\mathbb F_q}}) \xrightarrow{\sim} \mathcal I_J^{-}/R_u(\mathbf P_{\sigma_J(\mathbf J),\overline{\mathbb F_q}}).
\end{equation*}
We have therefore constructed a $\mathbf G_I$-zip $(\mathcal I_J,\mathcal I_J^{+},\mathcal I_J^{-},\iota_J)$ of type $\mathbf P_{\mathbf J}$ on $\mathrm{KR}_J$. This gives rise to a morphism of algebraic stacks
\begin{equation*}
\zeta_J: \mathrm{KR}_J \longrightarrow \mathbf G_{I,\overline{\mathbb F_q}}\text{-}\mathrm{Zip}^{J} \simeq [\mathbf E_{J}\backslash \mathbf G_{I,\overline{\mathbb F_q}}].
\end{equation*}
For $u\in{}^{\mathbf J}W_{\mathbb S_a^I}$, let $\mathbf G_I^u\subset \mathbf G_{I,\overline{\mathbb F_q}}$ be the corresponding $\mathbf E_{J}$-orbit, as in \eqref{Eq:DefiOrbitZipGroup}. We define
\begin{equation*}
    \mathrm{EKOR}_{J,u} \coloneq  \left(\mathrm{KR}_J \times_{\mathbf G_{I,\overline{\mathbb F_q}}\text{-}\mathrm{Zip}^{J}} [\mathbf E_{J}\backslash \mathbf G_{I,\overline{\mathbb F_q}}^u]\right)_{\mathrm{red}}.
\end{equation*}
If we let $u$ vary, we obtain a decomposition
\begin{equation*}
    \mathrm{KR}_J = \bigsqcup_{u\in{}^{\mathbf J}W_{\mathbb S_a^I}} \mathrm{EKOR}_{J,u},
\end{equation*}
into a disjoint union of (potentially empty) locally closed subschemes. Finally, if we let $J\in\mathcal P_{0,I}$ vary, we obtain the EKOR stratification 
\begin{equation*}
    \mathcal N_{\mathrm{red}}^{I,\varepsilon} = \bigsqcup_{J\in\mathcal P_{0,I}} \bigsqcup_{u\in{}^{\mathbf J}W_{\mathbb S_a^I}} \mathrm{EKOR}_{J,u}.
\end{equation*}
The indexing set of the EKOR stratification can be rewritten as follows. We define
\begin{equation*}
    {}^{I}\mathrm{Adm}(\{\mu'\}) \coloneq {}^{\mathbb S_a^I}\widetilde W \cap \mathrm{Adm}(\{\mu'\}) = {}^{\mathbb S_a^I}\widetilde W \cap \mathrm{Adm}(\{\mu'\})^I,
\end{equation*}
the last equality being due to Theorem 6.1 of \cite{heKottwitzRapoportConjectureUnions2016}. We also consider 
\begin{equation*}
    {}^{I}\mathrm{Adm}(\{\mu'\})_0 \coloneq  {}^{\mathbb S_a^I}\widetilde W \cap \mathrm{Adm}(\{\mu'\})_0 \subseteq {}^{\mathbb S_a^I}\widetilde W \cap \mathrm{Adm}(\{\mu'\})_0^I \subseteq {}^{I}\mathrm{Adm}(\{\mu'\}).
\end{equation*}
We stress that the first inclusion above can be strict in general. For $J\in\mathcal P_{0,I}$, since $v_J$ is the minimal length element in the double coset $W_{\mathbb S_a^I}v_JW_{\mathbb S_a^I}$, and since $\mathbf J = \mathbb S_a^I\cap \mathrm{Ad}(v_J^{-1})(\mathbb S_a^I)$, we have a bijection
\begin{equation}\label{Eq:BijectionKRPosetEKORPoset}
    {}^{\mathbf J}W_{\mathbb S_a^I} \xrightarrow{\sim} {}^{\mathbb S_a^I}\widetilde W \cap W_{\mathbb S_a^I}v_JW_{\mathbb S_a^I}, \qquad u\longmapsto v_Ju.
\end{equation}
Taking union over all $J \in \mathcal P_{0,I}$, we get a bijection
\begin{equation}\label{Eq:ImageOfEKORIndicesInBasicKR}
    \bigsqcup_{J\in\mathcal P_{0,I}}{}^{\mathbf J}W_{\mathbb S_a^I} \xrightarrow{\sim} {}^{\mathbb S_a^I}\widetilde W \cap \mathrm{Adm}(\{\mu'\})_0^I, \qquad (J,u)\longmapsto v_Ju.
\end{equation}
If $v \in {}^{\mathbb S_a^I}\widetilde W \cap \mathrm{Adm}(\{\mu'\})_0^I$ satisfies $v = v_Ju$ for some $J \in \mathcal P_{0,I}$ and $u \in {}^{\mathbf J}W_{\mathbb S_a^I}$, we will also write
\begin{equation*}
    \mathrm{EKOR}_v\coloneq \mathrm{EKOR}_{J,u}.
\end{equation*}
\begin{rk}
    In Corollary \ref{Corol:NonEmptyEKORStrata}, we prove that $\mathrm{EKOR}_v \not = \emptyset$ if and only if $v \in {}^{I}\mathrm{Adm}(\{\mu'\})_0$. This fact can also be recovered from \cite{gortzFullyHodgeNewton2019} Proposition 5.6. We point out that the emptiness of $\mathrm{EKOR}_v$ for $v \not \in {}^{I}\mathrm{Adm}(\{\mu'\})_0$ does not contradict \cite{shenEKORStrataShimura2021} Corollary 3.5.3. Indeed, it simply means that in the corresponding Shimura variety, the EKOR stratum indexed by $v$ lies in a non-basic Newton stratum, while the KR stratum containing it intersects the basic locus non-trivially. In fact, since the datum $(G,\{\mu\})$ is fully Hodge-Newton decomposable as defined in \cite{gortzFullyHodgeNewton2019}, the EKOR stratification is finer than the Newton stratification, but in general the same can not be said of the KR stratification. 
\end{rk}
We can make the sets ${}^{I}\mathrm{Adm}(\{\mu'\})$, ${}^{\mathbb S_a^I}\widetilde W \cap \mathrm{Adm}(\{\mu'\})_0^I$ and ${}^{I}\mathrm{Adm}(\{\mu'\})_0$ more explicit. 

\begin{prop}\label{Prop:DescriptionIndexSetOfEKORStrata}
    There is a bijection
    \begin{equation*}
        {}^{I}\mathrm{Adm}(\{\mu'\}) \xrightarrow{\sim} \left\{\text{tuples } \mathbf c = (c_i)_{i\in \overline I} \in \prod_{i \in \overline I}\{0,\ldots ,\ell_i\} \;\middle|\; \exists i \in \overline I, \quad c_i \leq \ell_i -1 \right\},
    \end{equation*}
    whose inverse is denoted $\mathbf c \mapsto v_{\mathbf c}$, inducing the following identifications 
    \begin{align*}
        {}^{\mathbb S_a^I}\widetilde W \cap \mathrm{Adm}(\{\mu'\})_0^I & \simeq \left\{v_{\mathbf c} \in {}^{I}\mathrm{Adm}(\{\mu'\}) \mid \exists i \in \overline I, \quad c_i,c_{\theta_I(i)} \leq \ell_i-1\right\},\\
        {}^{I}\mathrm{Adm}(\{\mu'\})_0 & \simeq \left\{v_{\mathbf c} \in {}^{I}\mathrm{Adm}(\{\mu'\}) \mid \exists i \in \overline I, \quad c_i + c_{\theta_I(i)} \leq \ell_i-1\right\}.
    \end{align*}
\end{prop}
\begin{proof}
    Recall from Remark \ref{Rk:DescriptionAdmMu} that any $v = w\tau_0 \in \mathrm{Adm}(\{\mu'\})$ is determined by its support $\mathrm{Supp}(w)$, which is a proper subset of $\mathbb S_a$. The element $v$ is left $\mathbb S_a^I$-reduced if and only if for every $s_j \in \mathrm{Supp}(w)$, we have $s_j \in \mathbb S_a^I \implies s_{j-1} \in \mathrm{Supp}(w)$. Thus, the support can be written in the form 
    \begin{equation*}
        \mathrm{Supp}(w) = \left\{s_j \;\middle|\; j \in \bigsqcup_{i \in \overline I} \{i, i+1,  \ldots , i+c_i-1\} \right\},
    \end{equation*}
    for a uniquely determined tuple $\mathbf c = (c_i)_{i \in \overline I}$ where $0 \leq c_i \leq \ell_i$ for all $i$. Besides, the fact that $\mathrm{Supp}(w) \not = \mathbb S_a$ is equivalent to $c_i \leq \ell_i-1$ for some $i$. The map $v \mapsto \mathbf c$ is the desired bijection. Now $v$ belongs to ${}^{\mathbb S_a^I}\widetilde W \cap \mathrm{Adm}(\{\mu'\})_0^I$ if and only if it can be written as $v = v_Ju_{\mathbf a}$ for some $J \in \mathcal P_{0,I}$ and $u_{\mathbf a} \in {}^{\mathbf J}W_{\mathbb S_a^I}$. We have $v_Ju_{\mathbf a} = w_J (\tau_0 u_{\mathbf a} \tau_0^{-1})\tau_0$ and 
    \begin{equation*}
        \mathrm{Supp}(w_J\tau_0 u_{\mathbf a} \tau_0^{-1}) = \left\{s_j \;\middle|\; j \in \left( \bigsqcup_{i \in J} B_i \right) \sqcup \left(\bigsqcup_{i \in \overline I \setminus J} \{i, \ldots, i+a_i-1\}\right) \right\}.
    \end{equation*}
    Thus, the corresponding tuple $\mathbf c$ is given by $c_i = \ell_i$ for $i \in J$ and $c_i = a_i$ for $i \in \overline I \setminus J$. If $O = \{i,\theta_I(i)\}$ is a $\theta_I$-orbit contained in $\overline I \setminus J$, we have $c_i,c_{\theta_I(i)} \leq \ell_i-1$, proving the first claim. Lastly, $v_Ju_{\mathbf a}$ belongs to ${}^{I}\mathrm{Adm}(\{\mu'\})_0$ if and only if $\mathrm{Supp}_{\sigma}(w_J\tau_0 u_{\mathbf a} \tau_0^{-1}) \not = \mathbb S_a$. If a simple reflection $s_j$ belongs to some block $B_i$, then $\mathrm{Ad}(\tau_0)\circ \sigma(s_j) = s_{-j-\underline{\varepsilon}-1}$ belongs to $B_{\theta_I(i)}$. Thus, $s_j$ and $\mathrm{Ad}(\tau_0)\circ \sigma(s_j)$ both lie out of $\mathrm{Supp}(w_J\tau_0 u_{\mathbf a} \tau_0^{-1})$ if and only if $s_j \in B_i$ for some $i$ such that $\{i, \theta_I(i)\} \subseteq \overline I \setminus J$, and we have $j = i + d$ for some $0 \leq d \leq \ell_i-1$ satisfying both $d \geq a_i$ and $d \leq \ell_i-1-a_{\theta_I(i)}$. In other words, $v_Ju_{\mathbf a}$ belongs to ${}^{I}\mathrm{Adm}(\{\mu'\})_0$ if and only if the corresponding tuple $\mathbf c$ satisfies $c_i + c_{\theta_I(i)} \leq \ell_i-1$ for some $i$.
\end{proof}
Given an element $v_{\mathbf c} \in {}^{I}\mathrm{Adm}(\{\mu'\})_0$, we define the set 
\begin{equation}\label{Eq:KRIndexAttachedToEKORIndex}
    J_{\mathbf c} \coloneq \{ i \in \overline I \mid c_i = \ell_i \} \in \mathcal P_{0,I}.
\end{equation}
It follows from the proof that $\mathrm{EKOR}_{v_{\mathbf c}} \subseteq \mathrm{KR}_{J_{\mathbf c}}$ (this holds also in the corresponding Shimura variety for $v_{\mathbf c} \in {}^{\mathbb S_a^I}\widetilde W \cap \mathrm{Adm}(\{\mu'\})_0^I$).

\section{The Bruhat-Tits stratification}\label{Section5}
\subsection{The BT strata}

In this section, we recall the main results of \cite{mullerBruhatTitsStratificationBasic2026}. Let $k$ be a perfect field containing $\mathbb F_{q^2}$ and let $Z$ be a $W_{\mathcal O_F}(k)$-lattice in $\mathbb V_{F_k,0}$. For $c \geq 1$ we write 
\begin{equation*}
    T_c(Z) \coloneq  Z + \tau(Z) + \cdots + \tau^{c-1}(Z) \subset \mathbb V_{F_k,0}.
\end{equation*}
By \cite{rapoportPeriodSpacesPdivisible1996} Proposition 2.17, there is an integer $1 \leq c \leq n$ such that $T_c(Z)$ is $\tau$-stable, that is $\tau(T_c(Z)) = T_c(Z)$, or equivalently $T_c(Z) = T_{c+1}(Z)$. In this case, there exists a unique lattice $\Lambda_Z \subset \mathbb V_{K,0}$ such that 
\begin{equation*}
    T_c(Z) = \Lambda_{Z,\mathcal O_{F_k}} = \Lambda_Z \otimes_{\mathcal O_K} \mathcal O_{F_k}.
\end{equation*}
Indeed, $\Lambda_Z$ is simply the $\mathcal O_{K}$-lattice formed by the vectors of $T_c(Z)$ which are fixed by $\tau$. If $k$ contains $\overline{\mathbb F_q}$ and $(Z_r)_{r\in I_{\mathbb Z}} \in \mathcal N_{j}^{I,\varepsilon}(k)$ as in Proposition \ref{Prop:PointsArbitraryParahoricBis}, we will write $\Lambda_{Z_r} \subset \mathbb V_{K,0}$ for the lattice associated by the process described above.
\begin{defi}
    Let $j \in \mathbb Z$. The set of \textit{vertex lattices of rank $j$} is the set 
    \begin{equation*}
        \mathcal L_j \coloneq  \{\mathcal O_K\text{-lattices } \Lambda \subset \mathbb V_{K,0} \mid \pi^{j+1}\Lambda^{\vee} \subseteq \Lambda \subseteq \pi^j\Lambda^{\vee}\}.
    \end{equation*}
    The \textit{orbit type} of a vertex lattice $\Lambda \in \mathcal L_j$ is the integer $t(\Lambda) \coloneq  [\Lambda:\pi^{j+1}\Lambda^{\vee}]$.
\end{defi}
\begin{rk} 
    We point out that for any $\Lambda \in \mathcal L_j$,
    \begin{equation*}
        t(\Lambda) \equiv nj + (1-\underline{\varepsilon}) \mod 2.
    \end{equation*}
    According to \cite{mullerBruhatTitsStratificationBasic2026} Proposition 2.15, the sets $\mathcal L_{j}$ are disjoint for varying $j$, except when $\varepsilon = -1$ in which case we have 
    \begin{equation*}
        \mathcal L_j \cap \mathcal L_{j+1} = \{\Lambda \subset \mathbb V_{K,0} \mid \Lambda = \pi^{j+1}\Lambda^{\vee}\}.
    \end{equation*}
    Because of this, the orbit type of $\Lambda$ is not entirely well-defined. However, the context should avoid any confusion throughout the exposition.
\end{rk}
We introduce some notation that will be useful for the BT stratification. We define
\begin{equation*}
    I^+ \coloneq \{r \in I \mid h_{r}^{\varepsilon} < n\} \subseteq \mathbb Z_{\geq 0}, \qquad I^- \coloneq \{-r-\underline{\varepsilon} \mid r \in I, h_r^{\varepsilon} > 0\} \subseteq \mathbb Z_{<0}.
\end{equation*}
Reduction modulo $n$ induces a bijection $I^+ \sqcup I^- \xrightarrow{\sim} \overline I$. We consider the set $\Omega_{\overline I} \coloneq \overline I / \langle \theta_I \rangle$ of $\theta_I$-orbits on $\overline I$, and we equip it with a total order by setting 
\begin{equation}\label{Eq:OrderOnOrbits}
    \forall O, O' \in \Omega_{\overline I}, \quad O' \leq O \xLeftrightarrow{\text{def}} \max \{r \in I^+ \sqcup I^- \mid \overline r \in O'\} \leq \max \{r \in I^+ \sqcup I^- \mid \overline r \in O\}.
\end{equation}
For any $O \in \Omega_{\overline I}$ with $\# O = 2$, there is a unique pair $(r_O^+,r_O^-) \in I^+ \times I^-$ such that $O = \{\overline{r_O^+}, \overline{r_O^-}\}$. If $\# O = 1$, there is a unique $r_O \in I^+ \sqcup I^-$ such that $O = \{\overline{r_O}\}$. More generally, if $\mathcal T \subseteq \overline I$ is any $\theta_I$-stable subset, we define $\Omega_{\mathcal T} \coloneq \mathcal T/\langle \theta_I\rangle \subseteq \Omega_{\overline I}$ and we equip it with the restriction of the total order $\leq$.
\begin{defi}
    Let $k$ be a perfect field containing $\overline{\mathbb F_q}$, let $j$ be such that $nj$ is even, and let $z = (Z_r)_{r \in I_{\mathbb Z}} \in \mathcal N_{j}^{I,\varepsilon}(k)$. The \textit{Bruhat-Tits type} of $z$ is the subset $\mathrm{Type}(z) \subseteq \overline I$ defined by
    \begin{equation*}
        \mathrm{Type}(z) \coloneq  \left\{\overline r \in \overline I \; \middle| \; \Lambda_{Z_{-r -\underline{\varepsilon}}} \subseteq \pi^{j+1}\Lambda_{Z_{r + \ell_{\overline r}}}^{\vee}\right\}.
    \end{equation*}
\end{defi}
By $\pi$-periodicity of $(Z_r)_{r \in I_{\mathbb Z}}$, the condition defining $\mathrm{Type}(z)$ does not depend on the choice of the representative for $\overline r$. The set $\mathrm{Type}(z)$ is stable under $\theta_I$, since if $\overline r \in \mathrm{Type}(z)$ then we have 
\begin{equation*}
    \Lambda_{Z_{r+\ell_{\overline r}}} \subseteq \pi^{j+1}\Lambda_{Z_{-r-\underline{\varepsilon}}}^{\vee},
\end{equation*}
simply by taking duals. This is equivalent to $\theta_I(\overline r) \in \mathrm{Type}(z)$. Thus, $\mathrm{Type}(z)$ is a union of orbits of $\theta_I$. 

\begin{rk}
    We explain the relation between $\mathrm{Type}(z)$ and the original definition of BT type given in \cite{mullerBruhatTitsStratificationBasic2026} Definition 2.29. Let us write temporarily $I = \{r_1 < \cdots < r_s \}$. In the notation of \cite{mullerBruhatTitsStratificationBasic2026}, for $1 \leq t \leq s$ we have
    \begin{equation*}
        A_t=Z_{r_t}, \qquad B_t = Z_{-r_t-\underline{\varepsilon}}.
    \end{equation*}
    For $1\leq t\leq s-1$, since $\ell_{\overline{r_t}}=r_{t+1}-r_t$, we have
    \begin{equation*}
        \overline{r_t}\in\mathrm{Type}(z) \iff \Lambda_{B_t} \subseteq \pi^{j+1}\Lambda_{A_{t+1}}^\vee.
    \end{equation*}
    Besides, we have the additional conditions
    \begin{align*}
        \text{if } h_{r_1}^{\varepsilon} \not = 0, & & \overline{-r_1-\underline{\varepsilon}} \in \mathrm{Type}(z) & \iff \Lambda_{A_1} \in \mathcal L_{j+1}, \\
        \text{if } h_{r_s}^{\varepsilon} \not = n, & & \overline{r_s} \in \mathrm{Type}(z) & \iff \Lambda_{B_s}\in\mathcal L_j.
    \end{align*}
    For $1 \leq t \leq s-1$, let us write $O_{t} \coloneq \{\overline{r_t}, \overline{-r_{t+1}-\underline{\varepsilon}}\} \subseteq \overline I$. We also define $O_0 \coloneq \{\overline{-r_1-\underline{\varepsilon}}\}$ if $h_{r_1}^{\varepsilon} \not = 0$, and $O_s \coloneq \{\overline{r_s}\}$ if $h_{r_s}^{\varepsilon} \not = n$. Then $\Omega_{\overline I} = \{O_0, O_1, \ldots , O_s\}$ and we have $O_0 < O_1 < \ldots < O_s$ (with $O_0$ and $O_s$ omitted if they are not defined). The BT type of $z$ defined in \cite{mullerBruhatTitsStratificationBasic2026} is exactly the subset $X \subseteq \{0, \ldots , s\}$ such that 
    \begin{equation*}
        \mathrm{Type}(z) = \bigsqcup_{t \in X} O_t.
    \end{equation*}
\end{rk}
In \cite{mullerBruhatTitsStratificationBasic2026}, the so-called ``crucial Lemma'' (Lemma 2.30), which originates from \cite{vollaardSupersingularLocusShimura2010} and \cite{choBasicLocusUnitary2018}, is the following statement. 
\begin{lem}\label{Lem:CrucialLemma}
    Let $k$ be a perfect field containing $\overline{\mathbb F_q}$ and let $j\in\mathbb Z$ be such that $nj$ is even. For any $z \in \mathcal N^{I,\varepsilon}_j(k)$ we have 
    \begin{equation*}
        \mathrm{Type}(z) \not = \emptyset.
    \end{equation*}
\end{lem}
Besides, Proposition 2.28 of \cite{mullerBruhatTitsStratificationBasic2026} also shows that, if $z \in \mathcal N^{I,\varepsilon}_j(k)$,
\begin{equation*}
    \forall r \in I^+ \sqcup I^-, \qquad \overline r \in \mathrm{Type}(z) \implies \Lambda_{Z_{-r-\underline{\varepsilon}}} \in \begin{cases}
        \mathcal L_{j} & \text{if } r \in I^+,\\
        \mathcal L_{j+1} & \text{if } r \in I^-,
    \end{cases}
\end{equation*}
but the converse does not hold. 
\begin{defi}\label{Defi:BTIndexInducedByAPoint}
    For $z \in \mathcal N^{I,\varepsilon}_j(k)$, we define a collection of vertex lattices $\mathbf{\Lambda}^z := (\Lambda_i^z)_{i \in \mathrm{Type}(z)}$ of rank $j$ as follows: 
    \begin{equation*}
        \Lambda_i^z \coloneq \begin{cases}
            \Lambda_{Z_{-r-\underline{\varepsilon}}} & \text{if } i = \overline r \text{ for some } r \in I^+,\\
            \pi^{j+1}\Lambda_{Z_{-r-\underline{\varepsilon}}}^{\vee} & \text{if } i = \overline r \text{ for some } r \in I^-.
        \end{cases}
    \end{equation*} 
\end{defi}
Then $\Lambda_i^z \in \mathcal L_j$ for all $i \in \mathrm{Type}(z)$.

\begin{defi}\label{Defi:NumMathcalT}
    Let $\mathcal T \subseteq \overline I$ be a non-empty $\theta_I$-stable subset. We define 
    \begin{equation*}
        \mathrm{Num}(\mathcal T) := \left\{(m_i)_{i \in \mathcal T} \in \mathbb Z_{\geq 0}^{\mathcal T} \mid \forall i \in \mathcal T, \quad m_i + m_{\theta_I(i)} \leq \ell_i-1\right\}.
    \end{equation*}
\end{defi}

\begin{defi}\label{Defi:BTIndex}
    Let $j\in\mathbb Z$ be such that $nj$ is even. A \textit{Bruhat-Tits index of rank $j$} is a pair $(\mathcal T,\mathbf{\Lambda})$, where $\mathcal T \subseteq \overline I$ is non-empty and $\theta_I$-stable, and 
    \begin{equation*}
        \mathbf{\Lambda} = (\Lambda_i)_{i\in\mathcal T}
    \end{equation*}
    is a collection of $\mathcal O_K$-lattices in $\mathbb V_{K,0}$ for which there exists $\mathbf m = (m_i)_{i\in\mathcal T}\in \mathrm{Num}(\mathcal T)$ such that
    \begin{enumerate}
        \item for every \(r\in I^+\sqcup I^-\) such that $\overline r \in \mathcal T$, we have $\Lambda_{\overline r} \in \mathcal L_j$ and
        \begin{equation*}
            t(\Lambda_{\overline r}) = \left|2(r+m_{\overline r}) +\underline{\varepsilon}+1\right|,
        \end{equation*}
        \item $\mathbf{\Lambda}$ forms a chain under inclusion.
    \end{enumerate}
\end{defi}
If $(\mathcal T, \mathbf{\Lambda})$ is a BT index, then the corresponding element $\mathbf m \in \mathrm{Num}(\mathcal T)$ is uniquely determined. 

\begin{rk}
    The set of vertices of the Bruhat-Tits building of $J^{\mathrm{der}}$ can be identified with $\mathcal L_j$. Then, a Bruhat-Tits index is nothing but a facet of the building, with certain numerical constraints, encoded by an element of $\mathrm{Num}(\mathcal T)$, on the types of its vertices.
\end{rk}

\begin{lem}\label{Lem:BTIndexAssociatedToPoint}
    Let $k$ be a perfect field containing $\overline{\mathbb F_q}$, let $j\in\mathbb Z$ be such that $nj$ is even, and let $z=(Z_r)_{r\in I_{\mathbb Z}} \in \mathcal N_j^{I,\varepsilon}(k)$. Then $(\mathrm{Type}(z),\mathbf{\Lambda}^z)$ is a Bruhat-Tits index of rank $j$.
\end{lem}

\begin{proof}
    By Lemma \ref{Lem:CrucialLemma} the set $\mathrm{Type}(z)$ is non-empty, and we have already seen that it is stable under $\theta_I$. Moreover, we also know that $\Lambda_i^z\in \mathcal L_j$ for every $i\in \mathrm{Type}(z)$. Let $i \in \mathrm{Type}(z)$. If $i = \overline r$ for some $r \in I^+$ then we have 
    \begin{equation*}
        \pi^{j+1}(\Lambda_i^z)^{\vee}_k \subseteq \pi^{j+1} Z_{-r-\underline{\varepsilon}}^\vee \overset{1}{\subset} Z_{r} \overset{2r+\underline{\varepsilon}}{\subseteq} Z_{-r-\underline{\varepsilon}} \subseteq (\Lambda_i^z)_k,
    \end{equation*}
    from which it follows that $t(\Lambda_i^z) \geq 2r+\underline{\varepsilon}+1$. On the other hand, if $i = \overline r$ for some $r \in I^-$ then we have
    \begin{equation*}
        (\Lambda_i^z)_k \subseteq \pi^{j+1}Z_{-r-\underline{\varepsilon}}^{\vee} \overset{1}{\subset} Z_{r} \overset{n+2r+\underline{\varepsilon}}{\subseteq} \pi^{-1}Z_{-r-\underline{\varepsilon}} \subseteq \pi^j(\Lambda_i^z)_k^{\vee},
    \end{equation*}
    from which it follows that $n-t(\Lambda_i^z) \geq n+2r+\underline{\varepsilon}+1$, so that $t(\Lambda_i^z) \leq -2r - \underline{\varepsilon}-1$. We define 
    \begin{equation*}
        \text{if } r \in I^+, \quad m^z_{\overline r} \coloneq \frac{t(\Lambda_{\overline r}^z) - 2r - \underline{\varepsilon}-1}{2}, \qquad \text{if } r \in I^-, \quad m_{\overline r}^z \coloneq \frac{-t(\Lambda_{\overline r}^z) - 2r - \underline{\varepsilon}-1}{2}.
    \end{equation*}
    By the inequalities above and since $t(\Lambda_i^z) \equiv 1-\underline{\varepsilon}\mod 2$, each $m_i^z$ is a non-negative integer. Let $O = \{i,\theta_I(i)\} \in \Omega_{\mathrm{Type}(z)}$ be an orbit with two elements. Up to exchanging $i$ and $\theta_I(i)$, we can assume that $i = \overline{r_O^+}$. We have  
    \begin{equation*}
        r_{O}^- =  -r_O^+-\underline{\varepsilon}-\ell_i \in I^-,
    \end{equation*}
    and since $i\in\mathrm{Type}(z)$, we have $\Lambda_i^z\subseteq\Lambda_{\theta_I(i)}^z$. Consequently,
    \begin{equation*}
        t(\Lambda_{\theta_I(i)}^z) - t(\Lambda_i^z) = 2(\ell_i-1-m_i^z-m_{\theta_I(i)}^z) \geq 0,
    \end{equation*}
    so that $m_i^z+m_{\theta_I(i)}^z\leq\ell_i-1$. Suppose now that $O = \{i\}$ is a singleton. If $r_O\in I^+$, then $\ell_i=n-2r_O-\underline{\varepsilon}$ and the inequality $t(\Lambda_i^z)\leq n$ gives
    \begin{equation*}
        2m_i^z = t(\Lambda_i^z)-2r_O-\underline{\varepsilon}-1 \leq\ell_i-1.
    \end{equation*}
    If $r_O\in I^-$, then $\ell_i=-2r_O-\underline{\varepsilon}$ and the inequality $t(\Lambda_i^z)\geq 0$ gives
    \begin{equation*}
        2m_i^z = -t(\Lambda_i^z)-2r_O-\underline{\varepsilon}-1 \leq\ell_i-1.
    \end{equation*}
    This proves that $\mathbf m^z := (m_i^z)_{i \in \mathrm{Type}(z)} \in \mathrm{Num}(\mathrm{Type}(z))$. It remains to check that $\mathbf{\Lambda}^z$ forms a chain. For every $r \in I_{\mathbb Z}$, we have 
    \begin{equation*}
        \pi^{j+1}(\Lambda_{Z_{-r-\underline{\varepsilon}}})_k^{\vee} \subseteq \pi^{j+1}Z_{-r-\underline{\varepsilon}}^\vee \subseteq Z_r \subseteq (\Lambda_{Z_r})_k.
    \end{equation*}
    Taking duals and multiplying by $\pi^{j+1}$, we obtain $\pi^{j+1}\Lambda_{Z_r}^{\vee} \subseteq\Lambda_{Z_{-r-\underline{\varepsilon}}}$. Moreover, if $r<r'$ belong to $I$, the inclusions $Z_{r'}\subseteq Z_r$ and $Z_{-r-\underline{\varepsilon}} \subseteq Z_{-r'-\underline{\varepsilon}}$ give
    \begin{equation*}
        \pi^{j+1}\Lambda_{Z_r}^{\vee} \subseteq\pi^{j+1}\Lambda_{Z_{r'}}^{\vee}, \qquad \Lambda_{Z_{-r-\underline{\varepsilon}}} \subseteq\Lambda_{Z_{-r'-\underline{\varepsilon}}}.
    \end{equation*}
    Finally, if $r<r'$ are consecutive in $I$ and $\overline r\in\mathrm{Type}(z)$, then $\Lambda_{Z_{-r-\underline{\varepsilon}}} \subseteq\pi^{j+1}\Lambda_{Z_{r'}}^\vee$. All together, these inclusions prove that if $O',O \in \Omega_{\mathrm{Type}(z)}$ satisfy $O' < O$, then $\Lambda_{i'}^z \subseteq \Lambda_i^z$ for every $i\in O$ and $i' \in O'$. This concludes the proof.
\end{proof}

Bruhat-Tits indices serve as the indexing set of the Bruhat-Tits stratification. We define a partial order on the set of Bruhat-Tits indices of rank $j$ as in \cite{mullerBruhatTitsStratificationBasic2026} Definition 2.35. 
\begin{defi}
    Given two Bruhat-Tits indices $(\mathcal T, \mathbf{\Lambda}), (\mathcal T', \mathbf{\Lambda}')$ of rank $j$, we write 
    \begin{equation*}
        (\mathcal T', \mathbf{\Lambda}') \leq (\mathcal T, \mathbf{\Lambda}),
    \end{equation*}
    if and only if $\mathcal T \subseteq \mathcal T'$ and, for every $r \in I^+ \sqcup I^-$ such that $\overline r\in \mathcal T$, we have 
    \begin{equation*}
        \begin{cases}
            \Lambda_{\overline r}'\subseteq\Lambda_{\overline r}
            &\text{if }r\in I^+,\\
            \Lambda_{\overline r}\subseteq\Lambda_{\overline r}'
            &\text{if }r\in I^-.
        \end{cases}
    \end{equation*}
\end{defi}

The following statement is a summary of \cite{mullerBruhatTitsStratificationBasic2026} Section 4.1. 
\begin{theo}
    Let $j \in \mathbb Z$ be such that $nj$ is even. For every Bruhat-Tits index $(\mathcal T, \mathbf{\Lambda})$ of rank $j$, there exists a locally closed subscheme 
    \begin{equation*}
        \mathcal N_{\mathcal T,\mathbf{\Lambda}}^{0} \hookrightarrow \mathcal N_{j,\mathrm{red}}^{I,\varepsilon},
    \end{equation*}
    such that, for every perfect field $k$ containing $\overline{\mathbb F_q}$, we have
    \begin{equation*}
        \mathcal N_{\mathcal T,\mathbf{\Lambda}}^{0}(k) = \left\{z \in \mathcal N_{j}^{I,\varepsilon}(k) \mid (\mathcal T, \mathbf{\Lambda}) = (\mathrm{Type}(z), \mathbf{\Lambda}^z) \right\} \not = \emptyset.
    \end{equation*}
    Moreover, we have 
    \begin{equation*}
        \mathcal N_{j}^{I,\varepsilon}(k) = \bigsqcup_{\substack{\text{BT index } (\mathcal T, \mathbf{\Lambda})\\ \text{of rank } j}} \mathcal N_{\mathcal T,\mathbf{\Lambda}}^{0}(k),
    \end{equation*}
    and the Zariski closure of $\mathcal N_{\mathcal T,\mathbf{\Lambda}}^{0}$ is projective, smooth and given by 
    \begin{equation*}
        \overline{\mathcal N_{\mathcal T,\mathbf{\Lambda}}^{0}} = \bigsqcup_{(\mathcal T', \mathbf{\Lambda}') \leq (\mathcal T, \mathbf{\Lambda})} \mathcal N_{\mathcal T',\mathbf{\Lambda}'}^{0},
    \end{equation*}
    where $(\mathcal T', \mathbf{\Lambda}')$ runs over all the Bruhat-Tits indices of rank $j$ smaller or equal to $(\mathcal T,\mathbf{\Lambda})$.
\end{theo}

Recall the group $J \simeq \mathrm{GU}(\mathbb V_{K,0}, (\cdot,\cdot))$ of Proposition \ref{Prop:GroupJIsomorphism}, as well as the map $\alpha_J:J(F) \to \mathbb Z$ defined as in Proposition \ref{Prop:ImageOfalpha}. An element $g \in J(F)$ induces an orbit type-preserving bijection 
\begin{equation*}
    g:\mathcal L_j \xrightarrow{\sim} \mathcal L_{j+\alpha_J(g)}, \qquad \Lambda \mapsto g(\Lambda).
\end{equation*}
It induces an action of $J(F)$ on the set of Bruhat-Tits indices via
\begin{equation*}
    g\cdot (\mathcal T, \mathbf{\Lambda}) \coloneq  (\mathcal T, g(\mathbf{\Lambda})),
\end{equation*}
where $g(\mathbf{\Lambda}) \coloneq (g(\Lambda_i))_{i\in \mathcal T}$. If $(\mathcal T, \mathbf{\Lambda})$ has rank $j$, then $g\cdot (\mathcal T, \mathbf{\Lambda})$ has rank $j + \alpha_J(g)$. The following statement is \cite{mullerBruhatTitsStratificationBasic2026} Proposition 4.10.

\begin{prop}\label{Prop:OrbitBruhatTitsIndex}
    Two Bruhat-Tits indices $(\mathcal T, \mathbf{\Lambda})$ and $(\mathcal T', \mathbf{\Lambda}')$ are in the same $J(F)$-orbit if and only if $\mathcal T = \mathcal T'$ and $t(\Lambda_i) = t(\Lambda_i')$ for all $i \in \mathcal T$.
\end{prop}

Since the integers $t(\Lambda_i)$ are completely determined by the corresponding tuple $\mathbf m \in \mathrm{Num}(\mathcal T)$, it follows that $\mathrm{Num}(\mathcal T)$ is a full set of invariants for the action of $J(F)$ on the set of BT indices associated to $\mathcal T$. Recall the standard lattice chain $\Lambda_r = \eta_0L_{r,\mathcal O_K}$ defined in \eqref{Eq:DefiLatticeLambda}. If $0 \leq t \leq n$ has the same parity as $nj + (1-\underline{\varepsilon})$ then
\begin{equation*}
    \Lambda_{\frac{nj+(1-\underline{\varepsilon})-t}{2}} \in \mathcal L_j \text{ and has orbit type } t.
\end{equation*}
Fix a non-empty $\theta_I$-stable subset $\mathcal T$ and a tuple $\mathbf m \in \mathrm{Num}(\mathcal T)$. We define a representative $(\mathcal T, \mathbf{\Lambda}(\mathbf m))$ of the $J(F)$-orbit corresponding to $\mathbf m$ by setting, for all $i \in \mathcal T$, 
\begin{equation*}
    \Lambda(\mathbf m)_{i} \coloneq \Lambda_{\frac{(1-\underline{\varepsilon})-\left|2(r+m_i) + \underline{\varepsilon} + 1\right|}{2}},
\end{equation*}
where $r \in I^+ \sqcup I^-$ is such that $i = \overline r$. Then $(\mathcal T, \mathbf{\Lambda}(\mathbf m))$ is a BT index of rank $0$, and any BT index is in the $J(F)$-orbit of some $(\mathcal T, \mathbf{\Lambda}(\mathbf m))$ for a unique pair $(\mathcal T, \mathbf m)$. According to Proposition 4.8 of \cite{mullerBruhatTitsStratificationBasic2026}, the $J(F)$ action on $\mathcal N_{\mathrm{red}}^{I,\varepsilon}$ is compatible with the Bruhat-Tits stratification, in the sense that any $g \in J(F)$ induces an isomorphism
\begin{equation*}
    g: \mathcal N_{\mathcal T, \mathbf{\Lambda}}^{0} \xrightarrow{\sim} \mathcal N_{\mathcal T, g(\mathbf{\Lambda})}^{0}.
\end{equation*}
From the discussion above, it is enough to focus on the Bruhat-Tits strata associated to Bruhat-Tits indices of the form $(\mathcal T,\mathbf{\Lambda}(\mathbf m))$ for $\mathbf m \in \mathrm{Num}(\mathcal T)$.

\subsection{Stratification by fine Deligne-Lusztig varieties}\label{Section5.2}

One of the main results of \cite{mullerBruhatTitsStratificationBasic2026} is the isomorphism between the Bruhat-Tits strata $\mathcal N_{\mathcal T, \mathbf{\Lambda}}^{0}$ and certain disjoint unions of explicitly determined fine Deligne-Lusztig varieties. In this section, we will assume familiarity with the definitions of (fine) Deligne-Lusztig varieties, see for instance Section 3.2.1 of \cite{mullerBruhatTitsStratificationBasic2026} for a summary and for the notations that we will borrow. Throughout this section, we fix a non-empty $\theta_I$-stable subset $\mathcal T\subseteq\overline I$ and an element $\mathbf m\in\mathrm{Num}(\mathcal T)$. For every $O\in\Omega_{\mathcal T}$, we define two lattices $\Lambda_O^+\subseteq\Lambda_O^-$ as follows
\begin{equation*}
    \Lambda_O^+ \coloneq \begin{cases}
        \Lambda(\mathbf m)_{\overline{r_O^+}} & \text{if } \#O = 2,\\
        \Lambda(\mathbf m)_i
        &\text{if }O=\{i\},
    \end{cases}
    \qquad \Lambda_O^- \coloneq \begin{cases}
        \Lambda(\mathbf m)_{\overline{r_O^-}} & \text{if } \#O = 2,\\
        \Lambda(\mathbf m)_i
        &\text{if }O=\{i\},
        \end{cases}
\end{equation*}
The inequalities defining $\mathrm{Num}(\mathcal T)$ imply that $\Lambda_O^+\subseteq\Lambda_O^-$. Moreover, for all $O,O'\in\Omega_{\mathcal T}$, we have
\begin{equation*}
    O'\leq O \iff \Lambda_{O'}^-\subseteq\Lambda_O^- \iff \Lambda_{O'}^+\subseteq\Lambda_O^+.
\end{equation*}
We write $O'\prec O$ if $O'<O$ and there is no $O''\in\Omega_{\mathcal T}$ such that $O'<O''<O$. We also denote by $O_{\mathrm{min}}$ and $O_{\mathrm{max}}$ the minimal and maximal elements of $\Omega_{\mathcal T}$. We define
\begin{align*}
    \forall O' \prec O, & & W_{O',O} & \coloneq \Lambda_O^+ / \Lambda_{O'}^-, &  W_{O',O}' & \coloneq \pi(\Lambda_{O'}^-)^\vee /\pi(\Lambda_O^+)^\vee,\\
    \forall O \in \Omega_{\mathcal T} \text{ such that } \# O = 2, & &
    H_O & \coloneq \Lambda_O^- / \Lambda_O^+, & H_O' & \coloneq \pi(\Lambda_O^+)^\vee / \pi(\Lambda_O^-)^\vee,\\
    & & W_{\mathrm{min}} & \coloneq \Lambda_{O_{\mathrm{min}}}^+ / \pi(\Lambda_{O_{\mathrm{min}}}^+)^\vee,  & W_{\mathrm{max}} & \coloneq \pi(\Lambda_{O_{\mathrm{max}}}^-)^\vee / \pi\Lambda_{O_{\mathrm{max}}}^-.
\end{align*}
All these quotients are $\mathbb F_{q^2}$-vector spaces. We have
\begin{align*}
    \dim W_{\mathrm{min}} & = t(\Lambda_{O_{\mathrm{min}}}^+),\\
    \dim W_{\mathrm{max}} & = n - t(\Lambda_{O_{\mathrm{max}}}^-),\\
    \dim W_{O',O} = \dim W_{O',O}' & = \frac{t(\Lambda_O^+) - t(\Lambda_{O'}^-)}{2},\\
    \dim H_O = \dim H_O' & = \frac{t(\Lambda_O^-) - t(\Lambda_O^+)}{2} = \ell_{\overline{r_O^+}}-1
       -m_{\overline{r_O^+}}-m_{\overline{r_O^-}}.
\end{align*}
The spaces $W_{\mathrm{min}}$ and $W_{\mathrm{max}}$, when non-zero, inherit non-degenerate hermitian forms induced respectively by $(\cdot,\cdot)$ and by $\pi^{-1}(\cdot,\cdot)$ modulo $\pi$. Moreover, $(\cdot,\cdot)$ also induces a perfect pairing between $W_{O',O}$ and $W_{O',O}'$, and between $H_O$ and $H_O'$. 

We denote by $\mathcal J_{\mathcal T,\mathbf m} \to \mathrm{Spec}(\mathcal O_F)$ the parahoric group scheme of $J$ attached to the self-dual lattice chain generated by $\mathbf{\Lambda}(\mathbf m)$. In particular, we have $\mathcal J_{\mathcal T,\mathbf m}(\mathcal O_F) = \mathrm{Stab}_{J(F)}(\mathbf{\Lambda}(\mathbf m))$. We define
\begin{equation*}
\mathbf J_{\mathcal T,\mathbf m} \coloneq  \mathcal J_{\mathcal T,\mathbf m,\mathbb F_q} /R_u(\mathcal J_{\mathcal T,\mathbf m,\mathbb F_q}).
\end{equation*}
This is the maximal reductive quotient of the special fiber of $\mathcal J_{\mathcal T,\mathbf m}$. Similarly to $\mathbf G_{I}$, it can be described explicitly as follows. 
\begin{equation*}
    \mathbf J_{\mathcal T,\mathbf m} \simeq \left(\mathrm{GU}(W_{\mathrm{min}}) \times_{\mathbb G_m} \mathrm{GU}(W_{\mathrm{max}})\right) \times \prod_{O'\prec O} \mathrm{Res}_{\mathbb F_{q^2}/\mathbb F_q} \mathrm{GL}(W_{O',O}) \times \prod_{\substack{O\in\Omega_{\mathcal T}\\ \#O = 2}} \mathrm{Res}_{\mathbb F_{q^2}/\mathbb F_q} \mathrm{GL}(H_O).
\end{equation*}
If $W_{\mathrm{min}}$ or $W_{\mathrm{max}}$ is zero, then the corresponding $\mathrm{GU}$-factor must be omitted. The preceding formulas give 
\begin{equation*}
    \dim W_{\mathrm{min}}+\dim W_{\mathrm{max}} +2\sum_{O'\prec O}\dim W_{O',O}     +2\sum_{\substack{O\in\Omega_{\mathcal T}\\\#O=2}}\dim H_O=n.
\end{equation*}
Finally, we point out that 
\begin{align*}
    (\mathrm{Res}_{\mathbb F_{q^2}/\mathbb F_q} \mathrm{GL}(W_{O',O}))_{\mathbb F_{q^2}}  & \simeq \mathrm{GL}(W_{O',O})\times\mathrm{GL}(W_{O',O}'), \\
    (\mathrm{Res}_{\mathbb F_{q^2}/\mathbb F_q} \mathrm{GL}(H_O))_{\mathbb F_{q^2}}  & \simeq \mathrm{GL}(H_O)\times\mathrm{GL}(H_O').
\end{align*}
Using the basis $\mathbf e^{(0)}$, we may think of $\mathbf J_{\mathcal T,\mathbf m}$ in terms of matrices. The full \textit{decreasing} lattice chain $\Lambda_{\bullet}$ determines a pair $(\mathbf B_{\mathcal T,\mathbf m}, \mathbf  T_{\mathcal T,\mathbf m})$ consisting of an $\mathbb F_q$-rational Borel subgroup $\mathbf B_{\mathcal T,\mathbf m} \subset \mathbf J_{\mathcal T,\mathbf m, \overline{\mathbb F_q}}$ containing the $\mathbb F_q$-rational maximal torus $\mathbf  T_{\mathcal T,\mathbf m}$. In the basis $\mathbf e^{(0)}$, $\mathbf  B_{\mathcal T,\mathbf m}$ is identified with the Borel subgroup of \textit{lower-triangular} matrices, while $\mathbf  T_{\mathcal T,\mathbf m}$ consists of the diagonal matrices.\\

In the following paragraphs, we will define the families of fine Deligne-Lusztig varieties introduced in \cite{mullerBruhatTitsStratificationBasic2026} Section 3.3 while adapting to our current notations. We warn the reader that our settings lead us to consider parameters suited to \textit{lower triangular} parabolic subgroups, whereas \cite{mullerBruhatTitsStratificationBasic2026} has been written for \textit{upper triangular} parabolic subgroups. This amounts to conjugating all sets of simple reflections and Weyl group elements by the element of maximal length.\\

There is a unique integer $x_0\in\mathbb Z$ such that $\Lambda_{O_{\mathrm{max}}}^-=\Lambda_{x_0}$. We define the integer interval
\begin{equation}\label{Eq:DefiIntervalWidehatMathcalK}
    \widehat{\mathcal K}_{\mathcal T,\mathbf m} \coloneq \{x_0,x_0+1,\ldots,x_0+n-1\}.
\end{equation}
For every $i\in \mathbb Z/n \mathbb Z$, we denote by $r_i^{\mathcal T,\mathbf m}$ the unique representative of $i$ contained in $\widehat{\mathcal K}_{\mathcal T,\mathbf m}$. We will usually suppress the dependence on $(\mathcal T,\mathbf m)$ from the notation and write $r_i \coloneq r_i^{\mathcal T,\mathbf m}$. We then have
\begin{equation}\label{Eq:DefiIntervalMathcalK}
    \mathcal K_{\mathcal T,\mathbf m} \coloneq I_{\mathbb Z} \cap \widehat{\mathcal K}_{\mathcal T,\mathbf m} = \{r_i\mid i\in\overline I\}.
\end{equation}
Let $V$ be any one of the quotient vector spaces introduced above. Then $V$ has the form $V = \Lambda_{\alpha_V}/\Lambda_{\beta_V}$ for some unique integers $x_0\leq\alpha_V\leq\beta_V\leq x_0+n$. We associate with $V$ the integer intervals
\begin{equation*}
    \widehat{\mathcal K}_V \coloneq \{\alpha_V,\alpha_V+1,\ldots,\beta_V-1\}, \qquad \mathcal K_V \coloneq \widehat{\mathcal K}_V \cap I_{\mathbb Z}.
\end{equation*}
The intervals $\widehat{\mathcal K}_V$ are pairwise disjoint and give a decomposition
\begin{equation*}
    \widehat{\mathcal K}_{\mathcal T,\mathbf m} = \widehat{\mathcal K}_{W_{\mathrm{min}}} \sqcup \widehat{\mathcal K}_{W_{\mathrm{max}}} \sqcup \bigsqcup_{\substack{O,O' \in \Omega_{\mathcal T}\\ O' \prec O}} \left( \widehat{\mathcal K}_{W_{O',O}} \sqcup \widehat{\mathcal K}_{W_{O',O}'} \right) \sqcup \bigsqcup_{\substack{O\in\Omega_{\mathcal T}\\\#O=2}} \left( \widehat{\mathcal K}_{H_O} \sqcup \widehat{\mathcal K}_{H_O'} \right).
\end{equation*}
The inequalities defining $\mathrm{Num}(\mathcal T)$ imply that $\mathcal K_{H_O}=\mathcal K_{H_O'}=\emptyset$. Thus, we have 
\begin{equation*}
    \mathcal K_{\mathcal T,\mathbf m} = \mathcal K_{W_{\mathrm{min}}} \sqcup \mathcal K_{W_{\mathrm{max}}} \sqcup \bigsqcup_{\substack{O,O' \in \Omega_{\mathcal T}\\ O' \prec O}} \left( \mathcal K_{W_{O',O}} \sqcup \mathcal K_{W_{O',O}'}\right).
\end{equation*}
For every component $V=\Lambda_{\alpha_V}/\Lambda_{\beta_V}$ and every $r\in\mathcal K_V$, we define
\begin{equation*}
    U_r^\circ \coloneq \Lambda_r/\Lambda_{\beta_V} \subseteq \Lambda_{\alpha_V}/\Lambda_{\beta_V}=V.
\end{equation*}
The subspaces $(U_r^{\circ})_{r\in\mathcal K_V}$ form a (decreasing) partial flag in $V$, satisfying $\dim U_r^\circ=\beta_V-r$. If $V = W_{\mathrm{min}}$ or $W_{\mathrm{max}}$, the stabilizer of $(U_r^{\circ})_{r\in\mathcal K_V}$ is a standard parabolic subgroup $\mathbf P_{\mathcal K_{V}} \subseteq \mathrm{GU}(V)_{\overline{\mathbb F}_q}$. On the other hand, for all $O, O' \in \Omega_{\mathcal T}$ with $O' \prec O$, the stabilizer of the double flag $(U_r^{\circ})_{r \in \mathcal K_{W_{O',O}}}$ and $(U_r^{\circ})_{r \in \mathcal K_{W_{O',O}'}}$ is a standard parabolic subgroup $\mathbf P_{\mathcal K_{W_{O',O}} \sqcup \mathcal K_{W_{O',O}'}}$ of $(\mathrm{Res}_{\mathbb F_{q^2}/\mathbb F_q} \mathrm{GL}(W_{O',O}))_{\overline{\mathbb F}_q}$. The product $\mathbf P_{\mathcal K_{\mathcal T, \mathbf m}}$ of all these parabolic subgroups is the standard parabolic subgroup of $\mathbf J_{\mathcal T, \mathbf m,\overline{\mathbb F}_q}$ stabilizing the concatenated flags $(U_r^{\circ})_{r \in \mathcal K_{\mathcal T,\mathbf m}}$.\\
We now describe the Coxeter system associated with the pair $(\mathbf B_{\mathcal T,\mathbf m}, \mathbf T_{\mathcal T,\mathbf m})$. Let $V$ be one of the quotient spaces introduced above. The images of $\mathbf e_{\alpha_V+1}^{(0)},\ldots, \mathbf e_{\beta_V}^{(0)}$ form a basis of $V$. We index these basis vectors by the elements of $\widehat{\mathcal K}_V$ by associating $x\in\widehat{\mathcal K}_V$ with the image of $\mathbf e_{x+1}^{(0)}$. With respect to this basis, the diagonal torus and the lower-triangular Borel subgroup determine an identification of the Weyl group $\mathbf W_V$ of $\mathrm{GL}(V)$ with $\mathrm{Bij}(\widehat{\mathcal K}_V)$. Its set of simple reflections is
\begin{equation*}
    \mathbf S_V \coloneq \left\{s_x = (x \ x+1) \;\middle|\; \alpha_V \leq x \leq \beta_V-2
    \right\}.
\end{equation*}
Let $\mathbf W_{\mathcal T,\mathbf m}$ denote the Weyl group of $\mathbf J_{\mathcal T,\mathbf m,\overline{\mathbb F_q}}$ with respect to $\mathbf T_{\mathcal T,\mathbf m}$. The preceding identifications give
\begin{equation*} \hspace{-15mm}
    \mathbf W_{\mathcal T,\mathbf m} \simeq \mathrm{Bij}(\widehat{\mathcal K}_{W_{\mathrm{min}}}) \times \mathrm{Bij}(\widehat{\mathcal K}_{W_{\mathrm{max}}}) \times \prod_{\substack{O,O' \in \Omega_{\mathcal T}\\ O' \prec O}} \left( \mathrm{Bij}(\widehat{\mathcal K}_{W_{O',O}}) \times \mathrm{Bij}(\widehat{\mathcal K}_{W_{O',O}'}) \right) \times \prod_{\substack{O\in\Omega_{\mathcal T}\\ \# O = 2}} \left( \mathrm{Bij}(\widehat{\mathcal K}_{H_O}) \times \mathrm{Bij}(\widehat{\mathcal K}_{H_O'}) \right).
\end{equation*}
Under this identification, the set of simple reflections is
\begin{equation*}
    \mathbf S_{\mathcal T,\mathbf m} = \bigsqcup_V \mathbf S_V.
\end{equation*}
Thus, whenever $x$ and $x+1$ belong to the same $\widehat{\mathcal K}_V$ for some $V$, we will write $s_x \coloneq(x \ x+1)\in \mathbf W_{\mathcal T,\mathbf m}$. Besides, if $x, x+1, y, y+1$ all belong to the same $\widehat{\mathcal K}_V$ for some $V$, we write 
\begin{equation*}
    s_{[x,y]} \coloneq \begin{cases}
        s_x s_{x+1} \cdots s_{y} & \text{if } x \leq y, \\
        \mathrm{id} & \text{if } x > y.
    \end{cases}
\end{equation*}
We now introduce a new set of parameters
\begin{equation}\label{Eq:DefiSetMathcalBT}
    \mathcal B_{\mathcal T} \coloneq \left\{(b_i)_{i \in \overline I \setminus \mathcal T} \in \mathbb Z_{\geq 0}^{\overline I \setminus \mathcal T} \;\middle|\; \text{for every }i\in\overline I\setminus\mathcal T, \quad \begin{array}{c}
            0\leq b_i\leq\ell_i, \\
            b_i+b_{\theta_I(i)}\geq\ell_i.
        \end{array} \right\}.
\end{equation}
Let us fix some $\mathbf b\in\mathcal B_{\mathcal T}$, and define a new tuple $\mathbf c( \mathbf b) = (c_i)_{i\in\overline I}$ by
\begin{equation}\label{Eq:DefinitionTuplec(b)}
    c_i\coloneq \begin{cases}
        m_i & \text{if } i \in \mathcal T,\\
        b_i & \text{if } i \in \overline I \setminus \mathcal T.
    \end{cases}
\end{equation}
Thus, $\mathbf c(\mathbf b)$ is just the concatenation of $\mathbf m$ and $\mathbf b$. When $\mathbf b$ is clear from the context, we simply write $\mathbf c \coloneq \mathbf c(\mathbf b)$. We note that the triple $(\mathcal T, \mathbf m, \mathbf b)$ can be recovered from $\mathbf c$ alone. Indeed, $\mathcal T$ is the subset of $i \in \overline I$ such that $c_i + c_{\theta_I(i)} \leq \ell_i-1$, and $\mathbf m$ and $\mathbf b$ can then be read directly from $\mathbf c$.\\
For every $i\in\overline I$, the inequalities defining $\mathrm{Num}(\mathcal T)$ and $\mathcal B_{\mathcal T}$ imply that the interval
\begin{equation*}
    \{r_{-i-\underline{\varepsilon}}-c_i, \ldots , r_{-i-\underline{\varepsilon}}\}
\end{equation*}
is contained in a unique one of the intervals $\widehat{\mathcal K}_{W_{\mathrm{min}}}$, $\widehat{\mathcal K}_{W_{\mathrm{max}}}$, $\widehat{\mathcal K}_{W_{O',O}}$ or $\widehat{\mathcal K}_{W_{O',O}'}$. We may therefore define
\begin{equation*}
    w_i(\mathbf c) \coloneq s_{[r_{-i-\underline{\varepsilon}}-c_i, r_{-i-\underline{\varepsilon}}-1]}.
\end{equation*}
We then set
\begin{equation}\label{Eq:DefinitionWc}
    w_{\mathbf c} \coloneq \prod_{i\in\overline I}w_i(\mathbf c),
\end{equation}
where the factors are written from left to right in increasing order of $r_{-i-\underline{\varepsilon}}$. In particular, $w_{\mathbf c}$ belongs to $\mathbf W_{\mathcal T,\mathbf m}$, and its components in the factors attached to $H_O$ and $H_O'$ are trivial. Moreover, $w_{\mathbf c}$ is right $\mathcal K_{\mathcal T, \mathbf m}$-reduced, that is $w_{\mathbf c}$ is the minimal length element of the right coset $w_{\mathbf c} \mathbf W_{\mathcal K_{\mathcal T,\mathbf m}}$, where $\mathbf W_{\mathcal K_{\mathcal T,\mathbf m}}$ is the parabolic subgroup of $\mathbf W_{\mathcal T,\mathbf m}$ generated by all the simple reflections $s_x$ such that $s_x \in \mathbf S_V$ for some $V$, but $x+1 \not \in \mathcal K_V$.

We may consider the fine Deligne-Lusztig variety
\begin{equation*}
    X_{\mathcal K_{\mathcal T,\mathbf m}} \{w_{\mathbf c}^{-1}\} \subseteq \mathbf J_{\mathcal T,\mathbf m,\overline{\mathbb F_q}} / \mathbf P_{\mathcal K_{\mathcal T,\mathbf m}}.
\end{equation*}
More explicitly, let $k$ be an algebraically closed field containing $\overline{\mathbb F_q}$. A $k$-point of $X_{\mathcal K_{\mathcal T,\mathbf m}}\{w_{\mathbf c}^{-1}\}$ is a collection $(U_r)_{r\in\mathcal K_{\mathcal T,\mathbf m}}$ of partial flags conjugate to $(U_r^{\circ})_{r \in \mathcal K_{\mathcal T, \mathbf m}}$, which can be lifted to a collection of full flags
\begin{equation*}
    (\mathcal F_r)_{r\in \widehat{\mathcal K}_{\mathcal T,\mathbf m}}, \qquad \mathcal F_r = U_r \text{ for all } r\in\mathcal K_{\mathcal T,\mathbf m},
\end{equation*}
such that $(\mathcal F_r)_{r\in \widehat{\mathcal K}_{\mathcal T,\mathbf m}} \in X(w_{\mathbf c}^{-1})(k)$. Here, $X(w_{\mathbf c}^{-1})$ denotes the classical Deligne-Lusztig variety of $\mathbf J_{\mathcal T,\mathbf m,\overline{\mathbb F_q}}$ with parameter $w_{\mathbf c}^{-1}$. We have 
\begin{equation}\label{Eq:DimensionFineDLVarieties}
    \dim X_{\mathcal K_{\mathcal T,\mathbf m}} \{w_{\mathbf c}^{-1}\}  = \ell(w_{\mathbf c}) = \sum_{i \in \overline I} c_i = \sum_{i \in \mathcal T} m_i + \sum_{i \in \overline I \setminus \mathcal T} b_i.
\end{equation}
Observe that $\mathbf J_{\mathcal T, \mathbf m}(\mathbb F_q)$ is the quotient of $\mathrm{Stab}_{J(F)}(\mathbf{\Lambda}(\mathbf m))$ by its pro-$p$ radical. We have an action of $\mathrm{Stab}_{J(F)}(\mathbf{\Lambda}(\mathbf m))$ on the Bruhat-Tits stratum $\mathcal N_{\mathcal T,\mathbf{\Lambda}(\mathbf m)}^{0}$. The following theorem is a combination of \cite{mullerBruhatTitsStratificationBasic2026} Theorem 3.36 and Theorem 4.3.

\begin{theo}\label{Theo:IsomorphismBTStrataAndFineDLVarieties}
    Let $\mathcal T \subset \overline I$ be a non-empty $\theta_I$-stable subset and let $\mathbf m \in \mathrm{Num}(\mathcal T)$. The action of $\mathrm{Stab}_{J(F)}(\mathbf{\Lambda}(\mathbf m))$ on $\mathcal N_{\mathcal T,\mathbf{\Lambda}(\mathbf m)}^{0}$ factors through $\mathbf J_{\mathcal T, \mathbf m}(\mathbb F_q)$, and we have a $\mathbf J_{\mathcal T, \mathbf m}(\mathbb F_q)$-equivariant isomorphism
    \begin{equation*}
        \mathcal N_{\mathcal T,\mathbf{\Lambda}(\mathbf m)}^{0} \xrightarrow{\sim} \bigsqcup_{\mathbf b \in \mathcal B_{\mathcal T}} X_{\mathcal K_{\mathcal T, \mathbf m}}\{w_{\mathbf c(\mathbf b)}^{-1}\},
    \end{equation*}
    where the right-hand side is a smooth, irreducible locally closed subvariety of the partial flag variety $\mathbf J_{\mathcal T,\mathbf m,\overline{\mathbb F_q}} / \mathbf P_{\mathcal K_{\mathcal T,\mathbf m}}$.
\end{theo}
By abuse of notation, we will still write $X_{\mathcal K_{\mathcal T, \mathbf m}}\{w_{\mathbf c}^{-1}\}$ for the locally closed subscheme of $\mathcal N_{\mathrm{red}}^{I,\varepsilon}$ transported through this isomorphism. We note that the components $\mathrm{Res}_{\mathbb F_{q^2}/\mathbb F_q} \mathrm{GL}(H_{O})$ all act trivially on $\mathcal N_{\mathcal T,\mathbf{\Lambda}(\mathbf m)}^{0}$.

\begin{rk}
    There is a polysemy around the meaning of ``Bruhat-Tits stratification'' in the literature. In \cite{gortzBasicLociCoxeter2024} Section 2.4 and in the context of affine Deligne-Lusztig varieties, it is the resulting decomposition 
    \begin{equation}\label{Eq:FineBTStratification}
        \mathcal N_{\mathrm{red}}^{I,\varepsilon} = \bigsqcup_{\mathcal T, \mathbf m} \quad \bigsqcup_{g \in J(F)/\mathrm{Stab}_{J(F)}(\mathbf{\Lambda}(\mathbf m))} \quad \bigsqcup_{\mathbf b \in \mathcal B_{\mathcal T}} g\cdot X_{\mathcal K_{\mathcal T,\mathbf m}}\{w_{\mathbf c(\mathbf b)}^{-1}\},
    \end{equation}
    which should be called the \textit{(weak) Bruhat-Tits stratification}, where $\mathcal T$ runs over non-empty $\theta_I$-stable subsets of $\overline I$ and $\mathbf m \in \mathrm{Num}(\mathcal T)$. In other words, the (weak) Bruhat-Tits strata are intended to be fine Deligne-Lusztig varieties. Here, the adjective ``weak'' is added only in case the triple $(G,\{\mu\},\mathcal G_{I})$ is not of Coxeter type (see Definition 2.4 of \cite{gortzBasicLociCoxeter2024}), to mark the difference with the Coxeter case. We point out that this triple turns out to be of Coxeter type only if ($I = \{0\}$ and $\varepsilon = 1$), ($I = \{\nu_{\varepsilon}\}$ and $2\nu_{\varepsilon}+\underline{\varepsilon} = n$), or ($I = \{0,\nu_{\varepsilon}\}$, $n$ is even and $\varepsilon = 1$). In the two first cases, the parahoric level $\mathcal G_I$ is hyperspecial. \\
    In \cite{mullerBruhatTitsStratificationBasic2026} and in this paper, we choose to call the decomposition of \eqref{Eq:FineBTStratification} the \textit{fine Bruhat-Tits stratification}, since it is a refinement of the Bruhat-Tits stratification, as defined here, by fine Deligne-Lusztig varieties.
\end{rk}

\subsection{Comparison of the fine BT and the KR stratifications}

In this section, we fix $\mathcal T\subseteq\overline I$ a non-empty $\theta_I$-stable subset, $\mathbf m\in\mathrm{Num}(\mathcal T)$ and $\mathbf b\in\mathcal B_{\mathcal T}$. We write $\mathbf c \coloneq \mathbf c(\mathbf b) = (c_i)_{i \in \overline I}$ as in \eqref{Eq:DefinitionTuplec(b)}. We define
\begin{equation*}
    X_{\mathcal T,\mathbf m,\mathbf b} \coloneq \bigsqcup_{g\in J(F)/\mathrm{Stab}_{J(F)}(\mathbf{\Lambda}(\mathbf m))} g\cdot X_{\mathcal K_{\mathcal T,\mathbf m}}\{w_{\mathbf c}^{-1}\}.
\end{equation*}
Then $X_{\mathcal T,\mathbf m,\mathbf b}$ is a $J(F)$-stable locally closed subscheme of $\mathcal N_{\mathrm{red}}^{I,\varepsilon}$. Recall the definition of $J_{\mathbf c}$ from \eqref{Eq:KRIndexAttachedToEKORIndex}. We have $J_{\mathbf c} \in \mathcal P_{0,I}$ because $\mathcal T$ is not empty. 

\begin{theo}\label{Theo:ComparisonFineBTStrataAndKR}
    We have
    \begin{equation*}
        X_{\mathcal T,\mathbf m,\mathbf b} \hookrightarrow \mathrm{KR}_{J_{\mathbf c}}.
    \end{equation*}
\end{theo}

\begin{proof}
    Since the KR strata are $J(F)$-invariant, it is enough to prove that $X_{\mathcal K_{\mathcal T,\mathbf m}}\{w_{\mathbf c}^{-1}\} \hookrightarrow \mathrm{KR}_{J_{\mathbf c}}$. Let $k$ be an algebraically closed field containing $\overline{\mathbb F_q}$ and let $z = (Z_r)_{r \in I_{\mathbb Z}} \in X_{\mathcal K_{\mathcal T,\mathbf m}}\{w_{\mathbf c}^{-1}\}(k)$. Let $J \in \mathcal P_{I}$ denote the index of the KR stratum containing $z$. We will prove that $J = J_{\mathbf c}$ using Proposition \ref{Prop:PointsOfKRStrata}. 

    First, let $i\in\mathcal T$ and let $r\in I_{\mathbb Z}$ be any representative of $i$. Since $z$ belongs to the Bruhat-Tits stratum $\mathcal N^0_{\mathcal T, \mathbf{\Lambda}(\mathbf m)}$, we have $\mathrm{Type}(z) = \mathcal T$. Therefore
    \begin{equation*}
        \tau(Z_{r+\ell_i}) \subseteq (\Lambda_{Z_{r+\ell_i}})_k \subseteq \pi(\Lambda_{Z_{-r-\underline{\varepsilon}}}^{\vee})_k \subseteq \pi Z_{-r-\underline{\varepsilon}}^{\vee}.
    \end{equation*}
    It follows that $J \subseteq \overline I \setminus \mathcal T$, and in particular $J \in \mathcal P_{0,I}$. 
    
    Next, let $V = \Lambda_{\alpha_V}/\Lambda_{\beta_V}$ be one of the quotient spaces such that $\mathcal K_V \not = \emptyset$. The point $z$ induces the following partial flag of $V_k$
    \begin{equation*}
        \forall r \in \mathcal K_V, \qquad U_r \coloneq Z_r / (\Lambda_{\beta_V})_k.
    \end{equation*}
    All together, they form a collection of partial flags $(U_r)_{r \in \mathcal K_{\mathcal T, \mathbf m}}$, that is a $k$-point of the partial flag variety $\mathbf J_{\mathcal T,\mathbf m,\overline{\mathbb F_q}} / \mathbf P_{\mathcal K_{\mathcal T,\mathbf m}}$. This partial flag can be lifted to a complete flag $(\mathcal F_r)_ {r\in\widehat{\mathcal K}_{\mathcal T,\mathbf m}} \in X(w_{\mathbf c}^{-1})(k)$.
    
    Let now $i\in\overline I\setminus\mathcal T$. Let $V$ be the unique component among $W_{\mathrm{min}}$, $W_{\mathrm{max}}$, $W_{O',O}$ and $W_{O',O}'$ such that $r_{-i-\underline{\varepsilon}} \in \mathcal K_V$. By construction, if $V = W_{\mathrm{min}}$ or $W_{\mathrm{max}}$, then $r_{i+\ell_i} \in \mathcal K_V$, while if $V$ is one of $W_{O',O}$ and $W_{O',O}'$, then $r_{i+\ell_i}$ belongs to the interval attached to the other one. If $V = W_{\mathrm{min}}$ or $W_{\mathrm{max}}$, the symbol $(\cdot)^{\perp}$ denotes the orthogonal complement in $V_k$ with respect to its hermitian form. If $V$ is one of $W_{O',O}$ and $W_{O',O}'$, it denotes the annihilator in the other space with respect to their perfect pairing. In either case, if $\tau$ denotes the operator acting as $x\mapsto x^{q^2}$ on the scalars, then $(U^{\perp})^{\perp} = \tau(U)$. We claim that
    \begin{equation}\label{Eq:FlagCalculationKR}
        \tau(U_{r_{i+\ell_i}})\subseteq U_{r_{-i-\underline{\varepsilon}}}^{\perp} \iff b_i < \ell_i.
    \end{equation}
    Indeed, the block $w_i(\mathbf c) = s_{[r_{-i-\underline{\varepsilon}}-b_i,r_{-i-\underline{\varepsilon}}-1]}$ translates to the following condition regarding the relative position of $(\mathcal F_r)_{r\in\widehat{\mathcal K}_{\mathcal T,\mathbf m}}$ and of $(\mathcal F_r^{\perp})_{r\in\widehat{\mathcal K}_{\mathcal T,\mathbf m}}$:
    \begin{center}
    \begin{tikzcd}[column sep=small]
    \mathcal F_{r_{i+\ell_i}-\ell_i+1}^{\perp} \arrow[d,equal,"/" marking] \arrow[dr,hook,shorten >=10pt] \arrow[r,symbol=\overset{1}{\subset}] & \ldots \arrow[r,symbol=\overset{1}{\subset}] & \mathcal F_{r_{i+\ell_i}-\ell_i+b_i}^{\perp} \arrow[d,equal,"/" marking] \arrow[dr,hook,shorten >=10pt] \arrow[r,symbol=\overset{1}{\subset}] & \mathcal F_{r_{i+\ell_i}-\ell_i+b_i+1}^{\perp} \arrow[d,equal] \arrow[r,symbol=\overset{1}{\subset}] & \ldots \arrow[r,symbol=\overset{1}{\subset}] & U_{r_{i+\ell_i}}^{\perp} \arrow[d,equal] \\
    U_{r_{-i-\underline{\varepsilon}}} \arrow[r,symbol=\overset{1}{\subset}] & \ldots \arrow[r,symbol=\overset{1}{\subset}] & \mathcal F_{r_{-i-\underline{\varepsilon}}-b_i+1} \arrow[r,symbol=\overset{1}{\subset}] & \mathcal F_{r_{-i-\underline{\varepsilon}}-b_i} \arrow[r,symbol=\overset{1}{\subset}] & \ldots \arrow[r,symbol=\overset{1}{\subset}] & \mathcal F_{r_{-i-\underline{\varepsilon}}-\ell_i+1}.
    \end{tikzcd}
    \end{center}
    Here, the first $b_i$ vertical relations are inequalities and the remaining $\ell_i-b_i$ are equalities. We note that the extremal cases $b_i=0$ or $b_i=\ell_i$ are possible. Interchanging $i$ and $\theta_I(i)$ gives the companion diagram
    \begin{center}
    \begin{tikzcd}[column sep=small]
    \mathcal F_{r_{-i-\underline{\varepsilon}}-\ell_i+1}^{\perp} \arrow[d,equal,"/" marking] \arrow[dr,hook,shorten >=10pt] \arrow[r,symbol=\overset{1}{\subset}] & \ldots \arrow[r,symbol=\overset{1}{\subset}] & \mathcal F_{r_{-i-\underline{\varepsilon}}-\ell_i+b_{\theta_I(i)}}^{\perp} \arrow[d,equal,"/" marking] \arrow[dr,hook,shorten >=10pt] \arrow[r,symbol=\overset{1}{\subset}] & \mathcal F_{r_{-i-\underline{\varepsilon}}-\ell_i+b_{\theta_I(i)}+1}^{\perp} \arrow[d,equal] \arrow[r,symbol=\overset{1}{\subset}] & \ldots \arrow[r,symbol=\overset{1}{\subset}] & U_{r_{-i-\underline{\varepsilon}}}^{\perp} \arrow[d,equal] \\
    U_{r_{i+\ell_i}} \arrow[r,symbol=\overset{1}{\subset}] & \ldots \arrow[r,symbol=\overset{1}{\subset}] & \mathcal F_{r_{i+\ell_i}-b_{\theta_I(i)}+1} \arrow[r,symbol=\overset{1}{\subset}] & \mathcal F_{r_{i+\ell_i}-b_{\theta_I(i)}} \arrow[r,symbol=\overset{1}{\subset}] & \ldots \arrow[r,symbol=\overset{1}{\subset}] & \mathcal F_{r_{i+\ell_i}-\ell_i+1}.
    \end{tikzcd}
    \end{center}
    If $V=W_{\mathrm{min}}$ or $W_{\mathrm{max}}$, both diagrams lie in the same hermitian space $V_k$. Otherwise, the first diagram lies in one of $(W_{O',O})_k$ and $(W_{O',O}')_k$, and the second lies in the other. If $i = \theta_I(i)$, then $r_{i+\ell_i} = r_{-i-\underline{\varepsilon}}$ and the two diagrams coincide, so that only the first one is needed.

    Suppose first that $b_i = \ell_i$. All the vertical relations in the first diagram are then inequalities, and the successive diagonal inclusions give
    \begin{equation*}
        \mathcal F_{r_{-i-\underline{\varepsilon}}-\ell_i+1}
        =\mathcal F_{r_{i+\ell_i}-1}^{\perp}+U_{r_{-i-\underline{\varepsilon}}}.
    \end{equation*}
    If we had $\tau(U_{r_{i+\ell_i}})\subseteq U_{r_{-i-\underline{\varepsilon}}}^{\perp}$, then, by taking orthogonal complements, we would have $U_{r_{-i-\underline{\varepsilon}}}\subseteq U_{r_{i+\ell_i}}^{\perp}$. Since also $\mathcal F_{r_{i+\ell_i}-1}^{\perp}\subseteq U_{r_{i+\ell_i}}^{\perp}$, it would follow that
    \begin{equation*}
        \mathcal F_{r_{-i-\underline{\varepsilon}}-\ell_i+1}\subseteq U_{r_{i+\ell_i}}^{\perp}.
    \end{equation*}
    These two subspaces have the same dimension, contradicting the rightmost inequality in the first diagram. Hence $\tau(U_{r_{i+\ell_i}})\not\subseteq U_{r_{-i-\underline{\varepsilon}}}^{\perp}$.

    Suppose now that $b_i < \ell_i$. The rightmost vertical relation in the first diagram is then an equality, so that
    \begin{equation*}
        \mathcal F_{r_{-i-\underline{\varepsilon}}-\ell_i+1}=U_{r_{i+\ell_i}}^{\perp}.
    \end{equation*}
    Taking orthogonal complements and using the top row of the companion diagram gives $\tau(U_{r_{i+\ell_i}}) = \mathcal F_{r_{-i-\underline{\varepsilon}}-\ell_i+1}^{\perp} \subseteq U_{r_{-i-\underline{\varepsilon}}}^{\perp}$. This proves \eqref{Eq:FlagCalculationKR}.

    Observe that \eqref{Eq:FlagCalculationKR} is equivalent, for any $r\in I_{\mathbb Z}$ such that $\overline r = i$, to the condition
    \begin{equation*}
        \tau(Z_{r+\ell_i}) \subseteq\pi Z_{-r-\underline{\varepsilon}}^\vee \iff b_i<\ell_i.
    \end{equation*}
    Proposition \ref{Prop:PointsOfKRStrata} therefore gives, for every $i\in\overline I\setminus\mathcal T$,
    \begin{equation*}
        i\in\overline I\setminus J \iff b_i < \ell_i \iff i\in\overline I\setminus J_{\mathbf c}.
    \end{equation*}
    Since we already know that $J\subseteq\overline I\setminus\mathcal T$, this proves that $J=J_{\mathbf c}$ and concludes the proof. 
\end{proof}

As pointed out in Remark \ref{Rk:KRNonEmptyIFFInAdmMu0}, we can now directly confirm a fact expected from the general theory on the non-emptiness of the KR strata.

\begin{corol}\label{Corol:KRStrataNonEmpty}
    Let $J \in \mathcal P_I$. We have $\mathrm{KR}_J \not = \emptyset \iff J\in \mathcal P_{0,I}$.  
\end{corol}
\begin{proof}
    Assume that $J\in\mathcal P_{0,I}$ and set $\mathcal T \coloneq \overline I \setminus \left(J \cup \theta_I(J)\right)$. Then $\mathcal T$ is non-empty and $\theta_I$-stable. Take $\mathbf m = (0)_{i\in\mathcal T}$ and define
    \begin{equation*}
        \forall i \in \overline I \setminus \mathcal T, \qquad b_i \coloneq \begin{cases}
            \ell_i & \text{if } i \in J,\\
            0 & \text{if } i \not \in J.
        \end{cases}
    \end{equation*}
    Then $\mathbf b\in\mathcal B_{\mathcal T}$ and we have $J_{\mathbf c(\mathbf b)} = J$. By Theorem \ref{Theo:ComparisonFineBTStrataAndKR}, we have $X_{\mathcal T, \mathbf m, \mathbf b} \hookrightarrow \mathrm{KR}_J$. Since fine Deligne-Lusztig varieties are not empty, we have $\mathrm{KR}_J \not = \emptyset$.\\
    
    Conversely, assume that $\mathrm{KR}_J \not = \emptyset$ and let $z$ be a geometric point in it. Let $j \in \mathbb Z$ such that $z \in \mathcal N_j^{I,\varepsilon}$. By Lemma \ref{Lem:CrucialLemma}, the Bruhat-Tits type $\mathrm{Type}(z)$ of $z$ is not empty. Any integer $r\in I_{\mathbb Z}$ such that $\overline r \in \mathrm{Type}(z)$ yields the following sequences of inclusions 
    \begin{align*}
        \tau(Z_{r+\ell_{\overline r}}) & \subseteq \Lambda_{Z_{r+\ell_{\overline r}}} \subseteq \pi^{j+1}\Lambda_{Z_{-r-\underline{\varepsilon}}}^{\vee} \subseteq \pi^{j+1}Z_{-r-\underline{\varepsilon}}^{\vee}, \\ 
        \tau(Z_{-r-\underline{\varepsilon}})  & \subseteq \Lambda_{Z_{-r-\underline{\varepsilon}}}  \subseteq \pi^{j+1} \Lambda_{Z_{r+\ell_{\overline r}}}^{\vee} \subseteq \pi^{j+1}Z_{r+\ell_{\overline r}}^{\vee}, 
    \end{align*}
    proving that the $\theta_I$-orbit $O = \{\overline r, \theta_I(\overline r)\}$ is contained in $\overline I \setminus J$. Thus $J \in \mathcal P_{0,I}$.
\end{proof}
It follows from Theorem \ref{Theo:ComparisonFineBTStrataAndKR} that a single Bruhat-Tits stratum $\mathcal N_{\mathcal T, \mathbf{\Lambda}}^{0}$ is in general far from being contained in a single KR stratum. In fact, it intersects all the strata $\mathrm{KR}_{J_{\mathbf c}}$ for varying $\mathbf b \in \mathcal B_{\mathcal T}$.

\subsection{Comparison of the fine BT and the EKOR stratifications}

Recall from Proposition \ref{Prop:DescriptionIndexSetOfEKORStrata} that the elements of ${}^{I}\mathrm{Adm}(\{\mu'\})_0$ are naturally indexed by the tuples $\mathbf c=(c_i)_{i\in\overline I}$ such that $0 \leq c_i \leq \ell_i$ for all $i$, and $c_i + c_{\theta_I(i)} \leq \ell_i-1$ for at least one $i \in \overline I$. The corresponding element of ${}^{I}\mathrm{Adm}(\{\mu'\})_0$ is denoted by $v_{\mathbf c}$.

Let $\mathcal T\subseteq\overline I$ be a non-empty $\theta_I$-stable subset, let $\mathbf m \in \mathrm{Num}(\mathcal T)$ and let $\mathbf b \in \mathcal B_{\mathcal T}$. Recall the associated tuple $\mathbf c = \mathbf c(\mathbf b)$ defined by
\begin{equation*}
    c_i = \begin{cases}
        m_i & \text{if } i \in \mathcal T,\\
        b_i & \text{if } i \in \overline I \setminus \mathcal T.
    \end{cases}
\end{equation*}
Since $\mathcal T$ is non-empty, the inequalities defining $\mathrm{Num}(\mathcal T)$ show that $v_{\mathbf c}\in{}^{I}\mathrm{Adm}(\{\mu'\})_0$.

\begin{prop}\label{Prop:BijectionIndexSetsBTEKOR}
    The map $(\mathcal T,\mathbf m,\mathbf b) \mapsto v_{\mathbf c(\mathbf b)}$ induces a bijection
    \begin{equation*} 
        \left\{ (\mathcal T,\mathbf m,\mathbf b) \;\middle|\; \begin{array}{l}
            \mathcal T \subseteq \overline I \text{ is non-empty and } \theta_I\text{-stable},\\
            \mathbf m \in \mathrm{Num}(\mathcal T),\\
            \mathbf b \in \mathcal B_{\mathcal T}.
        \end{array} \right\} \xrightarrow{\sim} {}^{I}\mathrm{Adm}(\{\mu'\})_0.
    \end{equation*}
\end{prop}

\begin{proof}
    Let $v_{\mathbf c}\in{}^{I}\mathrm{Adm}(\{\mu'\})_0$ and define
    \begin{equation*}
        \mathcal T_{\mathbf c} \coloneq \left\{i \in \overline I \;\middle|\; c_i + c_{\theta_I(i)} \leq \ell_i-1 \right\}.
    \end{equation*}
    Then $\mathcal T_{\mathbf c}$ is non-empty and $\theta_I$-stable. Define $\mathbf m_{\mathbf c} \coloneq(c_i)_{i\in\mathcal T_{\mathbf c}}$ and $\mathbf b_{\mathbf c} \coloneq(c_i)_{i\in\overline I\setminus\mathcal T_{\mathbf c}}$. By construction, $\mathbf m_{\mathbf c} \in \mathrm{Num}(\mathcal T_{\mathbf c})$ and $\mathbf b_{\mathbf c}\in\mathcal B_{\mathcal T_{\mathbf c}}$. This construction is inverse to the map of the Proposition.
\end{proof}

Given a tuple $\mathbf c$ as above, under the bijection \eqref{Eq:BijectionKRPosetEKORPoset}, recall that we have $v_{\mathbf c} = v_{J_{\mathbf c}} \prod_{i \in \overline I \setminus J_{\mathbf c}} u_{i,c_i}$. We can state the main theorem.

\begin{theo}\label{Theo:ComparisonBTEKOR}
    Let $\mathcal T\subseteq\overline I$ be a non-empty $\theta_I$-stable subset, let
    $\mathbf m\in\mathrm{Num}(\mathcal T)$ and let $\mathbf b\in\mathcal B_{\mathcal T}$. If $\mathbf c=\mathbf c(\mathbf b)$, then
    \begin{equation*}
        X_{\mathcal T,\mathbf m,\mathbf b} = \mathrm{EKOR}_{v_{\mathbf c}}.
    \end{equation*}
\end{theo}

\begin{proof}
    It is enough to prove the inclusion $X_{\mathcal T,\mathbf m,\mathbf b} \subseteq \mathrm{EKOR}_{v_{\mathbf c}}$. Indeed, Proposition \ref{Prop:BijectionIndexSetsBTEKOR} identifies the indexing sets of two stratifications of $\mathcal N_{\mathrm{red}}^{I,\varepsilon}$. Besides, since the EKOR strata are $J(F)$-invariant, it is enough to prove that $X_{\mathcal K_{\mathcal T,\mathbf m}}\{w_{\mathbf c}^{-1}\} \subseteq \mathrm{EKOR}_{v_{\mathbf c}}$. Let $k$ be an algebraically closed field containing $\overline{\mathbb F_q}$ and let $z = (Z_r)_{r\in I_{\mathbb Z}} \in X_{\mathcal K_{\mathcal T,\mathbf m}}\{w_{\mathbf c}^{-1}\}(k)$. As in the proof of Theorem \ref{Theo:ComparisonFineBTStrataAndKR}, the corresponding partial flag $(U_r)_{r\in\mathcal K_{\mathcal T,\mathbf m}}$ can be lifted to a full flag $(\mathcal F_r)_{r\in\widehat{\mathcal K}_{\mathcal T,\mathbf m}} \in X(w_{\mathbf c}^{-1})(k)$.

    Choose $g \in \mathbf J_{\mathcal T,\mathbf m}(k)$ sending the standard full flag induced by $(\Lambda_r)_{r\in\widehat{\mathcal K}_{\mathcal T,\mathbf m}}$ to $(\mathcal F_r)_{r\in\widehat{\mathcal K}_{\mathcal T,\mathbf m}}$. Since $\mathcal J_{\mathcal T,\mathbf m}$ is smooth, $g$ admits a lift $\widetilde g \in \mathcal J_{\mathcal T,\mathbf m}(W_{\mathcal O_F}(k))$. For every $r \in \mathbb Z$, set 
    \begin{equation*}
        \widetilde Z_r \coloneq \widetilde g \cdot\left(\Lambda_r \otimes_{\mathcal O_K}W_{\mathcal O_F}(k) \right).
    \end{equation*}
    Then $(\widetilde Z_r)_{r\in\mathbb Z}$ is a full decreasing $\pi$-periodic lattice chain satisfying
    \begin{equation*}
        \pi\widetilde Z_{-r-\underline{\varepsilon}}^\vee \overset{1}{\subset} \widetilde Z_r \subset \widetilde Z_{-r-\underline{\varepsilon}}^\vee.
    \end{equation*}
    It therefore defines a point $\widetilde z \in \mathcal N_0^{I_{\mathrm{Iw}},\varepsilon}(k)$, where $I_{\mathrm{Iw}} = \{0,\ldots,\nu_{\varepsilon}\}$ corresponds to the Iwahori level, and whose image under the change of parahoric morphism $\mathcal N_0^{I_{\mathrm{Iw}},\varepsilon} \to \mathcal N_0^{I,\varepsilon}$ is $z$. Let $J^{\mathrm{Iw}} \subsetneq \mathbb Z/n\mathbb Z$ be the index of the Iwahori-level KR stratum containing $\widetilde z$. We claim that
    \begin{equation}\label{Eq:ComputationOfJIwahori}
        J^{\mathrm{Iw}} = \bigsqcup_{i \in \overline I} \{i,i+1,\ldots,i+c_i-1\},
    \end{equation}
    where the set indexed by $i$ is understood to be empty if $c_i=0$. Recall that the subsets $B_i = \{i, i+1, \ldots, i+\ell_i-1 \}$ for $i \in \overline I$ form a partition of $\mathbb Z/n\mathbb Z$. Thus every element can be written uniquely as $i+d$, where $i\in\overline I$ and $0\leq d\leq\ell_i-1$. Consequently, we have to prove that
    \begin{equation}\label{Eq:MembershipJIwahori}
        i+d \in J^{\mathrm{Iw}} \iff d < c_i.
    \end{equation}
    Fix $i\in\overline I$ and $0\leq d\leq\ell_i-1$. By definition, $r_{-i-\underline{\varepsilon}-d} \in \widehat{\mathcal K}_{\mathcal T,\mathbf m}$ is the chosen representative of $-i-\underline{\varepsilon}-d$ in $\widehat{\mathcal K}_{\mathcal T,\mathbf m}$. Therefore $-r_{-i-\underline{\varepsilon}-d} -\underline{\varepsilon}$ is a representative of $i+d$. Proposition \ref{Prop:PointsOfKRStrata}, applied at Iwahori level to this representative, gives
    \begin{equation}
        i+d \not \in J^{\mathrm{Iw}} \iff \tau \left(\widetilde Z_{-r_{-i-\underline{\varepsilon}-d} -\underline{\varepsilon}+1} \right) \subseteq \pi \widetilde Z_{r_{-i-\underline{\varepsilon}-d}}^\vee.
    \end{equation}
    Hence \eqref{Eq:MembershipJIwahori} is equivalent to proving that
    \begin{equation}\label{Eq:CriterionToBeCheckedOnFlags}
        \tau\left(\widetilde Z_{-r_{-i-\underline{\varepsilon}-d}-\underline{\varepsilon}+1} \right) \subseteq \pi \widetilde Z_{r_{-i-\underline{\varepsilon}-d}}^\vee \iff d \geq  c_i.
    \end{equation}
    We now verify \eqref{Eq:CriterionToBeCheckedOnFlags} by distinguishing two cases.\\

    \textbf{Case 1:} Let $i\in \overline I \setminus \mathcal T$. Then $c_i = b_i$ and there is a unique quotient space $V$ introduced in Section \ref{Section5.2} such that
    \begin{equation}\label{Eq:Case1IntervalInclusion}
        \{r_{-i-\underline{\varepsilon}}-\ell_i+1,\ldots, r_{-i-\underline{\varepsilon}}\} \subseteq \widehat{\mathcal K}_V.
    \end{equation}
    The first diagram in the proof of Theorem \ref{Theo:ComparisonFineBTStrataAndKR} becomes
    \begin{center}
    \begin{tikzcd}[column sep=small]
        \mathcal F_{r_{i+\ell_i}-\ell_i+1}^{\perp} \arrow[d,equal,"/" marking] \arrow[dr,hook,shorten >=10pt] \arrow[r,symbol=\overset{1}{\subset}] & \ldots \arrow[r,symbol=\overset{1}{\subset}] & \mathcal F_{r_{i+\ell_i}-\ell_i+b_i}^{\perp} \arrow[d,equal,"/" marking] \arrow[dr,hook,shorten >=10pt] \arrow[r,symbol=\overset{1}{\subset}] & \mathcal F_{r_{i+\ell_i}-\ell_i+b_i+1}^{\perp} \arrow[d,equal] \arrow[r,symbol=\overset{1}{\subset}] & \ldots \arrow[r,symbol=\overset{1}{\subset}] & U_{r_{i+\ell_i}}^{\perp} \arrow[d,equal] \\
        U_{r_{-i-\underline{\varepsilon}}} \arrow[r,symbol=\overset{1}{\subset}] & \ldots \arrow[r,symbol=\overset{1}{\subset}] & \mathcal F_{r_{-i-\underline{\varepsilon}}-b_i+1} \arrow[r,symbol=\overset{1}{\subset}] & \mathcal F_{r_{-i-\underline{\varepsilon}}-b_i} \arrow[r,symbol=\overset{1}{\subset}] & \ldots \arrow[r,symbol=\overset{1}{\subset}] & \mathcal F_{r_{-i-\underline{\varepsilon}}-\ell_i+1}
    \end{tikzcd}
    \end{center}
    The first $b_i$ vertical relations are inequalities, while the last $\ell_i-b_i$ vertical relations are equalities. Let $0\leq d\leq\ell_i-1$. By \eqref{Eq:Case1IntervalInclusion} we have $r_{-i-\underline{\varepsilon}-d} = r_{-i-\underline{\varepsilon}}-d$. Thus, the lower row contains the term $\mathcal F_{r_{-i-\underline{\varepsilon}-d}}$, and the term lying directly above it is $\mathcal F_{r_{i+\ell_i}-\ell_i+d+1}^{\perp}$. Moreover,
    \begin{equation*}
        r_{i+\ell_i}-\ell_i+d+1 = \begin{cases}
        -r_{-i-\underline{\varepsilon}-d} -\underline{\varepsilon}+1 & \text{if } V \not  = W_{\mathrm{max}},\\
        -r_{-i-\underline{\varepsilon}-d} -\underline{\varepsilon}+1+n & \text{if } V = W_{\mathrm{max}}.
        \end{cases}
    \end{equation*}
    By taking orthogonal complements and using $\pi$-periodicity when $V=W_{\mathrm{max}}$, we therefore obtain
    \begin{equation*}
        \mathcal F_{r_{-i-\underline{\varepsilon}-d}} = \mathcal F_{r_{i+\ell_i}-\ell_i+d+1}^{\perp} \iff \tau\left( \widetilde Z_{-r_{-i-\underline{\varepsilon}-d} -\underline{\varepsilon}+1} \right) \subseteq \pi \widetilde Z_{r_{-i-\underline{\varepsilon}-d}}^\vee.
    \end{equation*}
    The diagram shows that these two flag terms are distinct for $0\leq d<b_i$ and equal for $b_i\leq d\leq\ell_i-1$. Consequently, we have proved \eqref{Eq:CriterionToBeCheckedOnFlags} for every $i\in\overline I\setminus\mathcal T$.\\

    \textbf{Case 2:} Let $i\in\mathcal T$, so that $c_i=m_i$. There is a unique quotient space $V$ such that 
    \begin{equation}\label{Eq:Case2IntervalInclusion}
        \{r_{-i-\underline{\varepsilon}}-m_i,\ldots, r_{-i-\underline{\varepsilon}}\} \subseteq \widehat{\mathcal K}_V.
    \end{equation}
    Moreover, we have $r_{-i-\underline{\varepsilon}}-m_i=\alpha_V$. The block $w_i(\mathbf c) = s_{[r_{-i-\underline{\varepsilon}}-m_i, r_{-i-\underline{\varepsilon}}-1]}$ gives the following diagram inside $V_k$.
    \begin{center}
    \begin{tikzcd}[column sep=small]
        \mathcal F_{r_{i+1}}^{\perp} \arrow[d,equal,"/" marking] \arrow[dr,hook,shorten >=10pt] \arrow[r,symbol=\overset{1}{\subset}] & \ldots \arrow[r,symbol=\overset{1}{\subset}] & \mathcal F_{r_{i+m_i}}^{\perp} \arrow[d,equal,"/" marking] \arrow[dr,hook,shorten >=10pt] \arrow[r,symbol=\overset{1}{\subset}] & \{0\}^{\perp} \arrow[d,equal] \\
        U_{r_{-i-\underline{\varepsilon}}} \arrow[r,symbol=\overset{1}{\subset}] & \ldots \arrow[r,symbol=\overset{1}{\subset}] & \mathcal F_{r_{-i-\underline{\varepsilon}}-m_i+1} \arrow[r,symbol=\overset{1}{\subset}] & V_k
    \end{tikzcd}
    \end{center}
    If $m_i=0$, this diagram consists only of its rightmost vertical equality. Let $0 \leq d \leq m_i$. By \eqref{Eq:Case2IntervalInclusion}, we have $r_{-i-\underline{\varepsilon}-d} = r_{-i-\underline{\varepsilon}}-d$. Thus, the lower row contains the term $\mathcal F_{r_{-i-\underline{\varepsilon}-d}}$, and the term lying directly above it is $\mathcal F_{r_{i+d+1}}^{\perp}$. Moreover,
    \begin{equation*}
        r_{i+d+1} = \begin{cases}
        -r_{-i-\underline{\varepsilon}-d} -\underline{\varepsilon}+1 & \text{if } V \not = W_{\mathrm{max}},\\
        -r_{-i-\underline{\varepsilon}-d} -\underline{\varepsilon}+1+n & \text{if }V = W_{\mathrm{max}}.
        \end{cases}
    \end{equation*}
    As in Case 1, taking orthogonal complements and using $\pi$-periodicity when $V=W_{\mathrm{max}}$ gives
    \begin{equation*}
        \mathcal F_{r_{-i-\underline{\varepsilon}-d}} = \mathcal F_{r_{i+d+1}}^{\perp} \iff \tau\left(\widetilde Z_{-r_{-i-\underline{\varepsilon}-d} -\underline{\varepsilon}+1} \right) \subseteq \pi\widetilde Z_{r_{-i-\underline{\varepsilon}-d}}^\vee.
    \end{equation*}
    By the diagram, we deduce that $\tau\left(\widetilde Z_{-r_{-i-\underline{\varepsilon}-d} -\underline{\varepsilon}+1} \right) \not \subseteq \pi\widetilde Z_{r_{-i-\underline{\varepsilon}-d}}^\vee$ when $0 \leq d < m_i$, but that this inclusion holds for $d = m_i$. 
    
    It remains to prove that it also holds for $m_i+1 \leq d \leq \ell_i-1$. We do this by distinguishing on the size of the $\theta_I$-orbit $O$ containing $i$.\\

    \textbf{Case 2.a:} Suppose first that $\#O=2$. Up to exchanging $i$ and $\theta_I(i)$, we may assume that $i = \overline{r_O^+}$ and $r_O^- = -r_O^+ - \underline{\varepsilon}-\ell_i$. We have
    \begin{align*}
        \widehat{\mathcal K}_{H_O} & = \{r_O^-+m_{\theta_I(i)}+1,\ldots, r_O^-+\ell_i-m_i-1\},\\
        \widehat{\mathcal K}_{H_O'} & = \{r_O^++m_i+1,\ldots, r_O^++\ell_i-m_{\theta_I(i)}-1\}.
    \end{align*}
    Since the components of $w_{\mathbf c}$ on $H_O$ and $H_O'$ are trivial, we have the following diagram inside $(H_O)_k$.
    \begin{center}
    \begin{tikzcd}[column sep=small]
    (H_O')_k^{\perp} \arrow[d,equal] \arrow[r,symbol=\overset{1}{\subset}] & \mathcal F_{r_O^++m_i+2}^{\perp} \arrow[d,equal] \arrow[r,symbol=\overset{1}{\subset}] & \ldots \arrow[r,symbol=\overset{1}{\subset}] & \mathcal F_{r_O^++\ell_i-m_{\theta_I(i)}-1}^{\perp} \arrow[d,equal] \arrow[r,symbol=\overset{1}{\subset}] & \{0\}^{\perp} \arrow[d,equal] \\
    \{0\} \arrow[r,symbol=\overset{1}{\subset}] & \mathcal F_{r_O^-+\ell_i-m_i-1} \arrow[r,symbol=\overset{1}{\subset}] & \ldots \arrow[r,symbol=\overset{1}{\subset}] & \mathcal F_{r_O^-+m_{\theta_I(i)}+2} \arrow[r,symbol=\overset{1}{\subset}] & (H_O)_k.
    \end{tikzcd}
    \end{center}
    The leftmost equality of this diagram is the equality already obtained for $d=m_i$ in Case 2 above. Let $m_i < d \leq \ell_i-m_{\theta_I(i)}-1$. Then $r_O^- + \ell_i-d \in \widehat{\mathcal K}_{H_O}$ and
    \begin{equation*}
        r_O^-+\ell_i-d \equiv-r_O^+ -\underline{\varepsilon}-d \equiv -i-\underline{\varepsilon}-d \mod n.
    \end{equation*}
    It follows that $r_{-i-\underline{\varepsilon}-d} = r_O^-+\ell_i-d$. The term lying directly above $\mathcal F_{r_O^-+\ell_i-d}$ is $\mathcal F_{r_O^++d+1}^{\perp}$, so that we have
    \begin{equation*}
        r_O^+ +d+1 = -r_{-i-\underline{\varepsilon}-d} -\underline{\varepsilon}+1.
    \end{equation*}
    Thus, each equality displayed in the diagram translates into the inclusions $\tau(\widetilde Z_{-r_{-i-\underline{\varepsilon}-d} -\underline{\varepsilon}+1}) \subseteq \pi\widetilde Z_{r_{-i-\underline{\varepsilon}-d}}^\vee$, for all $m_i < d \leq \ell_i-m_{\theta_I(i)}-1$.

    It remains to consider the case $\ell_i - m_{\theta_I(i)} \leq d \leq \ell_i-1$. This range is empty if $m_{\theta_I(i)} = 0$. Assume therefore that $m_{\theta_I(i)} > 0$. We have
    \begin{equation*}
        \{r_O^-+1,\ldots,r_O^-+m_{\theta_I(i)}\} \subseteq \begin{cases} 
        \widehat{\mathcal K}_{W_{O,O'}} & \text{if } O \prec O',\\ 
        \widehat{\mathcal K}_{W_{\mathrm{max}}} & \text{if } O = O_{\mathrm{max}},
    \end{cases}
    \end{equation*}
    where $O'$ is the unique orbit such that $O\prec O'$ in the first case. Let $V$ denote the quotient space appearing on the right-hand side. We have the following diagram inside $V_k$.
    \begin{center}
    \begin{tikzcd}[column sep=small]
        \mathcal F_{r_{i+\ell_i}-m_{\theta_I(i)}+1}^{\perp} \arrow[d,equal] \arrow[r,symbol=\overset{1}{\subset}] & \ldots \arrow[r,symbol=\overset{1}{\subset}] & U_{r_{i+\ell_i}}^{\perp} \arrow[d,equal] \\
        \mathcal F_{r_O^-+m_{\theta_I(i)}} \arrow[r,symbol=\overset{1}{\subset}] & \ldots \arrow[r,symbol=\overset{1}{\subset}] & \mathcal F_{r_O^-+1}.
    \end{tikzcd}
    \end{center}
    All the vertical relations are equalities. Indeed, the integer $r_O^-+\ell_i-d$ belongs to $\widehat{\mathcal K}_V$ and is congruent to $-i-\underline{\varepsilon}-d$ modulo $n$. Therefore $r_{-i-\underline{\varepsilon}-d} = r_O^-+\ell_i-d$. Under the identification $\mathbf W_V\simeq\mathrm{Bij}(\widehat{\mathcal K}_V)$, none of the simple reflections $s_{r_{-i-\underline{\varepsilon}-d}-1}$ for $\ell_i - m_{\theta_I(i)} \leq d \leq \ell_i-1$ occurs in $w_{\mathbf c}$. For otherwise, it would occur in some factor $w_j(\mathbf c)$, and by definition of $w_j(\mathbf c)$ there would exist an integer $0 \leq e < c_j$ such that
    \begin{equation*}
        r_{-i-\underline{\varepsilon}-d}-1 = r_{-j-\underline{\varepsilon}}-e-1.
    \end{equation*}
    Reducing modulo $n$ would give $i+d=j+e$. Since $i+d\in B_i$ and $j+e\in B_j$, the fact that the subsets $B_j$ form a partition of $\mathbb Z/n\mathbb Z$ implies that $j=i$ and $e=d$. This is impossible because $e < c_i = m_i < d$.\\
    The lower terms of the diagram have the form $\mathcal F_{r_{-i-\underline{\varepsilon}-d}}$, whereas the term directly above is $\mathcal F_{r_{i+\ell_i}-\ell_i+d+1}^{\perp}$. Their indices satisfy
    \begin{equation*}
        r_{i+\ell_i}-\ell_i+d+1 = \begin{cases} 
        -r_{-i-\underline{\varepsilon}-d} -\underline{\varepsilon}+1 & \text{if } V \not = W_{\mathrm{max}},\\ 
        -r_{-i-\underline{\varepsilon}-d} -\underline{\varepsilon}+1+n & \text{if } V = W_{\mathrm{max}}.
    \end{cases}
    \end{equation*}
    Taking orthogonal complements and using $\pi$-periodicity when $V = W_{\mathrm{max}}$ now gives
    \begin{equation*}
        \tau\left(\widetilde Z_{-r_{-i-\underline{\varepsilon}-d} -\underline{\varepsilon}+1} \right) \subseteq \pi\widetilde Z_{r_{-i-\underline{\varepsilon}-d}}^\vee.
    \end{equation*}
    This concludes the proof of \eqref{Eq:CriterionToBeCheckedOnFlags} when $i=\overline{r_O^+}$. If $i=\overline{r_O^-}$, the same argument applies after interchanging $i$ and $\theta_I(i)$, and $H_O$ and $H_O'$.\\

    \textbf{Case 2.b:} Suppose now that $O = \{i\}$. We then have $\theta_I(i) = i$ and $2m_i \leq \ell_i-1$. Moreover, 
    \begin{equation*}
        2r_O+\underline{\varepsilon}+\ell_i = \begin{cases}
        n & \text{if } r_O\in I^+,\\
        0 & \text{if } r_O\in I^-.
        \end{cases}
    \end{equation*}
    Set $V = W_{\mathrm{max}}$ if $r_O\in I^+$, and $V = W_{\mathrm{min}}$ if $r_O\in I^-$. Then
    \begin{equation*}
        \widehat{\mathcal K}_V = \{r_O+m_i+1,\ldots,r_O+\ell_i-m_i-1\},
    \end{equation*}
    and we have the following diagram inside $V_k$.
    \begin{center}
    \begin{tikzcd}[column sep=small]
        V_k^{\perp} \arrow[d,equal] \arrow[r,symbol=\overset{1}{\subset}] & \mathcal F_{r_O+m_i+2}^{\perp} \arrow[d,equal] \arrow[r,symbol=\overset{1}{\subset}] & \ldots \arrow[r,symbol=\overset{1}{\subset}] & \mathcal F_{r_O+\ell_i-m_i-1}^{\perp} \arrow[d,equal] \arrow[r,symbol=\overset{1}{\subset}] & \{0\}^{\perp} \arrow[d,equal] \\
        \{0\} \arrow[r,symbol=\overset{1}{\subset}] & \mathcal F_{r_O+\ell_i-m_i-1} \arrow[r,symbol=\overset{1}{\subset}] & \ldots \arrow[r,symbol=\overset{1}{\subset}] & \mathcal F_{r_O+m_i+2} \arrow[r,symbol=\overset{1}{\subset}] & V_k.
    \end{tikzcd}
    \end{center}
    The leftmost equality of this diagram is the equality already obtained for $d=m_i$ in Case 2 above. Let $m_i < d \leq \ell_i-m_i-1$. Then $r_O+\ell_i-d\in\widehat{\mathcal K}_V$ and we have
    \begin{equation*}
        r_O+\ell_i-d \equiv -i-\underline{\varepsilon}-d \mod n.
    \end{equation*}
    Consequently, $r_{-i-\underline{\varepsilon}-d} = r_O+\ell_i-d$. The term lying directly above $\mathcal F_{r_O+\ell_i-d}$ is $\mathcal F_{r_O+d+1}^{\perp}$, so that we have
    \begin{equation*}
        r_O+d+1 = \begin{cases}
        -r_{-i-\underline{\varepsilon}-d} -\underline{\varepsilon}+1+n & \text{if } r_O \in I^+,\\
        -r_{-i-\underline{\varepsilon}-d} -\underline{\varepsilon}+1 & \text{if } r_O \in I^-.
    \end{cases}
    \end{equation*}
    We deduce that the equalities displayed in the diagram give the inclusion $\tau(\widetilde Z_{-r_{-i-\underline{\varepsilon}-d} -\underline{\varepsilon}+1}) \subseteq \pi\widetilde Z_{r_{-i-\underline{\varepsilon}-d}}^\vee$ for all $m_i < d \leq \ell_i-m_i-1$.

    It remains to consider the case $\ell_i-m_i\leq d\leq\ell_i-1$. This range is empty if $m_i = 0$. Assume that $m_i > 0$. We have
    \begin{equation*}
    \{r_O+1,\ldots,r_O+m_i\} \subseteq \begin{cases}
        \widehat{\mathcal K}_{W_{O',O}'} & \text{if } r_O \in I^+ \text{ and } O'\prec O,\\
        \widehat{\mathcal K}_{W_{\mathrm{min}}} & \text{if } r_O \in I^+ \text{ and }O = O_{\mathrm{min}},\\
        \widehat{\mathcal K}_{W_{O,O'}} & \text{if } r_O \in I^- \text{ and } O\prec O',\\
        \widehat{\mathcal K}_{W_{\mathrm{max}}} & \text{if } r_O \in I^- \text{ and }O = O_{\mathrm{max}},
        \end{cases}
    \end{equation*}
    where $O'$ is the unique adjacent orbit in the indicated direction. Let $V$ denote the quotient space appearing on the right-hand side. We have the following diagram inside $V_k$.
    \begin{center}
    \begin{tikzcd}[column sep=small]
        \mathcal F_{r_{i+\ell_i}-m_i+1}^{\perp} \arrow[d,equal] \arrow[r,symbol=\overset{1}{\subset}] & \ldots \arrow[r,symbol=\overset{1}{\subset}] & U_{r_{i+\ell_i}}^{\perp} \arrow[d,equal] \\
        \mathcal F_{r_O+m_i} \arrow[r,symbol=\overset{1}{\subset}] & \ldots \arrow[r,symbol=\overset{1}{\subset}] & \mathcal F_{r_O+1}.
    \end{tikzcd}
    \end{center}
    As in Case 2.a, all the vertical relations are equalities. Since $r_O+\ell_i-d \in \widehat{\mathcal K}_V$ and this integer is congruent to $-i-\underline{\varepsilon}-d$ modulo $n$, we have $r_{-i-\underline{\varepsilon}-d} = r_O+\ell_i-d$. The lower terms in the diagram are therefore $\mathcal F_{r_{-i-\underline{\varepsilon}-d}}$, and the term directly above is $\mathcal F_{r_{i+\ell_i}-\ell_i+d+1}^{\perp}$. Their indices satisfy
    \begin{equation*}
        r_{i+\ell_i}-\ell_i+d+1 = \begin{cases}
            -r_{-i-\underline{\varepsilon}-d} -\underline{\varepsilon}+1 & \text{if } V \not = W_{\mathrm{max}},\\
            -r_{-i-\underline{\varepsilon}-d} -\underline{\varepsilon}+1+n & \text{if }V = W_{\mathrm{max}}.
        \end{cases}
    \end{equation*}
    Taking orthogonal complements and using $\pi$-periodicity in the latter case shows that
    \begin{equation*}
        \tau\left(\widetilde Z_{-r_{-i-\underline{\varepsilon}-d} -\underline{\varepsilon}+1} \right) \subseteq \pi\widetilde Z_{r_{-i-\underline{\varepsilon}-d}}^\vee.
    \end{equation*}
    This concludes the proof of \eqref{Eq:CriterionToBeCheckedOnFlags}, as well as Case 2.\\

    Combining Cases 1 and 2, we deduce that $J^{\mathrm{Iw}}$ is given by the formula of \eqref{Eq:ComputationOfJIwahori}. It follows that $v_{J^{\mathrm{Iw}}} = v_{\mathbf c} \in \mathrm{Adm}(\{\mu'\})_0$. At Iwahori level, the KR and EKOR stratifications agree. Moreover, since we also have $v_{\mathbf c} \in {}^I\mathrm{Adm}(\{\mu'\})$, the change of parahoric morphism $\mathcal N_0^{I_{\mathrm{Iw}},\varepsilon} \to \mathcal N_0^{I,\varepsilon}$ maps the KR stratum indexed by $v_{\mathbf c}$ onto $\mathrm{EKOR}_{v_{\mathbf c}}$, see \cite{shenEKORStrataShimura2021} Proposition 3.5.2 and \cite{heStratificationsReductionShimura2017} Proposition 6.11. Since $\widetilde z$ maps to $z$, we obtain $z \in \mathrm{EKOR}_{v_{\mathbf c}}(k)$, and this concludes the proof.
\end{proof}

We record the compatibility of the fine Bruhat-Tits and EKOR stratifications with the duality isomorphism $\Phi_I : \mathcal N^{I,\varepsilon} \xrightarrow{\sim} \mathcal N^{\widetilde I,\varepsilon'}$ constructed in Proposition \ref{Prop:IsomorphismDualityArbitraryParahoric}, where $\varepsilon' = (-1)^n \varepsilon$ and $\widetilde I$ is defined in Equation \ref{Eq:DefiWidetildeI}. Recall also the bijection $\varphi_I : \overline I \xrightarrow{\sim} \overline{\widetilde I}$ introduced in \ref{Eq:DefinitionMapVarphiI}. Proposition \ref{Prop:DualityKRStrata} describes the compatibility of $\Phi_I$ with the KR stratifications. We now establish the analogous statement for the Bruhat-Tits and EKOR stratifications.

Let $\mathcal T \subseteq \overline I$ be non-empty and $\theta_I$-stable, let $\mathbf m \in \mathrm{Num}(\mathcal T)$ and let $\mathbf b \in \mathcal B_{\mathcal T}$. We define
\begin{equation*}
    \widetilde{\mathcal T} \coloneq \varphi_I(\mathcal T) \subseteq \overline{\widetilde I},
\end{equation*}
together with tuples $\widetilde{\mathbf m}$ and $\widetilde{\mathbf b}$ determined by
\begin{equation*}
    \widetilde m_{\varphi_I(i)} \coloneq m_{\theta_I(i)} \quad \text{for all } i \in \mathcal T, \qquad \widetilde b_{\varphi_I(i)} \coloneq b_{\theta_I(i)} \quad \text{for all } i \in \overline I \setminus \mathcal T.
\end{equation*}
If $\mathbf c = \mathbf c(\mathbf b)$ is the tuple defined in Equation \ref{Eq:DefinitionTuplec(b)}, and if $\widetilde{\mathbf c} = \mathbf c(\widetilde{\mathbf b})$ denotes its counterpart for $\widetilde I$, then $\widetilde c_{\varphi_I(i)} = c_{\theta_I(i)}$ for all $i \in \overline I$. The properties of $\varphi_I$ established in Section \ref{Section3.2} show that $\widetilde{\mathcal T}$ is $\theta_{\widetilde I}$-stable, that $\widetilde{\mathbf m} \in \mathrm{Num}(\widetilde{\mathcal T})$, and that $\widetilde{\mathbf b} \in \mathcal B_{\widetilde{\mathcal T}}$.

\begin{prop}\label{Prop:DualityFineBTAndEKORStrata}
    With the notation above, the duality isomorphism induces isomorphisms
    \begin{equation*}
        \Phi_I: X_{\mathcal T,\mathbf m,\mathbf b}^{I,\varepsilon} \xrightarrow{\sim} X_{\widetilde{\mathcal T}, \widetilde{\mathbf m}, \widetilde{\mathbf b}}^{\widetilde I,\varepsilon'}, \qquad \Phi_I: \mathrm{EKOR}_{v_{\mathbf c}}^{I,\varepsilon} \xrightarrow{\sim} \mathrm{EKOR}_{v_{\widetilde{\mathbf c}}}^{\widetilde I, \varepsilon'}.
    \end{equation*}
\end{prop}

\begin{proof}
    We prove the assertion for the EKOR strata. The corresponding statement for $X_{\mathcal T, \mathbf m, \mathbf b}$ is a direct consequence of Theorem \ref{Theo:ComparisonBTEKOR}. Let
    \begin{equation*}
        p_I:\mathcal N_{\mathrm{red}}^{I_{\mathrm{Iw}},\varepsilon} \longrightarrow \mathcal N_{\mathrm{red}}^{I,\varepsilon}, \qquad p_{\widetilde I}: \mathcal N_{\mathrm{red}}^{I_{\mathrm{Iw}},\varepsilon'} \longrightarrow \mathcal N_{\mathrm{red}}^{\widetilde I,\varepsilon'},
    \end{equation*}
    denote the change of parahoric morphisms. Recall that $\nu_{\varepsilon} = \nu_{\varepsilon'}$, so that the Iwahori index set $I_{\mathrm{Iw}}$ is the same for the source and the target. The definition of the duality isomorphisms in Proposition \ref{Prop:IsomorphismDualityArbitraryParahoric} gives a commutative diagram
    \begin{center}
    \begin{tikzcd}
        \mathcal N_{\mathrm{red}}^{I_{\mathrm{Iw}},\varepsilon} \arrow[r,"\Phi_{I_{\mathrm{Iw}}}"] \arrow[d,swap,"p_I"] & \mathcal N_{\mathrm{red}}^{I_{\mathrm{Iw}},\varepsilon'} \arrow[d,"p_{\widetilde I}"] \\
        \mathcal N_{\mathrm{red}}^{I,\varepsilon} \arrow[r,"\Phi_I"] & \mathcal N_{\mathrm{red}}^{\widetilde I,\varepsilon'}.
    \end{tikzcd}
    \end{center}
    Associated with $\mathbf c$ and $\widetilde{\mathbf c}$, consider the Iwahori-level KR indices
    \begin{equation*}
        J_{\mathbf c}^{\mathrm{Iw}} \coloneq \bigsqcup_{i\in\overline I} \{i, i+1, \ldots, i+c_i-1\}, J_{\widetilde{\mathbf c}}^{\mathrm{Iw}} \coloneq \bigsqcup_{\widetilde i\in\overline{\widetilde I}} \{\widetilde i, \widetilde i+1, \ldots, \widetilde i+\widetilde c_{\widetilde i}-1\}.
    \end{equation*}
    By the description \eqref{Eq:ComputationOfJIwahori} and the change of parahoric result of \cite{shenEKORStrataShimura2021} Proposition 3.5.2 and \cite{heStratificationsReductionShimura2017} Proposition 6.11, we have
    \begin{equation}\label{Eq:EKORAsImageIwahoriKRUnderDuality}
        p_I\left(\mathrm{KR}_{J_{\mathbf c}^{\mathrm{Iw}}}^{I_{\mathrm{Iw}},\varepsilon}\right) = \mathrm{EKOR}_{v_{\mathbf c}}^{I,\varepsilon}, \qquad p_{\widetilde I}\left( \mathrm{KR}_{J_{\widetilde{\mathbf c}}^{\mathrm{Iw}}}^{I_{\mathrm{Iw}},\varepsilon'} \right) = \mathrm{EKOR}_{v_{\widetilde{\mathbf c}}}^{\widetilde I,\varepsilon'}.
    \end{equation}
    At Iwahori level, we have $\varphi_{I_{\mathrm{Iw}}} \left(\theta_{I_{\mathrm{Iw}}}(i)\right) = i+\nu_{\varepsilon}+\underline{\varepsilon}$ for every $i \in \mathbb Z/n\mathbb Z$. Proposition \ref{Prop:DualityKRStrata} therefore gives
    \begin{equation*}
        \Phi_{I_{\mathrm{Iw}}}\left(\mathrm{KR}_{J_{\mathbf c}^{\mathrm{Iw}}}^{I_{\mathrm{Iw}},\varepsilon} \right) = \mathrm{KR}_{J_{\mathbf c}^{\mathrm{Iw}}+\nu_{\varepsilon}+\underline{\varepsilon}}^{I_{\mathrm{Iw}},\varepsilon'}.
    \end{equation*}
    On the other hand, Equation \eqref{Eq:DefinitionMapVarphiI} gives $\varphi_I(\theta_I(i)) = i+\nu_{\varepsilon}+\underline{\varepsilon}$. Since $\widetilde c_{\varphi_I(\theta_I(i))} = c_i$, we obtain
    \begin{equation*}
        J_{\mathbf c}^{\mathrm{Iw}} +\nu_{\varepsilon}+\underline{\varepsilon} = \bigsqcup_{i\in\overline I} \left\{ \varphi_I(\theta_I(i)), \ldots, \varphi_I(\theta_I(i))+c_i-1 \right\} = J_{\widetilde{\mathbf c}}^{\mathrm{Iw}}.
    \end{equation*}
    The commutative diagram and
    \eqref{Eq:EKORAsImageIwahoriKRUnderDuality} now give
    \begin{align*}
        \Phi_I\left(\mathrm{EKOR}_{v_{\mathbf c}}^{I,\varepsilon} \right) = \Phi_I\left(p_I\left(\mathrm{KR}_{J_{\mathbf c}^{\mathrm{Iw}}}^{I_{\mathrm{Iw}},\varepsilon} \right)\right) & = p_{\widetilde I}\left( \Phi_{I_{\mathrm{Iw}}}\left( \mathrm{KR}_{J_{\mathbf c}^{\mathrm{Iw}}}^{I_{\mathrm{Iw}},\varepsilon} \right)\right)\\
        & = p_{\widetilde I}\left( \mathrm{KR}_{J_{\widetilde{\mathbf c}}^{\mathrm{Iw}}}^{I_{\mathrm{Iw}},\varepsilon'} \right) = \mathrm{EKOR}_{v_{\widetilde{\mathbf c}}}^{\widetilde I,\varepsilon'},
    \end{align*}
    which concludes the proof.
\end{proof}

Recall the description of the indexing set of the EKOR stratification given in Proposition \ref{Prop:DescriptionIndexSetOfEKORStrata}.

\begin{corol}\label{Corol:NonEmptyEKORStrata}
    Let $v \in {}^{\mathbb S_a^I}\widetilde W \cap \mathrm{Adm}(\{\mu'\})_{0}^I$. The EKOR stratum $\mathrm{EKOR}_{v}$ is non-empty if and only if $v \in {}^{I}\mathrm{Adm}(\{\mu'\})_0$. In this case, it is smooth of pure dimension $\ell(v)$. 
\end{corol}

\begin{proof}
    The non-emptiness statement is a direct consequence of Theorem \ref{Theo:ComparisonBTEKOR} and Proposition \ref{Prop:BijectionIndexSetsBTEKOR}. As for the dimension, while it is already well known from \cite{shenEKORStrataShimura2021} Theorem 3.4.12, it also follows from \eqref{Eq:DimensionFineDLVarieties}. 
\end{proof}

In \cite{shenEKORStrataShimura2021} Theorem 3.4.12, for $J \in \mathcal P_I$, the $J$-ordinary locus of any KR stratum $\mathrm{KR}_J$ is defined, in the context of Shimura varieties, as the unique top dimensional EKOR stratum contained in $\mathrm{KR}_J$. When $J \in \mathcal P_{0,I}$, it is natural to ask whether the KR stratum $\mathrm{KR}_J$ is fully included in the basic locus or not. This amounts to determining whether the $J$-ordinary locus is in the basic locus or not. Corollary \ref{Corol:NonEmptyEKORStrata} allows us to answer this question. 

\begin{prop}\label{Prop:EntirelyBasicKRStrata}
    Let $J \in \mathcal P_{0,I}$. The KR stratum indexed by $J$ is basic if and only if there exists $i \in \overline I \setminus (J \cup \theta_I(J))$ such that $\ell_i = 1$.
\end{prop}

\begin{proof}
    The index of the $J$-ordinary locus is $v_{\mathbf c^{\mathrm{top}}}$ where $\mathbf c^{\mathrm{top}} = (c_i^{\mathrm{top}})_{i \in \overline I}$ is given by 
    \begin{equation*}
        c_i^{\mathrm{top}} \coloneq \begin{cases}
            \ell_i & \text{if } i \in J, \\
            \ell_i-1 & \text{if } i \in \overline I \setminus J.
        \end{cases}
    \end{equation*}
    Then $v_{\mathbf c^{\mathrm{top}}}$ belongs to ${}^{I}\mathrm{Adm}(\{\mu'\})_0$ if and only if there is a $\theta_I$-orbit $O = \{i, \theta_I(i)\} \subseteq \overline I \setminus J$ such that $c_i^{\mathrm{top}} + c_{\theta_I(i)}^{\mathrm{top}} = 2(\ell_i-1) \leq \ell_i-1$. This is only possible if $\ell_i = 1$.
\end{proof}

Another natural question to ask is whether the irreducible components of the closure of EKOR strata are smooth or not. We can answer this question for some of them.

\begin{prop}\label{Prop:SmoothnessIrreducibleComponentsClosureEKOR}
    Let $v_{\mathbf c} \in {}^{I}\mathrm{Adm}(\{\mu'\})_{0}$. Suppose that for every $i \in \overline I$, we have 
    \begin{equation*}
        c_i \leq \ell_i - 1 \implies c_i + c_{\theta_I(i)} \leq \ell_i-1,
    \end{equation*}
    then every irreducible component of the Zariski closure $\overline{\mathrm{EKOR}_{v_{\mathbf c}}}$ is smooth.
\end{prop}

\begin{proof}
    As we will prove in Lemma \ref{Lem:ClosureXTmbIsUnion}, given a non-empty $\theta_I$-stable subset $\mathcal T \subseteq \overline I$, a tuple $\mathbf m \in \mathrm{Num}(\mathcal T)$ and an element $\mathbf b \in \mathcal B_{\mathcal T}$, the closure of $X_{\mathcal T, \mathbf m, \mathbf b}$ is the union of the closures of $g\cdot X_{\mathcal K_{\mathcal T,\mathbf m}}\{w_{\mathbf c}^{-1}\}$ for $g \in J(F)/\mathrm{Stab}_{J(F)}(\mathbf{\Lambda}(\mathbf m))$. If we have $b_i = \ell_i$ for all $\overline I \setminus \mathcal T$, then the closure of $g\cdot X_{\mathcal K_{\mathcal T,\mathbf m}}\{w_{\mathbf c}^{-1}\}$ is 
    \begin{equation*}
        \overline{g\cdot X_{\mathcal K_{\mathcal T,\mathbf m}}\{w_{\mathbf c}^{-1}\}} = \overline{\mathcal N_{\mathcal T, g\cdot\mathbf{\Lambda}(\mathbf m)}^{0}},
    \end{equation*}
    that is a closed BT stratum. According to \cite{mullerBruhatTitsStratificationBasic2026} Corollary 3.37, the closed BT strata are smooth.\\
    Now, let $(\mathcal T_{\mathbf c},\mathbf m_{\mathbf c},\mathbf b_{\mathbf c})$ be the triple associated with $\mathbf c$ by Proposition \ref{Prop:BijectionIndexSetsBTEKOR}, and let $J_{\mathbf c}$ as in \eqref{Eq:KRIndexAttachedToEKORIndex}. We have $X_{\mathcal T_{\mathbf c}, \mathbf m_{\mathbf c}, \mathbf b_{\mathbf c}} = \mathrm{EKOR}_{v_{\mathbf c}}$ by Theorem \ref{Theo:ComparisonBTEKOR}. The proof is over once we notice that we have $(\mathbf b_{\mathbf c})_i = \ell_i$ for all $i \in \overline I \setminus \mathcal T_{\mathbf c}$ if and only if $\mathcal T_{\mathbf c} = \overline I \setminus J_{\mathbf c}$.
\end{proof}

While we may conjecture that Proposition \ref{Prop:SmoothnessIrreducibleComponentsClosureEKOR} remains true for all EKOR strata, the proof of \cite{mullerBruhatTitsStratificationBasic2026} does not generalize well. Indeed, the approach used in Propositions 3.17, 3.21 and 3.25 to prove that the closure of the fine Deligne-Lusztig varieties involved is smooth, fails for $X_{\mathcal K_{\mathcal T,\mathbf m}}\{w_{\mathbf c}^{-1}\}$ when some $b_i$ is less than $\ell_i$. Namely, the proof involves identifying the closure of $X_{\mathcal K_{\mathcal T,\mathbf m}}\{w_{\mathbf c}^{-1}\}$ with the closure of a coarse Deligne-Lusztig variety (terminology of \cite{mullerBruhatTitsStratificationBasic2026} Definition 3.5), which is smoothly equivalent to the closure of an associated Schubert variety. Then, one may use some pattern avoidance criterion to deduce smoothness. When $b_i < \ell_i$ for some $i$, it appears that there is no coarse Deligne-Lusztig variety whose closure agrees with the closure of $X_{\mathcal K_{\mathcal T,\mathbf m}}\{w_{\mathbf c}^{-1}\}$, so that the proof does not carry over.

\section{Closure relations}\label{Section6}

\subsection{Auxiliary combinatorics}\label{Section6.1}

We begin with the combinatorial statement which will be used to compute the closure of the fine Deligne-Lusztig varieties defined in Section \ref{Section5.2}. In this paragraph only, let $d\geq 1$ and consider the symmetric group $\mathfrak S_d$, with simple reflections $s_i=(i\ i+1)$ for $1\leq i\leq d-1$. Let $\ell$ and $\leq$ denote the associated length function and Bruhat order, and let $\omega_0\in\mathfrak S_d$ be the element of maximal length. Let $H \coloneq \langle s_1,\ldots,s_{d-2} \rangle \subseteq \mathfrak S_d$ denote the subgroup of permutations fixing $d$. For $0 \leq a \leq d-1$, we define
\begin{equation*}
    u_a\coloneq s_{d-a}s_{d-a+1}\cdots s_{d-1}\in\mathfrak S_d,
\end{equation*}
with the convention that $u_0 = \mathrm{id}$.

\begin{prop}\label{Prop:CombinatorialProposition}
    Let $a, a', b, b' \in \{0, \ldots, d-1\}$. There exist $x,y \in H$ such that
    \begin{equation*}
        \mathrm{Ad}(\omega_0)(y)u_{a'}x^{-1} \leq u_a, \qquad \mathrm{Ad}(\omega_0)(x)u_{b'}y^{-1} \leq u_b,
    \end{equation*}
    if and only if
    \begin{equation*}
        (a' \leq a \text{ and } b' \leq b) \quad \text{or} \quad (a'+b' \leq d-2 \text{ and } a+b \geq d-1).
    \end{equation*}
\end{prop}

\begin{proof}
    The assertion is immediate when $d=1$, so we assume $d\geq2$. We start with the reverse implication. If $a' \leq a$ and $b' \leq b$, we may take $x = y = \mathrm{id}$. Assume now that $a'+b' \leq d-2$ and $a+b \geq d-1$. We may also assume, up to exchanging $(a,a')$ with $(b,b')$, that $a' > a$. It follows that $b' < b$. For $1 \leq s \leq t \leq d$, let $\omega_{[s,t]} \in \mathfrak S_d$ denote the permutation which reverses $\{s,\ldots,t\}$ and fixes its complement. We take
    \begin{equation*}
        x \coloneq \omega_{[d-a',d-a-1]} \in H, \qquad y \coloneq \omega_{[a+1,a'+1]} \in H.
    \end{equation*}
    It is clear that $x \in H$, whereas for $y$ it follows from the hypothesis $a' \leq d-2$. A direct computation gives
    \begin{equation*}
        \mathrm{Ad}(\omega_0)(y)u_{a'}x^{-1}=u_a, \qquad \mathrm{Ad}(\omega_0)(x)u_{b'}y^{-1}= (s_{a+1}\cdots s_{a'})u_{b'}.
    \end{equation*}
    Since $d-b \leq a+1$ and $a' \leq d-b'-2$, the expression on the right is a subword of $u_b$. This proves the reverse implication.\\

    Conversely, let us now assume that there exists some $x,y \in H$ as in the Proposition. We write 
    \begin{equation*}
        v \coloneq \mathrm{Ad}(\omega_0)(y)u_{a'}x^{-1}, \qquad w \coloneq \mathrm{Ad}(\omega_0)(x)u_{b'}y^{-1},
    \end{equation*}
    so that we have $v \leq u_a$ and $w \leq u_b$ in the Bruhat order. Let us rewrite it as 
    \begin{equation*}
        u_{a'}x^{-1}=\mathrm{Ad}(\omega_0)(y)^{-1}v, \qquad u_{b'}y^{-1}=\mathrm{Ad}(\omega_0)(x)^{-1}w,
    \end{equation*}
    Taking the length, we deduce that 
    \begin{equation*}
        a' + \ell(x) = \ell(u_{a'}x^{-1}) = \ell(\mathrm{Ad}(\omega_0)(y)^{-1}v) \leq \ell(\mathrm{Ad}(\omega_0)(y)^{-1}) + \ell(v) \leq \ell(y) + a,
    \end{equation*}
    and similarly $b' + \ell(y) \leq \ell(x) + b$. Here, we used the fact that $u_{a'}$ and $u_{b'}$ are right $H$-reduced, and that the length is compatible with the Bruhat order. We deduce that
    \begin{equation}\label{Eq:SumInequality}
        a'+b' \leq a+b.
    \end{equation}
    Suppose that $a' = d-1$. Since $\mathrm{Ad}(\omega_0)(y)$ fixes $1$, the definition of $v$ implies that $v^{-1}(1) = d$. Since $v \leq u_a$, this is only possible if $a = d-1$ takes the maximal value. But then \eqref{Eq:SumInequality} gives $b'\leq b$, and we are done. The same argument applies if $b' = d-1$. We may therefore assume that $a',b'\leq d-2$ from now on. We define 
    \begin{equation}\label{Eq:DefinitionOfQ}
        Q \coloneq \omega_0 v \omega_0 w =y\left(\omega_0u_{a'}\omega_0u_{b'}\right)y^{-1} \in \mathfrak S_d.
    \end{equation}
    
    \begin{lem}\label{Lem:LemmaCombinatorialProposition}
        The following statements hold.
        \begin{enumerate}
            \item If $a+b\leq d-2$, then the cycle of $Q$ containing $d$ is contained in $\{d-b, \ldots, d\}$, and the cycle containing $1$ is contained in $\{1, \ldots, a+1\}$. 
            \item Suppose that $b \leq d-2$ and that $1$ and $d$ belong to the same cycle of $Q$. If $r \geq 1$ is the smallest integer such that $Q^r(1) = d$, and $s \geq 1$ is the smallest integer such that $Q^s(d) = 1$, then we have $r \geq d-a$ and $s \geq d-b-1$.
        \end{enumerate}
    \end{lem}
    
    \begin{proof}[Proof of Lemma \ref{Lem:LemmaCombinatorialProposition}]
        We write $v' \coloneq \omega_0 v \omega_0 \leq s_a s_{a-1}\cdots s_1$. Thus $v'$ fixes $\{a+2, \ldots, d\}$, while $w$ fixes $\{1, \ldots, d-b-1\}$. If $a+b \leq d-2$, these two intervals are disjoint. Since $Q = v'w$, the cycles of $Q$ containing $1$ and $d$ are therefore contained respectively in $\{1, \ldots, a+1\}$ and in $\{d-b, \ldots, d\}$. This proves (1).\\

        We now prove (2). Since $v'$ is represented by a subword of $s_a \cdots s_1$, and $w$ by a subword of $s_{d-b} \cdots s_{d-1}$, we have 
        \begin{equation}\label{Eq:LemmaCombinatorialProposition}
            \forall 1 \leq k \leq d, \qquad v'(k) \geq k-1, \qquad w(k) \leq k+1.
        \end{equation}
        Moreover, $v'$ preserves $\{1, \ldots, a+1\}$ and fixes every larger integer. Consequently, $Q(k) = v'(w(k)) \leq \max\{a+1, w(k)\} \leq \max\{a+1,k+1\}$. It follows that
        \begin{equation*}
            \forall 1 \leq k \leq d, \qquad \max\{0,Q(k)-a\} \leq \max\{0,k-a\}+1.
        \end{equation*}
        The assumption that $1$ and $d$ lie in the same cycle forces $a > 0$, because $b \leq d-2$ and hence $w$ fixes $1$. Iterating the preceding inequality with $k = d = Q^r(1), Q^{r-1}(1), \ldots , Q(1), 1$, we obtain $d-a\leq r$.
        
        On the other hand, if $k < d-b$ then $w(k) = k$, and \eqref{Eq:LemmaCombinatorialProposition} gives $Q(k) \geq k-1$. If $k \geq d-b$, then $w(k) \geq d-b$, because $w$ preserves $\{d-b, \ldots, d\}$, and hence $Q(k) \geq d-b-1$. We deduce that
        \begin{equation*}
            \forall 1 \leq k \leq d, \qquad \max\{0, d-b-Q(k)\} \leq \max\{0, d-b-k\}+1.
        \end{equation*}
        Again, iterating this inequality with $k = 1 = Q^s(d), Q^{s-1}(d), \ldots , Q(d), d$, we obtain $d-b-1\leq s$.
    \end{proof}

    If $a+b\geq d-1$ and $a'+b'\leq d-2$, then there is nothing to prove. Thus, we can assume that this condition does not hold, and we prove that $a' \leq a$ and $b' \leq b$. Before that, we record the following observation. If $a,b\leq d-2$, then $v^{-1}$, $u_{a'}^{-1}$ and $\mathrm{Ad}(\omega_0)(y)^{-1}$ all fix $1$, and the definition of $v$ gives $x(1) = 1$. Similarly, the definition of $w$ gives $y(1) = 1$. Thus, in this case, $x$ and $y$ fix both $1$ and $d$. 
    
    We now distinguish two exclusive cases. \\
    
    \textbf{Case 1: $a+b \leq d-2$.} In particular, $a,b \leq d-2$, and \eqref{Eq:SumInequality} gives $a'+b' \leq a+b \leq d-2$. A direct computation shows that the cycles of $\omega_0 u_{a'} \omega_0 u_{b'}$ containing $1$ and $d$ are respectively
    \begin{equation*}
        (1\ a'+1\ a'\ \ldots \ 2), \qquad (d\ d-b'\ d-b'+1\ \ldots \ d-1),
    \end{equation*}
    Their lengths are $a'+1$ and $b'+1$. Since $y$ fixes both $1$ and $d$, \eqref{Eq:DefinitionOfQ} shows that the cycles of $Q$ containing $1$ and $d$ have the same respective lengths. Point (1) of Lemma \ref{Lem:LemmaCombinatorialProposition} therefore gives $a' \leq a$ and $b' \leq b$. This completes Case 1.\\

    \textbf{Case 2: $a+b\geq d-1$ and $a'+b'\geq d-1$.} Since $a',b' \leq d-2$, we necessarily have $a',b' > 0$. If $a = b = d-1$, there is nothing to do. Otherwise, after possibly exchanging $(a,a')$ with $(b,b')$, we may assume that $b \leq d-2$. If also $a \leq d-2$, we have already seen that $y$ fixes $1$. If $a = d-1$, then $w$, $\mathrm{Ad}(\omega_0)(x)$ and $u_{b'}$ fix $1$, and the definition of $w$ shows that $y^{-1}(1) = 1$. Thus, in every case under consideration, $y$ fixes both $1$ and $d$. A direct computation shows that $1$ and $d$ lie in the same cycle of $\omega_0 u_{a'} \omega_0 u_{b'}$, namely
    \begin{equation*}
        (1\ a'+1\ a'+2\ \ldots \ d\ d-b'-1\ d-b'-2\ \ldots \ 2).
    \end{equation*}
    Starting from $1$, we reach $d$ after $d-a'$ successive steps. Likewise, starting from $d$, we reach $1$ after $d-b'-1$ successive steps. Since $y$ fixes $1$ and $d$, we deduce that
    \begin{equation*}
        Q^{d-a'}(1)=d, \qquad Q^{d-b'-1}(d)=1,
    \end{equation*}
    and these exponents are minimal. Point (2) of Lemma \ref{Lem:LemmaCombinatorialProposition} therefore yields $d-a' \geq d-a$ and $d-b'-1 \geq d-b-1$, that is $a' \leq a$ and $b' \leq b$. This completes Case 2 and concludes the proof.
\end{proof}

\subsection{Closure of EKOR strata}\label{Section6.2}

In this section, we determine explicitly the poset structure on ${}^I\mathrm{Adm}(\{\mu'\})_0$ underlying the closure relations of the EKOR stratification on $\mathcal N_{\mathrm{red}}^{I,\varepsilon}$. Recall from Proposition \ref{Prop:DescriptionIndexSetOfEKORStrata} that the elements of ${}^{I}\mathrm{Adm}(\{\mu'\})_0$ are indexed by the tuples $\mathbf c = (c_i)_{i\in\overline I}$ where $0 \leq c_i \leq \ell_i$ for all $i \in \overline I$, and $c_i + c_{\theta_I(i)} \leq \ell_i-1$ for at least one $i \in \overline I$. The corresponding element of ${}^{I}\mathrm{Adm}(\{\mu'\})_0$ is still denoted $v_{\mathbf c}$. We define a partial order $\mathrel{\dot{\leq}}$ on ${}^{I}\mathrm{Adm}(\{\mu'\})_0$ as follows. We use the symbol $\mathrel{\dot{\leq}}$ so that it does not get confused with the Bruhat order.

\begin{defi}\label{Defi:OrderOnTuplesC}
    Let $v_{\mathbf c}, v_{\mathbf c'} \in {}^{I}\mathrm{Adm}(\{\mu'\})_0$. We write $v_{\mathbf c'} \mathrel{\dot{\leq}} v_{\mathbf c}$ if, for every $i \in \overline I$, we have
    \begin{equation*}
        (c_i' \leq c_i \text{ and } c_{\theta_I(i)}' \leq c_{\theta_I(i)}) \quad \text{or} \quad (c_i' + c_{\theta_I(i)}' \leq \ell_i-1 \text{ and } c_i + c_{\theta_I(i)} \geq \ell_i).
    \end{equation*}
\end{defi}

Both conditions are invariant under the substitution $i \mapsto \theta_I(i)$, so it is enough to check the condition once on every orbit in $\Omega_{\overline I}$. If $\theta_I(i) = i$, then the condition boils down to $c_i' \leq c_i$. 

\begin{lem}\label{Lem:OrderOnTuplesC}
    The relation $\mathrel{\dot{\leq}}$ is a partial order on ${}^{I}\mathrm{Adm}(\{\mu'\})_0$. Moreover, let $(\mathcal T, \mathbf m, \mathbf b)$ and $(\mathcal T', \mathbf m', \mathbf b')$ be the triples corresponding respectively to $\mathbf c$ and to $\mathbf c'$ under Proposition \ref{Prop:BijectionIndexSetsBTEKOR}. Then
    \begin{equation*}
        v_{\mathbf c'} \mathrel{\dot{\leq}} v_{\mathbf c} \iff \begin{cases}
            \mathcal T \subseteq \mathcal T',  & \\
            m_i' \leq m_i & \text{for every } i \in \mathcal T,\\
            b_i' \leq b_i & \text{for every } i \in \overline I \setminus \mathcal T'.
        \end{cases}
    \end{equation*}
\end{lem}

\begin{proof}
    Over a singleton orbit $O = \{i\} \in \Omega_{\overline I}$, the condition defining $\mathrel{\dot{\leq}}$ amounts to the usual order on $\{0, \ldots , \ell_i\}$. Over an orbit $O = \{i, \theta_I(i)\}$ of size $2$, the condition defining $\mathrel{\dot{\leq}}$ amounts to the partial order on the finite grid $\{0,\ldots , \ell_i\}^2$ given by the coordinatewise order, with the extra condition that any point lying below the diagonal $x+y = \ell_i$ is smaller than all the points lying on it or above it. Then $\mathrel{\dot{\leq}}$ is the product of all these order relations for $O$ running through $\Omega_{\overline I}$. It follows that $\mathrel{\dot{\leq}}$ is a partial order.

    Suppose that $v_{\mathbf c'} \mathrel{\dot{\leq}} v_{\mathbf c}$. If $i \in \mathcal T$, then $c_i + c_{\theta_I(i)} \leq \ell_i-1$, so that we must have $c_i' \leq c_i$ and $c_{\theta_I(i)}' \leq c_{\theta_I(i)}$. It follows that $i \in \mathcal T'$ and $m_i' \leq  m_i$. Thus $\mathcal T \subseteq \mathcal T'$. If $i \in \overline I \setminus \mathcal T'$, then $c_i' + c_{\theta_I(i)}' \geq \ell_i$, so that we must have $b_i' = c_i' \leq c_i = b_i$. Conversely, if the three conditions on $(\mathcal T, \mathbf m, \mathbf b)$ and $(\mathcal T', \mathbf m', \mathbf b')$ hold, then we have $c_i' \leq c_i$ for all $i \in \mathcal T$ as well as for all $i \in \overline I \setminus \mathcal T'$. On the other hand, we have $c_i' + c_{\theta_I(i)}' \leq \ell_i-1$ and $c_i + c_{\theta_I(i)} \geq \ell_i$ for all $i \in \mathcal T' \setminus \mathcal T$. This proves that $v_{\mathbf c'} \mathrel{\dot{\leq}} v_{\mathbf c}$. 
\end{proof}

We will repeatedly use the following elementary property of the Bruhat order, so let us record it for good.

\begin{lem}\label{Lem:BruhatOrderForProducts}
    Let $(W,S)$ be a Coxeter system and let $\leq$ denote the Bruhat order. Let $S_1, \ldots, S_t \subseteq S$ be pairwise disjoint subsets, and let $y_j, z_j \in W_{S_j}$ for all $1 \leq j \leq t$, where $W_{S_j} \coloneq \langle S_j \rangle \subseteq W$ is the corresponding parabolic subgroup. Then 
    \begin{equation*}
        y_1 \cdots y_t \leq z_1 \cdots z_t \iff \forall 1 \leq j \leq t, \quad y_j \leq z_j.
    \end{equation*}
\end{lem}

\begin{proof}
    The concatenation of any reduced expressions of $y_1, \ldots , y_t$ is automatically a reduced expression of the product $y_1 \cdots y_t$. The Lemma then follows easily from the subword criterion for the Bruhat order. 
\end{proof}

From now on, we fix a non-empty $\theta_I$-stable subset $\mathcal T \subseteq \overline I$ and a tuple $\mathbf m = (m_i)_{i \in \mathcal T} \in \mathrm{Num}(\mathcal T)$. Recall the notations and definitions introduced in Section \ref{Section5.2}. In particular, recall that $\mathbf P_{\mathcal K_{\mathcal T,\mathbf m}}$ is the standard parabolic subgroup of $\mathbf J_{\mathcal T,\mathbf m,\overline{\mathbb F_q}}$ stabilizing the partial flag $(U_r^{\circ})_{r \in \mathcal K_{\mathcal T,\mathbf m}}$. Its Weyl group with respect to the maximal torus $\mathbf T_{\mathcal T,\mathbf m}$ is
\begin{equation*}
    \mathbf W_{\mathcal K_{\mathcal T,\mathbf m}} \coloneq \left\langle s_x \;\middle|\; s_x \in \mathbf S_{V} \text{ for some quotient space } V \text{ and } x+1 \not \in \mathcal K_V \right\rangle \subseteq \mathbf W_{\mathcal T, \mathbf m}.
\end{equation*}
We denote its set of simple reflections by $\mathbf S_{\mathcal K_{\mathcal T,\mathbf m}}$. The $\mathbb F_q$-structure of $\mathbf J_{\mathcal T,\mathbf m}$ is defined by a Frobenius morphism $F : \mathbf J_{\mathcal T,\mathbf m,\overline{\mathbb F_q}} \longrightarrow \mathbf J_{\mathcal T,\mathbf m,\overline{\mathbb F_q}}$. Since $F$ stabilizes $(\mathbf B_{\mathcal T,\mathbf m},\mathbf T_{\mathcal T,\mathbf m})$, it induces an automorphism of the Coxeter system $(\mathbf W_{\mathcal T,\mathbf m}, \mathbf S_{\mathcal T,\mathbf m})$ which we denote by the same letter. We define $\mathbf b^{\mathrm{max}} \coloneq (\ell_i)_{i\in\overline I\setminus\mathcal T} \in \mathcal B_{\mathcal T}$ and let
\begin{equation*}
    \mathbf c^{\mathrm{max}} \coloneq \mathbf c(\mathbf b^{\mathrm{max}}) = (c_i^{\mathrm{max}})_{i\in\overline I},
\end{equation*}
be the tuple associated with $(\mathcal T,\mathbf m,\mathbf b^{\mathrm{max}})$ as in \eqref{Eq:DefinitionTuplec(b)}. Thus we have $c_i^{\mathrm{max}} = \ell_i$ for all $i \in \overline I \setminus \mathcal T$, and $c_i^{\mathrm{max}} = m_i$ for all $i \in \mathcal T$. We consider tuples $\mathbf d$ as follows
\begin{equation}\label{Eq:BoundsBoundaryTuple}
    \mathbf d = (d_i)_{i\in\overline I} \in \prod_{i\in\overline I} \{0, \ldots, c_i^{\mathrm{max}}\}.
\end{equation}
For every $i \in \overline I$, the integer interval $\{r_{-i-\underline{\varepsilon}}-c_i^{\mathrm{max}}, \ldots, r_{-i-\underline{\varepsilon}}\}$ is contained in $\widehat{\mathcal K}_V$ for a unique quotient space $V$. Denote this unique quotient space by $V_i = V_i^{\mathcal T,\mathbf m}$. Recall that we write $\mathbf W_{V_i} \coloneq \mathrm{Bij}(\widehat{\mathcal K}_{V_i}) \subseteq\mathbf W_{\mathcal T,\mathbf m}$. We may define
\begin{equation}\label{Eq:DefinitionBoundaryWord}
    w_i^{\mathcal T,\mathbf m}(\mathbf d) \coloneq s_{[r_{-i-\underline{\varepsilon}}-d_i, r_{-i-\underline{\varepsilon}}-1]} \in \mathbf W_{V_i}, \qquad w_{\mathbf d}^{\mathcal T,\mathbf m} \coloneq \prod_{i\in\overline I} w_i^{\mathcal T,\mathbf m}(\mathbf d) \in \mathbf W_{\mathcal T,\mathbf m},
\end{equation}
where the factors in the product are written from left to right in increasing order of
$r_{-i-\underline{\varepsilon}}$. The resulting expression of $w_{\mathbf d}^{\mathcal T,\mathbf m}$ is reduced, and this element is right $\mathcal K_{\mathcal T,\mathbf m}$-reduced. When $\mathbf d = \mathbf c^{\mathrm{max}}$, it agrees with the element $w_{\mathbf c^{\mathrm{max}}} \in \mathbf W_{\mathcal T,\mathbf m}$ defined in Section \ref{Section5.2}. 

On the set of right $\mathcal K_{\mathcal T,\mathbf m}$-reduced elements of $\mathbf W_{\mathcal T,\mathbf m}$, we introduce the following twisted Bruhat order: 
\begin{equation}\label{Eq:RightTwistedOrderFineDL}
    w' \leq_{\mathcal K_{\mathcal T,\mathbf m},F} w \iff \exists x \in \mathbf W_{\mathcal K_{\mathcal T,\mathbf m}}, \quad F(x) w' x^{-1} \leq w.
\end{equation}
We will prove that $\leq_{\mathcal K_{\mathcal T,\mathbf m},F}$ amounts to a collection of local inequalities, each of which boils down to the situation of Proposition \ref{Prop:CombinatorialProposition}.

\begin{lem}\label{Lem:IntervalsOutsideMathcalT}
    Let $i\in\overline I\setminus\mathcal T$. Then $r_{-i-\underline{\varepsilon}}-\ell_i = r_{\theta_I(i)} \in \widehat{\mathcal K}_{\mathcal T,\mathbf m}$, and the two integers $r_{\theta_I(i)}$ and $r_{-i-\underline{\varepsilon}}$ are consecutive elements of $\mathcal K_{\mathcal T,\mathbf m}$, belonging to the same interval $\widehat{\mathcal K}_{V_i}$.
\end{lem}

\begin{proof}
    Since $i \in \overline I \setminus \mathcal T$, we know that $\widehat{\mathcal K}_{V_i}$ contains the integer interval $\{r_{-i-\underline{\varepsilon}}-\ell_i,\ldots, r_{-i-\underline{\varepsilon}}\}$. Since the reduction modulo $n$ of $r_{-i-\underline{\varepsilon}}-\ell_i$ is $\theta_I(i)$, we have $r_{-i-\underline{\varepsilon}}-\ell_i = r_{\theta_I(i)}$. 
\end{proof}

We define 
\begin{align}\label{Eq:DefinitionWiMinusPlus}
    \mathbf W_i^- & \coloneq \left\langle s_{r_{-i-\underline{\varepsilon}}-\ell_i}, \ldots, s_{r_{-i-\underline{\varepsilon}}-2} \right\rangle \subseteq \mathbf W_{\mathcal K_{\mathcal T,\mathbf m}} \cap \mathbf W_{V_i}, \\
    \mathbf W_i^+ & \coloneq \left\langle s_{r_{-i-\underline{\varepsilon}}-\ell_i+1}, \ldots, s_{r_{-i-\underline{\varepsilon}}-1} \right\rangle \subseteq \mathbf W_{V_i}. \nonumber
\end{align}
The Frobenius morphism satisfies
\begin{equation}\label{Eq:FrobeniusOnLocalSimpleReflections}
    \forall 0 \leq j \leq \ell_i-2, \qquad F(s_{r_{-\theta_I(i)-\underline{\varepsilon}}-\ell_i+j}) = s_{r_{-i-\underline{\varepsilon}}-1-j} \in \mathbf W_i^+.
\end{equation}
Consequently, $F(\mathbf W_{\theta_I(i)}^-) = \mathbf W_i^+$. We next recall the decomposition of a closed BT stratum into fine Deligne-Lusztig varieties and identify the EKOR index of every piece.

\begin{lem}\label{Lem:BoundaryDecompositionClosedBT}
    Inside $\mathcal N_{\mathrm{red}}^{I,\varepsilon}$, the closure of the Bruhat-Tits stratum $\mathcal N_{\mathcal T, \mathbf\Lambda(\mathbf m)}^{0}$ decomposes as 
    \begin{equation*}
        \overline{\mathcal N_{\mathcal T, \mathbf\Lambda(\mathbf m)}^{0}} = \bigsqcup_{\mathbf d} X_{\mathcal K_{\mathcal T,\mathbf m}}\{(w_{\mathbf d}^{\mathcal T,\mathbf m})^{-1}\},
    \end{equation*}
    where $\mathbf d$ runs over all the tuples as in \eqref{Eq:BoundsBoundaryTuple}. For such a $\mathbf d$, put 
    \begin{equation*}
        \mathcal T_{\mathbf d} \coloneq \left\{i \in \overline I \;\middle|\; d_i + d_{\theta_I(i)} \leq \ell_i-1 \right\} \subseteq \overline I.
    \end{equation*}
    Then $\mathcal T \subseteq \mathcal T_{\mathbf d}$, $v_{\mathbf d} \in {}^{I}\mathrm{Adm}(\{\mu'\})_0$, and 
    \begin{equation*}
        X_{\mathcal K_{\mathcal T,\mathbf m}}\{(w_{\mathbf d}^{\mathcal T,\mathbf m})^{-1}\} = \overline{\mathcal N_{\mathcal T, \mathbf\Lambda(\mathbf m)}^{0}} \cap \mathrm{EKOR}_{v_{\mathbf d}}.
    \end{equation*}
\end{lem}

\begin{proof}
    The tuple $\mathbf c^{\mathrm{max}}$ corresponds to the choice $b_i = \ell_i$ for every $i \in \overline I \setminus \mathcal T$. By \cite{mullerBruhatTitsStratificationBasic2026} Theorem 3.36, the closed BT stratum $\overline{\mathcal N_{\mathcal T, \mathbf\Lambda(\mathbf m)}^{0}}$ is the closure of the fine Deligne-Lusztig variety $X_{\mathcal K_{\mathcal T,\mathbf m}}\{w_{\mathbf c^{\mathrm{max}}}^{-1}\} \subseteq \mathbf J_{\mathcal T,\mathbf m,\overline{\mathbb F_q}} /\mathbf P_{\mathcal K_{\mathcal T,\mathbf m}}$. The explicit decomposition of this closure is given by \cite{mullerBruhatTitsStratificationBasic2026} Propositions 3.18, 3.22 and 3.26. After conjugating by the elements of maximal length in order to suit our conventions here (recall that we use the Borel subgroup of \textit{lower-triangular} matrices), the parameters of these fine Deligne-Lusztig varieties are precisely the elements $w_{\mathbf d}^{\mathcal T,\mathbf m} \in \mathbf W_{\mathcal T,\mathbf m}$ of \eqref{Eq:DefinitionBoundaryWord}, for the tuples $\mathbf d$ as in \eqref{Eq:BoundsBoundaryTuple}. 

    If $i\in\mathcal T$, then $d_i + d_{\theta_I(i)} \leq m_i + m_{\theta_I(i)} \leq \ell_i-1$. Hence $\mathcal T \subseteq \mathcal T_{\mathbf d}$ and, in particular, $v_{\mathbf d}$ belongs to ${}^{I}\mathrm{Adm}(\{\mu'\})_0$.

    It remains to determine the index of the EKOR stratum containing $X_{\mathcal K_{\mathcal T,\mathbf m}}\{(w_{\mathbf d}^{\mathcal T,\mathbf m})^{-1}\}$. Let $k$ be an algebraically closed field containing $\overline{\mathbb F_q}$ and let $z \in X_{\mathcal K_{\mathcal T,\mathbf m}}\{(w_{\mathbf d}^{\mathcal T,\mathbf m})^{-1}\}(k)$. Lift its partial flag to a full flag of relative position $(w_{\mathbf d}^{\mathcal T,\mathbf m})^{-1}$. This gives rise to a point $\widetilde z$ at Iwahori level, which maps to $z$ via the change of parahoric morphism, just as in the proof of Theorem \ref{Theo:ComparisonBTEKOR}. Just as in said proof, one determines that the index of the Iwahori-level KR stratum containing $\widetilde z$ is 
    \begin{equation*}
        J^{\mathrm{Iw}} = \bigsqcup_{i\in\overline I} \{i, i+1, \ldots, i+d_i-1\} \subsetneq \mathbb Z/n\mathbb Z.
    \end{equation*}
    The final argument of Theorem \ref{Theo:ComparisonBTEKOR} now gives $z \in \mathrm{EKOR}_{v_{\mathbf d}}(k)$. Since the intersection $\overline{\mathcal N_{\mathcal T, \mathbf\Lambda(\mathbf m)}^{0}} \cap \mathrm{EKOR}_{v_{\mathbf d}}$ is reduced, we deduce that $X_{\mathcal K_{\mathcal T,\mathbf m}}\{(w_{\mathbf d}^{\mathcal T,\mathbf m})^{-1}\} \subseteq \overline{\mathcal N_{\mathcal T, \mathbf\Lambda(\mathbf m)}^{0}} \cap \mathrm{EKOR}_{v_{\mathbf d}}$. Since this holds for every $\mathbf d$ and since the map $\mathbf d \mapsto v_{\mathbf d}$ is injective, this inclusion is actually an equality.
\end{proof}

We can now compute the twisted Bruhat order $\leq_{\mathcal K_{\mathcal T,\mathbf m},F}$. We start with a Lemma. For $i\in\overline I\setminus\mathcal T$, define
\begin{equation}\label{Eq:DefinitionLocalFullSupport}
    \mathbf S_i \coloneq \left\{s_{r_{-i-\underline{\varepsilon}}-\ell_i}, \ldots, s_{r_{-i-\underline{\varepsilon}}-1} \right\} \subseteq \mathbf S_{V_i},
\end{equation}
and $\mathbf W_i \coloneq \langle \mathbf S_i \rangle \subseteq \mathbf W_{V_i}$. The subsets $\mathbf S_i$ are pairwise disjoint. Indeed, by Lemma \ref{Lem:IntervalsOutsideMathcalT}, they are the intervals of simple reflections lying between distinct pairs of consecutive elements of $\mathcal K_{\mathcal T,\mathbf m}$. We point out that we have $\mathbf W_i^-, \mathbf W_i^+ \subset \mathbf W_i$. If we identify $\mathbf W_i$ with $\mathfrak S_{\ell_i+1}$, then $\mathbf W_i^-$ and $\mathbf W_i^+$ are the subgroups of permutations fixing respectively $\ell_i+1$ and $1$. 

Let $V = \Lambda_{\alpha_V}/\Lambda_{\beta_V}$ be one of the quotient spaces of Section \ref{Section5.2}. If $\mathcal K_V \not = \emptyset$, we have
\begin{align}\label{Eq:YoungFactorsParabolicWeylGroup}
    \mathbf W_{\mathcal K_{\mathcal T,\mathbf m}} \cap \mathbf W_V = & \mathrm{Bij}\left( \{\alpha_V, \ldots, \min(\mathcal K_V)-1\} \right) \nonumber\\
    & \times \bigl( \prod_{\substack{r<r'\\ \text{consecutive} \\ \text{in } \mathcal K_V}} \mathrm{Bij}\left( \{r, \ldots, r'-1\} \right) \bigr) \times \mathrm{Bij}\left( \{\max(\mathcal K_V), \ldots, \beta_V-1\} \right),
\end{align}
and if $\mathcal K_V = \emptyset$, then $\mathbf W_{\mathcal K_{\mathcal T,\mathbf m}} \cap \mathbf W_V = \mathbf W_V$. The group $\mathbf W_{\mathcal K_{\mathcal T,\mathbf m}}$ is then the product of all the components $\mathbf W_{\mathcal K_{\mathcal T,\mathbf m}} \cap \mathbf W_V$ for $V$ running through all the quotient spaces. Likewise, if $\mathcal K_V \not = \emptyset$ we have 
\begin{align}\label{Eq:YoungFactorsFrobeniusParabolicWeylGroup}
    F\left(\mathbf W_{\mathcal K_{\mathcal T,\mathbf m}}\right) \cap \mathbf W_V = & \mathrm{Bij} \left( \{\alpha_V, \ldots, \min(\mathcal K_V)\} \right) \nonumber\\
    & \times \bigl( \prod_{\substack{r < r'\\ \text{consecutive}\\ \text{in }\mathcal K_V}} \mathrm{Bij} \left( \{ r+1, \ldots, r'\} \right) \bigr) \times \mathrm{Bij} \left( \{\max(\mathcal K_V)+1, \ldots, \beta_V-1\} \right),
\end{align}
if $\mathcal K_V = \emptyset$, then $F\left(\mathbf W_{\mathcal K_{\mathcal T,\mathbf m}}\right) \cap \mathbf W_V = \mathbf W_V$.

For every $i \in \overline I \setminus \mathcal T$, the groups $\mathbf W_i^-$ and $\mathbf W_i^+$ are precisely the factors associated with the consecutive elements $r_{\theta_I(i)} = r_{-i-\underline{\varepsilon}} - \ell_i < r_{-i-\underline{\varepsilon}}$ of $\mathcal K_{\mathcal T,\mathbf m}$ respectively in \eqref{Eq:YoungFactorsParabolicWeylGroup} and in \eqref{Eq:YoungFactorsFrobeniusParabolicWeylGroup}.

\begin{lem}\label{Lem:LocalizationFineDLTwistedInequality}
    Let $\mathbf b \in \mathcal B_{\mathcal T}$, write $\mathbf c = \mathbf c(\mathbf b)$ and let $\mathbf d$ satisfy \eqref{Eq:BoundsBoundaryTuple}. Then the following conditions are equivalent.
    \begin{enumerate}
        \item There exists $x \in \mathbf W_{\mathcal K_{\mathcal T,\mathbf m}}$ such that
        \begin{equation*}
            F(x) w_{\mathbf d}^{\mathcal T,\mathbf m} x^{-1} \leq w_{\mathbf c} \quad \text{in } \mathbf W_{\mathcal T,\mathbf m}.
        \end{equation*}
        \item There exists a family $(x_i)_{i\in\overline I\setminus\mathcal T}$ with $x_i \in \mathbf W_i^-$ such that, for all $i \in \overline I \setminus \mathcal T$,
        \begin{equation*}
            F(x_{\theta_I(i)}) w_i^{\mathcal T,\mathbf m}(\mathbf d) x_i^{-1} \leq w_i(\mathbf c) \quad \text{in } \mathbf W_i.
        \end{equation*}
    \end{enumerate}
\end{lem}

\begin{proof}
    Suppose first that (1) holds. Fix a quotient space $V$ and let $i_1, \ldots , i_q$ denote the elements of $\overline I \setminus \mathcal T$ such that $V_{i_j} = V$ for all $1 \leq j \leq q$, and numbered in increasing order of $r_{-i_j-\underline{\varepsilon}}$. The set of simple reflections $\mathbf S_V$ of $\mathbf W_V$ decomposes as
    \begin{equation}\label{Eq:AlternatingPartitionSimpleReflectionsV}
        \mathbf S_V = \mathbf S_{V,0}^{\partial} \sqcup\mathbf S_{i_1} \sqcup \mathbf S_{V,1}^{\partial} \sqcup \cdots \sqcup \mathbf S_{i_q} \sqcup \mathbf S_{V,q}^{\partial},
    \end{equation}
    for certain potentially empty intervals of simple reflections $\mathbf S_{V,0}^{\partial}, \ldots , \mathbf S_{V,q}^{\partial}$. For $0 \leq j \leq q$ we write $\mathbf W_{V,j}^{\partial} \coloneq \langle \mathbf S_{V,j}^{\partial} \rangle \subseteq \mathbf W_V$. We decompose $x$ in the product \eqref{Eq:YoungFactorsParabolicWeylGroup}, and denote by $x_i \in \mathbf W_i^-$ its component in the factor indexed by $i \in \overline I \setminus \mathcal T$. Finally, let $\mathrm{pr}_V : \mathbf W_{\mathcal T,\mathbf m} = \prod_{V'} \mathbf W_{V'} \to \mathbf W_V$ denote the projection. We claim that there exist $y_{V,j}, z_{V,j} \in \mathbf W_{V,j}^{\partial}$ for $0 \leq j \leq q$ such that
    \begin{align}\label{Eq:AlternatingFactorizationsFineDL}
        \mathrm{pr}_V \left( F(x) w_{\mathbf d}^{\mathcal T,\mathbf m} x^{-1} \right) & = y_{V,0} \prod_{j = 1}^{q} \left[ \left( F(x_{\theta_I(i_j)}) w_{i_j}^{\mathcal T,\mathbf m}(\mathbf d) x_{i_j}^{-1} \right) y_{V,j} \right],\\
        \mathrm{pr}_V (w_{\mathbf c}) & = z_{V,0} \prod_{j = 1}^{q} \left[ w_{i_j}(\mathbf c) z_{V,j} \right], \nonumber
    \end{align}
    where the products are ordered from $j = 1$ to $q$. 
    
    We prove the claim. For $0\leq j\leq q$, let $\eta_{V,j}, \xi_{V,j} \in \mathbf W_{V,j}^{\partial}$ be respectively the products of the factors of $\mathrm{pr}_V(x)$ in \eqref{Eq:YoungFactorsParabolicWeylGroup} and of the factors of $\mathrm{pr}_V(F(x))$ in \eqref{Eq:YoungFactorsFrobeniusParabolicWeylGroup} whose simple reflections belong to $\mathbf S_{V,j}^{\partial}$. The factors in each of these products commute. We obtain
    \begin{equation*}
        \mathrm{pr}_V(x) = \eta_{V,0} \prod_{j=1}^{q} \left[ x_{i_j} \eta_{V,j} \right], \qquad \mathrm{pr}_V(F(x)) = \xi_{V,0} \prod_{j=1}^{q} \left[ F(x_{\theta_I(i_j)})\xi_{V,j} \right].
    \end{equation*}
    Let $i\in\mathcal T$ be such that $V_i = V$. We have $d_i \leq m_i = c_i$. If $m_i > 0$, the partition \eqref{Eq:AlternatingPartitionSimpleReflectionsV} gives a unique $0 \leq j \leq q$ such that
    \begin{equation*}
        \left\{s_{r_{-i-\underline{\varepsilon}}-d_i}, \ldots, s_{r_{-i-\underline{\varepsilon}}-1}\right\} \subseteq \left\{s_{r_{-i-\underline{\varepsilon}}-m_i}, \ldots, s_{r_{-i-\underline{\varepsilon}}-1}\right\} \subseteq \mathbf S_{V,j}^{\partial}.
    \end{equation*}
    We let $w_{\mathbf d,V,j}^{\partial}, w_{\mathbf c,V,j}^{\partial} \in \mathbf W_{V,j}^{\partial}$ denote respectively the products, in the order prescribed by \eqref{Eq:DefinitionBoundaryWord} and by the definition of $w_{\mathbf c}$, of the blocks indexed by $i \in \mathcal T$ and supported on $\mathbf S_{V,j}^{\partial}$. We obtain
    \begin{equation*}
        \mathrm{pr}_V(w_{\mathbf d}^{\mathcal T,\mathbf m}) = w_{\mathbf d,V,0}^{\partial} \prod_{j = 1}^{q} \left[ w_{i_j}^{\mathcal T,\mathbf m}(\mathbf d) w_{\mathbf d,V,j}^{\partial} \right], \qquad \mathrm{pr}_V(w_{\mathbf c}) = w_{\mathbf c,V,0}^{\partial} \prod_{j=1}^{q} \left[ w_{i_j}(\mathbf c) w_{\mathbf c,V,j}^{\partial} \right].
    \end{equation*}
    For every $1 \leq j \leq q$, we have
    \begin{align*}
        F(x_{\theta_I(i_j)}) & \in \left\langle s_{r_{-i_j-\underline{\varepsilon}}-\ell_{i_j}+1}, \ldots, s_{r_{-i_j-\underline{\varepsilon}}-1} \right\rangle,\\
        x_{i_j} & \in \left\langle s_{r_{-i_j-\underline{\varepsilon}}-\ell_{i_j}}, \ldots, s_{r_{-i_j-\underline{\varepsilon}}-2} \right\rangle.
    \end{align*}
    Similarly, if $\mathbf S_{V,j}^{\partial} = \{s_a,\ldots,s_b\} \not = \emptyset$, then \eqref{Eq:YoungFactorsParabolicWeylGroup} and \eqref{Eq:YoungFactorsFrobeniusParabolicWeylGroup} give
    \begin{equation*}
        \xi_{V,j} \in \langle s_{a+1}, \ldots, s_b\rangle \quad\text{if } j \geq 1, \qquad \eta_{V,j} \in \langle s_a,\ldots,s_{b-1}\rangle \quad \text{if } j < q.
    \end{equation*}
    Since $s_a$ and $s_b$ commute whenever $|a-b|>1$, we can reorder the factors to write
    \begin{equation*}
        \mathrm{pr}_V \left(F(x)w_{\mathbf d}^{\mathcal T,\mathbf m}x^{-1}\right) = \left( \xi_{V,0}w_{\mathbf d,V,0}^{\partial}\eta_{V,0}^{-1} \right) \prod_{j = 1}^{q} \left[ \left( F(x_{\theta_I(i_j)}) w_{i_j}^{\mathcal T,\mathbf m}(\mathbf d)x_{i_j}^{-1} \right) \left( \xi_{V,j}w_{\mathbf d,V,j}^{\partial}\eta_{V,j}^{-1} \right) \right].
    \end{equation*}
    Thus the claim \eqref{Eq:AlternatingFactorizationsFineDL} holds with $y_{V,j} \coloneq \xi_{V,j} w_{\mathbf d,V,j}^{\partial} \eta_{V,j}^{-1}$ and $z_{V,j} \coloneq w_{\mathbf c,V,j}^{\partial}$ for all $0 \leq j \leq q$. \\

    Since the Bruhat order on $\mathbf W_{\mathcal T,\mathbf m} = \prod_V \mathbf W_V$ is componentwise, condition (1) and \eqref{Eq:AlternatingFactorizationsFineDL} give an inequality in every $\mathbf W_V$. Applying Lemma \ref{Lem:BruhatOrderForProducts} to the ordered family of pairwise disjoint subsets in \eqref{Eq:AlternatingPartitionSimpleReflectionsV} yields
    \begin{equation*}
        F(x_{\theta_I(i_j)}) w_{i_j}^{\mathcal T,\mathbf m}(\mathbf d)x_{i_j}^{-1} \leq w_{i_j}(\mathbf c) \quad \text{in }\mathbf W_{i_j}
    \end{equation*}
    for every $1\leq j\leq q$. Letting $V$ vary proves (2).\\

    Conversely, suppose that (2) holds and define $x \coloneq \prod_{i\in\overline I\setminus\mathcal T} x_i \in \prod_{i \in \overline I \setminus \mathcal T} \mathbf W_i^- \subseteq \mathbf W_{\mathcal K_{\mathcal T,\mathbf m}}$. Since $F(\mathbf W_{\theta_I(i)}^-)=\mathbf W_i^+$, all the factors of $\mathrm{pr}_V(x)$ and $\mathrm{pr}_V(F(x))$ whose simple reflections lie in some $\mathbf S_{V,j}^{\partial}$ are trivial. Hence all the elements $\eta_{V,j}$ and $\xi_{V,j}$ in the construction above are trivial. Consequently, we have $y_{V,j} = w_{\mathbf d,V,j}^{\partial}$ and $z_{V,j} = w_{\mathbf c,V,j}^{\partial}$. If $i \in \mathcal T$, then $d_i \leq c_i^{\mathrm{max}} = m_i = c_i$. Hence, in every $\mathbf W_{V,j}^{\partial}$, the reduced expression defining $w_{\mathbf d,V,j}^{\partial}$ is a subword of the reduced expression defining $w_{\mathbf c,V,j}^{\partial}$. Therefore $y_{V,j} \leq z_{V,j}$. Together with the local inequalities in (2), Lemma \ref{Lem:BruhatOrderForProducts} applied to \eqref{Eq:AlternatingFactorizationsFineDL} gives
    \begin{equation*}
        \mathrm{pr}_V \left(F(x)w_{\mathbf d}^{\mathcal T,\mathbf m}x^{-1}\right) \leq \mathrm{pr}_V(w_{\mathbf c}) \quad \text{in } \mathbf W_V
    \end{equation*}
    for every quotient space $V$. Since $\mathbf W_{\mathcal T,\mathbf m} = \prod_V \mathbf W_V$, this is exactly condition (1).
\end{proof}

\begin{prop}\label{prop:ComputationTwistedOrderBoundaryWords}
    Let $\mathbf b\in\mathcal B_{\mathcal T}$, write $\mathbf c = \mathbf c(\mathbf b)$ and let $\mathbf d$ satisfy \eqref{Eq:BoundsBoundaryTuple}. Then
    \begin{equation*}
        w_{\mathbf d}^{\mathcal T,\mathbf m} \leq_{\mathcal K_{\mathcal T,\mathbf m},F} w_{\mathbf c} \iff v_{\mathbf d} \mathrel{\dot{\leq}} v_{\mathbf c}.
    \end{equation*}
\end{prop}

\begin{proof}
    Suppose first that $w_{\mathbf d}^{\mathcal T,\mathbf m} \leq_{\mathcal K_{\mathcal T,\mathbf m},F} w_{\mathbf c}$. Choose $x \in \mathbf W_{\mathcal K_{\mathcal T,\mathbf m}}$ satisfying $F(x) w_{\mathbf d}^{\mathcal T,\mathbf m} x^{-1} \leq w_{\mathbf c}$, and let $x_i\in\mathbf W_i^-$ be its component for $i \in \overline I \setminus \mathcal T$. By Lemma
    \ref{Lem:LocalizationFineDLTwistedInequality}, we have $F(x_{\theta_I(i)}) w_i^{\mathcal T,\mathbf m}(\mathbf d) x_i^{-1} \leq w_i(\mathbf c)$ for every such $i$.

    Let $O = \{i,\theta_I(i)\} \in \Omega_{\overline I}$ be a $\theta_I$-orbit of size $2$ contained in $\overline I \setminus \mathcal T$. We can identify $\mathbf W_i$ with $\mathfrak S_{\ell_i+1}$ in such a way that, under this identification, we have $\mathbf W_i^- \simeq H = \langle s_1, \ldots, s_{\ell_i-1}\rangle$ and $\mathbf W_i^+ \simeq \langle s_{2}, \ldots , s_{\ell_i} \rangle$. Besides, in the notations of Section \ref{Section6.1}, $w_i^{\mathcal T,\mathbf m}(\mathbf d)$ is identified with $u_{d_i} \in \mathfrak S_{\ell_i+1}$ and $w_i(\mathbf c)$ is identified with $u_{c_i} \in \mathfrak S_{\ell_i+1}$. By \eqref{Eq:FrobeniusOnLocalSimpleReflections}, the isomorphism $F: \mathbf W_{\theta_I(i)}^- \xrightarrow{\sim} \mathbf W_i^+$ is identified with conjugation by $\omega_0$, the longest element of $\mathfrak S_{\ell_i+1}$. The two local inequalities indexed by $i$ and $\theta_I(i)$ therefore become
    \begin{equation*}
        \mathrm{Ad}(\omega_0)(y)u_{d_i}x^{-1} \leq u_{c_i}, \qquad \mathrm{Ad}(\omega_0)(x)u_{d_{\theta_I(i)}}y^{-1} \leq u_{c_{\theta_I(i)}},
    \end{equation*}
    for some $x, y \in H \subseteq \mathfrak S_{\ell_i+1}$. Proposition \ref{Prop:CombinatorialProposition} shows that $(d_i \leq c_i \text{ and } d_{\theta_I(i)} \leq c_{\theta_I(i)})$ or $(d_i + d_{\theta_I(i)} \leq \ell_i-1 \text{ and } c_i + c_{\theta_I(i)} \geq \ell_i)$.

    Suppose now that $O = \{i\} \subseteq \overline I \setminus \mathcal T$ is a singleton. Under the same identification, the local inequality has the form
    \begin{equation*}
        \mathrm{Ad}(\omega_0)(x)u_{d_i}x^{-1}\leq u_{c_i} \quad\text{in }\mathfrak S_{\ell_i+1},
    \end{equation*}
    for some $x \in H \subseteq \mathfrak S_{\ell_i+1}$. Let us denote by $v$ the left-hand side of this inequality. Since $u_{d_i}$ is right $H$-reduced, we have
    \begin{equation*}
        \ell(x)+d_i = \ell(u_{d_i}x^{-1}) = \ell(\mathrm{Ad}(\omega_0)(x)^{-1}v) \leq \ell(x) + \ell(v) \leq \ell(x) + c_i.
    \end{equation*}
    Hence $d_i\leq c_i$. 

    Finally, if $i \in \mathcal T$, then $d_i \leq c_i^{\mathrm{max}} = m_i = c_i$ by \eqref{Eq:BoundsBoundaryTuple}. We have therefore proved that $v_{\mathbf d} \mathrel{\dot{\leq}} v_{\mathbf c}$. The converse is proved similarly by using Proposition \ref{Prop:CombinatorialProposition}.
\end{proof}

\begin{corol}\label{Corol:ClosureFineBTStratum}
    Let $\mathbf b\in\mathcal B_{\mathcal T}$ and write $\mathbf c = \mathbf c(\mathbf b)$. Inside the partial flag variety $\mathbf J_{\mathcal T,\mathbf m,\overline{\mathbb F_q}} / \mathbf P_{\mathcal K_{\mathcal T,\mathbf m}}$, we have
    \begin{equation*}
        \overline{ X_{\mathcal K_{\mathcal T,\mathbf m}} \{w_{\mathbf c}^{-1}\}} = \bigsqcup_{\substack{\mathbf d \\ v_{\mathbf d} \mathrel{\dot{\leq}} v_{\mathbf c}}} X_{\mathcal K_{\mathcal T,\mathbf m}} \{(w_{\mathbf d}^{\mathcal T,\mathbf m})^{-1}\}.
    \end{equation*}
\end{corol}

\begin{proof}
    By \cite{he}, the closure of the fine Deligne-Lusztig variety is described by the twisted Bruhat order $\leq_{\mathcal K_{\mathcal T,\mathbf m},F}$. The proof then follows directly from Proposition \ref{prop:ComputationTwistedOrderBoundaryWords}.
\end{proof}

We also record the local finiteness needed to pass from one standard fine
Deligne--Lusztig variety to its $J(F)$-saturation.

\begin{lem}\label{Lem:ClosureXTmbIsUnion}
    For every $\mathbf b \in \mathcal B_{\mathcal T}$ together with $\mathbf c = \mathbf c(\mathbf b)$, we have
    \begin{equation*}
        \overline{X_{\mathcal T,\mathbf m,\mathbf b}} = \bigcup_{g \in J(F)/\mathrm{Stab}_{J(F)}(\mathbf\Lambda(\mathbf m))} \overline{ g \cdot X_{\mathcal K_{\mathcal T,\mathbf m}} \{w_{\mathbf c}^{-1}\}}.
    \end{equation*}
\end{lem}

\begin{proof}
    By \cite{mullerBruhatTitsStratificationBasic2026}, the closed Bruhat-Tits strata $\left(\overline{\mathcal N_{\mathcal T, g\cdot \mathbf{\Lambda}(\mathbf m)}^0}\right)_g$ for $g$ running through $J(F)/\mathrm{Stab}_{J(F)}(\mathbf{\Lambda}(\mathbf m))$ form a locally finite family, that is, every point has an open neighborhood meeting only finitely many members of this family. Each translate $g \cdot X_{\mathcal K_{\mathcal T,\mathbf m}}\{w_{\mathbf c}^{-1}\}$ is contained in the corresponding member $\overline{\mathcal N_{\mathcal T, g\cdot \mathbf{\Lambda}(\mathbf m)}^0}$ of this locally finite family. The result follows, since the closure of a locally finite union is the union of the closures.
\end{proof}

We can now prove the closure theorem.

\begin{theo}\label{Theo:ClosureFineBTAndEKORStrata}
    Let $\mathcal T \subseteq \overline I$ be a non-empty $\theta_I$-stable subset, let $\mathbf m \in \mathrm{Num}(\mathcal T)$, let $\mathbf b \in \mathcal B_{\mathcal T}$, and write $\mathbf c = \mathbf c(\mathbf b)$. Then
    \begin{equation*}
        \overline{\mathrm{EKOR}_{v_{\mathbf c}}} = \bigsqcup_{\substack{v_{\mathbf c'} \in {}^{I}\mathrm{Adm}(\{\mu'\})_0\\ v_{\mathbf c'} \mathrel{\dot{\leq}} v_{\mathbf c}}} \mathrm{EKOR}_{v_{\mathbf c'}}.
    \end{equation*}
    Equivalently,
    \begin{equation}\label{Eq:ClosureFineBTStratumTriples}
        \overline{X_{\mathcal T,\mathbf m,\mathbf b}} = \bigsqcup_{(\mathcal T', \mathbf m', \mathbf b')} X_{\mathcal T',\mathbf m',\mathbf b'},
    \end{equation}
    where $(\mathcal T',\mathbf m',\mathbf b')$ runs over all the triples such that $\mathcal T \subseteq \mathcal T'$, $m_i' \leq m_i$ for all $i \in \mathcal T$ and $b_i' \leq b_i$ for all $i \in \overline I \setminus \mathcal T'$.
\end{theo}

\begin{proof}
    By Theorem \ref{Theo:ComparisonBTEKOR}, we have $X_{\mathcal T, \mathbf m, \mathbf b} = \mathrm{EKOR}_{v_{\mathbf c}}$. Combining Corollary \ref{Corol:ClosureFineBTStratum} with
    Lemma \ref{Lem:ClosureXTmbIsUnion} gives
    \begin{equation*}
        \overline{X_{\mathcal T, \mathbf m, \mathbf b}} = \bigcup_{g \in J(F)/\mathrm{Stab}_{J(F)}(\mathbf\Lambda(\mathbf m))} \bigcup_{\substack{\mathbf d \\ v_{\mathbf d} \mathrel{\dot{\leq}} v_{\mathbf c}}} g \cdot X_{\mathcal K_{\mathcal T,\mathbf m}} \{(w_{\mathbf d}^{\mathcal T,\mathbf m})^{-1}\}.
    \end{equation*}
    By Lemma \ref{Lem:BoundaryDecompositionClosedBT}, the variety in this inner union indexed by $\mathbf d$ is contained in $\mathrm{EKOR}_{v_{\mathbf d}}$. Since the EKOR strata are $J(F)$-invariant, we obtain
    \begin{equation*}
        \overline{\mathrm{EKOR}_{v_{\mathbf c}}} \subseteq \bigsqcup_{\substack{\mathbf d \\ v_{\mathbf d} \mathrel{\dot{\leq}} v_{\mathbf c}}} \mathrm{EKOR}_{v_{\mathbf d}}.
    \end{equation*}
    Conversely, let $v_{\mathbf c'} \in {}^{I}\mathrm{Adm}(\{\mu'\})_0$ such that $v_{\mathbf c'} \mathrel{\dot{\leq}} v_{\mathbf c}$. Let $(\mathcal T', \mathbf m', \mathbf b')$ be the triple corresponding to $\mathbf c'$ under Proposition \ref{Prop:BijectionIndexSetsBTEKOR}. By Lemma \ref{Lem:OrderOnTuplesC}, we have $\mathcal T \subseteq \mathcal T'$, $m_i' \leq m_i$ for all $i \in \mathcal T$ and $b_i' \leq b_i$ for all $i \in \overline I \setminus \mathcal T'$. Let $k$ be an algebraically closed field containing $\overline{\mathbb F_q}$ and let $z \in \mathrm{EKOR}_{v_{\mathbf c'}}(k)$. By Theorem \ref{Theo:ComparisonBTEKOR}, we have
    \begin{equation*}
        \mathrm{EKOR}_{v_{\mathbf c'}} = X_{\mathcal T', \mathbf m', \mathbf b'} = \bigsqcup_{g \in J(F)/\mathrm{Stab}_{J(F)}(\mathbf\Lambda(\mathbf m'))} g \cdot X_{\mathcal K_{\mathcal T',\mathbf m'}} \{w_{\mathbf c'}^{-1}\}.
    \end{equation*}
    Hence there exists $g \in J(F)$ such that $z \in \mathcal N_{\mathcal T', g \cdot \mathbf\Lambda(\mathbf m')}^{0}(k)$. We claim that the BT indices satisfy
    \begin{equation}\label{Eq:ComparisonBTIndicesClosureProof}
        (\mathcal T', g \cdot \mathbf\Lambda(\mathbf m')) \leq (\mathcal T, g \cdot \mathbf\Lambda(\mathbf m)).
    \end{equation}
    It is enough to check the assertion before applying $g$. Let $r\in I^+ \sqcup I^-$ satisfy $\overline r \in \mathcal T$. If $r \in I^+$, then the fact that the standard lattice chain is decreasing gives
    \begin{equation*}
        \Lambda(\mathbf m')_{\overline r} = \Lambda_{-r-m_{\overline r}'-\underline{\varepsilon}} \subseteq \Lambda_{-r-m_{\overline r}-\underline{\varepsilon}} = \Lambda(\mathbf m)_{\overline r}.
    \end{equation*}
    Likewise, if $r\in I^-$ then
    \begin{equation*}
        \Lambda(\mathbf m)_{\overline r} = \Lambda_{r+m_{\overline r}+1} \subseteq \Lambda_{r+m_{\overline r}'+1} = \Lambda(\mathbf m')_{\overline r}.
    \end{equation*}
    The claim follows, and the closure relations of the BT stratification give $z \in \overline{\mathcal N_{\mathcal T, g \cdot \mathbf\Lambda(\mathbf m)}^{0}}(k)$. By Lemma \ref{Lem:BoundaryDecompositionClosedBT}, we have 
    \begin{equation*}
        z \in \left( \overline{\mathcal N_{\mathcal T, g \cdot \mathbf\Lambda(\mathbf m)}^{0}} \cap \mathrm{EKOR}_{v_{\mathbf c'}}\right)(k) = g \cdot X_{\mathcal K_{\mathcal T,\mathbf m}} \{(w_{\mathbf c'}^{\mathcal T,\mathbf m})^{-1}\}(k).
    \end{equation*}
    Thus, $z \in \overline{X_{\mathcal T, \mathbf m, \mathbf b}}(k) = \overline{\mathrm{EKOR}_{v_{\mathbf c}}}(k)$, and this concludes the proof.
\end{proof}

We finish by recording the resulting closure of a KR stratum in $\mathcal N_{\mathrm{red}}^{I,\varepsilon}$. Given $v_{\mathbf c} \in {}^I\mathrm{Adm}(\{\mu'\})_0$, recall the notations 
\begin{equation*}
    J_{\mathbf c} = \{i \in \overline I \mid c_i = \ell_i \}, \qquad \mathcal T_{\mathbf c} = \{i \in \overline I \mid c_i + c_{\theta_I(i)} \leq \ell_i-1\}.
\end{equation*}

\begin{corol}\label{Corol:ClosureKRStrata}
    Let $J \in \mathcal P_{0,I}$. We have
    \begin{equation*}
        \overline{\mathrm{KR}_J} = \bigsqcup_{\substack{v_{\mathbf c} \in {}^I\mathrm{Adm}(\{\mu'\})_0\\ \mathcal T_{\mathbf c} \cap (\overline I \setminus (J \cup \theta_I(J))) \not = \emptyset\\ J_{\mathbf c} \subseteq J}}
        \mathrm{EKOR}_{v_{\mathbf c}}.
    \end{equation*}
\end{corol}

\begin{proof}
    We already know that 
    \begin{equation*}
        \mathrm{KR}_J = \bigsqcup_{\substack{v_{\mathbf c} \in {}^I\mathrm{Adm}(\{\mu'\})_0\\ J_{\mathbf c}=J}} \mathrm{EKOR}_{v_{\mathbf c}}.
    \end{equation*}
    This is a finite disjoint union, so Theorem \ref{Theo:ClosureFineBTAndEKORStrata} shows that $\mathrm{EKOR}_{v_{\mathbf d}}$ is contained in $\overline{\mathrm{KR}_J}$ if and only if there exists a tuple $\mathbf c$ satisfying $J_{\mathbf c} = J$ and $v_{\mathbf d} \mathrel{\dot{\leq}} v_{\mathbf c}$. 
    
    Suppose first that such a tuple exists. If $d_i = \ell_i$, then $v_{\mathbf d} \mathrel{\dot{\leq}} v_{\mathbf c}$ forces $c_i = \ell_i$ as well. Hence $J_{\mathbf d} \subseteq J$. Since $v_{\mathbf c} \in{}^{I}\mathrm{Adm}(\{\mu'\})_0$, there exists $i \in \overline I$ such that $c_i + c_{\theta_I(i)} \leq \ell_i-1$. Thus $v_{\mathbf d} \mathrel{\dot{\leq}} v_{\mathbf c}$ again forces $d_i \leq c_i$ and $d_{\theta_I(i)} \leq c_{\theta_I(i)}$. It follows that $d_i + d_{\theta_I(i)} \leq c_i + c_{\theta_I(i)} \leq \ell_i-1$, thus $i \in \mathcal T_{\mathbf d} \cap(\overline I\setminus(J\cup\theta_I(J)))$.
    
    Conversely, assume that $\mathcal T_{\mathbf d} \cap (\overline I \setminus (J \cup \theta_I(J))) \not = \emptyset$ and that $J_{\mathbf d} \subseteq J$. Define a tuple $\mathbf c = (c_i)_{i\in \overline I}$ by
    \begin{equation*}
        c_i \coloneq \begin{cases}
            \ell_i & \text{if } i \in J,\\
            d_i & \text{if } i \in \overline I \setminus J.
        \end{cases}
    \end{equation*}
    The condition $J_{\mathbf d} \subseteq J$ gives $J_{\mathbf c} = J$. Choose $i \in \mathcal T_{\mathbf d} \cap(\overline I\setminus(J\cup\theta_I(J)))$. Then $c_i = d_i$ and $c_{\theta_I(i)} = d_{\theta_I(i)}$, so $c_i + c_{\theta_I(i)} \leq \ell_i-1$. Hence $v_{\mathbf c}\in{}^{I}\mathrm{Adm}(\{\mu'\})_0$. Finally, $d_i \leq c_i$ for every $i \in \overline I$, so $v_{\mathbf d} \mathrel{\dot{\leq}} v_{\mathbf c}$, and the proof is
    complete.
\end{proof}

\section*{Acknowledgements}

The author is thankful to Chia-Fu Yu for helpful discussions on the EKOR stratification.

\bibliographystyle{amsplain}
\bibliography{bibliobis}

\end{document}